\documentclass[11pt]{article}
\usepackage{latexsym}
\usepackage{amsmath}
\usepackage{amssymb}
\usepackage{amsthm}
\usepackage{epsfig}
\usepackage{xcolor}
\usepackage{graphicx}
\usepackage{bm}
\usepackage{enumitem}

\usepackage[top=1in, bottom=1in, left=.7in, right=.7in]{geometry}

\renewcommand{\P}{\mathbb{P}}
\newcommand{\E}{\mathbb{E}}

\newcommand{\Z}{\mathbb{Z}}
\newcommand{\R}{\mathbb{R}}

\newcommand{\N}{\mathbb{N}}
\renewcommand{\S}{\mathbb{S}}

\newcommand{\eps}{\varepsilon} 

\def\id{{\mathbf I}}

\newcommand{\<}{\langle}
\renewcommand{\>}{\rangle}

\newcommand{\diag}{\text{diag}}
\newcommand{\tr}{\text{tr}}
\newcommand{\op}{{\rm op}}
\newcommand{\ones}{\bm{1}}
\newcommand{\what}{\widehat}

\def\sT{{\mathsf T}}
\def\bzero{{\boldsymbol 0}}

\DeclareMathOperator*{\argmin}{arg\,min}

\newtheorem{theorem}{Theorem}
\newtheorem*{theorem*}{Theorem}
\newtheorem{lemma}{Lemma}

\newtheorem{assumption}{Assumption}

\newtheorem{proposition}{Proposition}

\newtheorem{corollary}{Corollary}

\theoremstyle{definition}

\newtheoremstyle{myremark} 
    {\topsep}                    
    {\topsep}                    
    {\rm}                        
    {}                           
    {\bf}                        
    {.}                          
    {.5em}                       
    {}  

\theoremstyle{myremark}
\newtheorem{remark}{Remark}[section]

\usepackage[toc,page]{appendix}
\usepackage{graphicx}
\usepackage{bbm}

\DeclareSymbolFont{rsfs}{U}{rsfs}{m}{n}
\DeclareSymbolFontAlphabet{\mathscrsfs}{rsfs}

\def\bA{{\boldsymbol A}}
\def\bB{{\boldsymbol B}}
\def\bC{{\boldsymbol C}}
\def\bD{{\boldsymbol D}}
\def\bE{{\boldsymbol E}}

\def\bG{{\boldsymbol G}}

\def\bH{{\boldsymbol H}}

\def\bL{{\boldsymbol L}}
\def\bM{{\boldsymbol M}}
\def\bN{{\boldsymbol N}}

\def\bP{{\boldsymbol P}}
\def\bQ{{\boldsymbol Q}}
\def\bR{{\boldsymbol R}}
\def\bS{{\boldsymbol S}}
\def\bT{{\boldsymbol T}}
\def\bU{{\boldsymbol U}}
\def\bV{{\boldsymbol V}}
\def\bW{{\boldsymbol W}}
\def\bX{{\boldsymbol X}}

\def\bZ{{\boldsymbol Z}}

\def\ba{{\boldsymbol a}}
\def\bb{{\boldsymbol b}}
\def\be{{\boldsymbol e}}
\def\boldf{{\boldsymbol f}}

\def\bh{{\boldsymbol h}}

\def\bu{{\boldsymbol u}}
\def\bv{{\boldsymbol v}}

\def\bx{{\boldsymbol x}}
\def\by{{\boldsymbol y}}
\def\bz{{\boldsymbol z}}

\def\beps{{\boldsymbol \eps}}

\def\bpsi{{\boldsymbol \psi}}
\def\bphi{{\boldsymbol \phi}}
\def\btheta{{\boldsymbol \theta}}

\def\bDelta{{\boldsymbol \Delta}}
\def\bLambda{{\boldsymbol \Lambda}}
\def\bPsi{{\boldsymbol \Psi}}
\def\bPhi{{\boldsymbol \Phi}}
\def\bSigma{{\boldsymbol \Sigma}}
\def\bTheta{{\boldsymbol \Theta}}

\def\bfzero{{\boldsymbol 0}}
\def\bfone{{\boldsymbol 1}}

\def\hba{{\hat {\boldsymbol a}}}
\def\hf{{\hat f}}
\def\ha{{\hat a}}

\def\spn{{\rm span}}

\def\KR{{\rm KR}}
\def\bbHe{{\rm He}}

\def\de{{\rm d}}
\def\Trace{{\rm Tr}}

\def\Coeff{{\rm Coeff}}
\def\de{{\rm d}}
\def\Unif{{\rm Unif}}

\def\RF{{\rm RF}}

\def\kernel{\rm Ker}
\def\image{{\rm Im}}

\def\ddiag{{\rm ddiag}}
\def\KR{{\rm KR}}

\def\bbHe{{\rm He}}

\def\spn{{\rm span}}

\def\cV{{\mathcal V}}

\def\cL{{\mathcal L}}
\def\cF{{\mathcal F}}
\def\cE{{\mathcal E}}

\def\cV{{\mathcal V}}

\def\cH{{\mathcal H}}
\def\cA{{\mathcal A}}

\def\Unif{{\sf Unif}}
\def\normal{{\sf N}}
\def\proj{{\mathsf P}}

\def\RF{{\sf RF}}

\def\reals{{\mathbb R}}
\def\integers{{\mathbb Z}}
\def\naturals{{\mathbb N}}
\def\Top{{\mathbb T}}

\def\normal{{\sf N}}
\def\proj{{\mathsf P}}

\def\Unif{{\sf Unif}}
\def\normal{{\sf N}}
\def\Uop{{\mathbb U}}
\def\Hop{{\mathbb H}}

\def\proj{{\mathsf P}}

\def\RF{{\sf RF}}

\def\reals{{\mathbb R}}
\def\integers{{\mathbb Z}}
\def\naturals{{\mathbb N}}
\def\proj{{\mathsf P}}
\def\Hop{{\mathbb H}}
\def\Uop{{\mathbb U}}
\def\App{{\rm App}}

\def\stest{\mbox{\tiny\rm test}}

\def\seff{\mbox{\tiny\rm eff}}

\def\ocV{\overline{{\mathcal V}}}

\def\hba{{\hat {\boldsymbol a}}}
\def\hf{{\hat f}}
\def\hy{{\hat y}}
\def\hU{\widehat{U}}
\def\hUop{\widehat{\mathbb U}}

\def\cE{{\mathcal E}}
\def\cD{{\mathcal D}}
\def\cX{{\mathcal X}}
\def\cF{{\mathcal F}}

\def\He{{\rm He}}
\def\Trace{{\rm Tr}}

\def\Coeff{{\rm Coeff}}
\def\de{{\rm d}}
\def\Unif{{\rm Unif}}
\def\RF{{\rm RF}}

\def\stest{\mbox{\tiny\rm test}}

\def\cE{{\mathcal E}}

\def\normal{{\sf N}}

\def\bDelta{{\boldsymbol \Delta}}

\def\cX{{\mathcal X}}

\def\Coeff{{\rm Coeff}}
\def\RF{{\rm RF}}

\def\bA{{\boldsymbol A}}
\def\btheta{{\boldsymbol \theta}}
\def\bTheta{{\boldsymbol \Theta}}
\def\bLambda{{\boldsymbol \Lambda}}

\def\Tr{{\rm Tr}}

\def\cV{{\mathcal V}}
\def\bP{{\boldsymbol P}}
\def\diag{{\rm diag}}
\def\bS{{\boldsymbol S}}

\def\bD{{\boldsymbol D}}
\def\bPsi{{\boldsymbol \Psi}}

\def\bL{{\boldsymbol L}}

\def\tbu{\Tilde \bu}
\def\tbZ{\Tilde \bZ}
\def\tbphi{\Tilde \bphi}
\def\tbpsi{\Tilde \bpsi}
\def\tbD{\Tilde \bD}
\def\tbf{\Tilde \boldf}
\def\hbU{\hat{{\boldsymbol U}}_\lambda }
\def\hbUi{\hat{{\boldsymbol U}}_\lambda^{-1} }
\def\bb{{\boldsymbol b}}
\def\bsigma{{\boldsymbol \sigma}}

\def\hf{\hat f}
\def\hbf{\hat \boldf}
\def\bR{{\boldsymbol R}}
\def\bpsi{{\boldsymbol \psi}}
\def\cuH{\mathscrsfs{H}}

\def\noise{\sigma_{\varepsilon}}

\def\evn{{\mathsf m}}
\def\evN{{\mathsf M}}

\def\lvn{{\mathsf s}}
\def\lvN{{\mathsf S}}

\def\bc{{\boldsymbol c}}
\def\bC{{\boldsymbol C}}
\def\oba{\overline{{\boldsymbol a}}}
\def\uba{\underline{{\boldsymbol a}}}

\def\barsigma{\bar{\sigma}}

\def\tbN{\Tilde \bN}

\def\Cube{{\mathscrsfs Q}^d}

\def\bzeta{{\boldsymbol \zeta}}
\def\hbzeta{{\hat {\boldsymbol \zeta}}}
\def\oproj{{\overline \proj}}

\usepackage{hyperref}
\hypersetup{
    colorlinks,
    linkcolor={blue!80!black},
    citecolor={green!50!black},
}
\colorlet{linkequation}{blue}

\begin{document}
\title{Generalization error of random features and kernel methods: hypercontractivity and kernel matrix concentration}

\author{Song Mei\thanks{Department of Statistics, University of California, Berkeley},\;\; Theodor Misiakiewicz\thanks{Department of
    Statistics, Stanford University}, \;\;
  Andrea Montanari\footnotemark[2] \thanks{Department of Electrical Engineering,
    Stanford University} }

\maketitle

\begin{abstract}
  Consider the classical supervised learning problem: we are given data $(y_i,\bx_i)$, $i\le n$,
  with $y_i$ a response and $\bx_i\in\cX$ a  covariates vector, and try to learn a model $f:\cX\to\reals$
  to predict future responses. Random features methods map the covariates vector $\bx_i$ to a point
  $\bphi(\bx_i)$ in a higher dimensional space $\reals^N$, via a random featurization map $\bphi$.
  We study the use of  random features methods in conjunction with ridge regression in the feature space $\reals^N$.
  This can be viewed as a finite-dimensional approximation of kernel ridge regression (KRR), or as a stylized model
  for neural networks in the so called lazy training regime.

  We define a class of  problems satisfying certain spectral conditions on the underlying kernels,
  and a hypercontractivity assumption on the associated eigenfunctions. These conditions are verified by
  classical high-dimensional examples. 
  Under these conditions,  we prove a sharp characterization of the
  error of random features ridge regression.
  In particular, we address two fundamental questions: $(1)$~What is the generalization error of KRR?
  $(2)$~How big $N$ should be for the random features approximation  to achieve the same error as KRR?

  In this setting, we prove that KRR is well approximated by a projection onto the top $\ell$ eigenfunctions of the kernel,
  where $\ell$ depends on the sample size $n$. 
  We show that the  test error of random features ridge regression is dominated by its approximation error  and is larger than the error of KRR
  as long as $N\le n^{1-\delta}$ for some $\delta>0$. We characterize this gap. 
   For $N\ge n^{1+\delta}$,  random features achieve the same error as the corresponding KRR,
  and further increasing $N$ does not lead to a significant change in test error.
\end{abstract}

\tableofcontents

\section{Introduction}

\subsection{Background}
\label{sec:IntroBackground}

Consider the supervised learning problem in which we are given i.i.d. samples $(y_i,\bx_i)$, $i\le n$,
from a common probability distribution on $\reals\times \cX$. Here $\bx_i\in \cX$ is a vector of covariates, and
$y_i$ is a response variable. We are interested in learning a model $\hf:\cX\to \reals$ which, given
a new point $\bx_{\stest}$, predicts the corresponding response $y_{\stest}$ via $\hf(\bx_{\stest})$.

A number of statistical learning methods can be viewed as a combination of two steps: featurization and training.
Featurization maps sample points into a convenient `feature space' $\cH$ (a vector space)
via a featurization map $\bphi:\cX\to \cH$, $\bx_i\mapsto \bphi(\bx_i)$. Training fits a model that is linear in the feature space:
$\hf(\bx) = \<\ba,\bphi(\bx)\>_{\cH}$.
In this paper we will be concerned with a relatively simple method for training, ridge regression:
\begin{align}
  \hba(\lambda) :=
  \arg\min_{\ba}\Big\{\sum_{i=1}^n\big(y_i-\<\ba,\bphi(\bx_i)\>_{\cH}\big)^2+\lambda\|\ba\|_{\cH}^{2}\Big\}\,
  .\label{eq:FirstRidge}
\end{align}
Here it is implicitly assumed that $\cH$ is an Hilbert space, and therefore $\ba\in \cH$ and $\<\,\cdot\,,\,\cdot\,\>_{\cH}$,
$\|\,\cdot\,\|_{\cH}$ are the scalar product and norm in $\cH$.

It is useful to discuss a few examples of this paradigm, some of which will play a role in what follows (we refer
to Section \ref{sec:RFSetting}  for formal definitions).

\vspace{0.2cm}

\noindent\emph{Feature engineering.} We use this term to refer to the classical approach of crafting a
set of $N$ features $\bphi(\bx) = (\phi_1(\bx),\dots,\phi_N(\bx))\in\cH = \reals^N$ for a specific application,
by leveraging domain expertise. This has been the standard approach to computer vision for a long time
\cite{lowe2004distinctive,bay2008speeded},
and is still the state of the art in most of applied statistics \cite{Hastie}.

\vspace{0.2cm}

\noindent\emph{Kernel methods.} In this case $\cH$ is a reproducing kernel Hilbert space (RKHS) defined
implicitly via a positive definite kernel $H:\cX\times \cX\to \reals$ \cite{berlinet2011reproducing}.
Rather than manually constructing features, the
statistician/data analyst only needs to encode in $H(\bx_1,\bx_2)=\<\bphi(\bx_1), \bphi(\bx_2)\>_{\cH}$
a suitable notion of similarity in the input space $\cX$.
The resulting model only depends on the kernel $H$, and a crucial role is played by its eigenvalue decomposition
$H(\bx_1,\bx_2) = \sum_{\ell=1}^{\infty}\lambda_{\ell}^2\psi_{\ell}(\bx_1)\psi_{\ell}(\bx_2)$. 
Ridge regression  with RKHS featurization is referred to as kernel ridge regression (KRR).
Formally, the KRR estimator takes the form: 
\begin{align}
  \hf_{\lambda}(\bx) &= \sum_{\ell=1}^{\infty}\hf_{\lambda,\ell}\psi_{\ell}(\bx)\, ,\;\;\;\;
                       \hf_{\lambda,\ell}= \sum_{\ell'=1}^{\infty}(( \lambda /n) \cdot  \id+\bG)^{-1}_{\ell,\ell'}\lambda_{\ell}\lambda_{\ell'}\<\psi_{\ell'},y\>_n\, ,
 \label{eq:FirstKRR}\\
  G_{\ell,\ell'}&:= \lambda_{\ell}\lambda_{\ell'}\<\psi_{\ell},\psi_{\ell'}\>_n\, .
\end{align}
Here $\<f,g\>_n:= n^{-1}\sum_{i=1}^nf(\bx_i)g(\bx_i)$ denotes the scalar product with respect to the empirical measure. 

For large $n$, we can imagine to replace the empirical scalar product with the population one, and therefore
$G_{\ell,\ell'}\approx \lambda_{\ell}^2\bfone_{\ell=\ell'}$, whence
$\hf_{\lambda,\ell}\approx  (( \lambda /n) +\lambda_{\ell}^2)^{-1}\lambda^2_{\ell}\<\psi_{\ell},y\>_n $.
In words, KRR attempts to estimate accurately the projection of $f(\bx) = \E[y|\bx]$ onto the
eigenvectors of the kernel $H$, corresponding to large eigenvalues $\lambda_{\ell}$.
On the other hand, it shrinks towards $0$ the projections of $f$  onto eigenvectors corresponding to smaller eigenvalues.

\vspace{0.2cm}

\noindent\emph{Random Features (RF).} Instead of constructing the featurization map $\bphi$
on the basis of domain expertise, or, implicitly, via a kernel,
RF methods  use a random map $\bphi:\cX\to \reals^N$.
We will study a general construction that generalizes the original proposal of \cite{rahimi2008random,balcan2006kernels}.
We sample $N$ point in a space $\Omega$ via $\btheta_1$,\dots $\btheta_N\sim _{iid}\tau$
(for a certain probability measure $\tau$ on $\Omega$), and then define the mapping $\bphi$ by letting
$\bphi(\bx) = (\sigma(\bx;\btheta_1),\dots,\sigma(\bx;\btheta_N))$.
Here $\sigma: \cX\times \Omega\to\reals$ is a square integrable function. We endow the feature space $\cH_N = \R^N$  with
the inner product $\< \ba_1, \ba_2\>_{\cH_N} = \ba_1^\sT \ba_2 / N$. 

Because of the connection to two-layers neural networks (see below) we shall refer to $N$
as the `number of neurons'  (although,  `number of parameters' would be more appropriate),
and to $\sigma$ as the  `activation
function.'  The resulting function $\hf$ takes the form
\begin{align}
  \hf(\bx;\ba) := \<\ba,\bphi(\bx) \>_{\cH}= \frac{1}{N}\sum_{i=1}^N a_i\sigma(\bx;\btheta_i)\, .
\end{align}

We will refer to the procedure defined by Eq.~\eqref{eq:FirstRidge}  with $\bphi$ the random feature map defined here
as `random features ridge regression' (RFRR).
RFRR is closely related to  KRR. First of all,  we can view RFRR as an example
of KRR, with kernel 
\[
H_N(\bx_1,\bx_2) = \< \bphi(\bx_1), \bphi(\bx_2) \>_{\cH} = \frac{1}{N}\sum_{i=1}^N \sigma(\bx_1;\btheta_i)\sigma(\bx_2;\btheta_i).
\]
Notice however that the kernel $H_N$ has finite rank and is random, because of the random features
$\btheta_1,\dots,\btheta_N$.

Second, for large $N$, we can expect $H_N$ to be a good approximation of its expectation
\begin{align}
\E H_N(\bx_1,\bx_2) = H(\bx_1,\bx_2):=\int_{\Omega} \!\sigma(\bx_1;\btheta)\,\sigma(\bx_2;\btheta)\, \tau(\de\btheta)\, .\label{eq:ExpectedKernel}
\end{align}
Hence, for large $N$, we expect RFRR to have similar generalization properties as the underlying RKHS, while
possibly exhibiting lower complexity because it only  operates on $N\times n$
matrices (instead of $n\times n$ matrices as for KRR).

\vspace{0.2cm}

\noindent\emph{Neural networks in the linear (lazy) regime.} The methods described above fit the 
general paradigm of Eq.~\eqref{eq:FirstRidge}. Training does not affect the feature map $\bphi$. The
model $\hf_{\lambda}(\,\cdot\,)$  is linear in  $\by$, as a consequence of the fact that the loss is quadratic
(see also Eq.~\eqref{eq:FirstKRR}).
In contrast, neural networks aim at learning  the best feature representation of the data. The feature
map changes during training, and indeed there is no clear separation between the feature map
$\bphi(\bx)$ and the coefficients $\ba$.

Nevertheless a copious line of recent research shows that ---under certain training schemes---
neural networks are well approximated by their linearization around a random initialization \cite{jacot2018neural, li2018learning, du2018gradient, du2018gradientb, allen2018convergence, allen2018learning, arora2019fine, zou2018stochastic, oymak2019towards}.
It is useful to recall the basic argument here.
Denote by $\bx\mapsto f(\bx;\btheta)$ the neural network, with parameters (weights) $\btheta\in\reals^N$,
and by $\btheta_0$ the initialization for gradient-based training. For highly overparametrized networks,
a small change in the parameters $\btheta$ is sufficient to change significantly the
evaluations of $f$ at the data points, i.e., the vector $(f(\bx_1;\btheta),\dots,f(\bx_n;\btheta))$.
As a consequence, an empirical risk minimizer can be found in a small neighborhood of the initialization
$\btheta_0$, and it is legitimate to approximate $f$ by its first order Taylor expansion in the parameters:
\begin{align}
  f(\bx;\btheta_0+\ba) \approx  f(\bx;\btheta_0) + \<\ba,\nabla_{\btheta}f(\bx;\btheta_0)\>\, .
\end{align}
Apart from the zero-th order term $f(\bx;\btheta_0)$   (which has no free parameters, and hence plays the role of an
offset), this linearized model takes the same form $\hf(\bx) = \<\ba,\bphi(\bx)\>$. The 
featurization map is given by $\bphi(\bx) = \nabla_{\btheta} f(\bx;\btheta_0)$.
We refer to the model $\bx\mapsto  \<\ba,\nabla_{\btheta}f(\bx;\btheta_0)\>$ as the neural tangent (NT)
model.

Notice that the NT featurization map is random, because of the random initialization $\btheta_0$.
However, in general it does not take the form of the RF model, because the entries
of   $\nabla_{\btheta} f(\bx;\btheta_0)$ are not independent. 
Despite this important difference, we expect key properties of the RF model to
generalize to suitable classes of NT models. Examples of this phenomenon were studied recently
in \cite{ghorbani2019linearized,montanari2020interpolation}.

\vspace{0.2cm}
 
The present paper focuses on KRR and RFRR. We introduce a set of assumptions on the data distribution,
the choice of activation function, and the probability distribution
$\tau$ on the $\btheta_i$'s,
under which we can characterize the large $n$, $N$ behavior of the generalization (test) error.
While our results apply to an abstract input space $\cX$, our assumptions aim at capturing
the behavior observed when $\cX$ is high-dimensional, and the distribution $\nu$ on $\cX$ satisfies strong
concentration properties.
For instance, our results apply to $\cX= \S^{d-1}$ (the sphere in $d$ dimension) or
$\cX = \{+1,-1\}^d$, both endowed with the uniform measure.

Our results do not require the true regression function $f$ to belong to the associate RKHS and
they characterize the test error (with respect to the square loss) pointwise, i.e., for
any function $f$.
This characterization holds up to error terms that are negligible compared to the null risk $\E\{f(\bx)^2\}$.

In particular our results allow to answer in a quantitative way two sets of key questions that emerge
from the above discussion:
\begin{itemize}
\item[{\sf Q1.}] How does the test error of KRR depends on the sample size $n$, on the target
  function $f$, and on the kernel $H$?
  While this question has attracted considerable attention in the past (see Section \ref{sec:Related}
  for an overview), a very precise answer can be given in the present setting.
\item[{\sf Q2.}] How does the test error of RFRR depend on the sample size $n$, and the number of neurons $N$?
  In particular, for a given sample size, how big $N$ should be to achieve the same error as for the
  associated KRR (which corresponds formally to $N=\infty$)?
\item[{\sf Q3.}] How do the answers to the previous questions depend on the regularization parameter $\lambda$?
  In particular, in which cases the optimal test error is achieved by choosing $\lambda\to 0$, i.e. by
  using the minimum norm interpolator to the training data?
\end{itemize}
Let us emphasize that the second question is technically more challenging than the first one, because
it amounts to studying KRR \emph{with a random kernel}. The setting introduced here is particularly motivated
by the objective to address {\sf Q2} (and its ramifications in {\sf Q3}). Indeed, to the best of our knowledge, we provide the first set
of results on the optimal choice of the overparametrization $N/n$ under polynomial scalings of $N,n,d$.

\subsection{Summary of main results}

Before summarizing our results, it is useful to describe informally our assumptions: we refer to
Sections \ref{sec:RF_Assumptions} and \ref{sec:assumptions_Kernel} for a formal statement of the same assumptions.
We consider $(\bx_i)_{i\le n}\sim_{iid}\nu$ with $\nu$ a probability distribution of the covariates
space $\cX$, and $y_i = f(\bx_i)+\eps_i$, where $f$ is the target function and
$\eps_i \sim \normal(0,\noise^2)$ independent of $\bx_i $ is noise.

An RFRR problem is specified by $\nu,f,\noise$ (which determine the data distribution),
 $\sigma,\tau$ (which determine the RF representation), and the parameters $n,N$ (sample size and number of neurons).
The associated kernel problem is obtained by using the kernel \eqref{eq:ExpectedKernel}.
It is also useful to introduce a kernel in the $\btheta$ space via
$U(\btheta_1,\btheta_2):=\E_{\bx\sim \nu}\{\sigma(\bx;\btheta_1)\,\sigma(\bx;\btheta_2)\}$.

We will consider sequences of such problems indexed by an integer $d$, and characterize
their behavior as $N,n,d\to\infty$. In applications, $d$ typically corresponds to the dimension of the
covariates space $\cX$. In this informal summary, we drop any reference to $d$ for simplicity.

We next describe informally our key assumptions, which depends on integers $(\evn,\evN,u)$,
with $u \ge \max( \evN , \evn)$.
(For the sake of simplicity, we omit some assumptions of a more technical nature.) 
\begin{enumerate}
\item \emph{Hypercontractivity.} The top $u$ eigenvectors of $H$ are `delocalized'.
  We formalize this condition by requiring that, for any function $g\in \spn(\psi_j:j\le u)$, and for any integer
  $k$, $\E_{\nu}\{g(\bx)^{2k}\}\le C_{k,u}\E_{\nu}\{g(\bx)^{2}\}^k$.
  We assume a same condition for the eigenvectors of $U$.
\item \emph{Concentration of diagonal elements of the kernels.} Denote by $H_{>\evn}$ the kernel obtained from $H$ by setting to zero the eigenvalues $\lambda_1,\dots, \lambda_\evn$. We require that the diagonal elements
  $\{H_{>\evn}(\bx_i,\bx_i)\}_{i\le n}$ concentrate around their expectation with respect to the measure $\nu$ on $\cX$.
  Analogously, we require the diagonal elements $\{U_{>\evN}(\btheta_i,\btheta_i)\}_{i\le N}$ to concentrate around their expectation.

  This assumption amounts to a condition of symmetry: most points $\bx$ in the support of $\nu$
  are roughly equivalent, in the sense of having the same value of  $H_{>\evn}(\bx,\bx)$, and similarly for most $\btheta$
  in the support of $\nu$.
\item \emph{Spectral gap.} Recall that $(\lambda_{j}^2)_{j\ge 1}$ denote the eigenvalues of the kernel $H$
  in decreasing order. We then assume one of the following two conditions to hold:
  \begin{itemize}
  \item[] \emph{Undeparametrized regime.} We have $N\ll n$ and 
    \begin{align}
 \frac{1}{\lambda^2_\evN}\sum_{k=\evN+1}^{\infty}\lambda_k^2\ll N \ll   \frac{1}{\lambda^2_{\evN+1}}\sum_{k=\evN+1}^{\infty}\lambda_k^2\, ,\label{eq:EllDef}
      \end{align}
    \item[] \emph{Overparametrized regime.} We have $n\ll N$ and 
  \begin{align}
    \frac{1}{\lambda^2_\evn}\sum_{k=\evn+1}^{\infty}\lambda_k^2\ll n \ll
    \frac{1}{\lambda^2_{\evn+1}}\sum_{k=\evn+1}^{\infty}\lambda_k^2\, .\label{eq:QDef}
  \end{align}
\end{itemize}
This assumption is ensures a clear separation between the subspace of $\cD_d$ which is estimated  accurately
(spanned by the eigenfunction of $H$ corresponding to the top eigenvalues) and the subspace that is estimated trivially by $0$
(corresponding to the low eigenvalues of $H$.) As we will see,  a spectral gap condition holds for classical high-dimensional examples. 
On the other hand, we believe it should be possible to avoid this condition at the price of a more
complicate characterization of the risk, and indeed we do not require it for KRR. 
\end{enumerate}

As explained above, KRR attempts to estimate accurately the projection of the target function
$f_*$ onto the top eigenvectors of the kernel $H$, and shrinks to zero its other components.
RFRR behaves similarly, except that it only constructs a finite rank approximation of the kernel $H$.
How many components of the target function are estimated accurately? There are of course  two limiting
factors: the statistical error which depends on the sample size $n$, and the approximation error which depends
on the number of neurons $N$.

It turns out that, in the present setting, the interplay between $n$ and $N$ takes a particularly simple form.
In a nutshell what matters is the smaller of $n$ and $N$. If $n\ll N$, then the statistical error dominates
and ridge regression estimates correctly the projection of $f_*$ onto the top $\evn$ eigenfunctions of $H$
(where $\evn$ is defined per Eq.~\eqref{eq:QDef}). If on the other hand $N\ll n$, then the approximation error dominates
and ridge regression estimates correctly the projection of $f_*$ onto the top $\evN$ eigenfunctions of $H$
(where $\evN$ is defined per Eq.~\eqref{eq:EllDef}).

In formulas, we denote by $R_{\RF}(f_*;\lambda) = \E\{(f_*(\bx)-f_{\lambda}(\bx))^2\}$ the test error of
RFRR (for square loss) when the target function is $f_*$ and the regularization parameter equals $\lambda$.
Our main result establishes that for all $\lambda\in [0,\lambda_*]$ (with a suitable choice of $\lambda_*$), in a
certain asymptotic sense, the following hold:
\begin{align}
  R_{\RF}(f_*;\lambda) =\begin{cases}
    \E\{ (\proj_{>\evn}f_*(\bx))^2\} +o(1)\cdot \E\{f_*(\bx)^2\} & \mbox{ if $n\ll N$,}\\
    \E\{ (\proj_{>\evN}f_*(\bx))^2\} +o(1)\cdot \E\{f_*(\bx)^2\} & \mbox{ if $n\gg N$,}
  \end{cases}
                                                          \label{eq:BasicInformal}
  \end{align}
  where $\proj_{>\ell}$ is the projector onto the span of the eigenfunctions $\{\psi_j:\; j> \ell\}$.
  This statement also applies to KRR, if we interpret the latter as the $N=\infty$ case of RFRR.
  Further, no kernel machine achieves a smaller error.
  
  This characterization implies a relatively simple answer to questions {\sf Q1}, {\sf Q2}, and {\sf Q3},
  which we posed in the previous section. We summarize some of the insights that follow from this result.
  \begin{description}
    \item[KRR acts as a projection.] As mentioned above, Eq.~\eqref{eq:BasicInformal} can be restated as saying 
      that (for the special case $N=\infty$), $\hf_{\lambda}(\bx) \approx \proj_{\le \evn}f_*(\bx)$. Indeed, we will prove a
      stronger result, which does not require the spectral gap assumption of Eq.~\eqref{eq:QDef}. The KRR estimator
      $\hf_{\lambda}$ is well approximated by the KRR estimator for the population problem ($n=\infty$), but with a larger value
      of the ridge regularization $\gamma>\lambda$. In other words KRR acts as a shrinkage operator
      along the eigenfunctions of the kernel.
    \item[Effects of overparametrization.] In random features models, we are free to choose the number of
      neurons $N$. Equation \eqref{eq:BasicInformal} indicates that any choice of $N$ has roughly the same
      test error (which is also the test error of KRR) as long as $N\gg n$. This is interesting in both directions.
      First, the test error does not deteriorate as the number of parameters increases, and becomes much larger than the
      sample size. This contrasts with a naive measure of the model complexity: indeed, counting the number of parameters would naively suggest that $N\gg n$ might hurt generalization.
      Second, the error does not improve with overparametrization either, as long as $N\gg n$.
    \item[Optimal overparametrization.] At what level of overparametrization should we operate? In view of the previous
      point, it is sufficient to use a model with a number of parameters much larger than the sample size (formally,
      $N\ge n^{1+\delta}$ for some $\delta>0$, although this specific condition is mainly dictated by our proof technique).
      Further overparametrization does not improve the statistical behavior.

      Let us also note that ---as proven in \cite{mei2019generalization}--- choosing $N/n =: \psi =O(1)$
      can lead to sub-optimal test error, with the suboptimality vanishing if $\psi\to\infty$ \emph{after}, $N,n\to\infty$.
    \item[Optimality of interpolation.] Finally, the above phenomena are obtained for all $\lambda\in [0,\lambda_*]$.
      The case $\lambda=0$ corresponds to minimum norm interpolators. We also prove that the risk of any kernel
      machine is lower bounded   by $\E\{ (\proj_{>\evn}f_*(\bx))^2\} +o(1)\cdot \E\{f_*(\bx)^2\}$. We therefore conclude that,
      in the overparametrized regime $N\gg n$, min-norm interpolators are optimal among all kernel methods. 
    \end{description}
    
    \subsection{Related literature}
 \label{sec:Related}

    The test error of KRR was studied by a number of authors in the past \cite{caponnetto2007optimal, jacot2020kernel}, \cite[Theorem 13.17]{wainwright2019high}. In particular, \cite{caponnetto2007optimal}
    establishes that KRR achieves minimax optimal rates over certain subclasses of the associated RKHS. However
    these results require a strictly positive ridge regularizer (and hence do not cover interpolation) and characterize the
    decay of the error as $n\to \infty$ in fixed dimension $d$. In contrast our focus is on the case in which both
    $d$ and $n$ grow simultaneously. Further, we provide upper and lower bounds that hold pointwise (for a given target function $f_*$) while earlier work mostly establish pointwise upper bound and minimax lower bounds (for the worst case $f_*$). The recent work \cite{jacot2020kernel} also derived pointwise upper and lower bounds for kernel ridge regression (but with strictly positive ridge regularizer), which is very similar to our Theorem \ref{thm:upper_bound_KRR}. However, these results are based on a universality assumption whose validity is unclear in specific settings.

    Recently, the ridge-less (interpolation) limit of KRR was studied by Liang, Rakhlin and Zhai
    \cite{liang2020just,liang2019risk}. Again, these authors provide minimax upper bounds that hold within the RKHS,
    holding for inner product kernel, when the feature vectors $\bx$ have independent coordinates. 
    Their results are related but not directly comparable to ours.
    
    The complexity of training kernel machine scales at least quadratically in the sample size. This has motivated the
    development of randomized techniques to lower the complexity of training and testing. While our focus is on random features
    methods, alternative approaches are  based on subsampling the columns-rows of the empirical
    kernel matrix, see e.g.
    \cite{bach2013sharp,alaoui2015fast,rudi2015less}. In particular,
    \cite{rudi2015less} compares the prediction errors using the
    sketched and the full kernel matrices, and shows that ---for a fixed RKHS---
    it is sufficient to use a number of rows/columns of the order of the square root of the sample size in order to achieve
    the minimax rate over that RKHS.

    The generalization properties of random features methods have been studied in a smaller number of papers \cite{rahimi2009weighted, rudi2017generalization, ma2020towards}. 
    Rahimi and Recht \cite{rahimi2009weighted} proved an upper bound of the order $1/\sqrt{N}+1/\sqrt{n}$ on
    the generalization error. The insight provided by this bound is similar to one of our points:
    about $N\asymp n$ neurons are sufficient for the error to be of the same order as for $N\to\infty$. On the other
    hand, \cite{rahimi2009weighted} proves  only a minimax upper bound, it is limited to Lipschitz losses,
    and, crucially, requires the coefficients $\max_{i\le N}|a_i|\le C$ so that $\| \ba \|_2^2 = O(N)$. In contrast, in the present setting,
    we typically have $\| \ba \|_2^2 = \Theta(n N)$ 
    The case of square loss was considered earlier by Rudi and Rosasco \cite{rudi2017generalization} who proved that,
    for a target function $f_*$ in the RKHS, $N= C\sqrt{n} \log n$ is sufficient to learn a random features
    model with test error of order $1/\sqrt{n}$. These authors interpret this finding as implying that roughly
    $\sqrt{n}$ random features are sufficient: we will discuss the difference between their setting and ours in Section \ref{sec:RF_Theorem}.
   
    Finally, \cite{bach2015equivalence} studies optimized distributions for sampling the random features,  while
    \cite{yang2012nystrom} provides a comparison between random features approaches and
    subsampling of the kernel matrix.

    As pointed out above, we find that taking $\lambda \to 0$ yields nearly optimal test error, within our setting.
    Optimality of minimum norm interpolators has attracted considerable attention recently \cite{belkin2019reconciling,belkin2019does,hastie2019surprises,bartlett2020benign,tsigler2020benign}.
    In particular our results point in the same direction as the general analysis of ridge regression in \cite{bartlett2020benign,tsigler2020benign}.
    Note however that the general results of \cite{bartlett2020benign,tsigler2020benign} do not apply to the present setting because they require subgaussian features $\bphi(\bx_i)$.
    Further, they only provide upper and lower bounds that match up to factors depending on the condition number of
    a certain random matrix.  In contrast, our characterization is
    specialized to the random features setting, does not require subgaussianity, and holds up to additive errors that are negligible compared to the null risk.
    
    The present paper solves a number of open problems that were left open in our earlier work \cite{ghorbani2019linearized}.
    First of all, \cite{ghorbani2019linearized} only considered the cases $n=\infty$ (approximation error of random
    features models) or $N=\infty$ (generalization error of KRR). Here instead we establish the complete picture for
    both $n$ and $N$ finite. Second,  \cite{ghorbani2019linearized} assumed a special data distribution
    ($\nu$ was the uniform distribution over the $d$-dimensional sphere), a special structure for the kernel
    (inner product kernels), and a special type of activation functions (depending on the inner product $\<\btheta,\bx\>)$).
    The present paper considers  general data distribution,
    kernel, and activation functions, under a set of assumptions that covers the previous example as a special case.
    Finally, the proofs of \cite{ghorbani2019linearized}  made use of the moment method, which is difficult to
    generalize beyond special examples. Here we use a decoupling approach and matrix concentration methods which
    are significantly more flexible.

    The results of \cite{ghorbani2019linearized} were generalized to certain anisotropic distributions in \cite{ghorbani2020neural}. For the inner product activation functions on the sphere, the precise asymptotics (for  $N, n, d\to\infty$ with $N/d\to\psi_1$, $n/d\to\psi_2$,
    $\psi_1,\psi_2\in (0,\infty)$) of generalization error of random features models was calculated in \cite{mei2019generalization}.

   \subsection{Notations}

For a positive integer, we denote by $[n]$ the set $\{1 ,2 , \ldots , n \}$. For vectors $\bu,\bv \in \R^d$, we denote $\< \bu, \bv \> = u_1 v_1 + \ldots + u_d v_d$ their scalar product, and $\| \bu \|_2 = \< \bu , \bu\>^{1/2}$ the $\ell_2$ norm. Given a matrix $\bA \in \R^{n \times m}$, we denote $\| \bA \|_{\op} = \max_{\| \bu \|_2 = 1} \| \bA \bu \|_2$ its operator norm and by $\| \bA \|_{F} = \big( \sum_{i,j} A_{ij}^2 \big)^{1/2}$ its Frobenius norm. If $\bA \in \R^{n \times n}$ is a square matrix, the trace of $\bA$ is denoted by $\Tr (\bA) = \sum_{i \in [n]} A_{ii}$.

We use $O_d(\, \cdot \, )$  (resp. $o_d (\, \cdot \,)$) for the standard big-O (resp. little-o) relations, where the subscript $d$ emphasizes the asymptotic variable. Furthermore, we write $f = \Omega_d (g)$ if $g(d) = O_d (f(d) )$, and $f = \omega_d (g )$ if $g (d) = o_d (f (d))$. Finally, $f =\Theta_d (g)$ if we have both $f = O_d (g)$ and $f = \Omega_d (g)$.

We use $O_{d,\P} (\, \cdot \,)$ (resp. $o_{d,\P} (\, \cdot \,)$) the big-O (resp. little-o) in probability relations. Namely, for $h_1(d)$ and $h_2 (d)$ two sequences of random variables, $h_1 (d) = O_{d,\P} ( h_2(d) )$ if for any $\eps > 0$, there exists $C_\eps > 0 $ and $d_\eps \in \Z_{>0}$, such that
\[
\begin{aligned}
\P ( |h_1 (d) / h_2 (d) | > C_{\eps}  ) \le \eps, \qquad \forall d \ge d_{\eps},
\end{aligned}
\]
and respectively: $h_1 (d) = o_{d,\P} ( h_2(d) )$, if $h_1 (d) / h_2 (d)$ converges to $0$ in probability.  Similarly, we will denote $h_1 (d) = \Omega_{d,\P} (h_2 (d))$ if $h_2 (d) = O_{d,\P} (h_1 (d))$, and $h_1 (d) = \omega_{d,\P} (h_2 (d))$ if $h_2 (d) = o_{d,\P} (h_1 (d))$. Finally, $h_1(d) =\Theta_{d,\P} (h_2(d))$ if we have both $h_1(d) =O_{d,\P} (h_2(d))$ and $h_1(d) =\Omega_{d,\P} (h_2(d))$.

\section{Generalization error of random features ridge regression}

In this section, we present our results on the generalization error of
random features models. We begin in Section \ref{sec:RFSetting} by introducing the general abstract setting
in which we work, and some of its basic properties. We then state our assumptions in Section \ref{sec:RF_Assumptions},
and state our main theorem (Theorem \ref{thm:RFK_generalization}) in Section \ref{sec:RF_Theorem}.

Finally, Section \ref{sec:RF_Examples}
presents applications of our general theorem to $(i)$~the case of feature vectors uniformly distributed over the sphere
$\bx_i\sim\Unif(\S^{d-1}(\sqrt{d}))$, and $(ii)$~the case of feature vectors uniformly distributed over the Hamming cube
$\bx_i\sim\Unif(\{+1,-1\}^d)$. While these applications are `simple' in the sense that checking the assumptions of our general
theorem is straightforward, they are in themselves quite interesting. In particular, our result for the uniform distribution
on the sphere (cf. Proposition \ref{prop:RFK_sphere}) closes the main problem left unsolved in \cite{ghorbani2019linearized}.

\subsection{Random features models, kernels, and their spectral decomposition}
\label{sec:RFSetting}

We consider two sequences of Polish probability spaces $(\cX_d, \nu_d)$ and $(\Omega_d, \tau_d)$, indexed by an integer $d$. We denote by $L^2(\cX_d) = L^2(\cX_d, \nu_d)$ the space of square integrable
functions on $(\cX_d, \nu_d)$, and by $L^2(\Omega_d) = L^2(\Omega_d, \tau_d)$ the space of square integrable functions on $(\Omega_d, \tau_d)$. Since $(\cX_d, \nu_d)$ and $(\Omega_d, \tau_d)$ are standard probability spaces
\cite[Theorem 13.1.1]{dudley2018real}, it follows that  $L^2(\cX_d)$ and $L^2(\Omega_d)$ are separable.

More generally for $p \geq 1$, we denote $\| f \|_{L^p (\cX)} = \E_{\bx \sim \nu} [ |f(\bx)|^p ]^{1/p} $ the $L^p$ norm of $f$. We will sometimes omit $\cX$ and write directly $\| f \|_{L^2}$ and $\| f \|_{L^p}$ when clear from context. 

Given two closed linear subspaces $\cD_d \subseteq L^2(\cX_d)$, $\cV_d \subseteq L^2(\Omega_d)$, and the
activation function
$\sigma_d \in L^2( \cX_d \times \Omega_d, \nu_d\otimes \tau_d)$, we define a Fredholm integral operator
$\Top_d: \cD_d \to \cV_d $ via
\begin{align}
\Top_d g (\btheta) \equiv \int_{\cX_d} \sigma_d(\bx, \btheta) g(\bx) \nu_d(\de \bx). 
\end{align}
Note that $\Top_d$ is a compact operator by construction.
We will assume that $\Top_dg\neq 0$ for any $g\in\cD_d\setminus\{0\}$.  Also,
without loss of generality, we can assume $\cV_d = \image(\Top_d)$ (which is closed since $\Top_d$ is bounded).
With an abuse of notation, we will sometimes denote
by  $\Top_d$ the extension of this operator obtained by setting $\Top_dg=0$ for $g\in \cD_d^\perp$.
Notice that we can choose the kernel $\sigma_d$ so that $\int_{\cX_d} \sigma_d(\bx, \btheta) g(\bx) \nu_d(\de \bx)=0$ for any $g \in \cD_d^\perp$: we will assume such a choice hereafter.

While in simple examples we might assume $\cD_d= L^2(\cX_d)$, the extra flexibility
afforded by a general subspace $\cD_d\subseteq L^2(\cX_d)$ allows to model some important applications
\cite{mei2021invariant}.

The adjoint operator $\Top_d^*: \cV_d \to \cD_d$ has kernel representation 
\[
\Top_d^* f (\bx) = \int_{\Omega_d} \sigma_d(\bx, \btheta) f(\btheta) \tau_d(\de \btheta). 
\]
As before, we will sometimes extend $\Top^*_d$ to $\cL^2(\Omega_d)$ by setting $\kernel(\Top^*_d)=\cV_d^{\perp}$.

The operator $\Top_d$ induces two compact self-adjoint positive definite operators:
$\Uop_d =  \Top_d \Top_d^*: \cV_d  \to \cV_d$, and $\Hop_d =\Top_d^* \Top_d: \cD_d \to \cD_d$.
These operators admit the kernel representations:
\begin{align}
\Uop_d f (\btheta) =&~ \int_{\Omega_d} U_d(\btheta, \btheta') f(\btheta') \tau_d(\de \btheta'),\label{eq:Ukernel}\\
\Hop_d g (\bx) =&~ \int_{\cX_d} H_d(\bx, \bx') g(\bx') \nu_d(\de \bx')\, ,
\end{align}
where $U_d: \Omega_d \times \Omega_d \to \R$ and $H_d: \cX_d \times \cX_d \to \R$ be two measurable functions,
satisfying $\int_{\Omega_d} U_d(\btheta,\btheta') f(\btheta')\, \tau_d(\de\btheta') = 0$ for $f\in \cV_d^\perp$, and $\int_{\cX_d} H_d(\bx,\bx') g(\bx')\, \nu_d(\de\bx') = 0$
for $g\in \cD_d^\perp$.
We immediately have 
\begin{align}
U_d(\btheta_1, \btheta_2) =&~ \E_{\bx \sim \nu_d}[\sigma_d(\bx, \btheta_1) \sigma_d(\bx, \btheta_2)],\\
H_d(\bx_1, \bx_2) =&~ \E_{\btheta \sim \tau_d}[\sigma_d(\bx_1, \btheta) \sigma_d(\bx_2, \btheta)]. \label{eq:Hdsigma}
\end{align}
By Cauchy-Schwartz inequality, we have $U_d \in L^2(\Omega_d \times \Omega_d)$ and $H_d \in L^2(\cX_d \times \cX_d)$. 

By the spectral theorem of compact operators, there exist two orthonormal bases
$(\psi_j)_{j \ge 1}$, $\spn(\psi_j, j \ge 1) = \cD_d \subseteq L^2(\cX_d)$ and $(\phi_{j})_{j \ge 1}$, $\spn(\phi_j, j \ge 1) = \cV_d \subseteq L^2(\Omega_d)$,
and eigenvalues $(\lambda_{d, j})_{j \ge 1} \subseteq \R$, with nonincreasing absolute values $\vert\lambda_{d, 1} \vert \ge \vert \lambda_{d, 2} \vert \ge \cdots$, and $\sum_{j \ge 1} \lambda_{d, j}^2 < \infty$ such that
\[
\begin{aligned}
\Top_d = \sum_{j = 1}^\infty \lambda_{d, j} \psi_j \phi_j^*, ~~~ \Uop_d = \sum_{j = 1}^\infty \lambda_{d,j}^2 \phi_{j} \phi_{j}^*, ~~~ \Hop_d = \sum_{j = 1}^\infty \lambda_{d,j}^2 \psi_{j} \psi_{j}^*.
\end{aligned}
\]
(Here convergence holds in operator norm.) In terms of the kernel, these identities read
\begin{align}
 \sigma_d(\bx, \btheta) = \sum_{j=1}^{\infty} \lambda_{d, j} \psi_{j}(\bx) \phi_{j}(\btheta), \;\;\;\;\;\;
   U_d(\btheta_1, \btheta_2) = \sum_{j=1}^{\infty} \lambda_{d, j}^2 \phi_{j}(\btheta_1) \phi_{j}(\btheta_2), \;\;\;\;\;
    H_d(\bx_1, \bx_2)   = \sum_{j=1}^{\infty} \lambda_{d, j}^2 \psi_{j}(\bx_1) \psi_{j}(\bx_2). \label{eq:Hdecomposition}
\end{align}
Here convergence  holds in $L^2(\cX_d \times \Omega_d)$, $L^2(\Omega_d \times \Omega_d)$, and $L^2(\cX_d \times \cX_d)$.

 Associated to the operator $\Hop$, we can define a reproducing kernel Hilbert space (RKHS)  $\cH \subseteq \cD_d$ defined as 
\[
\cH = \Big\{ f \in \cD : \| f \|_{\cH}^2 = \sum_{j = 1}^\infty \lambda_{d,j}^{-2} \< f , \psi_j \>_{L^2}^2 < \infty \Big\},
\]
where $\| \cdot \|_{\cH}$ denotes the RKHS norm associated to $\cH$. In particular, $\cH$ is dense in $\cD_d$,
provided $\lambda_{d,j}^2>0$ for all $j$.

For $S\subseteq \{1,2,\dots\}$, we denote $\proj_S$ to be the projection operator from $L^2(\cX_d)$ onto
$\cD_{d,S}:=\spn(\psi_j, j \in S)$. With a little abuse of notations, we also denote $\proj_S$ to be the projection operator from $L^2(\Omega_d)$ onto $\cV_{d,S}:=\spn(\phi_j, j \in S)$. We denote $\Top_{d, S}$ and $\sigma_{d, S}$ to be the corresponding operator and kernel 
\begin{align*}
  \Top_{d, S} =&~ \sum_{j \in S} \lambda_{d, j} \psi_{j} \phi_{j}^* ,\\
  \sigma_{d, S}(\bx, \btheta) =&~ \sum_{j \in S}  \lambda_{d, j} \psi_{j}(\bx) \phi_{j}(\btheta).
\end{align*}
We define  $\Uop_{d, S} =  \Top_{d, S} \Top_{d, S}^*$ and $\Hop_{d, S} = \Top_{d, S}^* \Top_{d, S}$, and denote by $U_{d, S}$ and $H_{d, S}$ the
corresponding kernels. If $S=\{j\in \naturals :\;\; j\le \ell\}$ we will write for brevity  $\Top_{d, \le \ell}$, $\Uop_{d, \le \ell}$, $\Hop_{d, \le \ell}$,
and similarly for $S=\{j\in \naturals :\;\; j>\ell\}$.

Since $\sigma_d\in L^2(\cX_d\times\Omega_d)$, it follows that $\Uop_{d, S}$ is trace class, for any $S\subseteq \naturals$,
with trace given by
\begin{align*}
\Trace(\Uop_{d, S}) \equiv \sum_{j\in S} \lambda_{d,j}^2 = \E_{\btheta \sim \tau_d}[U_{d, S}(\btheta, \btheta)] <\infty\, .
\end{align*}
Similarly, we have 
\begin{align*}
\Trace(\Hop_{d, S}) \equiv \sum_{j\in S} \lambda_{d,j}^2 = \E_{\bx \sim \nu_d}[H_{d, S}(\bx, \bx)] <\infty\, .
\end{align*}

\subsection{Assumptions}
\label{sec:RF_Assumptions}

Let $\bTheta = (\btheta_i)_{i \in [N]} \sim_{iid} \tau_d$. We define the random features function class to be 
\[
\cF_{\RF, N}(\bTheta) = \Big\{ \hf(\bx;\ba) = \frac{1}{N}\sum_{i=1}^N a_i \sigma_d(\bx, \btheta_i): a_i \in \R, i \in [N] \Big\}.
\]
Note that the factor $1/N$ is immaterial here, and only introduced in order to match the definition of
feature map and scalar product in Section \ref{sec:IntroBackground}.

We observe pairs  $(y_i, \bx_i)_{i \in [n]}$, with $(\bx_i)_{i \in [n]} \sim_{iid} \nu_d$, and $y_i = f_d(\bx_i) +\eps_i$, $f_d \in L^2(\cX_d)$ and $\eps_i \sim \normal ( 0 ,\noise^2)$ independently.
We fit the coefficients $(a_i)_{i\le N}$ using ridge regression, cf. Eq.~\eqref{eq:FirstRidge} that we reproduce here
\begin{equation}\label{eqn:hbiota_RFRR}
\hba(\lambda) = \argmin_{\ba} \left\{ \sum_{i=1}^n\big( y_i - \hf(\bx_i;\ba) \big)^2  + \frac{\lambda}{N} \| \ba \|_2^2  \right\}\, .
\end{equation}
We allow $\lambda$ to depend on the dimension parameter $d$.
The test error is  given by
\begin{equation}\label{eqn:test_error_RFRR}
R_\RF(f_d, \bX, \bTheta, \lambda) := \E_\bx\Big[ \Big(f_d (\bx) - \hf(\bx;\hba(\lambda))\Big)^2\Big]\, .
\end{equation}

We next state our assumptions on the sequences of probability spaces $(\cX_d,\nu_d)$ and $(\Omega_d,\tau_d)$,
and on the activation functions $\sigma_d$. The first set of assumptions concerns the concentration properties of the
feature map, and are grouped in the next definition. These assumptions are quantified
by four sequences of integers $\{ ( N(d) , \evN(d), n(d) , \evn(d)  ) \}_{d \ge 1}$, where $N(d)$ and $n(d)$
are, respectively, the number of neurons and the sample size. The integers $\evN(d)$ and $\evn(d)$ play a minor role
in this definition, but will encode the decomposition of $L^2(\Omega_d)$ and $L^2(\cX_d)$
(respectively) into the span of the top eigenvectors of $\Uop_d$ and $\Hop_d$ (of dimensions $\evN(d)$ and $\evn(d)$)
and their complements.
\begin{assumption}[$\{ ( N(d) , \evN(d), n(d) , \evn(d)  ) \}_{d \ge 1}$-Feature Map Concentration Property]\label{ass:FMPCP} We say that the sequence of activation functions $\{ \sigma_d \}_{d\ge 1}$ satisfies the \textit{Feature Map Concentration Property (FMCP)} with respect to the sequence $\{ ( N(d) , \evN(d) , n(d) , \evn(d))  \}_{d \ge 1}$ if there exists a sequence $\{u(d)\}_{d\ge 1}$ with $u(d) \ge \max( \evN (d), \evn (d) )$ such that the following hold.
\begin{itemize}
\item[(a)] (Hypercontractivity of finite eigenspaces) 
\begin{itemize}
\item[(i)] (Hypercontractivity of finite eigenspaces on $\cD_d$.) For any integer $k\ge 1$, there exists $C$ such that, for any $g \in \cD_{d, \le u(d)} = \spn ( \psi_s, 1 \leq s \leq u(d) )$, we have 
\[
\begin{aligned}
\| g \|_{L^{2k} (\cX_d) }  \le&~ C \cdot \| g \|_{L^{2} (\cX_d) }. 
\end{aligned}
\]
\item[(ii)] (Hypercontractivity of finite eigenspaces on $\cV_d$.) For any integer $k \ge 2$, there exists $C'$ such that, for any $g \in \cV_{d, \le u(d)}= \spn ( \phi_s, 1 \leq s \leq u(d) )$, we have 
\[
\begin{aligned}
\| g \|_{L^{2k} ( \Omega_d) }  \le&~ C' \cdot \| g \|_{L^{2} ( \Omega_d) }. 
\end{aligned}
\]
\end{itemize}
\item[(b)]  (Properly decaying eigenvalues.) There exists a fixed $\delta_0 >0$, such that , for all $d$ large enough
\begin{equation}
\max(N(d),n(d))^{2 + \delta_0} \le
\frac{\Big(\sum_{j=u(d)+1}^{\infty} \lambda_{d,j}^2\Big)^2}{\sum_{j=u(d)+1}^{\infty} \lambda_{d,j}^4}\, .\label{eq:ConcentratioB}
\end{equation}

\item[(c)](Hypercontractivity of the high degree part.) Let $\sigma_{d, >u(d)}$ corresponds to the projection on the high degree part of $\sigma_{d}$. Then there exists a fixed $\delta_0 >0$ and an integer $k$ such that 
\[
\min(n,N)^{1 + 2\delta_0} \max(N,n)^{1/k - 1} \log(\max(N,n) ) = o_d (1),
\]
and
\[
\E_{\bx,\btheta} [ \sigma_{>u(d)} ( \bx ; \btheta )^{2k} ]^{1/(2k)} =   O_d(1) \cdot \min(n,N)^{\delta_0} \cdot \E_{\bx,\btheta} [ \sigma_{>u(d)} ( \bx ; \btheta )^{2} ]^{1/2}. 
\]

\item[(d)](Concentration of diagonal elements) For $(\bx_i)_{i \in [n(d)]} \sim_{iid} \nu_d$ and $(\btheta_i)_{i \in [N(d)]} \sim_{iid} \tau_d$, we have
\[
\begin{aligned}
\sup_{i \in [n(d)]} \Big\vert H_{d, > \evn(d)}(\bx_i, \bx_i) - \E_{\bx} [H_{d, > \evn(d)}(\bx, \bx)] \Big\vert =&~ o_{d, \P}(1) \cdot \E_{\bx}[H_{d, > \evn(d)}(\bx, \bx)]\, ,\\
\sup_{i \in [N(d)]} \Big\vert U_{d, > \evN(d)}(\btheta_i, \btheta_i) - \E_{\btheta} [U_{d, > \evN(d)}(\btheta, \btheta)] \Big\vert =&~ o_{d, \P}(1) \cdot \E_{\btheta}[U_{d, > \evN(d)}(\btheta, \btheta)].
\end{aligned}
\]
\end{itemize}
\end{assumption}
This statement formalizes three assumptions. The first one is hypercontractivity (points $(a)$ and $(c)$).
Recall that $\cD_{d, \le u(d)}$ is the eigenspace spanned by  top eigenvectors of
the operator $\Hop_d$, and $\cV_{d, \le u(d)}$ is the eigenspace spanned by  top eigenvectors of
the operator $\Uop_d$. We request that functions in these spaces have comparable norms of all orders, which roughly amount to say that they
take values of the same order as their typical value for most $\bx$ (or most $\btheta$). This typically happens when the functions in the top eigenspaces are delocalized.

The second assumption (assumption $(b)$) requires that the eigenvalues of kernel operators do not decay too rapidly.
If this is not the case, the RKHS will be very close to a low-dimensional space. For instance, if $\lambda^2_{d,k}\asymp k^{-2\alpha}$, $\alpha>0$,
then this condition holds as long as we take $u(d)\ge \max(N(d),n(d))^{2+\delta_0}$ for some $\delta_0>0$. 

Finally, assumption $(d)$ concerns the diagonal elements of the kernel matrices. They require the truncated kernel functions $H_{d,>\evn(d)}$ and $U_{d,>\evN(d)}$ evaluated on covariates and weight vectors  to have nearly constant diagonal values. 

The second set of assumptions concerns the spectrum of the kernel operator, defined by the sequence of eigenvalues $(\lambda_{d,j}^2)_{j\ge 1}$.
We require that the spectrum has a gap: the location of this gap dictates the relationship between $N(d)$ and $\evN(d)$ and
between $n(d)$ and $\evn(d)$. 
\begin{assumption}[Spectral gap at level $\{ ( N(d) , \evN(d) , n(d) , \evn(d)  ) \}_{d \ge 1}$]\label{ass:spectral_gap} We say that the sequence of activation functions $\{ \sigma_d \}_{d\ge 1}$ has a \textit{spectral gap} at level $\{ ( N(d) , \evN(d) , n(d) , \evn(d)  ) \}_{d \ge 1}$ if one of the following conditions $(a)$, $(b)$ hold for all $d$ large enough.
 \begin{itemize}
\item[(a)] (Overparametrized regime.) We have $N(d) \ge n(d)$ and
\begin{itemize}
\item[(i)](Number of samples) There exists fixed $\delta_0 > 0$ such that $\evn (d) \le n(d)^{1 - \delta_0}$ and
\begin{align}
  \frac{1}{\lambda_{d,\evn(d)}^2}\sum_{k=\evn(d)+1}^{\infty}\lambda_{d,k}^2
  \le n(d)^{1-\delta_0}\le  n(d)^{1+\delta_0}\le   \frac{1}{\lambda_{d,\evn(d)+1}^2}\sum_{k=\evn(d)+1}^{\infty}\lambda_{d,k}^2\, .
  \label{eq:NumberOfSamples}
\end{align}
\item[(ii)](Number of features)
There exists fixed $\delta_0 > 0$ such that $\evN (d) \le N(d)^{1 - \delta_0}$, $\evN (d) \ge \evn (d)$ and 
\begin{align}
N(d)^{1+\delta_0} \le&~  \frac{1}{\lambda_{d,\evN(d)+1}^2}\sum_{k=\evN(d)+1}^{\infty}\lambda_{d,k}^2.\label{eq:NumberOfFeatures}
\end{align}
\end{itemize}
\item[(b)] (Underparametrized regime) We have $n(d) \ge N(d)$ and 
\begin{itemize}
\item[(i)](Number of features) There exists fixed $\delta_0 > 0$ such that $\evN(d) \le N(d)^{1 - \delta_0}$ and 
  \begin{align*}
     \frac{1}{\lambda_{d,\evN(d)}^2}\sum_{k=\evN(d)+1}^{\infty}\lambda_{d,k}^2
    \le N(d)^{1-\delta_0}\le  N(d)^{1+\delta_0} \le  \frac{1}{\lambda_{d,\evN(d)+1}^2}\sum_{k=\evN(d)+1}^{\infty}\lambda_{d,k}^2\, . 
\end{align*}

\item[(ii)](Number of samples)
There exists fixed $\delta_0 > 0$ such that $\evn(d) \le n(d)^{1 - \delta_0}$, $\evn (d) \ge \evN (d)$ and 
\begin{align*}
n(d)^{1+\delta_0} \le&~  \cdot  \frac{1}{\lambda_{d,\evn(d)+1}^2}\sum_{k=\evn(d)+1}^{\infty}\lambda_{d,k}^2\, .
\end{align*}
\end{itemize}
\end{itemize}
\end{assumption}

The assumption of a spectral gap is useful in that it leads to a clear-cut separation
in our main statement below. For instance, in the overparametrized regime $n(d)\ll N(d)$,
the projection of the target function onto $\cD_{d,\le \evn(d)}$ is estimated with negligible error, while
the projection onto $\cD_{d,> \evn(d)}$ is estimated with $0$.
If there was no spectral gap, the transition would not be as sharp. However, we expect this to affect only target functions
with a large projection onto eigenfunctions whose indices are  close to $\evn(d)$. In this sense, while restrictive, the spectral gap assumption
can be in fact a good model for a more generic situation.

\subsection{A general theorem}
\label{sec:RF_Theorem}

We are now in position to state our main results for random features ridge regression.
\begin{theorem}[Generalization error of Random Features Ridge Regression]\label{thm:RFK_generalization}
  Let $\{ f_d \in \cD_d \}_{d \ge 1}$ be a sequence of functions, $\bX = (\bx_i)_{i \in [n(d)]}$ and $\bTheta= (\btheta_j)_{j \in [N(d)]}$ with $(\bx_i)_{i \in [n(d)]} \sim \nu_d$ and $(\btheta_j )_{j \in [N(d) ]} \sim  \tau_d$ independently. Let $y_i = f_d(\bx_i) + \eps_i$ and $\eps_i \sim_{iid} \normal(0, \noise^2)$ for some $\noise>0$. Let $\{ \sigma_d \}_{d\ge 1}$ be a sequence of activation functions satisfying $\{ ( N(d) , \evN(d) , n(d) , \evn(d)  ) \}_{d \ge 1}$-FMCP (Assumption \ref{ass:FMPCP}) and spectral gap at level
  $\{ ( N(d) , \evN(d) , n(d) , \evn(d)) ) \}_{d \ge 1}$ (Assumption \ref{ass:spectral_gap}). Then the following hold for the test error of RFRR (see Eq. \eqref{eqn:test_error_RFRR}):
\begin{itemize}
\item[(a)] (Overparametrized regime) If $N(d) \geq d^\delta \cdot n(d)$ for some $\delta >0$, , let $\lambda_\star$ be such that $\lambda_*=o_{d} (\Trace(\Hop_{d, > \evn})) $. Then, for any regularization parameter $\lambda \in [0, \lambda_\star]$, and any fixed $\eta >0$ and $\eps > 0$, with high probability we have
\begin{align}
\vert R_{\RF}(f_d, \bX, \bTheta, \lambda) - \| \proj_{> \evn } f_d \|_{L^2}^2 \vert \le \eps \cdot (\| f_d \|_{L^2}^2+\| \proj_{>\evn}f_d \|_{L^{2 + \eta} }^2 + \noise^2). 
\end{align}

\item[(b)] (Underparametrized regime) If $n(d) \geq d^\delta \cdot N(d)$ for some $\delta >0$, let $\lambda_\star$ be such that $\lambda_\star = o_{d} ( n/N \cdot  \Trace(\Uop_{d, > \evN})) $. Then, for any regularization parameter $\lambda \in [0, \lambda_\star]$, and any fixed $\eta >0$ and $\eps > 0$, with high probability we have
\begin{align}
\vert R_{\RF}(f_d, \bX, \bTheta, \lambda) - \| \proj_{> \evN } f_d \|_{L^2}^2 \vert \le \eps \cdot (\| f_d \|_{L^2}^2 + \| \proj_{>\evN} f_d \|_{L^{2 + \eta}}^2 + \noise^2). 
\end{align}
\end{itemize}  
\end{theorem}

\begin{remark}
  The two limits $N=\infty$ and $n=\infty$ play a special role. For $N=\infty$, the random kernel $H_N(\bx_1,\bx_2) =N^{-1}\sum_{i=1}^N \sigma(\bx_1;\btheta_i)\sigma(\bx_2;\btheta_i)$ converges to its expectation, and we recover KRR. While this case is not
  technically covered by Theorem \ref{thm:RFK_generalization}, we establish the relevant characterization in
  Theorems \ref{thm:KR-general-main} and \ref{thm:upper_bound_KRR}.

  In the case $n=\infty$ the generalization error vanishes, and we are left with the approximation error. This case is covered
  separately in Appendix \ref{sec:approximation_error_RFK}. In both these limit cases we confirm the result that would have been obtained by naively setting $N=\infty$ or $n=\infty$ in the last theorem. 
\end{remark}

Notice that the sample size $n$ and the number of neurons $N$ play a nearly symmetric role in this statement,
and the smallest of the two determines the test error.
An important insight follows: in the present setting, the test error is nearly insensitive to the number of neurons as long as
we take $N\gg n$. If we want to minimize computational complexity subject to achieving nearly optimal generalization properties,
we should operate, say, at  $N\asymp n^{1+\delta}$ for some small $\delta>0$.

It is instructive to compare this result with \cite{rudi2017generalization} which instead suggests $N\asymp \sqrt{n}\log n$.
While our setting differs from the one of \cite{rudi2017generalization} in a number of technical aspects, we believe that the core difference
between the two results lies in the treatment of the target function $f_d$. Simplifying, the recommendation of
\cite{rudi2017generalization} is based on two results, the second of which proved in \cite{caponnetto2007optimal}
(with an abuse of notation, we indicate the number of neurons and sample size as arguments of $R_{\RF}(f_d) = R_{\RF}(f_d;N,n)$,
and use $N=\infty$ to denote the KRR limit case):
\begin{align}
  \sup_{\|f_d\|_{\cH}\le r}R_{\RF}(f_d;N_n,n) 
  &\le C_1(d) \frac{r^2}{\sqrt{n}}\, ,\;\;\;\; \;\;\;\; \;\;\mbox{for}\;\; N_n\asymp \sqrt{n}\log n\, ,\label{eq:RF-Rudi}\\
   \sup_{\|f_d\|_{\cH}\le r}R_{\RF}(f_d;\infty,n) 
  &\le C_2(d) r^2 \Big(\frac{\log n}{n}\Big)^{b/(b+1)}\, ,\label{eq:RKHS-Rudi}
\end{align}
where $b\in (1,\infty)$ encodes the decay of eigenvalues of the
kernel\footnote{The results of \cite{caponnetto2007optimal,rudi2017generalization} assume the weaker condition
  that $\inf_{g\in \cH}\|f_d-g\|_{L^2}$ is achieved in $\cH$: since $\cH$ is dense in $L^2(\cX_d)$ (provided the kernel is strictly positive definite), this
  is equivalent to $f_d\in\cH$.}. Now, considering the worst case decay $b\to 1$, the error rate achieved by RFRR, cf. Eq.~\eqref{eq:RF-Rudi},
is of the same order as the one achieved by KRR, cf. Eq.~\eqref{eq:RKHS-Rudi}.

Note several differences with respect to our results: $(i)$~The analysis of \cite{rudi2017generalization,caponnetto2007optimal}
is minimax, over balls in the RKHS, while our results hold \emph{pointwise}, i.e., for a given function
$f_d$; $(ii)$~Optimality in  \cite{rudi2017generalization} is established in terms of rates, i.e., up to multiplicative constant, while ours
hold up to additive errors (multiplicative constants are exactly characterized); $(iii)$~The results of
\cite{rudi2017generalization,caponnetto2007optimal} apply to a fixed RKHS (in particular, a fixed dimension $d$), while we
study the case in which $d$ is large and $N,n,d$ are polynomially related.

Some of these distinction are also relevant in comparing our work to recent results on KRR. In particular 
points $(i)$ and $(ii)$ apply when comparing with \cite{liang2020just,liang2019risk}.

\subsection{Examples: The binary hypercube and the sphere}
\label{sec:RF_Examples}

As examples we consider the case of feature vectors $\bx_i$ that are uniformly distributed over the
discrete hypercube $\Cube = \{ -1, +1\}^d$  or the sphere $\S^{d-1} (\sqrt{d}) = \{ \bx \in \R^d : \| \bx \|_2 = d\}$.
Namely, letting $\cA_d$ to be either $\Cube$ or $\S^{d-1} (\sqrt{d}) $ and $\rho_d = \Unif ( \cA_d )$,
we set $\cX_d = \cA_d$ and $\nu_d = \rho_d$. We further choose the $\btheta_i$'s to be distributed as the covariates vectors, namely
$\cV_d = \cA_d$ and $\tau_d = \rho_d$. Apart from simplifying our
analysis, this is a sensible choice: since the covariates vectors
do not align along any preferred direction, it is reasonable for the $\btheta_i$'s to be isotropic as well.

Given a function  $\barsigma_d : \R \to \R$ (which we allow to depend on the dimension $d$), we define the activation function $\sigma_d : \cA_d \times \cA_d \to \R$ by
\begin{align}
\sigma_d ( \bx ; \btheta ) = \barsigma_d ( \< \bx , \btheta \> / \sqrt{d} ).\label{eq:InnerProdActivation}
\end{align}
We denote by $\cE_{d,\leq \ell}$ the subspace of $L^2 ( \cA_d , \rho_d )$ spanned by polynomials of degree less or equal to $\ell$ and by
$\oproj_{\leq \ell}$ the orthogonal projection on $\cE_{d,\leq \ell}$ in $L^2 ( \cA_d , \rho_d )$.
The projectors $\oproj_{\ell}$ and $\oproj_{> \ell}$ are defined analogously (see Appendix \ref{sec:technical_background} for more details). Let us emphasize that the projectors $\oproj_{\leq \ell}$ are related but distinct from
the $\proj_{\le \evn}$: while $\oproj_{\leq \ell}$  projects onto eigenspaces of polynomials of degree at most
$\ell$, $\proj_{\le \evn}$ projects onto the top $\evn$-eigenfunctions\footnote{The two coincide if $\evn=\sum_{\ell'\le \ell}B(\cA_d;\ell')$, with $B(\cA_d;\ell')$ the dimension of the space of degree-$\ell'$ polynomials and the top $\evn$ eigenvalues verify $\lambda_{d,j}^2 = \Omega_{d} ( d^{-\ell} )$, see Appendix \ref{sec:proof_RFK_sphere}.}.

In order to apply Theorem \ref{thm:RFK_generalization}, we make the following assumption about $\barsigma_d$.
\begin{assumption}[Assumptions on $\cA_d$ at level $(\lvn , \lvN) \in\naturals^2$]\label{ass:activation_RFK_sphere}
  For $\{ \barsigma_d \}_{d \ge 1}$  a sequence of functions $\barsigma_d:\reals\to\reals$, we assume the following conditions to hold.

\begin{itemize}
\item[(a)] There exists an integer $k$ and constants $c_1 <1$ and $c_0 >0$, $\delta_0>1/k$  such that
  $n \le N^{1-\delta_0}$ or $N\le n^{1-\delta_0}$ and $|\barsigma_d ( x )| \leq c_0 \exp (c_1 x^2 / (4k) )$.

\item[(b)] We have
\begin{align}
 \min_{k \le \lvn}d^{\lvn-k}\|\oproj_{k}\barsigma_{d } (\< \be,\, \cdot\, \>) \|^2_{L^2(\cA_d , \rho_d)}  = & \Omega_d(1), \label{eq:ass_sphere_b1} \\
  \min_{k \le \lvN}d^{\lvN-k}\|\oproj_{k}\barsigma_{d } (\< \be,\, \cdot\, \>) \|^2_{L^2(\cA_d , \rho_d)}  = & \Omega_d(1), \label{eq:ass_sphere_b1p} \\
 \| \oproj_{> 2\max(\lvn , \lvN) + 1} \barsigma_{d } (\< \be, \,\cdot\, \>) \|_{L^2(\cA_d , \rho_d)} = & \Omega_d (1), \label{eq:ass_sphere_b2} 
\end{align}
where  $\be \in \cA_d$ is a fixed vector (it is easy to see that these quantities do not depend on $\be$).

\item[(c)] If $\cA_d = \Cube$, we have, for all $d$ large enough
\begin{align}
  \max_{k \le 2\max(\lvn , \lvN) +2 } d^{-k}\|\oproj_{d-k}\barsigma_{d } (\<
  \be, \,\cdot\, \>) \|^2_{L^2(\cA_d , \rho_d)}  \le d^{-2\max(\lvn , \lvN)-2}\,.
   \label{eq:ass_sphere_c} 
\end{align}
\end{itemize}
\end{assumption}
Assumption $(a)$ requires $n$, $N$ to be well separated and a technical integrability condition.
The latter is necessary for the hypercontractivity condition in Assumption \ref{ass:FMPCP}.$(c)$ to make sense.

Equations \eqref{eq:ass_sphere_b1} and \eqref{eq:ass_sphere_b1p} (Assumption $(b)$) are a quantitative version of a universality condition:  if
$\oproj_{k}\barsigma_{d } (\< \be, \cdot \>/\sqrt{d}) =0$ for some $k$, then
linear combinations of $\barsigma_d$ can only span a linear subspace of $L^2(\cA_d,\rho_d)$.
Equation \eqref{eq:ass_sphere_b2} (Assumption $(b)$) requires the high degree part of $ \barsigma_{d }$ to be non-vanishing (and therefore induce a non-zero regularization from the high degree non-linearity).

For $\cA_d = \Cube$, we further require Assumption $(c)$, namely that the last eigenvalues of $\barsigma_d$ decrease sufficiently fast. This is a necessary conditions to avoid pathological sequences $\{ \barsigma_{d} \}_{d \ge 1}$ which are very rapidly oscillating.
\begin{remark}
  If $\barsigma_d = \barsigma$ is independent of the dimension, then  Assumptions $(b)$, $(c)$ are easy to check:
  \begin{itemize}
  \item The first two parts of Assumption $(b)$  (Eqs.~\eqref{eq:ass_sphere_b1} and \eqref{eq:ass_sphere_b1p})  are
  satisfied if we require $\E\{\barsigma(G)\, p(G)\}\neq 0$ for all non-vanishing polynomials $p$  of degree at most $\max(\lvn , \lvN)$ (expectation being taken with respect to $G\sim \normal(0,1)$.)
  This is in turn equivalent to $\E\{\barsigma(G)\, \He_k(G)\}\neq 0$ for all $k\le \max(\lvn , \lvN)$, where
  $\He_k$ is the $k$-th Hermite polynomial.
\item The third part of Assumption $(b)$  (Eq.~\eqref{eq:ass_sphere_b2})  amounts to requiring $\barsigma$ not to be a degree-$ (2\max(\lvn , \lvN)+1)$ polynomial.
\item  In Appendix \ref{sec:hypercube_case} we check that Assumption $(c)$ holds if $\barsigma$ is smooth and there exists $c_0>0$ and $c_1 < 1$ constants such that
 the $( 2\max(\lvn , \lvN)+2)$-th derivative verifies $| \barsigma^{(2\max(\lvn , \lvN)+2)} (x) | \leq c_0 \exp ( c_1 x^2 / 4)$.
  \end{itemize}
\end{remark}

As an example, the shifted ReLu $\barsigma_d (x) = (x - c)_+$ with a a generic $c \in \R\setminus\{0\}$ verifies
Assumption \ref{ass:activation_RFK_sphere}. (The case $c=0$ violates Eq.~\eqref{eq:ass_sphere_b1}, since $\E\{\barsigma(G)
\He_k(G)\}=0$ for $k\ge 3$ odd. This is not a limitation of our result: the unshifted ReLU is not universal in the present setting.)

\begin{theorem}[Generalization error of RFRR on the sphere and hypercube]\label{prop:RFK_sphere}
Let $\{ f_d \in L^2(\cA_d, \rho_d)\}_{d \ge 1}$ be a sequence of functions. Let $\bTheta = (\btheta_i)_{i \in [N]}$ with $(\btheta_i)_{i \in [N]} \sim \rho_d$ independently and $\bX = (\bx_i)_{i \in [n]}$ with $(\bx_i)_{i \in [n]} \sim \rho_d$ independently. Let $y_i = f_d(\bx_i) + \eps_i$ and $\eps_i \sim_{iid} \normal(0, \noise^2)$ for some $\noise>0$. Assume $d^{\lvn+\delta_0} \leq n \le d^{\lvn+1-\delta_0}$ and $d^{\lvN + \delta_0 } \le N\le d^{\lvN+1-\delta_0}$ for fixed integers
  $\lvn, \lvN$ and for some $\delta_0 >0$. Let $\{ \barsigma_d \}_{d \geq 1}$ satisfy Assumption \ref{ass:activation_RFK_sphere} at level $(\lvn, \lvN)$. Then the following hold for the test error of RFRR (see Eq. \eqref{eqn:test_error_RFRR}): 
\begin{itemize}
\item[(a)] Assume $ N \ge n d^\delta$ for some $\delta >0$. Then for any regularization parameter $\lambda = O_d(1)$ (including $\lambda =0$ identically), any 
$\eta >0$ and $\eps > 0$, we have, with high probability,
\begin{align}
\vert R_{\RF}(f_d, \bX, \bTheta, \lambda) - \| \oproj_{> \lvn} f_d \|_{L^2}^2 \vert \le \eps \cdot (\| f_d \|_{L^2}^2+\| \oproj_{>\lvn} f_d \|_{L^{2 + \eta} }^2 + \noise^2). 
\end{align}
\item[(b)] Assume $ n \ge N d^\delta$ for some $\delta >0$.
  Then, for any regularization parameter $ \lambda = O_d(n/N)$ (including $\lambda =0$ identically), $\eta >0$ and $\eps > 0$, we have, with high probability,
\begin{align}
\vert R_{\RF}(f_d, \bX, \bTheta, \lambda) - \| \oproj_{> \lvN} f_d \|_{L^2}^2 \vert \le \eps \cdot (\| f_d \|_{L^2}^2 + \| \oproj_{>\lvN} f_d \|_{L^{2 + \eta} }^2 + \noise^2). 
\end{align}
\end{itemize}
\end{theorem}

As mentioned in the introduction, \cite{ghorbani2019linearized} proves this theorem 
in the cases $n=\infty$ (RF approximation error) and $N=\infty$ (generalization error of KRR),
for the uniform measure on the sphere. The general case follows here as a consequence of Theorem \ref{thm:RFK_generalization}.

To see the connection between Theorem \ref{thm:RFK_generalization} and the results given here
for the sphere and hypercube cases (see Appendix \ref{sec:proof_RFK_sphere} for details), notice that the  integral operator $\Top_d$ associated to the inner product activation function
\eqref{eq:InnerProdActivation} is in this case symmetric, and commutes with rotations in
${\rm SO}(d)$ (for the sphere) or with the action of $(\integers_2)^d$
(for the hypercube\footnote{In the $\{+1,-1\}^d$ representation, $\bz\in\{+1,-1\}^d$ acts on
  $\Cube$ via
  $\bx\mapsto \bD_{\bz}\bx$, where $\bD_{\bz}$ is the diagonal matrix with ${\rm diag}(\bD_{\bz})=\bz$.}). 
Hence, the eigenvectors of $\Top_d$ (which is self-adjoint by construction) are given by the
spherical harmonics of degree $\ell$ (for the sphere) or the homogeneous polynomials of degree $\ell$ (for the hypercube). The spaces spanned by the low degree spherical harmonics and homogeneous polynomials verify the hypercontractivity condition of Assumption \ref{ass:FMPCP}.$(a)$ (see Appendix \ref{app:hypercontractivity}). 
The corresponding distinct eigenvalues are $\xi_{d,\ell}$, with degeneracy
\begin{align}
  B(\S^{d-1};\ell) = \frac{d-2+2\ell}{d-2}\binom{d-3+\ell}{\ell}\, ,\;\;\;\;\;\;\;\; B(\Cube;\ell) = \binom{d}{\ell}\, .
\end{align}
Notice that in both cases $B(\cA_d;\ell)= (d^{\ell}/\ell!)(1+o_d(1))$ and, hence $\xi_{d,\ell}\lesssim d^{-\ell/2}$
(by construction $\Tr(\Hop_d)$ is bounded uniformly). Indeed, by Assumption \ref{ass:activation_RFK_sphere}.$(a)$,
we have $\xi_{d,\ell}\asymp d^{-\ell/2}$.

As a consequence, if we set $\evn = \sum_{\ell\le \lvn} B(\cA_d;\ell)$,  $\evN = \sum_{\ell\le \lvN} B(\cA_d;\ell)$,
we have $\sum_{k=\ell+1}^{\infty}\lambda_{k,d}^2  = \Theta(1)$ (indeed this sum is $O_d(1)$ because $\Tr(\Hop_d)$
is bounded uniformly, and it is $\Omega_d(1)$ by Assumption \ref{ass:activation_RFK_sphere}.$(b)$). 
Therefore, the conditions  \eqref{eq:EllDef} and \eqref{eq:QDef} (or, more formally, the conditions in Assumption
\ref{ass:spectral_gap}) can be rewritten as
\begin{align}
  &d^{\lvN}\asymp\frac{1}{\xi^2_\lvN}\ll N \ll   \frac{1}{\xi^2_{\lvN+1}}\asymp d^{\lvN+1}\, ,\\
  &d^{\lvn}\asymp\frac{1}{\xi^2_\lvn}\ll n \ll   \frac{1}{\xi^2_{\lvn+1}}\asymp d^{\lvn+1}\, ,
\end{align}
which matches the assumptions in Theorem \ref{prop:RFK_sphere}.

Figure \ref{fig:heatmap}  provides an illustration of Theorem \ref{prop:RFK_sphere}, for the case of the uniform distribution
over the sphere $\cA_d= \S^{d-1}(\sqrt{d})$. We fix $d=50$, and generate data $\{(\bx_i,y_i)\}_{i\le n}$
with no noise $\noise=0$. We use the target function
\begin{align}
  f_d(\bx) = g_d(\<\bv,\bx\>)\, ,\label{eq:TargetSimulations}
\end{align}
where $\bv\in\S^{d-1}(\sqrt{d})$  and $g$ is a fourth-order polynomial:
$g(z) = \frac{2}{\sqrt{5}} \widehat{Q}_{1}(z) +  \frac{2}{\sqrt{5}} \widehat{Q}_{2}(z)  +  \frac{1}{\sqrt{10}} \widehat{Q}_{3}(z) + \frac{1}{\sqrt{10}} \widehat{Q}_{4}(z)$ (here $\widehat{Q}_{\ell}$ is the $\ell$-th Gegenbauer polynomial, normalized so that
$\|\widehat{Q}_\ell(\<\bv,\,\cdot\,\>)\|_{L^2(\S^{d-1}(\sqrt{d}))}=1$). While the precise form of $f_d$
does not really matter here, we note that  $\|\oproj_{1}f_d\|_{L^2}^2= \|\oproj_{2}f_d\|_{L^2}^2= 0.4$, $\|\oproj_{3}f_d\|_{L^2}^2= \|\oproj_{4}f_d\|_{L^2}^2= 0.1$ and $\|\oproj_{>4}f_d\|_{L^2}^2=0$.
We plot the test error of RFRR using $\sigma(x) = \max( x - 0.5, 0)$ (shifted ReLu), and $\lambda=0+$ (min-norm interpolation).
We repeat this calculation for a grid of values of $n,N$, and for each point in the grid report the average risk
over $10$ realizations.

\begin{figure}[t]
\centering
\includegraphics[width=0.95\textwidth]{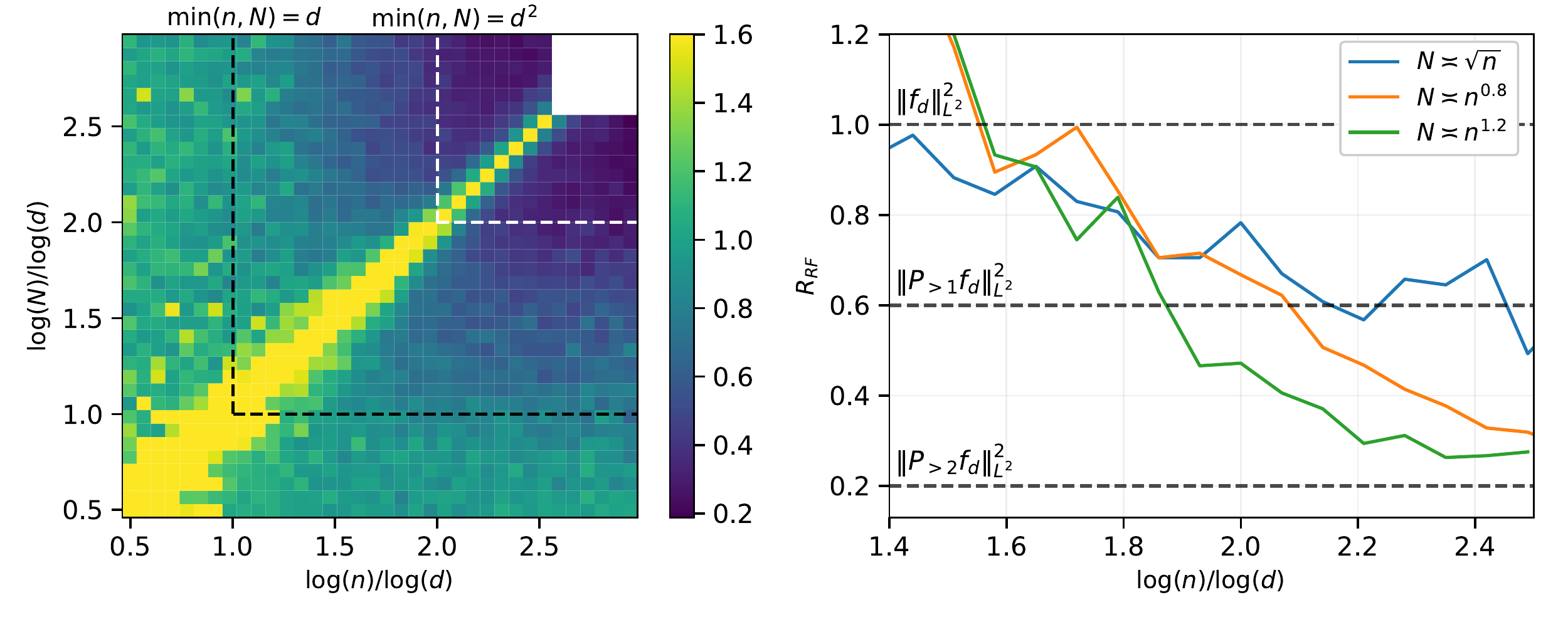} 
\caption{Learning a polynomial $f_d$ (cf. Eq.~\eqref{eq:TargetSimulations}) over the $d$-dimensional sphere, $d=50$, using a random features model and min-norm interpolation. We report the test error averaged over 10 realizations. Left: heatmap of the test error as a  function of the number of neurons $N$ and number of samples $n$. Notice the blow-up at the interpolation threshold $N\approx n$, and the symmetry around this line. Right: decrease of the test error as a function of sample size for scalings of the network size $N=n^{\alpha}$.}  
\label{fig:heatmap} 
\end{figure}

We plot the observed average risk in the sample-size/number-of-parameters plane whose axes are
$\log n/\log d$ and $\log N/\log d$ (corresponding to the exponents in the polynomial relation between $n$ and $d$,
and between $N$ and $d$). Several prominent features of this plot are worth of note:
\begin{itemize}
\item The risk has a large peak for $N\approx n$. This phenomenon was characterized precisely in the proportional regime $N\asymp d$, $n\asymp d$ in \cite{hastie2019surprises,mei2019generalization}.
\item The plot appears completely symmetric under exchange of $N$ and $n$: the number of parameters and sample size
  plays the same role in limiting the generalization abilities, as anticipated by  Theorem  \ref{thm:RFK_generalization} and Theorem \ref{prop:RFK_sphere}.
\item The risk is bounded away from zero even for $N,n\asymp d^3$. Indeed, Theorem \ref{prop:RFK_sphere} implies
  that consistent estimation would require $N,n\gg d^{4}$ in this case.
\item Finally, for a fixed $n$, near optimal test error is achieved when $N\asymp n^{1+\delta_*}$, for $\delta_*$ a small positive constant.
\end{itemize}

\section{Generalization error of kernel machines}
\label{sec:Kernel}

Formally, kernel ridge regression (KRR) corresponds to the limit $N\to\infty$ of random feature ridge regression.
Despite this, we cannot apply directly Theorem \ref{thm:RFK_generalization} with $N=\infty$.
We state therefore a separate theorems for kernel methods. As a side benefit, we establish somewhat
stronger results in this case. In particular:
\begin{itemize}
\item We simplify the set of assumptions (in particular, the assumptions concern only $H_d$ and not the
  activation function $\sigma_d$, as they should).
\item We prove a risk lower bound, Theorem \ref{thm:KR-general-main}, that holds for general kernel methods, not only
  KRR.
\item Crucially, we remove the spectral gap assumption. In this more general setting, the risk of KRR
  is not approximated by the square norm of the projection of $f_d$ orthogonal to the leading eigenfunctions of the kernel. We instead
  obtain an approximation in terms of a population-level ridge regression problem, with an effective value of the regularization parameter, which we determine. 
\end{itemize}

Throughout this section, the setting is the same as in the previous one: we observe  i.i.d. data $(y_i, \bx_i)_{i \in [n]}$,
with feature vectors $\bx_i$ from the probability space $(\cX_d, \nu_d)$.  Responses are given by
$y_i = f_d(\bx_i) +\eps_i$, $f_d \in \cD_d$ and $\eps_i \sim \normal ( 0 ,\noise^2)$ independently of $\bx_i$.

We introduce some general background in Section \ref{sec:KernelBackground}, then state our assumptions
in Section \ref{sec:assumptions_Kernel}, and formally state our results in Sections \ref{sec:lower_bound_Kernel} and \ref{sec:upper_bound_Kernel}.

\subsection{Background on kernel methods}
\label{sec:KernelBackground}

We consider a general RKHS defined on the probability space $(\cX_d, \nu_d)$, via
 $\Hop_d$ a compact self-adjoint positive definite operators: $\Hop_d : \cD_d \to \cD_d$ with kernel representation
\[
\Hop_d g ( \bx_1) = ~ \int_{\cX_d} H_d(\bx, \bx') g(\bx') \nu_d(\de \bx')\, ,
\]
where $H_d : \cX_d \times \cX_d \to \R$ is a square integrable function $H_d \in L^2( \cX_d \times \cX_d)$, with the property that $\int_{\cX_d} H_d(\bx, \bx') g(\bx') \nu_d(\de \bx') = 0 $ for $g \in \cD_d^\perp$.

Given a loss function $\ell:\reals\times\reals\to\reals_{\ge 0}$ a general kernel method
learns the function
\begin{align}
\hf_{\lambda} = \arg\min_{f}\left\{\sum_{i=1}^n\ell(y_i,f(\bx_i)) +\lambda\|f\|_{\cH}^2\right\}\, ,\label{eq:KernelMethods}
\end{align}
where $\|f\|_{\cH}$ is the RKHS norm associated to $H_d$. Kernel ridge regression corresponds to the special
case $\ell(y,\hy) = (y-\hy)^2$. As before, we will evaluate a kernel methods via their test error,
which we denote as follows in the case of KRR
\begin{align}\label{eqn:test_error_KRR}
R_\KR(f_d, \bX, \lambda) :=&~ \E_\bx\Big[ \Big(f_d(\bx) - \hf_{\lambda}(\bx)\Big)^2 \Big]\, .
\end{align}

As mentioned above any kernel method can be seen as the $N\to\infty$ limit of a RF model.
To see this, note any positive semidefinite kernel can be written in the form
$H_d(\bx_1, \bx_2) = \E_{\btheta \sim \tau_d}[\sigma_d(\bx_1, \btheta) \sigma_d(\bx_2, \btheta)]$,
for some activation function $\sigma_d$, and some probability space $(\Omega_d,\tau_d)$.
This is akin to  taking the square root of a matrix and ---as in the finite-dimensional case--- the square root is not unique.
For instance, we can let $\sigma_d$  be the symmetric square root $H_d(\bx_1, \bx_2)$
replacing $\lambda^2_{d,j}$ by $\lambda_{d,j}$ in Eq.~\eqref{eq:Hdecomposition}.

Given a choice of this square root, we can rewrite the estimator \eqref{eq:KernelMethods}
as $\hf_{\lambda}(\bx) = f(\bx;\ha_{\lambda})$, where $\ha_{\lambda}\in L^2(\Omega_d;\nu_d)$ and
\begin{align}
  \ha_{\lambda} &= \arg\min_{a}\left\{\sum_{i=1}^n\ell(y_i,f(\bx_i;a)) +\lambda\| a \|_{L^2}^2\right\}\, ,\\
  &f(\bx;a):= \int\sigma_d(\bx;\btheta)\, a(\btheta)\, \tau_d(\de\btheta)\, .
\end{align}
This can be informally seen as the $N\to\infty$ limit of Eq.~\eqref{eqn:hbiota_RFRR} if we choose the square loss function.

\subsection{Assumptions on the kernel}
\label{sec:assumptions_Kernel}

As for the case of RFRR, we collect our assumptions in two groups. The first one is mainly concerned with the
concentration properties of the kernel, which are quantified in terms of the sequences of integers $n(d)$, $\evn(d)$.
\begin{assumption}[$\{n(d),\evn(d)\}_{d \ge 1}$-Kernel Concentration Property]\label{ass:KRR} We say that the sequence of
  operators $\{ \Hop_d \}_{d\ge 1}$ satisfies the \textit{Kernel Concentration Property (KCP)} with respect to the sequence
  $\{ (n(d), \evn(d)) \}_{d \ge 1}$ if there exists a sequence of integers $\{u(d)\}_{d\ge 1}$ with $u(d) \ge \evn(d)$ such that
  the following conditions hold.
\begin{itemize}
\item[(a)] (Hypercontractivity of finite eigenspaces.) For any fixed $q \ge 1$, there exists a constant
  $C$ such that, for any $h\in \cD_{d, \le u(d)}= \spn ( \psi_s, 1 \leq s \leq u(d) )$, we have
\begin{align}\label{eqn:hypercontractivity_in_KRR}
\| h\|_{L^{2q}} \le C \cdot \| h\|_{L^2}.  
\end{align}
\item[(b)] (Properly decaying eigenvalues.) There exists fixed $\delta_0 > 0$, such that,  for all $d$ large enough,
\begin{align}
  n(d)^{2 + \delta_0} \le&~   \frac{\Big(\sum_{j=u(d)+1}^{\infty}\lambda_{d,j}^4\Big)^2}
                                                           {\sum_{j=u(d)+1}^{\infty}\lambda_{d,j}^8}, \label{eq:ass_kernel_b1} \\
  n(d)^{2 + \delta_0} \le&~    \frac{\Big(\sum_{j=u(d)+1}^{\infty}\lambda_{d,j}^2\Big)^2}
                                                           {\sum_{j=u(d)+1}^{\infty}\lambda_{d,j}^4}\, . \label{eq:ass_kernel_b2}
\end{align}

\item[(c)](Concentration of diagonal elements of kernel)
 For $(\bx_i)_{i \in [n(d)]} \sim_{iid} \nu_d$, we have: 
\begin{align}
  \max_{i \in [n(d)]} \Big\vert \E_{\bx \sim \nu_d}\big[H_{d, > \evn(d)}(\bx_i, \bx)^2 \big] - \E_{\bx, \bx' \sim \nu_d}\big[H_{d, > \evn(d)}(\bx, \bx')^2\big] \Big\vert = &o_{d, \P}(1) \cdot \E_{\bx, \bx' \sim \nu_d}\big[ H_{d, > \evn(d)}(\bx, \bx')^2\big], \label{eq:ass_kernel_e1}\\
  \max_{i \in [n(d)]} \Big\vert H_{d, > \evn(d)}(\bx_i, \bx_i) - \E_{\bx}[H_{d, > \evn(d)}(\bx, \bx)] \Big\vert =&~ o_{d, \P}(1) \cdot \E_{\bx}[H_{d, > \evn(d)}(\bx, \bx)]. \label{eq:ass_kernel_e2}
\end{align}
\end{itemize}
\end{assumption}
In the last definition, assumptions $(a)$ and $(c)$ have an interpretation that is similar to the one for RFRR.
Namely, assumption $(a)$ requires that the top eigenvectors of $\Hop_d$ are delocalized, and assumption
$(c)$ requires that `most points' in the sample space $\cX_d$ behave similarly, in the sense
of having similar values of the kernel diagonal $H_d(\bx,\bx)$. Condition $(b)$ is very mild in
high dimension, and concerns the tail of eigenvalues of $\Hop_d$.

The next condition essentially connects the sample size $n(d)$ to the eigenvalue index $\evn(d)$,
via the eigenvalues sequence.
\begin{assumption}[Eigenvalue condition at level $\{(n(d),\evn(d))\}_{d \ge 1}$]\label{ass:eig_decay} We say that the sequence of Kernel operators $\{ \Hop_d \}_{d\ge 1}$ satisfies the \textit{Eigenvalue Condition} at level $\{ (n(d), \evn(d)) \}_{d \ge 1}$ if the following conditions hold
  for all $d$ large enough.
 \begin{itemize}
\item[(a)] There exists fixed $\delta_0 > 0$, such that 
\begin{align}
n(d)^{1+\delta_0} \le&~  \frac{1}{\lambda_{d,\evn(d)+1}^4}
                     \sum_{k=\evn(d)+1}^{\infty}\lambda_{d,k}^4, \label{ass:n_parameters_KRR_lower_1}\\
  n(d)^{1+\delta_0} \le&~   \frac{1}{\lambda_{d,\evn(d)+1}^2}
\sum_{k=\evn(d)+1}^{\infty}\lambda_{d,k}^2.  \label{ass:n_parameters_KRR_lower_2}
\end{align}
\item[(b)] There exists fixed $\delta_0 > 0$, such that 
\[
\evn(d) \le n(d)^{1 - \delta_0}.
\]
\end{itemize}
\end{assumption}
Unlike in the case of RFRR, we do not require the existence of a spectral gap, but we assume two different upper bounds
$n(d)$ to hold simultaneously. In many cases of interest, the right hand
sides of \eqref{ass:n_parameters_KRR_lower_1} and \eqref{ass:n_parameters_KRR_lower_2} have roughly the same value,
which is given by the number of eigenvalues between $\lambda_{d,\evn(d)+1}$  and
$c_0\lambda_{d,\evn(d)+1}$ for a small $c_0$ (counting degeneracy).
The technical requirement $(b)$ is mild and we do not know of any interesting counterexample.
 
\subsection{Lower bound for general kernel methods}
\label{sec:lower_bound_Kernel}

Consider any regression method of the form \eqref{eq:KernelMethods}.
By the representer theorem, there exist coefficients
$\hat \zeta_1,\dots, \hat \zeta_n$ such that 
\begin{align}
\hf_{\lambda}(\bx) = \sum_{i=1}^n \hat \zeta_i \, H_d(\bx, \bx_i)\, .  \label{eq:KRRf}
\end{align}
We are therefore led to define the following data-dependent prediction risk function for kernel methods
\begin{align}\label{eqn:test_error_Kernel_Method}
R_{H}(f_d,\bX):= \min_{\bzeta} \E_\bx\Big\{ \Big(f_d(\bx) - \sum_{i=1}^n \zeta_i  H_d(\bx_i, \bx) \Big)^2 \Big\}.
\end{align}
This is a lower bound on the prediction error of any kernel methods of the form \eqref{eq:KernelMethods}.

The next theorem provides a lower bound on the generalization of kernel methods that is a consequence of the approximation bound in Theorem \ref{thm:RF_lower_upper_bound}.$(a)$ derived for the random features model, in Appendix \ref{sec:approximation_error_RFK}.

\begin{theorem}\label{thm:KR-general-main}
  Let $\{ f_d \in \cD_d \}_{d \ge 1}$ be a sequence of functions, $(\bx_i)_{i \in [n(d)]} \sim \nu_d$ independently, $\{ \Hop_d \}_{d\ge 1}$ be a sequence of kernel operators such that $\{ ( \Hop_d , n(d),\evn(d) )\}_{d \ge 1}$ satisfies Eqs.~\eqref{eqn:hypercontractivity_in_KRR},
  \eqref{eq:ass_kernel_b1}, \eqref{eq:ass_kernel_e1}, and \eqref{ass:n_parameters_KRR_lower_1}. Then we have (cf. Eq. \eqref{eqn:test_error_Kernel_Method})
\begin{align}
\Big \vert R_{H}(f_{d}, \bX) - R_{H}(\proj_{\le\evn(d)} f_d, \bX) - \| \proj_{>\evn(d)} f_d \|_{L^2}^2 \Big \vert &\le o_{d, \P}(1) \cdot \| f_d \|_{L^2} \| \proj_{>\evn(d)} f_d \|_{L^2}.
\label{eq:KRR_lower_bound}
\end{align}
\end{theorem}

\begin{proof}
  This follows immediately from Theorem \ref{thm:RF_lower_upper_bound} $(a)$ stated in Appendix \ref{sec:approximation_error_RFK}. Indeed, setting $\sigma_d(\bx, \bx') = H_d(\bx, \bx')$, we obtain $R_{H}(f_d, \bX) =R_{\RF}(f_d, \bX)$, whence the claim follows by applying Eq.~(\ref{eq:RC_lower_bound}).
\end{proof}

Notice that $R_{H}(\proj_{\le\evn(d)} f_d, \bX) \ge 0$ by construction and therefore this theorem immediately implies a lower bound of the test error of kernel ridge regression (cf. Eq. \eqref{eqn:test_error_KRR})
\begin{align}
  R_{\KR}(f_d;\bX,\lambda) \ge R_{H}(f_{d}, \bX)\ge  \| \proj_{>\evn(d)} f_d \|_{L^2}^2-o_{d, \P}(1) \cdot \| f_d \|_{L^2} \| \proj_{>\evn(d)} f_d \|_{L^2}.
  \end{align}
  In words, if we neglect the error term $o_{d, \P}(1) \cdot \| f_d \|_{L^2} \| \proj_{>\evn(d)} f_d \|_{L^2}$,
  no kernel method can achieve non-trivial accuracy on the projection of $f_d$ onto eigenvectors beyond the first
  $\evn(d)$. 
  
\subsection{The risk of kernel ridge regression}
\label{sec:upper_bound_Kernel}

Kernel ridge regression is one specific way of selecting the coefficients $\hbzeta$ in Eq.~\eqref{eq:KRRf}, namely
by using $\ell(\hy,y) = (\hy-y)^2$ in Eq.~\eqref{eq:KernelMethods}. Solving for the coefficients yields
\[
\hbzeta = (\bH + \lambda \id_n)^{-1} \by,
\]
where the kernel matrix $\bH = (H_{ij})_{ij \in [n]}$ is given by
$H_{ij} =  H_d(\bx_i, \bx_j)$,
and $\by = (y_1, \ldots, y_n)^\sT$. 

It is convenient to state our main results in terms of an effective ridge regression estimator 
\begin{align}
\hf^{\seff}_{\gamma} = \arg\min_{f}\Big\{\|f_d-f\|_{L^2}^2 +\frac{\gamma}{n} \|f\|_{\cH}^2\Big\}\, ,\label{eq:KernelPop}
\end{align}
This amounts to replacing the empirical risk  in Eq.~\eqref{eq:KernelMethods} by its population counterpart
$\|f_d-f\|_{L^2}^2=\E\{(f_d(\bx)-f(\bx))^2\}$. Also note that the regularization
  parameter does not coincide with $\lambda$: its precise value will be specified below.

  The solution of the population ridge problem \eqref{eq:KernelPop} can be explicitly written in terms of a
  shrinkage operator in the basis of eigenfunctions of $\Hop_d$:
  \begin{align}\label{eq:sol_population_KRR}
    f(\bx) = \sum_{\ell=1}^{\infty}c_\ell \psi_{d,\ell}(\bx)\;\;\mapsto \;\;
    \hf^{\seff}_{\gamma}(\bx) = \sum_{\ell=1}^{\infty}\frac{\lambda_{d,\ell}^2}{\lambda_{d,\ell}^2+\frac{\gamma}{n}}\, c_\ell
    \psi_{d,\ell}(\bx)\, .
  \end{align}

  \begin{theorem}\label{thm:upper_bound_KRR}
    Let $\{ f_d \in \cD_d \}_{d \ge 1}$ be a sequence of functions, $(\bx_i)_{i \in [n(d)]} \sim \nu_d$ independently, $\{ \Hop_d \}_{d\ge 1}$ be a sequence of kernel operators such that $\{ ( \Hop_d , n(d),\evn(d) )\}_{d \ge 1}$ satisfies
    $\{n(d),\evn(d) \}_{d \ge 1}$-KPCP (Assumption \ref{ass:KRR}) and eigenvalue condition at level $\{n(d),\evn(d) \}_{d \ge 1}$ (Assumption \ref{ass:eig_decay}).
Define the effective regularization
    \begin{align}
      \gamma^{\seff}:=\lambda + \Trace ( \Hop_{d, >\evn(d)})\, .
    \end{align}
      Then, for any regularization parameter $\lambda\in [0, \lambda_\star]$ where $\lambda_\star = \Trace(\Hop_{d, > \evn(d)})$,
    any $\eta> 0$, we have (cf. Eq. (\ref{eqn:test_error_KRR}))
\begin{align}
  \Big\vert R_{\KR}(f_d, \bX, \lambda) -\|f_d-\hf_{\gamma^{\seff}}^{\seff}\|_{L^2} \Big\vert = o_{d, \P}(1) \cdot ( \| f_d \|_{L^2}^2 + \| \proj_{>\evn} f_d \|_{L^{2 + \eta} }^2 +
  \noise^2). 
\end{align}
Further, the ridge regression estimator $\hf_{\lambda}$ is close to the effective estimator
$\hf_{\gamma^{\seff}}^{\seff}$, namely
\begin{align}
  \big\|\hf_{\lambda} - \hf_{\gamma^{\seff}}^{\seff}\big\|_{L^2}^2 = o_{d, \P}(1) \cdot ( \| f_d \|_{L^2}^2 + \| \proj_{>\evn} f_d \|_{L^{2 + \eta} }^2 +
  \noise^2). 
  \end{align}
\end{theorem}
The proof of Theorem \ref{thm:upper_bound_KRR} is deferred to Appendix \ref{sec:proof_KR}. 

In words, KRR behaves as ridge regression with respect to the population risk,
except that the regularization parameter is increased by $\Trace(\Hop_{d, > \evn})$.
The underlying mechanism is quite simple. The empirical kernel matrix is decomposed as $\bH = \bH_{\le \evn}+
\bH_{> \evn}$, and the second component can be approximated by a multiple of the identity:
$\bH_{> \evn}\approx \Trace ( \Hop_{d, >\evn})\cdot \id_n$. This term acts as an additional ridge regularizer.

As mentioned above, we do not assume here any eigenvalue gap condition. However,
formulas simplify if we assume an eigenvalue gap, e.g.:
\[
  n(d) = \omega_d(1) \cdot \frac{1}{\lambda_{d,\evn(d)+1}^2}
  \sum_{k=\evn(d)+1}^{\infty} \lambda^2_{d,k}\, .
\]
Under this additional assumption, Theorem \ref{thm:upper_bound_KRR}
implies the following simplified formula for the test error:
\[
\Big\vert R_{\KR}(f_d, \bX, \lambda)  - \| \proj_{> \evn(d)} f_d \|_{L^2}^2 \Big\vert = o_{d, \P}(1) \cdot ( \| f_d \|_{L^{2+ \eta}}^2 + \noise^2). 
\]
As anticipated, this coincides with the risk of RFRR, if we heuristically set $N=\infty$ in
Theorem \ref{thm:RFK_generalization}.

\section*{Acknowledgnements}

This work was supported by NSF through award DMS-2031883 and from the Simons Foundation through Award 814639 for the
Collaboration on the Theoretical Foundations of Deep Learning
We also acknowledge  NSF grants CCF-2006489, IIS-1741162 and the ONR
grant N00014-18-1-2729.

\bibliographystyle{amsalpha}

\newcommand{\etalchar}[1]{$^{#1}$}
\providecommand{\bysame}{\leavevmode\hbox to3em{\hrulefill}\thinspace}
\providecommand{\MR}{\relax\ifhmode\unskip\space\fi MR }
\providecommand{\MRhref}[2]{%
  \href{http://www.ams.org/mathscinet-getitem?mr=#1}{#2}
}
\providecommand{\href}[2]{#2}

\clearpage

\appendix

\section{Approximation error of random features model}
\label{sec:approximation_error_RFK}

In this section, we consider the approximation error of the random features function class. Formally, the approximation error can be seen as the generalization error of random features ridge regression for finite number of neurons $N < \infty$ and infinite data $n = \infty$. However, we cannot apply directly Theorem \ref{thm:RFK_generalization} with $n = \infty$. We therefore state a separate theorem. This is also used to prove the lower bound of Theorem \ref{thm:KR-general-main} on the generalization error of general kernel methods.

In Section \ref{sec:theorem_approximation}, we state our assumptions and theorem. Sections \ref{sec:RF_approx_lower} and \ref{sec:RF_approx_upper} provide a proof of the theorem, while Section \ref{sec:concentration_approx} gathers key technical concentration results that will also be used in the proofs of Theorem \ref{thm:RFK_generalization} and Theorem \ref{thm:upper_bound_KRR}.

\subsection{Assumptions and theorem}
\label{sec:theorem_approximation}

Recall the definition of the random features function class (see Section \ref{sec:RFSetting}): let $\bTheta = (\btheta_i)_{i \in [N]} \sim_{iid} \tau_d$,
\[
\cF_{\RF, N}(\bTheta) = \Big\{ \hf(\bx;\ba) = \sum_{i=1}^N a_i \sigma_d(\bx; \btheta_i): a_i \in \R, i \in [N] \Big\}.
\]
We define the approximation error of the random features function class for a target function $f_d \in L^2(\cX_d)$ as
\begin{equation}\label{eq:def_app_error}
R_{\App}(f_d, \bTheta) : = \inf_{\hat f \in \cF_{\RF, N}(\bTheta)} \E_{\bx \sim \tau_d}[(f_d(\bx) - \hat f(\bx))^2]. 
\end{equation}

Similarly to Sections \ref{sec:RF_Assumptions} and \ref{sec:assumptions_Kernel}, we will quantify our assumptions on the sequences of probability spaces $(\cX_d,\nu_d)$ and $(\Omega_d , \tau_d)$, and on the activation functions $\sigma_d \in L^2 ( \cX_d \times \Omega_d)$, in terms  of the sequences of integers $N(d) , \evN(d)$. We state the assumptions in two groups: Assumption \ref{ass:hyper_K_stronger} and Assumption \ref{ass:N_parameters} deal respectively with the concentration properties and the spectrum of the sequence of feature kernel operator $\{ U_d \}_{d\ge 1}$ defined as
\[
U_d ( \btheta_1 , \btheta_2 ) = \E_{\bx\sim \nu_d} [ \sigma_d ( \bx ; \btheta_1) \sigma_d ( \bx ; \btheta_2 ) ].
\]

\begin{assumption}[Feature kernel concentration at level $\{(N(d) , \evN (d)) \}_{d \ge 1}$]\label{ass:hyper_K_stronger}
The sequences of spaces $\{ \cV_d \}_{d\ge 1}$, operators $\{ \Uop_d \}_{d \ge 1}$ and numbers of neurons $\{ N(d) \}_{d \ge 1}$ satisfy feature kernel concentration at level $\{\evN (d) \}_{d \ge 1}$ if there exists a sequence $\{ u(d) \}_{d \ge 1}$ with $u(d) \ge \evN(d)$, such that the following hold.
\begin{itemize}
\item[(a)] (Hypercontractivity of finite eigenspaces.) For any fixed $q \ge 1$, there exists $C$ such that, for any $g \in \cV_{d, \le u(d)} = \spn(\phi_s, 1 \le s \le u(d))$, we have
\[
\begin{aligned}
\| g \|_{L^{2q} ( \Omega_d) } \le&~ C \cdot \| g \|_{L^2 ( \Omega_d) }. \\
\end{aligned}
\]
\item[(b)] (Properly decaying eigenvalues.) There exists a fixed $\delta_0 > 0$, such that 
\[
  N(d)^{2 + \delta_0} \le \frac{\Big(\sum_{j=u(d)+1}^{\infty} \lambda_{d,j}^2\Big)^2}{\sum_{j=u(d)+1}^{\infty} \lambda_{d,j}^4}\, .
\]
\item[(c)] (Upper bound on the diagonal elements of the kernel) For $(\btheta_i)_{i \in [N(d)]} \sim_{iid} \tau_d$ and any $\delta >0$, we have
\[
\begin{aligned}
\max_{i \in [N(d)]} U_{d, > \evN (d)}(\btheta_i, \btheta_i) =&~ O_{d, \P}(N(d)^\delta) \cdot \E_{\btheta}[U_{d, > \evN (d)}(\btheta, \btheta)]. \\
\end{aligned}
\]
\item[(d)] (Lower bound on the diagonal elements of the kernel) For $(\btheta_i)_{i \in [N(d)]} \sim_{iid} \tau_d$ and any $\delta >0$, we have
\[
\begin{aligned}
\min_{i \in [N(d)]} U_{d, > \evN (d)}(\btheta_i, \btheta_i) =&~ \Omega_{d, \P}(N(d)^{-\delta}) \cdot \E_{\btheta}[U_{d, > \evN (d)}(\btheta, \btheta)]. \\
\end{aligned}
\]
\end{itemize}
\end{assumption}

\begin{assumption}[Spectral gap at level $\{ (N(d) ,\evN(d)) \}_{d \ge 1}$]\label{ass:N_parameters}
The sequence of operators $\{ \Uop_d \}_{d \ge 1}$ has a spectral gap at level $\{ (N(d), \evN(d)) \}_{d \ge 1}$ if the following hold.
\begin{itemize}
\item[(a)] There exists a fixed $\delta_0 > 0$, such that 
\begin{align*}
N(d)^{1+\delta_0} \le \frac{1}{\lambda_{d,\evN(d)+1}^2}
  \sum_{j=\evN(d)+1}^{\infty}\lambda_{j,d}^2\,  .
\end{align*}
\item[(b)] There exists a fixed $\delta_0 > 0$, such that $\evN (d) \le N(d)^{1 - \delta_0}$ and
\[
N(d)^{1 - \delta_0} \ge \frac{1}{\lambda_{d,\evN(d)}^2}
  \sum_{j=\evN(d)+1}^{\infty}\lambda_{j,d}^2\, .
\] 
\end{itemize}
\end{assumption}

\begin{remark} In Assumption \ref{ass:hyper_K_stronger}.$(c)$, we can replace $U_{d, > \evN (d)}$ by $U_{d, > u (d)}$ (see Lemma \ref{lem:bound_max_hyper}).

\end{remark}


We are now in position to state our theorem on the approximation error of the random features function class. We state the lower and upper bounds and their assumptions separately.

\begin{theorem}[Approximation error of the random features function class]\label{thm:RF_lower_upper_bound}
Let $\{ f_d \in \cD_d \}_{d \ge 1}$ be a sequence of functions and $\bTheta = (\btheta_i)_{i \in [N(d)]}$ with $(\btheta_i)_{i \in [N(d)]} \sim \tau_d$ independently. Let $\{ \sigma_d \}_{d \ge 1}$ be a sequence of activation functions satisfying Assumptions \ref{ass:hyper_K_stronger}.(a) and  \ref{ass:hyper_K_stronger}.$(b)$ at level $\{ \evN (d) \}_{d \ge 1}$. Then the  following hold for the approximation error of the random features class (see Eq.~\eqref{eq:def_app_error}):
\begin{itemize}
\item[(a)] (Lower bound) If $\{ \sigma_d \}_{d \ge 1}$ satisfies further Assumptions \ref{ass:hyper_K_stronger}.$(d)$ and \ref{ass:N_parameters}.(a), then we have 
\begin{align}
\Big \vert R_{\App}(f_{d}, \bTheta) - R_{\App}(\proj_{\le\evN(d)} f_d, \bTheta) - \| \proj_{>\evN (d)} f_d \|_{L^2}^2 \Big \vert &\le o_{d, \P}(1) \cdot \| f_d \|_{L^2} \| \proj_{>\evN(d)} f_d \|_{L^2}.
\label{eq:RC_lower_bound}
\end{align}
\item[(b)] (Upper bound)  If $\{ \sigma_d \}_{d \ge 1}$ satisfies further Assumptions \ref{ass:hyper_K_stronger}.(c) and \ref{ass:N_parameters}, then we have 
\begin{align}
\Big \vert R_{\App}(\proj_{\le\evN(d)} f_d, \bTheta) \Big \vert &\le o_{d, \P}(1) \cdot \| f_d \|_{L^2} \| \proj_{\le \evN(d)} f_d \|_{L^2}.
\label{eq:RC_upper_bound}
\end{align}
\end{itemize}
\end{theorem}
Point $(a)$  is proved in Section \ref{sec:RF_approx_lower}, while point $(b)$ is proved in Section \ref{sec:RF_approx_upper}. 

The lower bound on general kernel methods in Theorem \ref{thm:KR-general-main} is obtained as a direct consequence of Theorem \ref{thm:RF_lower_upper_bound}.(a), by taking $\sigma_d ( \bx , \bx ' ) = H_d ( \bx , \bx')$. Indeed, it is easy to check that Eqs.~\eqref{eqn:hypercontractivity_in_KRR} and \eqref{eq:ass_kernel_b2} imply Assumptions \ref{ass:hyper_K_stronger}.(a) and \ref{ass:hyper_K_stronger}.$(b)$, Eq.~\eqref{eq:ass_kernel_e2} implies Assumptions \ref{ass:hyper_K_stronger}.$(c)$ and \ref{ass:hyper_K_stronger}.$(d)$, and Eq.~\eqref{ass:n_parameters_KRR_lower_2} implies Assumption \ref{ass:N_parameters}.(a).

\subsection{Proof of Theorem \ref{thm:RF_lower_upper_bound}.$(a)$: lower bound on the approximation error}\label{sec:RF_approx_lower}

We denote $\E_\btheta$ to be the expectation operator with respect to $\btheta \sim \tau_d$, $\E_\bx$ to be the expectation operator with respect to $\bx \sim \nu_d$.  We will denote $\evN = \evN(d)$ and $N = N(d)$. 

Define the random vectors $\bV = (V_1, \ldots, V_N)^\sT$, $\bV_{\le \evN} = (V_{1, \le\evN}, \ldots, V_{N, \le \evN})^\sT$, $\bV_{> \evN} = (V_{1, >\evN}, \ldots, V_{N, >\evN})^\sT$, with
\begin{align*}
V_{i, \le\evN} \equiv&~  \E_{\bx \sim \nu_d}[[\proj_{\le\evN} f_d](\bx)  \sigma_d(\bx; \btheta_i) ],\\
V_{i, >\evN} \equiv&~ \E_{\bx \sim \nu_d}[[\proj_{>\evN} f_d](\bx) \sigma_d(\bx; \btheta_i) ],\\
V_i \equiv&~  \E_{\bx \sim \nu_d} [ f_d(\bx) \sigma_d (\bx; \btheta_i) ] = V_{i, \le\evN } + V_{i, >\evN}. 
\end{align*}
Define the random matrix $\bU = (U_{ij})_{i, j \in [N]}$, with 
\begin{align}
U_{ij} = \E_{\bx \sim \nu_d} [ \sigma_d ( \bx; \btheta_i)  \sigma_d ( \bx ; \btheta_j)]. 
\label{eq:KernelMatrix}
\end{align}
In what follows, we write $R_{\App}(f_d) = R_{\App} (f_d, \bTheta)$ for the approximation error of the random features model, omitting the dependence on the weights $\bTheta$. By definition and a simple calculation, we have 
\[
\begin{aligned}
R_{\App}(f_d) =& \min_{\ba \in \R^N} \Big\{ \E_{\bx}[f_d(\bx)^2] - 2 \< \ba, \bV \> + \< \ba, \bU \ba\> \Big\} = \E_{\bx}[f_d(\bx)^2] - \bV^\sT \bU^{-1} \bV,\\
R_{\App}(\proj_{\le\evN} f_d) =& \min_{\ba \in \R^N} \Big\{ \E_{\bx}[\proj_{\le\evN} f_d(\bx)^2] - 2 \< \ba, \bV_{\le\evN} \> + \< \ba, \bU \ba\> \Big\} = \E_{\bx}[\proj_{\le\evN} f_d(\bx)^2] - \bV_{\le\evN}^\sT \bU^{-1} \bV_{\le\evN}. 
\end{aligned}
\]
By orthogonality, we have
\[
\E_{\bx}[f_d(\bx)^2] = \E_{\bx}[[\proj_{\le\evN} f_d](\bx)^2] + \E_{\bx}[[\proj_{>\evN} f_d](\bx)^2], 
\]
which gives
\begin{equation}\label{eqn:decomposition_risk}
\begin{aligned}
& \Big\vert R_{\App}(f_d) - R_{\App}(\proj_{\le\evN} f_d) - \E_{\bx}[[\proj_{>\evN} f_d](\bx)^2] \Big\vert \\
=& \Big\vert \bV_{\le\evN}^\sT \bU^{-1} \bV_{\le\evN} - \bV^\sT \bU^{-1} \bV \Big\vert = \Big\vert \bV_{\le\evN}^\sT \bU^{-1} \bV_{\le\evN} - (\bV_{\le\evN} + \bV_{>\evN})^\sT \bU^{-1} (\bV_{\le\evN} + \bV_{>\evN}) \Big\vert\\
=& \Big\vert 2 \bV^\sT \bU^{-1} \bV_{>\evN} - \bV_{>\evN}^\sT \bU^{-1} \bV_{>\evN} \Big\vert \le 2 \| \bU^{-1/2} \bV_{>\evN} \|_2 \| \bU^{-1/2} \bV \|_{2} +  \| \bU^{-1} \|_{\op} \| \bV_{>\evN}\|_2^2\\
\le&  2 \| \bU^{-1/2} \|_{\op} \| \bV_{>\evN} \|_2 \| f_d \|_{L^2}+  \| \bU^{-1} \|_{\op} \| \bV_{>\evN}\|_2^2,
\end{aligned}
\end{equation}
where the last inequality used the fact that
\[
0 \le R_{\App}(f_d) = \| f_d \|_{L^2}^2 - \bV^\sT \bU^{-1} \bV, 
\]
so that
\[
\| \bU^{-1/2} \bV \|_2^2 = \bV^\sT \bU^{-1} \bV \le \| f_d \|_{L^2}^2. 
\]

By Eq. (\ref{eqn:decomposition_risk}), to prove Theorem \ref{thm:RF_lower_upper_bound}.$(a)$, we need to bound $\| \bU^{-1} \|_{\op} \| \bV_{>\evN}\|_2^2$. This is achieved in the two following propositions. 

\begin{proposition}[Expected norm of $\bV$]\label{prop:expected_V_RF}
Let $\{ f_d \in \cD_d \}$ be a sequence of target functions. Define $\cE_{> \evN}$ by
\[
\cE_{> \evN} \equiv  \E_{\btheta}\Big[ \Big( \E_\bx [ \proj_{> \evN} f_d(\bx) \sigma_d (\bx ; \btheta)  ] \Big)^2 \Big].
\]
Then we have
\[
\cE_{> \evN} \le  \|\Uop_{d, >\evN}\|_{\op} \cdot \| \proj_{>\evN} f_d \|_{L^2}^2.
\]
\end{proposition}
\begin{proof}[Proof of Proposition \ref{prop:expected_V_RF}] We have
\begin{align*}
\cE_{>\evN} \equiv&~  \E_{\btheta \sim \tau_d}[\< \proj_{> \evN} f_d, \sigma_d (\,\cdot\,, \btheta) \>_{L^2(\cX_d)}^2]\\
 =&~ \E_{\btheta \sim \tau_d} \E_{\bx_1, \bx_2 \sim \nu_d}[\proj_{> \evN} f_d(\bx_1) \sigma_d (\bx_1, \btheta ) \sigma_d (\bx_2, \btheta) \proj_{> \evN} f_d(\bx_2)]\\
=&~ \E_{\bx_1, \bx_2 \sim \nu_d}[\proj_{> \evN} f_d(\bx_1) \E_{\btheta \sim \tau_d}[ \sigma_d (\bx_1, \btheta ) \sigma_d (\bx_2, \btheta )] \proj_{> \evN} f_d(\bx_2)] \\
=&~ \<\proj_{> \evN} f_d, \Hop_d\, \proj_{> \evN} f_d \>_{L^2} = \<\proj_{> \evN} f_d, \Hop_{d, > \evN} \proj_{> \evN} f_d \>_{L^2}\\
 \le&~ \|\Hop_{d, > \evN}\|_{\op} \|\proj_{> \evN} f_d \|^2_{L^2}= \| \Uop_{d, > \evN} \|_{\op} \|\proj_{> \evN} f_d \|^2_{L^2}. 
\end{align*}
This proves the proposition. 
\end{proof}

\begin{proposition}[Lower bound on the kernel matrix] \label{prop:lower_bound_U}
Let $\{\sigma_d \}_{d \ge 1}$ be a sequence of activation functions satisfying Assumptions \ref{ass:hyper_K_stronger}.$(a)$, \ref{ass:hyper_K_stronger}.$(b)$ and \ref{ass:N_parameters}.$(a)$ at level $\{ (N(d) , \evN(d) ) \}_{d \ge 1}$. Let $(\btheta_i)_{i \in [N]} \sim \tau_d$ independently and let $\bU \in \R^{N \times N}$ be the kernel matrix defined by Eq.~\eqref{eq:KernelMatrix}. Then, we have
\begin{align}\label{eqn:def_of_delta1}
\bU \succeq \kappa_{>\evN} (\bLambda + \bDelta) , 
\end{align}
with $\bLambda = \diag((U_{d, > \evN}(\btheta_i, \btheta_i) /\kappa_{>\evN})_{i \in [N]})$, $\kappa_{>\evN} = \Trace(\Uop_{d, > \evN})$, and $\bDelta$ is such that there exists some $\delta' > 0$, such that 
\[
\begin{aligned}
\E[\| \bDelta \|_{\op}] = O_{d}(N^{-\delta'}). 
\end{aligned}
\]
\end{proposition}

\begin{proof}[Proof of Proposition \ref{prop:lower_bound_U}]
This is a direct consequence of Theorem \ref{prop:expression_U}.$(a)$.
\end{proof}

By Proposition \ref{prop:expected_V_RF}, we have
\begin{align}
  \E[\| \bV_{>\evN} \|_2^2] = N \cE_{> \evN} \le&~ N  \cdot \| \Uop_{d, > \evN}\|_{\op} \cdot \| \proj_{> \evN} f_d \|_{L^2}^2.
\end{align}
Next, by Proposition \ref{prop:lower_bound_U} and Assumption \ref{ass:hyper_K_stronger}.$(d)$, for any fixed $\delta > 0$ with $\delta < \delta'$, we have 
\[
\| \bU^{-1} \|_{\op}  \cdot \Trace(\Uop_{d, > \evN}) \le  \Big[  \min_{i \in [N]} U_{d,>\evN}(\btheta_i, \btheta_i) / \Trace(\Uop_{d, > \evN}) - O_{d, \P}(N^{-\delta'}) \Big]^{-1} \le O_{d, \P}(N^\delta), 
\]
and hence by Markov inequality we have
\begin{align}\label{eqn:bound_U_V_RC_last}
 \frac{\| \bU^{-1} \|_{\op} \| \bV_{> \evN} \|_2^2}{\| \proj_{> \evN} f_d \|_{L^2}^2}
\le O_{d, \P}(N^{\delta}) \cdot N  \cdot \frac{\|\Uop_{d, > \evN}\|_{\op} }{ \Trace(\Uop_{d, > \evN})}\, .
\end{align}
By Assumption \ref{ass:N_parameters}.$(a)$, we have $N  \cdot \|\Uop_{d, > \evN}\|_{\op} / \Trace(\Uop_{d, > \evN}) = O_d(N^{-\delta_0})$ for some $\delta_0 > 0$. Plugging this equation into Eq. (\ref{eqn:bound_U_V_RC_last}) and choosing $\delta< \delta_0$, we have
\begin{align}
 \| \bU^{-1} \|_{\op} \| \bV_{>\evN} \|_2^2  =& o_{d,\P} (1) \cdot \| \proj_{>\evN} f_d \|_{L^2}^2.  \label{eqn:bound_U_V_RC}
\end{align}
Combining Eq. (\ref{eqn:bound_U_V_RC}) with Eq. (\ref{eqn:decomposition_risk}) proves Theorem \ref{thm:RF_lower_upper_bound}.$(a)$.

\subsection{Proof of Theorem \ref{thm:RF_lower_upper_bound}.$(b)$: upper bound on the approximation error}\label{sec:RF_approx_upper}

In the following, we would like to calculate the quantity $R_{\App}(\proj_{\le \evN} f_d, \bTheta)$. We have 
\[
R_{\App}(\proj_{\le \evN} f_{d}, \bTheta) = \| \proj_{\le \evN} f_{d}  \|_{L_2}^2 - \bV_{\le \evN}^\sT \bU^{-1} \bV_{\le \evN},
\]
where $\bV_{\le \evN} = (V_{\le \evN, 1}, \ldots, V_{\le \evN, N})^\sT$ and $\bU = (U_{ij})_{ij \in [N]}$ with
\[
\begin{aligned}
V_{\le \evN, i} =& \E_{\bx \sim \nu_d}[\proj_{\le \evN} f_{d}(\bx) \sigma_{d}(\bx ; \btheta_i)],\\
U_{ij} =& \E_{\bx \sim \nu_d}[\sigma_{d}(\bx ; \btheta_i) \sigma_{d}(\bx ; \btheta_j)]. \\
\end{aligned}
\]
Recall that $(\psi_k)_{k\ge 1}$ is the orthonormal eigenbasis of $\Hop_d$. We denote the decomposition of $\proj_{\le \evN} f_{d}$
in this basis by 
\[
\proj_{\le \evN} f_{d}(\bx) = \sum_{k = 1}^\evN \< f_d, \psi_k\>_{L^2} \psi_k(\bx) \equiv \sum_{k = 1}^\evN \hat f_{k} \psi_k(\bx)\, .
\]
Recall the decomposition of $\sigma_{d}$
\[
\sigma_{d}(\bx, \btheta) = \sum_{k = 1}^\infty \lambda_{d, k} \psi_k(\bx) \phi_k(\btheta). 
\]
By orthonormality of the $(\psi_k)_{k\ge 1}$, we have 
\[
V_{\le \evN, i} = \sum_{k = 1}^\evN \hat f_{k} \lambda_{d, k} \phi_k( \btheta_i). 
\]
Define
\[
\begin{aligned}
\hat \boldf =&~ ( \hat f_1, \ldots, \hat f_{\evN} )^\sT \in \R^{\evN}, \\
\bD =&~ \diag(\lambda_{d, 1}, \ldots, \lambda_{d, \evN}) \in \R^{\evN \times \evN},\\
\bPhi =&~ (\phi_k(\btheta_i))_{i \in [N], k \in [\evN]} \in \R^{N \times \evN}, \\
\bL =&~ \bPhi \bD \in \R^{N \times \evN}. 
\end{aligned}
\]
Then we have 
\[
\bV_{\le \evN} = \Big( \sum_{k = 1}^\evN \hat f_k  \lambda_{d, k} \phi_k(\btheta_i) \Big)_{i \in [N]} = \bPhi \bD \hat \boldf = \bL \hat \boldf. 
\]

By Eq.~\eqref{eqn:def_of_delta} in Theorem \ref{prop:expression_U}, there exists $\bDelta \in \R^{N \times N}$ such that
\[
 \bU = \bPhi \bD^2 \bPhi^\sT +  \kappa_{>\evN} (\bLambda + \bDelta)= \bL \bL^\sT + \kappa_{>\evN} (\bLambda + \bDelta), 
\]
where $\kappa_{>\evN} = \Trace(\Uop_{d, > \evN})$, $\bLambda = \diag((U_{d, > \evN}(\btheta_i, \btheta_i) / \kappa_{>\evN})_{i \in [N]})$. By simple algebra, we have 
\[
\bV_{\le \evN}^\sT \bU^{-1} \bV_{\le \evN} = \hat \boldf^\sT \bS \hat \boldf, 
\]
where
\[
\bS =  \bL^\sT (\bL \bL^\sT + \kappa_{>\evN} (\bLambda + \bDelta))^{-1} \bL. 
\]
Therefore, we have
\[
\begin{aligned}
R_{\App}(\proj_{\le \evN} f_d, \bW) = & \| \proj_{\le \evN} f_{d}  \|_{L_2}^2 - \bV_{\le \evN}^\sT \bU^{-1} \bV_{\le \evN} = \| \hat \boldf \|_2^2 - \< \hat \boldf, \bS \hat \boldf\> \\
\le & \| \id_{\evN} - \bS \|_{\op} \| \hat \boldf \|_2^2 = o_{d, \P}(1) \cdot \| \proj_{\le \evN} f_d \|_{L^2}^2. 
\end{aligned}
\]
The last equation is by Lemma \ref{lem:concentration_S} which is stated and proved below. This proves the theorem. 

\begin{lemma}[Concentration of $\bS$]\label{lem:concentration_S}
Let Assumptions \ref{ass:hyper_K_stronger}.$(a)$, \ref{ass:hyper_K_stronger}.$(b)$, \ref{ass:hyper_K_stronger}.$(c)$ and \ref{ass:N_parameters} hold. Then we have 
\[
\| \id_{\evN} - \bS \|_{\op} = o_{d, \P}(1). 
\]
\end{lemma}

\begin{proof}[Proof of Lemma \ref{lem:concentration_S}]

By the Sherman-Morrison-Woodbury formula, we have 
\[
\id_{\evN} - \bS = \id_{\evN} - \bL^\sT (\bL \bL^\sT + \kappa_{>\evN} (\bLambda + \bDelta))^{-1} \bL = (\id_{\evN} + \bL^\sT (\bLambda + \bDelta)^{-1} \bL / \kappa_{>\evN})^{-1},
\]
so that 
\begin{equation}\label{eqn:norm_S_minus_I}
\| \id_{\evN} - \bS \|_{\op} \le 1/ \lambda_{\min}( \bL^\sT (\bLambda + \bDelta)^{-1} \bL / \kappa_{>\evN}). 
\end{equation}
Note that we have 
\begin{equation}\label{eqn:YY_2}
\begin{aligned}
&~\lambda_{\min}(\bL^\sT (\bLambda + \bDelta)^{-1} \bL) / \kappa_{>\evN} = \lambda_{\min}(\bD \bPhi^\sT (\bLambda + \bDelta)^{-1} \bPhi \bD) / \kappa_{>\evN}  \\
\ge&~ \lambda_{\min}(\bPhi^\sT \bPhi / N) \cdot [ N \cdot\lambda_{\min}(\bD^2) / \kappa_{>\evN}] / \|\bLambda + \bDelta \|_{\op}  \\
=&~  \lambda_{\min} (\bPhi^\sT \bPhi / N) \cdot [N \cdot \lambda_{\min}(\Uop_{d, \le \evN}) / \kappa_{>\evN} ] / \|\bLambda + \bDelta \|_{\op} .
\end{aligned}
\end{equation}
By Theorem \ref{prop:expression_U}.$(b)$, we have
\[
\lambda_{\min} (\bPhi^\sT \bPhi / N) = \Theta_{d, \P}(1). 
\]
By Assumption \ref{ass:hyper_K_stronger}.$(c)$, we have $\| \bLambda \|_{\op} = O_{d,\P} (N^\delta)$ for any $\delta>0$. Therefore, by Theorem \ref{prop:expression_U}.$(a)$, for any $\delta > 0$, we have
\[
\|\bLambda + \bDelta \|_{\op} \le \| \bLambda \|_{\op} + \| \bDelta \|_{\op} = O_d(N^{\delta}). 
\]
By Assumption \ref{ass:N_parameters}.$(b)$, there exists $\delta_0 > 0$, such that
\[
[N \cdot \lambda_{\min}(\Uop_{d, \le \evN}) / \kappa_{>\evN} ] = \Omega_d(N^{\delta_0}). 
\]
Combining the above equalities with Eq. (\ref{eqn:YY_2}) and choosing $\delta$ such that $0 < \delta < \delta_0$, we have
\[
\lambda_{\min}( \bL^\sT (\bLambda + \bDelta)^{-1} \bL / \kappa_{>\evN}) = \omega_{d, \P}(1). 
\]
Combining with Eq. (\ref{eqn:norm_S_minus_I}) proves the lemma. 
\end{proof}

\subsection{Structure of the empirical kernel matrix}
\label{sec:concentration_approx}

In this section, we present a key theorem describing the structure of the empirical kernel matrix $\bU = ( U ( \btheta_i , \btheta_j ) )_{i,j \in [N]} \in \R^{N\times N}$. The proof of this theorem relies on two propositions: Proposition \ref{prop:YY_new} shows that the matrix of the top eigenvectors evaluated on the random weights $(\btheta_i)_{i \in [N]}$ is nearly orthogonal and is presented in Section \ref{sec:concentration_low_degree}; Proposition \ref{prop:generalized_Gram} shows the concentration to zero in operator norm of the off-diagonal part of the matrix $\bU_{>\evN}$ and is presented in Section \ref{sec:vanishing_off_diag}. The proof of Proposition \ref{prop:generalized_Gram} is deferred to Section \ref{sec:GeneralizedGram}.

\begin{theorem}[Structure of the empirical kernel matrix]\label{prop:expression_U}
Let Assumptions \ref{ass:hyper_K_stronger}.$(a)$, \ref{ass:hyper_K_stronger}.$(b)$ and \ref{ass:N_parameters}.$(a)$ hold. Let $(\btheta_i)_{i \in [N]} \sim \tau_d$ independently, and define $\bU = (U_{ij})_{ij \in [N]}$ with 
\[
  U_{ij}: = U_d(\btheta_i,\btheta_j)\, .
\]
(Recall that $U_d(\btheta_i,\btheta_j) \equiv\E_{\bx \sim \nu_d}[\sigma_d( \bx, \btheta_i) \sigma_d( \bx, \btheta_j)].$)
Then, we can rewrite $\bU$  (by choosing $\bDelta \in \R^{N \times N}$)
\begin{align}\label{eqn:def_of_delta}
\bU = \bPhi \bD^2 \bPhi^\sT + \kappa_{>\evN} (\bLambda + \bDelta), 
\end{align}
with $\kappa_{>\evN} = \Trace(\Uop_{d, > \evN})$ and
\[
\bPhi =( \phi_k(\btheta_i) )_{i \in [N], k \in [\evN]}, ~~~ \bD = \diag(\lambda_{d, 1}, \ldots, \lambda_{d, \evN}), ~~~ \bLambda = \diag((U_{d, > \evN}(\btheta_i, \btheta_i) / \kappa_{>\evN})_{i \in [N]}).
\]
The following hold:
\begin{enumerate}
\item[$(a)$] There exists a fixed $\delta' > 0$, such that 
\[
\begin{aligned}
\E[\| \bDelta \|_{\op}] = O_{d}(N^{-\delta ' }). 
\end{aligned}
\]

\item[(b)] If further we assume $\evN (d) \leq N(d)^{1 - \delta_0}$ for a fixed $\delta_0 >0$, then we have
\[
\Big\| \bPhi^\sT \bPhi / N - \id_{\evN} \Big\|_{\op} = o_{d,\P}(1).
\]

\end{enumerate} 
\end{theorem}

\begin{proof}[Proof of Theorem \ref{prop:expression_U}]
  For $S\subseteq \{1,2,3,\dots\}$, recall that
  \[
    \Uop_{d,S}\equiv\sum_{s\in S}\lambda_{d,s}^2\phi_s\phi_s^*\,,
    \]
and let $U_{d, S}$ to be the kernel associated to $\Uop_{d, S}$. Define $\bQ_{S} = (Q_{S, ij})_{i, j \in [N(d)]}$ by
\[
Q_{S, ij} = U_{d, S}(\btheta_i, \btheta_j) \ones_{i \neq j}. 
\]
By decomposing the entries of $\bU$ in the orthonormal basis $\{ \phi_j \}_{j \ge 1}$, we can write $\bU = \bU_{\leq \evN} + \bU_{>\evN}$ where
\[
\begin{aligned}
\bU_{\leq \evN} =&~ \bPhi \bD^2 \bPhi^\sT \, , \\
\bU_{> \evN} =&~ (U_{d, >\evN}(\btheta_i,\btheta_j))_{i, j \in [N]}\, .
\end{aligned}
\]

We begin by part $(b)$. By Assumption \ref{ass:hyper_K_stronger}.$(a)$ and $\evN (d) \leq N(d)^{1 - \delta_0}$, the assumptions of Proposition \ref{prop:YY_new} are satisfied with $D = \evN$ and $(\phi_1, \ldots , \phi_{\evN})$ the top $\evN$ eigenvectors of $\Uop$. Hence, there exists $C= C(q)>0$ a constant that depends only on $q$ such that 
\[
\E \Big[ \Big\| \bPhi^\sT \bPhi / N - \id_{\evN} \Big\|_{\op} \Big] \leq C \frac{\evN \log(N)}{N^{1 - 1/q}}.
\]
Taking $q > 1/\delta_0$, the right hand side become $o_d(1)$ and Theorem \ref{prop:expression_U}.$(b)$ follows by Markov's inequality. 

Next, we prove part $(a)$, namely that $\bU_{> \evN} = \kappa_{>\evN} \cdot (\bLambda + \bDelta)$ with $\| \bDelta \|_{\op} = O_{d, \P}(N^{-\delta'})$ for some $\delta' > 0$. 

Letting $\bQ \in \R^{N \times N}$ be the matrix with entries $Q_{ij} = (\bU_{> \evN})_{ij} \ones_{i \neq j}$. Then we have $\bU_{>\evN} = \kappa_{>\evN} \bLambda + \bQ$. 
We next apply Proposition \ref{prop:generalized_Gram} to the operator $\hUop_d=\Uop_{d, >\evN}$ and subspace $\what \cV_d = \cV_{d, > \evN}$. Notice that the assumptions of
Proposition \ref{prop:generalized_Gram} are satisfied by Assumptions \ref{ass:hyper_K_stronger}.$(a)$, \ref{ass:hyper_K_stronger}.$(b)$ and \ref{ass:N_parameters}.$(a)$.
We therefore conclude that $\E[\| \bQ \|_{\op}] = O_{d}(N^{-\delta'}) \cdot \Trace(\Uop_{d, > \evN}) = O_{d}(N^{-\delta'}) \cdot \kappa_{>\evN}$ for some $\delta' > 0$. This concludes the proof of Theorem \ref{prop:expression_U}.$(a)$.
\end{proof}

This theorem implies a particularly simple structure of the empirical kernel matrix $\bU$. Under the additional Assumptions \ref{ass:hyper_K_stronger}.$(c)$, \ref{ass:hyper_K_stronger}.$(d)$ and \ref{ass:N_parameters}.$(b)$, $\bU$ can be written as a sum of a `spike' $\bU_{\leq \evN}$ (of rank $\evN$ and eigenvalues $\gg  \Tr ( \Uop_{d, > \evN})$) and a full rank matrix $\bU_{>\evN}$ with eigenvalues of order $ \Tr ( \Uop_{d, > \evN})$. The `spike' matrix $\bU_{\leq \evN}$ has the following approximate diagonalization: 
\[
\bU_{\leq \evN} = \Tilde \bPhi \Tilde \bD^2 \Tilde \bPhi^\sT,
\]
where $\Tilde \bPhi = \bPhi/\sqrt{N} \in \R^{N \times \evN}$ is approximately an orthogonal matrix $\| \Tilde \bPhi^\sT \Tilde \bPhi - \id_{\evN} \|_{\op} = o_{d,\P}(1)$ and the diagonal matrix $\Tilde \bD^2 =  \diag(N \lambda_{d, 1}^2, \ldots, N \lambda_{d, \evN}^2)$ verifies $\Tilde \bD^2 \succeq N(d)^{\delta_0} \Tr ( \Uop_{d, > \evN}) \cdot \id_N$ (by Assumption \ref{ass:N_parameters}.$(b)$). Furthermore, by Assumptions \ref{ass:hyper_K_stronger}.$(c)$, and \ref{ass:hyper_K_stronger}.$(d)$, and Theorem \ref{prop:expression_U}.$(a)$, we have for any $\delta >0$,
\[
\Omega_{d,\P} ( N^{-\delta} )\cdot \Tr ( \Uop_{d, > \evN}) \cdot \id_N \preceq \bU_{>\evN} \preceq O_{d,\P} ( N^{\delta} ) \cdot \Tr ( \Uop_{d, > \evN}) \cdot \id_N.
\]

\subsubsection{Concentration of the top eigenvectors}
\label{sec:concentration_low_degree}

Here we state and prove a general matrix concentration result.
For each $d\ge 1$, let $(\Omega_d,\tau_d)$ be a (Polish) probability space, and $(\phi_k)_{k \ge 1}$ an orthonormal basis
of $L^2(\Omega_d,\tau_d)$. Define $\bphi(\btheta) \equiv ( \phi_1(\btheta), \ldots, \phi_D(\btheta) )^\sT \in \R^D$,
and let $( \btheta_i )_{i \le N} \sim_{iid} \tau_d$. The law of large numbers and orthonormality imply
that, for any fixed $D$,
\begin{align}
\lim_{N\to\infty}\frac{1}{N} \sum_{i=1}^N\bphi(\btheta_i)\bphi(\btheta_i)^{\sT} =\int_{\Omega_d} \bphi(\btheta)\bphi(\btheta)^{\sT}\, \tau_d(\de\btheta) =\id_D\, .
  \end{align}
  The next proposition establishes a generalization of this fact for the case in which both $D$ and $N$ diverge.

\begin{proposition}\label{prop:YY_new}
Let $\{ \phi_k \in L^2(\Omega, \tau) \}_{k = 1}^D$ be orthonormal functions. Let $\{ \btheta_i \}_{i \in [N]} \sim \tau$ independently. Define $\bphi_i = \bphi(\btheta_i) = ( \phi_1(\btheta_i), \ldots, \phi_D(\btheta_i) )^\sT \in \R^D$ for $i \in [N]$. We assume that, for any integer $q \ge 2$, there exists $C = C(q)$ such that we have 
\begin{align}\label{eqn:hyper_YY_new}
\sup_{k \in [D]} \| \phi_k \|_{L^{2q}} \le C(q). 
\end{align}
Then for any $q \ge 2$, there exists $K = K(q)$ that only depends on $C(q)$, such that denoting $\delta \equiv K(q) D \log (D \vee N) / N^{1 - 1/q}$, we have 
\[
\E \Big \| \frac{1}{N} \sum_{i=1}^N \bphi_i \bphi_i^\sT - \id_D \Big \|_{\op} \le (\delta \vee \sqrt \delta).  
\]
\end{proposition}

\begin{proof}[Proof of Proposition \ref{prop:YY_new}]~
By the hypercontractivity assumption, cf. Eq.~\eqref{eqn:hyper_YY_new}, we have 
\[
\begin{aligned}
\Gamma :=&~ \E\Big[\max_{i\in [N]} \| \bphi_i \|_2^2 \Big] \le \E\Big[\max_{i\in [N]} \| \bphi_i \|_2^{2q} \Big]^{1/q} \le N^{1/q} \cdot \E[\| \bphi_i \|_2^{2q}]^{1/q} \\
=&~  N^{1/q} \cdot \Big\| \sum_{k = 1}^D \phi_k^2 \Big \|_{L^q} \le N^{1/q} D \cdot \max_{k \in [D]} \| \phi_k \|_{L^{2q}}^2 \le C(q)^2 \cdot N^{1/q} D. 
\end{aligned}
\]
Applying Lemma \ref{lem:independent_row_operator_norm_isotropic} below proves the proposition. 
\end{proof}

\begin{lemma}[\cite{vershynin2010introduction} Theorem 5.45]\label{lem:independent_row_operator_norm_isotropic}
Let $\{ \ba_i \in \R^D \}_{i \in [N]}$ be independent random vectors with $\E[\ba_i \ba_i^\sT] = \id_D$. Denote $\Gamma \equiv \E[\max_{i\in [N]} \| \ba_i \|_2^2]$. Then there exists a universal constant $C$, such that denoting $\delta \equiv C \cdot \Gamma \cdot \log (N \wedge D) / N $, we have
\[
\E\Big[ \Big\| \frac{1}{N} \sum_{i=1}^N \ba_i \ba_i^\sT - \id_D  \Big\|_{\op} \Big] \le \delta \vee \sqrt \delta. 
\]
\end{lemma}

\subsubsection{Bounding the off-diagonal part of the matrix $U_{>\evN}$ }
\label{sec:vanishing_off_diag}

We state a key proposition) whose proof will be presented in Section \ref{sec:GeneralizedGram}.
The statement and the assumptions are self-contained. 

\begin{proposition}[Bound on the off-diagonal part of the matrix $\bU_{>\evN}$]\label{prop:generalized_Gram}
Let $(\btheta_i )_{i \in [N(d)]} \sim_{iid} \tau_d$. Let $\hUop_d$ be a self-adjoint positive definite operator $\hUop_d: \what \cV_d\to \what  \cV_d$, $\what \cV_d\subseteq L^2(\Omega_d)$ with kernel $\hU_d\in L^2(\Omega_d\times\Omega_d)$ (see Eq.~\eqref{eq:Ukernel}) satisfying $\int_{\Omega_d} \hU_d(\btheta,\btheta') f(\btheta')\, \tau_d(\de\btheta') = 0$ for any $f \in \cV_d^\perp$. Let $(\hat \phi_{j})_{j\ge 1}$ be an orthonormal basis of eigenfunctions with $\spn(\hat \phi_j, j \ge 1) = \what \cV_d \subseteq L^2(\Omega_d)$, and eigenvalues $(\hat \lambda_{d, j})_{j \ge 1} \subseteq \R$ with nonincreasing absolute values $\vert\hat \lambda_{d, 1} \vert \ge \vert \hat \lambda_{d, 2} \vert \ge \cdots$ and $\sum_{j \ge 1} \hat \lambda_{d, j}^2 < \infty$, such that
\[
\begin{aligned}
\what \Uop_d = \sum_{j = 1}^\infty \hat \lambda_{d,j}^2 \hat \phi_{j} \hat \phi_{j}^*, ~~~~~~ \hU_d(\btheta, \btheta') = \sum_{j = 1}^\infty \hat \lambda_{d, j}^2 \hat \phi_j(\btheta) \hat \phi_j(\btheta'). 
\end{aligned}
\]
When $S \subseteq \{ 1, 2, 3,  \ldots \}$, we denote 
\[
\begin{aligned}
\what \Uop_{d, S} = \sum_{j \in S} \hat \lambda_{d,j}^2 \hat \phi_{j} \hat \phi_{j}^*, ~~~~~~ \hU_{d, S}(\btheta, \btheta') = \sum_{j \in S} \hat \lambda_{d, j}^2 \hat \phi_j(\btheta) \hat \phi_j(\btheta').
\end{aligned}
\]

We make the following assumptions: 
\begin{itemize}
\item[{\rm (A1)}] There exists a sequence $\{ v(d) \}_{d \ge 1}$, such that for any fixed $q \ge 1$, there exists $C = C(q, \{ v(d) \}_{d \ge 1})$ such that, for any $f_d \in \what \cV_{d, \le v(d)} \equiv \spn(\hat \phi_s, 1 \le s \le v(d))$, we have
\[
\begin{aligned}
\| f_d \|_{L^{2q}} \le&~ C \cdot \| f_d \|_{L^2}. \\
\end{aligned}
\]
\item[{\rm (A2)}] For the same sequence $\{ v(d) \}_{d \ge 1}$ as in {\rm (A1)},  there exists fixed $\delta_0 > 0$, such that 
\[
\Trace( \hUop_{d, > v(d)}^2) \cdot N(d)^{2 + \delta_0} = O_d(1) \cdot \Trace(\hUop_{d, > v(d)})^2.
\]
\item[{\rm (A3)}] There exists $\delta_0 > 0$, such that 
\begin{equation}\label{eqn:generalized_Gram_N_parameter_assumption}
N(d)^{1 + \delta_0}  \cdot \| \hUop_d \|_{\op} = O_d(1) \cdot \Trace(\hUop_d).
\end{equation}
\end{itemize}
Consider the random matrix $\bQ = (Q_{ij})_{i, j \in [N(d)]} \in \R^{N \times N}$, with 
\[
Q_{ij} = \hU_d(\btheta_i, \btheta_j) \ones_{i \neq j}. 
\]
Then there exists $\delta' > 0$, such that 
\[
\E[\| \bQ \|_{\op}] = O_{d}(N^{- \delta'}) \cdot \Trace(\hUop_d). 
\]
\end{proposition}

\subsection{Proof of Proposition \ref{prop:generalized_Gram}}
\label{sec:GeneralizedGram}

We begin by stating two key estimates which are used in the proof of Proposition \ref{prop:generalized_Gram}. The notations of Lemma \ref{lem:generalized_Gram_higher} follow the notations of Proposition \ref{prop:generalized_Gram}. The notations and assumptions of Proposition \ref{prop:generalized_Gram_double_new} are self-contained. We collect a number of technical lemmas in Section \ref{sec:RF_lower_auxilliary}.

\begin{lemma}\label{lem:generalized_Gram_higher}
Consider the same setup as Proposition \ref{prop:generalized_Gram}. Let $\{ N(d) \}_{d \ge 1}$ and $\{ v(d) \}_{d \ge 1}$ be two sequences, and assume that there exists $\delta_0 > 0$ such that (this is Assumption (A2) in Proposition \ref{prop:generalized_Gram})
\begin{equation}\label{eqn:generalized_Gram_higher1}
N(d)^2 \cdot \Trace (\hUop_{d, > v(d)}^2) = O_d(N^{- \delta_0}) \cdot \Trace(\hUop_{d, > v(d)})^2. 
\end{equation}
Consider the random matrix $\bQ_{> v(d)} = (Q_{> v(d), ij})_{i, j \in [N(d)]} \in \R^{N \times N}$, with 
\[
Q_{> v(d), ij} = \hU_{d, > v(d)}(\btheta_i, \btheta_j) \ones_{i \neq j}. 
\]
Then we have 
\[
\E[\| \bQ_{> v(d)} \|_{\op}^2 ]^{1/2} = O_d(N^{- \delta_0}) \cdot \Trace(\hUop_{d, > v(d)}). 
\]
\end{lemma}

\begin{proposition}[Vanishing off-diagonal]\label{prop:generalized_Gram_double_new}
Let $\overline \Uop$ be a compact self-adjoint positive definite operator on a closed subspace $\overline \cV \subseteq L^2(\Omega, \tau)$, $\overline \Uop: \overline \cV \to \overline \cV$, with corresponding kernel $\overline U \in L^2(\Omega \times \Omega)$, satisfying $\int_{\Omega} \overline U(\btheta,\btheta') \, f(\btheta')\, \tau(\de\btheta')=0$ for all $f \in \overline \cV^{\perp}$. For any $q \ge 1$, we assume that there exists $C(q)$ such that
\begin{equation}\label{eqn:hypercontractivity_overline}
\begin{aligned}
\E_{\btheta_1, \btheta_2 \sim \tau}[\vert \overline U(\btheta_1, \btheta_2)\vert ^{2q}]^{1/(2q)} \le&~ C(q) \cdot \E_{\btheta_1, \btheta_2 \sim \tau}[\overline U(\btheta_1, \btheta_2)^{2}]^{1/2}, \\
\E_{\btheta \sim \tau}[\vert \overline U(\btheta, \btheta)\vert^{q}]^{1/q} \le&~ C(q) \cdot \E_{\btheta \sim \tau}[\overline U(\btheta, \btheta)]. 
\end{aligned}
\end{equation}
  
Moreover, let $\{ \btheta_i \}_{i \in [N]} \sim_{iid} \tau$ independently, and consider $\bDelta = (\Delta_{ij})_{i, j \in [N]} \in \R^{N \times N}$, with 
\[
\Delta_{ij} = \overline U(\btheta_i, \btheta_j) \ones_{i \neq j}. 
\]
Then for any integer $p > 0$, there exists a constant $K(p)$ which only depends on the constant $C(p)$, such that
\begin{align}
\E[\| \bDelta \|_{\op}] \le K(p) \cdot \Big\{ N \| \overline \Uop \|_{\op} + [\| \overline \Uop \|_{\op} \Trace(\overline \Uop) N^{1 + 2/p} \log N ]^{1/2}\Big\}. 
\end{align}
\end{proposition}

We are now in position to prove Proposition \ref{prop:generalized_Gram}.
\begin{proof}[Proof of Proposition \ref{prop:generalized_Gram}]
We decompose the operator $\hUop_d = \hUop_{d, \le v(d)} + \hUop_{d, > v(d)}$, and the kernel $\hU_d = \hU_{d, \le v(d)} + \hU_{d, > v(d)}$. Define $\bQ_{S} = (Q_{S, ij})_{i, j \in [N(d)]}$ with 
\[
Q_{S, ij} = U_{d, S}(\btheta_i, \btheta_j) \ones_{i \neq j}. 
\]
By Assumption {\rm (A2)} and by Lemma \ref{lem:generalized_Gram_higher}, we have 
\[
\E[\| \bQ_{> v(d)} \|_{\op}^2]^{1/2} = O_d(N^{-\delta_0}) \cdot \Trace(\hUop_{d, > v(d)}). 
\]
By Assumption {\rm (A1)} and Lemma \ref{lem:hypercontractivity_basis_kernel} which is stated in Section \ref{sec:RF_lower_auxilliary} below, the assumptions of Proposition \ref{prop:generalized_Gram_double_new} are satisfied, in which we take $\bDelta = \bQ_{\le v(d)}$, $\overline \Uop = \hUop_{d, \le v(d)}$, $\overline U = \hU_{d, \le v(d)}$, and $\ocV \equiv\spn(\phi_s:\; 1 \le s\le v(d))$. Further by Assumption (A3) as in Eq. (\ref{eqn:generalized_Gram_N_parameter_assumption}), we fix some $p > 4/ \delta_0$ in Proposition \ref{prop:generalized_Gram_double_new}, then for $\delta'  = \delta_0 / 4 > 0$, we have
\[
\E[\| \bQ_{\le v(d)} \|_{\op}] = O_d(N^{- \delta'}) \cdot \Trace(\hUop_{d, \le v(d)}). 
\]
Combining the  equations in the last two displays proves the proposition. 
\end{proof}

We next prove Lemma  \ref{lem:generalized_Gram_higher} and Proposition \ref{prop:generalized_Gram_double_new}.
\begin{proof}[Proof of Lemma \ref{lem:generalized_Gram_higher}] We have 
\[
\begin{aligned}
\E[\| \bQ_{> v(d)} \|_{\op}^2] \le&~ \E[\| \bQ_{> v(d)} \|_{F}^2] = N(N-1) \cdot \E[Q_{> v(d), ij}^2] \\
=&~ N (N-1)  \cdot \Trace(\hUop_{d, > v(d)}^2) = O_d(N^{- \delta_0}) \cdot \Trace(\hUop_{d, > v(d)})^2. 
\end{aligned}
\]
where the last equation is by Eq. (\ref{eqn:generalized_Gram_higher1}). This proves the lemma. 
\end{proof}

\begin{proof}[Proof of Proposition \ref{prop:generalized_Gram_double_new}]~
With a little abuse of notation, we define $\overline \bU = (\overline U(\btheta_i, \btheta_j))_{i, j \in [N]}\in \reals^{N\times N}$. 

\noindent
{\bf Step 1. Bound $\E[\| \bDelta \|_{\op}]$ using matrix decoupling. } For $T_1, T_2 \subseteq [N]$, we denote $\bA_{T_1, T_2} = (A_{ij})_{i \in T_1, j \in T_2}$. By Lemma \ref{lem:matrix_decoupling} which is stated in Section \ref{sec:RF_lower_auxilliary} below, we have 
\begin{align}\label{eqn:IPB1}
\E[\| \bDelta\|_{\op}] \le 4 \sup_{T \subseteq [N]} \E[\| \bDelta_{T T^c} \|_{\op}]. 
\end{align}
For any $S \subseteq [N]$, we denote $\E_S$ to be the expectation with respect to $\{ \btheta_i \}_{i \in S}$ and conditional on $\{ \btheta_j \}_{j \in S^c}$. Fix $T \subseteq [N]$. Using Lemma \ref{lem:independent_row_operator_norm} (which is stated in Section \ref{sec:RF_lower_auxilliary} below) conditioning on $\{ \btheta_j \}_{j \in T^c}$, we have 
\[
\begin{aligned}
\E_T[\| \bDelta_{T T^c} \|_{\op}] \le [ \Sigma(T) \cdot N]^{1/2} + C \cdot (\Gamma(T) \cdot  \log N )^{1/2},
\end{aligned}
\]
where $\Sigma(T) \equiv \| \E_{\btheta_u}[\bDelta_{T^c u} \bDelta_{u T^c}] \|_{\op}$ (for some $u \in T$) and $\Gamma(T) \equiv \E_T[\max_{i \in T} \| \bDelta_{i T^c} \|_2^2]$. Therefore, by Holder's inequality, we have 
\begin{align}\label{eqn:IPB2}
\begin{aligned}
\E[\| \bDelta\|_{\op}] \le&~ 4 \sup_{T \subseteq [N]} \E[\| \bDelta_{T T^c} \|_{\op}] = 4 \sup_{T \subseteq [N]}\E_{T^c} \E_T[\| \bDelta_{T T^c} \|_{\op}]  \\
\le &~ 4 \sup_{T \subseteq [N]} \Big\{  [\E_{T^c}[\Sigma(T)] \cdot N]^{1/2} + C \cdot (\E_{T^c}[\Gamma(T)] \cdot \log N )^{1/2} \Big\}.
\end{aligned}
\end{align}

\noindent
{\bf Step 3. Bound $\E_{T^c}[ \Sigma(T)]$. } By the compactness of operator $\overline \Uop \vert_{\overline \cV}$, there exists orthogonal basis $\{ \phi_k \}_{k \ge 1}$ and real numbers $\{ \lambda_k \}_{k \ge 1}$, such that $\overline U(\btheta_i, \btheta_j) = \sum_{k} \lambda_k^2 \phi_k(\btheta_i) \phi_k(\btheta_j)$. Therefore, we have
\[
\begin{aligned}
\Sigma(T) =&~ \| \E_{\btheta_u}[\bDelta_{T^c u} \bDelta_{u T^c}] \|_{\op} = \sup_{\| \bz \|_2 = 1} \sum_{i, j \in T^c} \sum_{k} \lambda_{k}^4 \phi_k(\btheta_i) \phi_k(\btheta_j) z_i z_j \\
\le&~ \| \overline \Uop \|_\op \cdot \sup_{\| \bz \|_2 = 1} \sum_{i,j \in T^c} \sum_{k} \lambda_{k}^2 \phi_k(\btheta_i) \phi_k(\btheta_j) z_i z_j \\
=&~ \| \overline \Uop \|_{\op}  \cdot \| ( \overline U_{ij} )_{i, j \in T^c} \|_{\op} \le \| \overline \Uop \|_{\op}  \cdot [ \| \ddiag(\overline \bU) \|_{\op} + \| \bDelta \|_{\op}].  
\end{aligned}
\]
Note by the hypercontractivity assumption as in Eq. (\ref{eqn:hypercontractivity_overline}), we have
\[
\begin{aligned}
\E[\| \ddiag(\overline \bU) \|_{\op}] \le&~ \E\Big[\sum_{i = 1}^N \overline U_{ii}^{p} \Big]^{1/p} \le N^{1/p} \cdot \E[\overline U_{ii}^p]^{1/p} \le C(p) N^{1/p} \cdot \E[\overline U_{ii}] \le C(p) N^{1/p} \cdot \Trace(\overline \Uop). 
\end{aligned}
\]
This gives
\begin{align}\label{eqn:IPB3}
\E_{T^c}[ \Sigma(T) ] \le&~  C(p) N^{1/p} \cdot \| \overline \Uop \|_{\op} \Trace(\overline \Uop) + \| \overline \Uop \|_{\op} \E[\| \bDelta \|_{\op}].  
\end{align}

\noindent
{\bf Step 4. Bound $\E_{T^c}[\Gamma(T)]$. } By the hypercontractivity assumption as in Eq. (\ref{eqn:hypercontractivity_overline}), we have
\begin{equation}\label{eqn:IPB4}
\begin{aligned}
\E_{T^c}[\Gamma(T)] \equiv&~ \E\Big[\max_{i \in T} \| \bDelta_{i T^c} \|_2^2 \Big] \le N \cdot \E \Big[\max_{i \in T}\max_{j \in T^c} \Delta_{ij}^2 \Big] \\
\le&~ N \cdot \E\Big[\max_{i \in T, j\in T^c} \Delta_{i j}^{2p} \Big]^{1/p} \le N^{1 + 2/p} \cdot \E[ \Delta_{i j}^{2p}]^{1/p} \\
=&~ C(p)^2 N^{1 + 2/p} \cdot \E[ \Delta_{i j}^2] \le C(p)^2 N^{1 + 2/p} \cdot \| \overline \Uop\|_{\op} \Trace(\overline \Uop).  
\end{aligned}
\end{equation}
The last inequality holds since $\E[\Delta_{ij}^2] = \E\{[\sum_k \lambda_k \phi_k(\btheta_i) \phi_k(\btheta_j)]^2\} = \sum_{k} \lambda_k^4 \le \| \overline \Uop \|_{\op} \Trace(\overline \Uop)$.

\noindent
{\bf Step 5. Combining the equations. } Combining Eq. (\ref{eqn:IPB2}), (\ref{eqn:IPB3}), and (\ref{eqn:IPB4}), we have 
\[
\begin{aligned}
\E[\| \bDelta\|_{\op}] \le&~ 4 \sup_{T \subseteq [N]} \Big\{ [\E_{T^c}[\Sigma(T)] \cdot N]^{1/2} + C \cdot (\E_{T^c}[\Gamma(T)] \cdot  \log N )^{1/2} \Big\}\\
\le&~ K(p) \Big\{ \{ \| \overline \Uop \|_{\op} \Trace(\overline \Uop) N^{1 + 2/p} \log N \}^{1/2} + \{ N \| \overline \Uop \|_{\op} \E[\| \bDelta \|_{\op}] \}^{1/2}\Big\}. 
\end{aligned}
\]
Denote $\eps_1 =  K(p) (N \| \overline \Uop \|_{\op} )^{1/2} \ge 0$ and $\eps_2 =  K(p)\{ \| \overline \Uop \|_{\op} \Trace(\overline \Uop) N^{1 + 2/p} \log N \}^{1/2} \ge 0$, $x = \E[\| \bDelta\|_{\op}]^{1/2}$. The above inequality implies $x^2 - \eps_1 x - \eps_2 \le 0$, which gives $x \le [\eps_1 + (\eps_1^2 + 4 \eps_2)^{1/2} ] /2 \le (\eps_1^2 + 4 \eps_2)^{1/2}$. This concludes the proof. 
\end{proof}

\subsubsection{Auxiliary lemmas}\label{sec:RF_lower_auxilliary}

The following standard decoupling trick follows, for instance, from \cite{vershynin2010introduction}
in Lemma 5.60.
\begin{lemma}[Matrix decoupling]\label{lem:matrix_decoupling}
Let $\bA \in \R^{N \times N}$ be a real symmetric random matrix. For $T_1, T_2 \subseteq \{ 1, 2, \ldots, N\}$, we denote $\bA_{T_1, T_2} = (A_{ij})_{i \in T_2, j \in T_2}$. Then we have 
\[
\E[\| \bA - \ddiag(\bA) \|_{\op} ] \le 4 \max_{T \subseteq [N]} \E[\| \bA_{T, T^c} \|_{\op}]. 
\]
\end{lemma}

\begin{proof}[Proof of Lemma \ref{lem:matrix_decoupling}]

Let $T$ be a random subset of $\{ 1, 2, \ldots, N\}$, with each element selected with probability $1/2$ independently. For any $\bx \in \S^{N-1}$, we have
\[
\< \bx, [\bA - \ddiag(\bA)] \bx \> = 4 \E_T \Big[ \sum_{i \in T, j \in T^c} A_{ij} x_i x_j \Big]. 
\]
By Jensen's inequality we have 
\[
\begin{aligned}
\E[\| \bA - \ddiag(\bA) \|_{\op}] =&~ \E_\bA\Big[ \sup_{\bx \in \S^{N -1}} \< \bx, [\bA - \ddiag(\bA)] \bx\> \Big] \le 4  \E_T \E_\bA \Big[ \sup_{\bx \in \S^{N - 1}}\sum_{i \in T, j \in T^c} A_{ij} x_i x_j \Big] \\
\le&~ 4  \sup_{T \subseteq [N]} \E[ \| A_{T T^c} \|_{\op}].  
\end{aligned}
\]
This completes the proof. 
\end{proof}

\begin{lemma}[\cite{vershynin2010introduction} Theorem 5.48]\label{lem:independent_row_operator_norm}
Let $\bA \in \R^{N \times n}$ with $\bA^\sT = [\ba_1, \ldots, \ba_N]$ where $\ba_i$ are independent random vectors in $\R^n$ with the common second moment matrix $\bSigma = \E[\ba_i \ba_i^\sT]$. Let $\Gamma \equiv \E[\max_{i\in [N]} \| \ba_i \|_2^2]$. Then there exists a universal constant $C$, such that
\[
\E[\| \bA \|_{\op}^2]^{1/2} \le ( \| \bSigma \|_{\op} \cdot N)^{1/2} + C \cdot (\Gamma \cdot \log (N \wedge n))^{1/2}. 
\]
\end{lemma}

\begin{lemma}\label{lem:hypercontractivity_basis_kernel}
Let $\{ \phi_k \}_{1 \le k \le Z} \subseteq L^2(\Omega, \tau)$ be a set of orthonormal functions. We assume that, for any fixed $q \ge 1$, there exists $C = C(q)$, such that for any $f \in \spn\{ \phi_k: 1 \le k \le Z \}$, we have
\[
\begin{aligned}
\| f \|_{L^{2q}} \le &~ C(q) \cdot \| f \|_{L^2}. 
\end{aligned}
\]
For $\btheta, \btheta' \in \Omega$, we denote $\overline U(\btheta, \btheta') = \sum_{k = 1}^Z \lambda_k^2 \phi_k(\btheta) \phi_k(\btheta')$ where $\{ \lambda_{k} \}_{1 \le k \le Z}$ are fixed real numbers. Then for any $q \ge 1$, we have
\begin{align}
\E_{\btheta_1, \btheta_2 \sim \tau}[\overline U(\btheta_1, \btheta_2)^{2 q}]^{1/(2q)} \le&~ C(q)^2 \cdot \E_{\btheta_1, \btheta_2 \sim \tau}[\overline U(\btheta_1, \btheta_2)^{2}]^{1/2}, \label{eqn:hypercontractivity_basis_kernel_1}\\
\E_{\btheta \sim \tau}[\overline U(\btheta, \btheta)^{q}]^{1/q} \le&~ C(q)^2 \cdot \E_{\btheta \sim \tau}[\overline U(\btheta, \btheta)].  \label{eqn:hypercontractivity_basis_kernel_2}
\end{align}
\end{lemma}

\begin{proof}[Proof of Lemma \ref{lem:hypercontractivity_basis_kernel}]

For any $q \ge 1$, we have
\[
\begin{aligned}
&~ \E_{\btheta_1, \btheta_2 \sim \tau}[\overline U(\btheta_1, \btheta_2)^{2 q}] =  \E_{\btheta_1 \sim \tau} \Bigg\{ \E_{\btheta_2 \sim \tau}\Big\{ \Big[\sum_{k = 1}^Z \lambda_k^2 \phi_k(\btheta_1) \phi_k(\btheta_2)\Big]^{2q} \Big\vert \btheta_1 \Big\} \Bigg\} \\
\stackrel{(a)}{\le}&~  C(q)^{2q} \cdot \E_{\btheta_1 \sim \tau} \Bigg\{ \E_{\btheta_2 \sim \tau}\Big\{ \Big[\sum_{k = 1}^Z \lambda_k^2 \phi_k(\btheta_1) \phi_k(\btheta_2)\Big]^{2} \Big\vert \btheta_1 \Big\}^q \Bigg\} \stackrel{(b)}{=}  C(q)^{2q} \cdot \E_{\btheta_1 \sim \tau} \Bigg\{ \Big[ \sum_{k = 1}^Z \lambda_k^4 \phi_k(\btheta_1)^2  \Big]^q \Bigg\} \\
\stackrel{(c)}{\le} &~  C(q)^{2q} \cdot  \Bigg\{  \sum_{k = 1}^Z  \lambda_k^4 \cdot \E_{\btheta_1 \sim \tau}[\phi_k(\btheta_1)^{2q}]^{1/q}   \Bigg\}^q \stackrel{(d)}{\le}  C(q)^{2q} \cdot  \Bigg\{  C(q)^2 \sum_{k = 1}^Z  \lambda_k^4 \cdot \E_{\btheta_1 \sim \tau}[\phi_k(\btheta_1)^2]   \Bigg\}^q \\
\stackrel{(e)}{=}&~ C(q)^{4q}\Big[ \sum_{k = 1}^Z \lambda_k^4 \Big]^q \stackrel{(f)}{=} C(q)^{4q} \cdot 
\Big\{ \E_{\btheta_1, \btheta_2 \sim \tau}[\overline U(\btheta_1, \btheta_2)^{2}] \Big\}^q. 
\end{aligned}
\]
Here, inequality $(a)$ follows by applying the hypercontractivity inequality with respect to $f(\btheta_2) = \sum_{k = 1}^Z \lambda_k^2 \phi_k(\btheta_1) \phi_k(\btheta_2)$ (and conditional on $\btheta_1$). Equality $(b)$ by the fact that $( \phi_k)_{1 \le k \le Z}$ are orthonormal functions. Inequality $(c)$ is by the Minkowski inequality. Inequality
$(d)$ follows by applying the hypercontractivity inequality with respect to $f(\btheta_1) =  \phi_k(\btheta_1)$. Equality $(e)$ holds because $( \phi_k)_{1 \le k \le Z}$ are orthonormal functions. Finally, equality $(f)$ follows by simple calculation. This proves Eq. (\ref{eqn:hypercontractivity_basis_kernel_1}). 

For any $q \ge 1$, we have
\[
\begin{aligned}
&~ \E_{\btheta \sim \tau}[\overline U(\btheta, \btheta)^{q}] = \E_{\btheta \sim \tau}\Big[ \Big(\sum_{k = 1}^Z \lambda_k^2 \phi_k(\btheta)^2 \Big)^q \Big] \stackrel{(a)}{\le} \Big[ \sum_{k = 1}^Z \lambda_k^2 \cdot \E_{\btheta \sim \tau}[\phi_k(\btheta)^{2q}]^{1/q} \Big]^q \\
\stackrel{(b)}{\le} &~C(q)^{2q} \Big[ \sum_{k = 1}^Z \lambda_k^2 \cdot \E_{\btheta \sim \tau}[\phi_k(\btheta)^{2}] \Big]^q \stackrel{(c)}{=} C(q)^{2q} \Big[ \sum_{k = 1}^Z \lambda_k^2 \Big]^q \stackrel{(d)}{=} C(q)^{2q} \Big\{ \E_{\btheta \sim \tau}[\overline U(\btheta, \btheta)] \Big\}^q. 
\end{aligned}
\]
Here, inequality $(a)$ holds by  Minkowski inequality. Inequality $(b)$ follows by applying the hypercontractivity inequality with respect to $f(\btheta) = \phi_k(\btheta)$. Equality $(c)$ holds because $( \phi_k)_{1 \le k \le Z}$ are orthonormal functions, and equality $(d)$ by a simple calculation. This proves Eq. (\ref{eqn:hypercontractivity_basis_kernel_2}). 
\end{proof}

\begin{lemma}[Bound on the maximum of diagonal]\label{lem:bound_max_hyper}
Consider a sequence of probability spaces $(\Omega_d, \tau_d)$ with $\{ \phi_{d,k} \}_{k \ge 1}$ an orthonormal basis of functions for $\cD_d \subseteq L^2 ( \Omega_d, \tau_d)$. Assume that there exists a sequence of integers $\{ u(d) \}_{d \ge 1}$ such that the subspace $\cD_{d,\leq u(d)} = \spn ( \phi_{d,k} : 1 \leq k \leq u(d))$ is hypercontractive, i.e., for any fixed $k\ge 1$, there exists a constant $C$ such that, for any $g \in \cD_{d, \leq u(d)}$, we have
\[
\| g \|_{L^{2k} ( \Omega_d)} \leq C \cdot \| g \|_{L^2 ( \Omega_d)}.
\]
Let $\{ U_d \}_{d\ge 1}$ be a sequence of positive definite kernels $U_d \in L^2 ( \Omega_d \times \Omega_d)$ with 
\[
U_d ( \btheta_1 , \btheta_2 ) = \sum_{j = 1}^\infty \lambda_{d,k}^2 \phi_{d,k} (\btheta_1) \phi_{d,k} (\btheta_2).
\]
Denote $U_{d, >\ell}$ the kernel function obtained by setting $\lambda_{d,1} = \ldots = \lambda_{d,\ell} = 0$. Letting $(\btheta_i)_{i \in [N(d)]} \sim_{iid} \tau_d$, if we assume that for any $\delta >0$,
\begin{equation}\label{eq:max_diag_u_cond}
\max_{i \in [N(d)]} U_{d,>u(d)} (\btheta_i, \btheta_i) = O_{d,\P} (N(d)^\delta) \cdot \E_{\btheta \sim \tau_d} [ U_{d,>u(d)} (\btheta, \btheta)],
\end{equation}
then for any $\delta > 0$,
\begin{equation}\label{eq:max_diag_u}
\max_{i \in [N(d)]} U_{d} (\btheta_i, \btheta_i) = O_{d,\P} (N(d)^\delta) \cdot  \E_{\btheta \sim \tau_d} [ U_{d} (\btheta, \btheta)].
\end{equation}
Furthermore, if we assume that for any $\delta >0$,
\begin{equation}\label{eq:max_diag_u2_cond}
\max_{i \in [N(d)]} \E_{\btheta \sim \tau_d} [ U_{d,>u(d)} (\btheta_i, \btheta)^2] = O_{d,\P} (N(d)^\delta) \cdot \E_{\btheta_1, \btheta_2 \sim \tau_d} [ U_{d,>u(d)} (\btheta_1, \btheta_2)^2],
\end{equation}
then for any $\delta > 0$,
\begin{equation}\label{eq:max_diag_u2}
\max_{i \in [N(d)]} \E_{\btheta \sim \tau_d} [ U_{d} (\btheta_i, \btheta)^2] = O_{d,\P} (N(d)^\delta) \cdot \E_{\btheta_1, \btheta_2 \sim \tau_d} [ U_{d} (\btheta_1, \btheta_2)^2].
\end{equation}

\end{lemma}

\begin{proof}[Proof of Lemma \ref{lem:bound_max_hyper}]
Let us decompose $U_d$ in a high and low degree parts, $U_d = U_{d,\leq u} + U_{d, >u}$ where
\[
\begin{aligned}
U_{d, \leq u} (\btheta_1, \btheta_2 ) =&~ \sum_{k = 1}^u \lambda_{d,k}^2 \phi_{d,k} (\btheta_1) \phi_{d,k} (\btheta_2), \\
U_{d, > u} (\btheta_1, \btheta_2 ) =&~ \sum_{k = u+1}^\infty \lambda_{d,k}^2 \phi_{d,k} (\btheta_1) \phi_{d,k} (\btheta_2). \\
\end{aligned}
\] 
By Lemma \ref{lem:hypercontractivity_basis_kernel}, we have for any $q \ge 1$,
\[
\begin{aligned}
\E \Big[ \max_{i \in [N(d)]} U_{d, \leq u} (\btheta_i, \btheta_i) \Big] \leq&~ \E \Big[ \max_{i \in [N(d)]} U_{d, \leq u} (\btheta_i, \btheta_i)^{q} \Big]^{1/q} \\
\leq~& N^{1/q} \E \Big[ U_{d, \leq u} (\btheta, \btheta)^{q} \Big]^{1/q} \\
\leq~&  C(q)^2 N^{1/q} \E \Big[ U_{d, \leq u} (\btheta, \btheta) \Big].
\end{aligned}
\]
Hence, by Markov's inequality and condition \eqref{eq:max_diag_u_cond}, we get for any $\delta>0$, taking $q$ sufficiently large,
\[
\max_{i \in [N(d)]} U_{d} (\btheta_i, \btheta_i) = O_{d,\P} (N(d)^\delta) \cdot  \E_{\btheta \sim \tau_d} [ U_{d} (\btheta, \btheta)].
\]
The proof of Eq.~\eqref{eq:max_diag_u2} follows from a similar argument.
\end{proof}

\section{Generalization error of random features model: Proof of Theorem \ref{thm:RFK_generalization}}
\label{sec:proof_generalization_RFK}

In this section, we prove Theorem \ref{thm:RFK_generalization}. The proof in the overparametrized regime is presented in Section
\ref{sec:proof_RFRR_over}. The proof in the underparametrized regime follows from a very similar argument: we will omit it and
simply add comments in the overparametrized proof where they differ. 

We defer the proofs of some  technical results to later sections. Section \ref{sec:structure_Z} proves a key proposition on the structure of the feature matrix $\bZ = ( \sigma_d ( \bx_i ; \btheta_j ) )_{i \in [n], j\in [N]}$. Section \ref{sec:technical_claims_RFRR} gather some technical bounds necessary for the proof of Theorem \ref{thm:RFK_generalization}. Finally, Section \ref{sec:RFRR_ZZ_concentration} contains concentration results on the high degree part of the feature matrix.

\subsection{Proof of Theorem \ref{thm:RFK_generalization} in the overparametrized regime}
\label{sec:proof_RFRR_over}

In this section, we prove Theorem \ref{thm:RFK_generalization} in the overparametrized regime. We defer the proofs of some of the technical lemmas and matrix concentration results to Sections \ref{sec:structure_Z}, \ref{sec:technical_claims_RFRR} and \ref{sec:RFRR_ZZ_concentration}. The underparametrized case follows from the same proof with the following mapping $n \leftrightarrow N$, $\evn \leftrightarrow \evN$ and $\lambda \rightarrow \lambda_N = N \lambda/n$. We will add remarks in the proof when a difference arises.

\noindent
{\bf Step 1. Rewrite the $\by$, $\bV$, $\bU$, $\bZ$ matrices. }

We recall that the random features ridge regression solution is given by
\[
\hba(\lambda) = \argmin_{\ba} \Big\{\sum_{i=1}^n\big( y_i - \hf(\bx_i;\ba) \big)^2  +  \frac{\lambda}{N} \| \ba \|_2^2\Big\}.
\]
Solving for the coefficients yields
\[
\hba ( \lambda) = ( \bZ^\sT \bZ / N + \lambda \id_N )^{-1} \bZ^\sT \by ,
\]
where $\by = (y_1 , \ldots , y_n )$ and $\bZ = (Z_{ij})_{i \in[n], j \in [N]} \in \R^{n \times N}$ with $Z_{ij} = \sigma_d ( \bx_i ; \btheta_j )$. Hence, the prediction function at location $\bx$ is given by
\[
\hf(\bx;\hba(\lambda)) = \by^\sT \bZ  ( \bZ^\sT \bZ / N + \lambda \id_N )^{-1} \bsigma (\bx) / N,
\]
where $\bsigma (\bx) = ( \sigma_d ( \bx ; \btheta_1) , \ldots , \sigma_d ( \bx ; \btheta_N ) ) \in \R^N$.

Expanding the test error, we get
\[
\begin{aligned}
R_{\RF}(f_d, \bX, \bTheta, \lambda) \equiv& \E_\bx\Big[ \Big( f_d(\bx) - \by^\sT \bZ (\bZ^\sT \bZ/N + \lambda \id_N)^{-1} \bsigma (\bx) /N \Big)^2 \Big]\\
=& \E_\bx [ f_d(\bx)^2] - 2 \by^\sT \bZ \hbUi \bV/N  + \by^\sT \bZ \hbUi \bU \hbUi \bZ^\sT \by / N^2 ,
\end{aligned}
\]
where $\bV = (V_1, \ldots, V_N)^\sT$ and $\bU = (U_{ij})_{ij \in [N]} $with
\[
\begin{aligned}
V_i =& \E_\bx[f_d (\bx) \sigma_d(\bx ; \btheta_i  )], \\
U_{ij} =& \E_{\bx}[\sigma_d(\bx ; \btheta_i ) \sigma_d(\bx ;  \btheta_j ) ], \\
\end{aligned}
\]
and $\hbU = \bZ^\sT \bZ / N + \lambda \id_N$ is the (rescaled) regularized empirical kernel matrix 
\[
\hat{U}_{\lambda, ij} = \frac{1}{N} \sum_{k \in [n]} \sigma_d ( \bx_k ; \btheta_i ) \sigma_d ( \bx_k ; \btheta_j ) + \lambda \delta_{ij}.
\]

We recall that the eigendecomposition of $\sigma_d$ is given by
\[
\sigma_{d}( \bx ; \btheta ) = \sum_{k = 1}^\infty \lambda_{d,k}  \psi_k(\bx) \phi_k ( \btheta).
\]
We write the orthogonal decomposition of $f_{d}$ in this basis as
\[
f_{d}(\bx) = \sum_{k = 1}^\infty \hf_{d,k} \psi_k(\bx), 
\]
Define
\begin{equation}\label{eq:def_quantities_RFRR_proof}
\begin{aligned}
\bpsi_{k}  =& (\psi_{k}(\bx_1), \ldots , \psi_k ( \bx_n) )^\sT  \in \R^{n},\\
\bphi_{k} =& (\phi_{k}(\btheta_1), \ldots , \phi_k ( \btheta_N) )^\sT  \in \R^{N},\\
\bD_{\le \evn} =& \diag(\lambda_{d,1}, \lambda_{d,2} , \ldots , \lambda_{d,\evn} ) \in \R^{\evn \times \evn},  \\
\bpsi_{\le \evn} =& (\bpsi_{k } (\bx_i ) )_{i \in [n], k \in [ \evn ] } \in \R^{n \times \evn},  \\
\bphi_{\le \evn} =& (\bphi_{k } (\btheta_i ) )_{i \in [N],  k \in [\evn]} \in \R^{N \times \evn},  \\
\hbf_{\le \evn}=&  ( \hf_{d,1} , \hf_{d,2} , \ldots , \hf_{d, \evn} )^\sT \in \R^{\evn}.
\end{aligned}
\end{equation}

Recall that $\by = ( y_1 , \ldots , y_n )^\sT = \boldf + \beps$ with
\[
\begin{aligned}
\boldf = & ( f_d ( \bx_1) , \ldots , f_d ( \bx_n))^\sT \\
\beps =& ( \eps_1 , \ldots , \eps_n )^\sT.
\end{aligned}
\]
Using the above notations, we can decompose the vectors and matrices $\boldf$, $\bV$,  $\bU$, and  as
\begin{equation}\label{eq:def_leq_ell_big_ell}
\begin{aligned}
  &\boldf =\boldf_{\le \evn} +\boldf_{> \evn},\qquad
  &\boldf_{\le \evn} =  \bpsi_{\le \evn} \hbf_{\le \evn}, \qquad  &\boldf_{> \evn} = \sum_{k = \evn + 1}^\infty \hf_{d,k} \bpsi_{k}, \\
   &\bV = \bV_{\le \evn} +\bV_{> \evn},\qquad &\bV_{\le \evn} =  \bphi_{\le \evn} \bD_{\le \evn} \hbf_{\le \evn}, \qquad  &\bV_{> \evn}  =  \sum_{k = \evn + 1}^\infty \hf_{d,k} \lambda_{d,k}  \bphi_{k},\\
   &\bU = \bU_{\le \evn} +\bU_{> \evn},\qquad
   &\bU_{\le \evn} =  \bphi_{\le \evn} \bD_{\le \evn}^2 \bphi_{\le \evn}^\sT, \qquad  &\bU_{> \evn} =  \sum_{k = \evn + 1}^\infty \lambda_{d,k}^2  \bphi_{k} \bphi_k^\sT,\\
   &\bZ = \bZ_{\le \evn} +\bZ_{> \evn},\qquad
   &\bZ_{\leq \evn} = \bpsi_{\leq \evn} \bD_{\leq \evn} \bphi_{\leq \evn}^\sT, \qquad &\bZ_{>\evn} = \sum_{k \geq \evn+1}  \lambda_{d,k} \bpsi_k \bphi_k^\sT\, .
\end{aligned}
\end{equation}

\noindent
{\bf Step 2. Decompose the risk.}

We decompose the risk with respect to $\by = \boldf + \beps$ as follows
\[
\begin{aligned}
R_{\KR}( f_d, \bX, \bW, \lambda) =& \| f_d \|_{L^2}^2  - 2 T_1 + T_2 +  T_3 - 2 T_4 + 2 T_5. 
\end{aligned}
\]
where
\begin{equation}\label{eq:decomposition_risk_Ti}
\begin{aligned}
T_1 =& \boldf^\sT \bZ \hbUi  \bV/N, \\
T_2 =&  \boldf^\sT \bZ \hbUi  \bU \hbUi  \bZ^\sT \boldf /N^2, \\
T_3 =& \beps^\sT \bZ\hbUi  \bU \hbUi  \bZ^\sT \beps /N^2,\\ 
T_4 =& \beps^\sT\bZ \hbUi  \bV/N ,\\
T_5 =& \beps^\sT \bZ\hbUi  \bU \hbUi \bZ^\sT \boldf / N^2. 
\end{aligned}
\end{equation}

The proof relies on the following key result on the structure of the feature matrix $\bZ$:

\begin{proposition}[Structure of the feature matrix $\bZ$]\label{prop:singular_values_bZ}
Follow the assumptions and the notations in the proof of Theorem \ref{thm:RFK_generalization} in the overparametrized regime (note in particular that $N \ge n^{1+\delta_0}$ and $n \ge \evn^{1 + \delta_0}$ for some fixed $\delta_0 >0$). Consider the singular value decomposition of $\bZ = ( Z_{ij} )_{i \in [n], j \in [N]}$ with $Z_{ij} = \sigma_d (\bx_i ; \btheta_j)$:
\[
\bZ / \sqrt{N} = \bP \bLambda \bQ^\sT = [\bP_1 , \bP_2 ] \diag ( \bLambda_1 , \bLambda_2 ) [\bQ_1 , \bQ_2]^\sT \in \R^{n \times N},
\]
where $\bP \in \R^{n \times n}$ and $\bQ \in \R^{N \times n}$, and $\bP_1 \in \R^{n \times \evn}$ and $\bQ_1 \in \R^{N \times \evn}$ correspond to the left and right singular vectors associated to the largest $\evn$ singular values $\bLambda_1$, while $\bP_2 \in \R^{n \times (n-\evn)}$ and $\bQ_2 \in \R^{N \times (n-\evn)}$ correspond to the left and right singular vectors associated to the last $(n-\evn)$ smallest singular values $\bLambda_2$.
Define $\kappa_{>\evn} = \Tr (\Hop_{d , > \evn} )$.

Then the singular value decomposition has the following properties: 
\begin{enumerate}
\item[(a)] Define $\bLambda = \diag ( (\sigma_i ( \bZ / \sqrt{N} ) )_{i \in [n]} )$ the singular values (in non increasing order) of $\bZ / \sqrt{N}$. Then the singular values verify
\begin{align}
\sigma_{\min} ( \bLambda_1) = &~ \min_{i \in [\evn]} \sigma_i (\bZ / \sqrt{N}) =  \kappa_{>\evn}^{1/2} \cdot \omega_{d,\P} ( 1), \label{eq:min_lambda_1} \\
\| \bLambda_2 - \kappa_{>\evn}^{1/2} \cdot \id_{n - \evn} \|_{\op} =&~ \max_{i = \evn+1 , \ldots , n } \big\vert \sigma_i ( \bZ / \sqrt{N} ) - \kappa_{>\evn}^{1/2}  \big\vert = \kappa_{>\evn}^{1/2} \cdot o_{d,\P} (1).\label{eq:bound_lambda_2}
\end{align}

\item[(b)] The left and right singular vectors associated to the $(n-\evn)$ smallest singular values verify
\begin{equation}\label{eq:singular_vectors_Z}
n^{-1/2} \| \bpsi_{\leq \evn}^\sT \bP_2 \|_\op = o_{d,\P}( 1 ), \qquad  N^{-1/2} \| \bphi_{\leq \evn}^\sT \bQ_2 \|_\op = o_{d,\P}( 1 ).
\end{equation}
\item[(c)] We have
\begin{equation}\label{eq:cross_term_Z}
N^{-1/2} \| \bP_1^\sT \bZ_{>\evn} \bQ_2 \|_\op = \kappa_{>\evn}^{1/2} \cdot  o_{d,\P}( 1 ).
\end{equation}
\end{enumerate}
\end{proposition}

We defer the proof of Proposition \ref{prop:singular_values_bZ} to Section \ref{sec:structure_Z}.
\begin{remark}
  Proposition \ref{prop:singular_values_bZ} shows that the feature matrix $\bZ = \bZ_{\leq \evn} + \bZ_{>\evn}$
  (cf. Eq.~\eqref{eq:def_leq_ell_big_ell}) is a spiked matrix, with $\evn$ spikes with singular values $\bLambda_1$ much larger than $\kappa_{>\evn}^{1/2}$ coming from the low-degree part $\bZ_{\leq \evn}$ (in particular, Proposition \ref{prop:singular_values_bZ}.$(b)$ shows that the left and right singular vectors of the spikes are approximately spanned by the left and right singular vectors of $\bZ_{\leq \evn}$) while the rest of the singular values are approximately constant equal to $\kappa_{>\evn}^{1/2}$. The proof of this proposition is based on the following observations:
  \begin{itemize}
  \item[$(a)$] $\bZ_{\leq \evn} / \sqrt{N} = \bpsi_{\leq \evn}  \bD_{\leq \evn}  \bphi_{\leq \evn}^\sT  / \sqrt{N} $ is a rank $\evn$ matrix with 
  \begin{itemize}
  \item[$(i)$] $\bpsi_{\leq \evn} / \sqrt{n}$ and $\bphi_{\leq \evn} / \sqrt{N}$ are approximately orthogonal matrices (see Eq.~\eqref{eq:tbpsi_tbphi_concentration}).
  
  \item[$(ii)$] $\sqrt{n } | \bD_{\leq \evn}  | = \diag ( \sqrt{n} | \lambda_{1} | , \ldots , \sqrt{n} | \lambda_{\evn} | ) \succeq  \omega_{d,\P} ( \kappa_{>\evn}^{1/2}) \cdot \id_{\evn}$ from condition \eqref{eq:NumberOfSamples} in Assumption \ref{ass:spectral_gap}.$(a)$.
  \end{itemize}

  \item[$(b)$] The high degree part $\bZ_{>\evn}/\sqrt{N}$ has nearly constant singular values $\| \bZ_{>\evn} \bZ_{>\evn}^\sT /N - \kappa_{>\evn} \id_n\|_\op = \kappa_{>\evn} \cdot o_{d,\P}(1)$ and is nearly orthogonal to the span of the right singular vectors of $\bZ_{\leq \evn}$, i.e., $\| \bZ_{>\evn} \bphi_{\leq \evn} /N \|_{\op} = \kappa_{>\evn}^{1/2} \cdot o_{d,\P}(1)$ (see Proposition \ref{prop:concentration_bZ} in Section \ref{sec:RFRR_ZZ_concentration}).
  \end{itemize}
\end{remark}

Using Proposition \ref{prop:singular_values_bZ}, we can prove the following list of bounds that will be the main tools for the rest of the proof of Theorem \ref{thm:RFK_generalization}.
\begin{proposition}\label{prop:technical_facts}
Follow the assumptions and the notations in the proof of Theorem \ref{thm:RFK_generalization} in the overparametrized regime. Then the following bounds hold. (Recall that $\kappa_{> \evn} = \Tr (\Hop_{d , > \evn} )$.)
\begin{enumerate}

\item[$(a)$]  Bounds on $\hbUi = (\bZ^\sT \bZ / N + \lambda \id_N)^{-1}$: 
\begin{align}
\bpsi_{\leq \evn}^\sT \bZ \hbUi  \bphi_{\leq \evn} \bD_{\leq \evn} / N =&~    \id_\evn + \bDelta, \label{eq:prop_a_1} \\
\| \bD_{\leq \evn}  \bphi_{\leq \evn}^\sT \hbUi   \bZ^\sT \boldf_{> \evn} /N \|_2 =&~ \| \proj_{>\evn} f_d \|_{L^{2+\eta}} \cdot o_{d,\P} (1 ),\label{eq:prop_a_2} \\
\sqrt{n} \| \bZ \hbUi \bphi_{\leq \evn} \bD_{\leq \evn} / N \|_\op =&~ O_{d,\P}(1 ),\label{eq:ZUL}
\end{align}
where $\| \bDelta \|_\op = o_{d,\P}(1 )$. Furthermore, we have
\begin{align}
\| \bZ \hbUi / \sqrt{N} \|_\op = &~ \kappa_{>\evn}^{-1/2} \cdot O_{d,\P} (   1 ).\label{eq:ZU_bound}
\end{align}  

\item[$(b)$]  Bound on $\bU_{>\evn}$: 
\[
\begin{aligned}
\frac{n}{N} \| \bU_{>\evn} \|_{\op} = \kappa_{>\evn} \cdot o_{d,\P}(1 ).
\end{aligned}
\]

\item[$(c)$]  Bounds on $\boldf$: 
\[
\begin{aligned}
\| \boldf \|_2 = & \sqrt{n} \| f_d \|_{L^2} \cdot O_{d,\P} (1), \\
\| \bpsi_{\leq \evn}^\sT \boldf_{> \evn}/n\|_2 = & \| \proj_{> \evn} f_d \|_{L^{2+\eta}} \cdot o_{d,\P}(1).
\end{aligned}
\]

\item[$(d)$]  Bound on $\bV_{> \evn}$:
\[
\sqrt{\frac{n}{N}} \| \bV_{> \evn} \|_2 =  \kappa_{>\evn}^{1/2} \| \proj_{> \evn}  f_d \|_{L^2} \cdot o_{d,\P}(1 ).
\]

\end{enumerate}

\end{proposition}

The proof of Proposition \ref{prop:technical_facts} is deferred to Section \ref{sec:technical_claims_RFRR}.

\begin{remark}
In the underparametrized case, the proofs and statements of Proposition \ref{prop:singular_values_bZ} and Proposition \ref{prop:technical_facts}.$(a)$ and \ref{prop:technical_facts}.$(c)$ are symmetric under the mapping $n \leftrightarrow N$, $\evn \leftrightarrow \evN$ and $\lambda \rightarrow \lambda_N = N \lambda/n$. The bounds in Propositions \ref{prop:technical_facts}.$(b)$ and \ref{prop:technical_facts}.$(d)$ can be easily replaced by
\[
\| \bU_{>\evN} \|_{\op} = \kappa_{>\evN} \cdot O_{d,\P} (1), \qquad \| \bV_{>\evN} \|_2 = \kappa_{>\evN}^{1/2} \| \proj_{> \evn}  f_d \|_{L^2} \cdot o_{d,\P}(1).
\]
In order to bound the term $T_{22}$ in Eq.~\eqref{eqn:term_R22_RF}, we will further use the following bound
\[
\| \hbUi \bZ^\sT \boldf / n \|_{\op} = \kappa_{>\evN}^{- 1/2} \cdot \| f_d \|_{L^2} \cdot o_{d,\P}(1),
\]
that we prove in Section \ref{sec:bounds_underparametrized}. It is easy to plug the new bounds below with the aforementioned mapping and check that the underparametrized case follows indeed from the same computation.
\end{remark}

The rest of the proof amounts to controlling each term separately using the claims listed in Proposition \ref{prop:technical_facts}. We will use extensively the following (basic) properties of the operator norm: for $\bA \in \R^{m \times p}$, $\bB \in \R^{p \times q}$, $\bu \in \R^m$ and $\bv \in \R^p$, we have
\[
\begin{aligned}
\| \bA \|_{\op} = & \| \bA^\sT \bA \|_{\op}^{1/2} =  \| \bA \bA^\sT \|_{\op}^{1/2},\\
\| \bA \bB \|_{\op} \leq & \| \bA \|_{\op} \|\bB \|_{\op},\\
\bu^\sT \bA \bv \leq & \| \bu\|_2 \| \bA \|_{\op} \| \bv \|_2.
\end{aligned}
\]

\noindent
{\bf Step 3. Term $T_1$.}

Let us decompose $T_1$ into
\[
T_1 = T_{11} + T_{12} + T_{13}, 
\]
where
\[
\begin{aligned}
T_{11} =& \boldf_{\le \evn}^\sT \bZ \hbUi \bV_{\le \evn} / N , \\
T_{12} =& \boldf_{> \evn}^\sT \bZ \hbUi  \bV_{\le \evn} / N , \\
T_{13} =& \boldf^\sT \bZ \hbUi  \bV_{> \evn} / N. \\
\end{aligned}
\]
Recall that $\bV_{\leq \evn} = \bphi_{\leq \evn} \bD_{\leq \evn} \hbf_{\leq \evn}$ and $\boldf_{\leq \evn} = \bpsi_{\leq \evn} \hbf_{\leq \evn}$. Hence by Eq.~\eqref{eq:prop_a_1} in Proposition \ref{prop:technical_facts}.$(a)$, we have
\begin{equation}\label{eqn:term_R11_RF}
\begin{aligned}
  T_{11} =  &\hbf_{\leq \evn}^\sT ( \bpsi_{\leq \evn}^\sT \bZ \hbUi \bphi_{\leq \evn} \bD_{\leq \evn} / N) \hbf_{\leq \evn}\\
  = &  \hbf_{\leq \evn}^\sT  ( \id_\evn   + \bDelta )  \hbf_{\leq \evn} \\
= & \| \proj_{\leq \evn} f_d \|_{L^2}^2 + \| \proj_{\leq \evn} f_d \|_{L^2}^2 \cdot o_{d,\P}(1).
\end{aligned}
\end{equation}
Similarly by Eq.~\eqref{eq:prop_a_2} in Proposition \ref{prop:technical_facts}.$(a)$, 
\begin{equation}\label{eqn:term_R12_RF}
\begin{aligned}
  | T_{12} |= & | \boldf_{> \evn}^\sT  \bZ \hbUi \bphi_{\leq \evn} \bD_{\leq \evn} \hbf_{\leq \evn} /N | \\
  \leq &  \| \bD_{\leq \evn}  \bphi_{\leq \evn}^\sT \hbUi   \bZ^\sT \boldf_{> \evn}  /N \|_2 \| \hbf_{\leq \evn} \|_2  \\
= &   \| \proj_{>\evn} f_d \|_{L^{2+\eta} } \| \proj_{\leq \evn} f_d \|_{L^2}  \cdot o_{d,\P} (1).
\end{aligned}
\end{equation}
Using Proposition \ref{prop:technical_facts}.$(c)$ and \ref{prop:technical_facts}.$(d)$ as well as Eq.~\eqref{eq:ZU_bound} in Proposition \ref{prop:technical_facts}.$(a)$, we get
\begin{equation}\label{eqn:term_R13_RF}
\begin{aligned}
| T_{13} |=  | \boldf^\sT \bZ \hbUi \bV_{>\evn} / N  |
\leq  & \| \boldf / \sqrt{n} \|_2 \| (\bZ / \sqrt{N}) \hbUi  \|_{\op} \cdot  \sqrt{n/N}  \| \bV_{>\evn} \|_2  \\
\leq  &  O_{d,\P} ( \| f_d \|_{L^2} ) \cdot O_{d,\P} (  \kappa_{>\evn}^{-1/2}  ) \cdot o_{d,\P} ( \kappa_{>\evn}^{1/2} \| \proj_{> \evn } f_d \|_{L^2} ) \\
= & \| f_d \|_{L^2}  \| \proj_{> \evn } f_d \|_{L^2} \cdot o_{d,\P}(1).
\end{aligned}
\end{equation}

Combining Eqs. (\ref{eqn:term_R11_RF}), (\ref{eqn:term_R12_RF}) and (\ref{eqn:term_R13_RF}) yields
\begin{equation}\label{eqn:RF_term_T1}
T_1 =  \| \proj_{\le \evn} f_d \|_{L^2}^2 + o_{d, \P}(1) \cdot (\| f_d \|_{L^2}^2 +  \| \proj_{>\evn} f_d \|_{L^{2+\eta} }^2 ). 
\end{equation}

\noindent
{\bf Step 4. Term $T_2$}

Recalling $\bU = \bphi_{\leq \evn} \bD_{\leq \evn}^2 \bphi_{\leq \evn}^\sT + \bU_{>\evn} $, we can decompose $T_2$ as 
\[
T_2 = T_{21}  +  T_{22},
\]
where
\[
\begin{aligned}
T_{21} =& (\boldf^\sT \bZ \hbUi \bphi_{\leq \evn} \bD_{\leq \evn}/N) (\bD_{\leq \evn} \bphi_{\leq \evn}^\sT \hbUi \bZ^\sT \boldf /N ) ,\\
T_{22} =& \boldf^\sT \bZ \hbUi \bU_{>\evn} \hbUi \bZ^\sT \boldf /N^2.
\end{aligned}
\]
From Eqs.~\eqref{eq:prop_a_1} and \eqref{eq:prop_a_2} in Proposition \ref{prop:technical_facts}.$(a)$, we have
\[
\begin{aligned}
\bD_{\leq \evn} \bphi_{\leq \evn}^\sT \hbUi \bZ^\sT \boldf /N = & \bD_{\leq \evn} \bphi_{\leq \evn}^\sT \hbUi \bZ^\sT \bpsi_{\leq \evn} \hbf_{\leq \evn} /N  + \bD_{\leq \evn} \bphi_{\leq \evn}^\sT \hbUi \bZ^\sT \boldf_{>\evn} /N  \\
= & ( \id_\evn + \bDelta_0 ) \hbf_{\leq \evn} + \| \proj_{>\evn} f_d \|_{L^{2 + \eta}} \cdot \bDelta_1,
\end{aligned}
\]
where $\| \bDelta_1 \|_\op = o_{d,\P}(1)$, $\| \bDelta_2 \|_2 = o_{d,\P}(1)$.
Hence,
\begin{equation}\label{eqn:term_R21_RF}
\begin{aligned}
T_{21} = & \| \proj_{\leq \evn} f_d \|_{L^2}^2 + (\| f_d \|_{L^2}^2 + \| \proj_{>\evn}  f_d \|_{L^{2+\delta}}^2 ) \cdot o_{d,\P}(1).
\end{aligned}
\end{equation}
From Eq.~\eqref{eq:ZU_bound} in Proposition \ref{prop:technical_facts}.$(a)$ as well as Proposition \ref{prop:technical_facts}.$(b)$, \ref{prop:technical_facts}.$(c)$, the second term is bounded by
\begin{equation}\label{eqn:term_R22_RF}
\begin{aligned}
 | T_{22} | = & |\boldf^\sT \bZ \hbUi \bU_{>\evn} \hbUi \bZ^\sT \boldf /N^2 |\\
 \leq & \| (n/N) \bU_{>\evn} \|_\op \| (\bZ / \sqrt{N} ) \hbUi  \|_{\op}^2 \| \boldf / \sqrt{n} \|_2^2  \\
 = & o_{d,\P}(\kappa_{>\evn}  ) \cdot O_{d,\P}( \kappa_{>\evn}^{-1} )  \cdot O_{d,\P}( \| f_d \|_{L^2}^2) = \| f_d \|_{L^2}^2 \cdot o_{d,\P}(1).
\end{aligned}
\end{equation}
 As a result, combining Eqs. (\ref{eqn:term_R21_RF}) and (\ref{eqn:term_R22_RF}), we have 
\begin{equation}\label{eqn:RF_term_T2}
T_2 =  \| \proj_{\le \evn} f_d \|_{L^2}^2 + o_{d, \P}(1) \cdot(\| f_d \|_{L^2}^2 + \| \proj_{>\evn}  f_d \|_{L^{2+\eta}}^2 ). 
\end{equation}

\noindent
{\bf Step 5. Terms $T_3, T_4$ and $T_5$. } 

Let us start with the term $T_3$. Decompose $\bU =  \bphi_{\leq \evn} \bD_{\leq \evn}^2 \bphi_{\leq \evn}^\sT + \bU_{>\evn}$:
\[
\begin{aligned}
\E_\beps [T_3] / \noise^2 = &  \tr(\bZ \hbUi \bU \hbUi  \bZ^\sT ) / N^2\\
 = &  \tr(\bZ \hbUi \bphi_{\leq \evn} \bD_{\leq \evn}^2 \bphi_{\leq\evn}^\sT \hbUi  \bZ^\sT ) / N^2 +  \tr(\bZ \hbUi \bU_{>\evn } \hbUi  \bZ^\sT ) / N^2.
\end{aligned}
\]
By Eq.~\eqref{eq:ZUL} in Proposition \ref{prop:technical_facts}.$(a)$, and since $\evn\le n^{1-\delta_0}$
by Assumption \ref{ass:spectral_gap}.$(a)$, we have
\[
\begin{aligned}
 \tr(\bZ \hbUi \bphi_{\leq \evn} \bD_{\leq \evn}^2 \bphi_{\leq \evn}^\sT \hbUi  \bZ^\sT ) / N^2 
\leq & \evn \cdot \| \bZ \hbUi \bphi_{\leq \evn} \bD_{\leq \evn} / N \|_{\op}^2 = \frac{\evn}{n} \cdot O_{d,\P} (1) = o_{d,\P}(1).
\end{aligned}
\]
By Eq.~\eqref{eq:ZU_bound} in Proposition \ref{prop:technical_facts}.$(a)$ as well as Proposition \ref{prop:technical_facts}.$(b)$, the second term is bounded by
\[
\begin{aligned}
 \tr(\bZ \hbUi \bU_{>\evn} \hbUi  \bZ^\sT ) / N^2 \leq & \| (n/N) \bU_{>\evn} \|_\op   \| \bZ \hbU^{-2}  \bZ^\sT/N \|_{\op} / n \\
   = & o_{d,\P}(\kappa_{>\evn} ) \cdot O_{d,\P} ( \kappa_{>\evn}^{-1}  ) \cdot n^{-1} = o_{d,\P}(1).
\end{aligned}
\]
Combining these two bounds and using Markov's inequality, we get
\begin{align}\label{eqn:term_varT3_RF}
T_3 = o_{d, \P}(1) \cdot \noise^2. 
\end{align}

Let us consider $T_4$ term. Recall that we can decompose $\bV = \bphi_{\leq \evn} \bD_{\leq \evn} \hbf_{\leq \evn} + \bV_{>\evn}$,
\[
\begin{aligned}
\E_{\beps} [T_4^2 ]/\noise^2 =  & \tr ( \bZ\hbUi \bV \bV^\sT \hbUi \bZ^\sT ) / N^2 \\
= & \bV^\sT \hbUi  \bZ^\sT \bZ \hbUi \bV / N^2 \\
\leq & 2( \|  \bZ \hbUi \bV_{\leq \evn} / N \|_2^2 + \|  \bZ \hbUi \bV_{> \evn} / N \|_2^2 ).  
\end{aligned}
\]
We have by Eq.~\eqref{eq:ZUL} in Proposition \ref{prop:technical_facts}.$(a)$,
\[
\begin{aligned}
\|  \bZ \hbUi \bV_{\leq \evn} / N \|_2 \leq & \|   \bZ \hbUi \bphi_{\leq \evn } \bD_{\leq \evn} / N \|_\op \| \hbf_{\leq \evn} \|_2 = \| \proj_{\leq \evn} f_d \|_{L^2} \cdot o_{d,\P}(1),
\end{aligned}
\]
and by Proposition \ref{prop:technical_facts}.$(d)$,
\[
\begin{aligned}
\|  \bZ \hbUi \bV_{> \evn} / N \|_2  \leq & \|  \bZ \hbUi  / \sqrt{N} \|_2 \| \bV_{>\evn} / \sqrt{N} \|_2 \\
=& O_{d,\P}( \kappa_{>\evn}^{-1/2} ) \cdot o_{d,\P} ( \kappa_{>\evn}^{1/2} \| \proj_{>\evn} f_d \|_{L^2} n^{-1/2 } ) =  \| \proj_{>\evn} f_d \|_{L^2} \cdot o_{d,\P}(1).
\end{aligned}
\]
Combining the two above bounds, we get by Markov's inequality
\begin{align}\label{eqn:term_varT4_RF}
T_4 = o_{d, \P}(1) \cdot \noise \| f_d \|_{L^2} = o_{d,\P} (1) \cdot ( \noise^2 +  \| f_d \|_{L^2}^2 ). 
\end{align}

Let us consider the last term $T_5$. We have
\[
\begin{aligned}
\E_{\beps} [T_5^2 ]/\noise^2  = & \tr ( \bZ \hbUi  \bU \hbUi  \bZ^\sT \boldf \boldf^\sT \bZ \hbUi  \bU \hbUi \bZ^\sT  )/ N^4 \\ 
=& \| \bZ \hbUi  \bU \hbUi  \bZ^\sT \boldf  /N^2 \|_2^2 \leq  \| \bZ \hbUi  \bU \hbUi  \bZ^\sT \sqrt{n}/N^2 \|_\op^2 \| \boldf / \sqrt{n} \|_2^2.
\end{aligned}
\]
By Eq.~\eqref{eq:prop_a_2} in Proposition \ref{prop:technical_facts}.$(a)$, and
Proposition \ref{prop:technical_facts}.$(b)$,
\[
\begin{aligned}
\| \bZ \hbUi  \bU \hbUi  \bZ^\sT \sqrt{n}/N^2 \|_\op \leq & \sqrt{n} \cdot  \| \bZ \hbUi  \bphi_{\leq \evn} \bD_{\leq \evn} / N \|_\op^2 + \| \sqrt{n/N^2} \bU_{>\evn} \|_\op   \|  \bZ  \hbU^{-2} \bZ^\sT /N \|_\op = o_{d,\P}(1).
\end{aligned}
\]
Hence, by Proposition \ref{prop:technical_facts}.$(c)$,
\[
\E_{\beps} [T_5^2 ]/\noise^2 =  o_{d,\P}(1) \cdot  \| \boldf / \sqrt{n} \|_2^2 =  \| f_d \|_2^2 \cdot o_{d,\P}(1),
\]
which gives by Markov's inequality
\begin{align}\label{eqn:term_varT5_RF}
T_5 = \noise \| f_d \|_{L^2}  \cdot o_{d, \P}(1)  =   ( \noise^2 +  \| f_d \|_{L^2}^2 ) \cdot  o_{d,\P} (1).
\end{align}

\noindent
{\bf Step 6. Finish the proof. }

Combining Eqs. (\ref{eqn:RF_term_T1}), (\ref{eqn:RF_term_T2}), (\ref{eqn:term_varT3_RF}), (\ref{eqn:term_varT4_RF}) and (\ref{eqn:term_varT5_RF}), we have
\[
\begin{aligned}
R_{\RF}( f_d, \bX, \bW, \lambda) = & \| f_d \|_{L^2}^2 - 2 T_{1} + T_{2} + T_3 - 2 T_4 + 2T_5 \\
= & \| f_d \|_{L^2}^2  - 2 \| \proj_{\leq \evn} f_d \|_{L^2}^2 + \| \proj_{\leq \evn} f_d \|_{L^2}^2 + o_{d, \P}(1) \cdot (\| f_d \|_{L^2}^2 + \|  \proj_{>\evn} f_d \|_{L^{2+\eta} }^2 + \noise^2) \\
=& \| \proj_{> \evn} f_d \|_{L^2}^2 + o_{d, \P}(1) \cdot  (\| f_d \|_{L^2}^2 + \| \proj_{>\evn} f_d \|_{L^{2+\eta} }^2 + \noise^2),
\end{aligned}
\]
which concludes the proof.

\subsection{Proof of Proposition \ref{prop:singular_values_bZ}: Structure of the feature matrix $Z$}
\label{sec:structure_Z}

Recall the definition $\bZ = ( \sigma_d (\bx_i ; \btheta_j) )_{i \in [n], j \in [N]}$.
Recall the decomposition $\bZ = \bZ_{\leq \evn} + \bZ_{>\evn}$ into a low and high degree parts,
as per Eq.~\eqref{eq:def_leq_ell_big_ell}.
For convenience, we will consider the normalized  quantities
\[
\begin{aligned}
\tbZ = \bZ /\sqrt{N},& \qquad \tbZ_{\leq \evn} = \bZ_{\leq \evn} /\sqrt{N},\qquad  &\tbZ_{> \evn} = \bZ_{>\evn} /\sqrt{N},\\
 \tbphi_{\leq \evn} = \bphi_{\leq \evn} /\sqrt{N},& \qquad  \tbpsi_{\leq \evn} = \bpsi_{\leq \evn} /\sqrt{n}, \qquad  &\tbD_{\leq \evn} = \sqrt{n} \bD_{\leq \evn} .
\end{aligned}
\]
In particular, notice that $\hbU = \tbZ^\sT \tbZ +  \lambda \id_N$ and $\tbZ_{\leq \evn} = \tbpsi_{\leq \evn} \tbD_{\leq \evn} \tbphi_{\leq \evn}^\sT$.

By Proposition \ref{prop:YY_new} applied to $\tbphi_{\leq \evn}$ and $\tbpsi_{\leq \evn}$ (with assumptions satisfied by Assumption \ref{ass:FMPCP}.$(a)$ and Assumption \ref{ass:spectral_gap}.$(a)$), we get
\begin{equation}
\tbphi_{\leq \evn}^\sT \tbphi_{\leq \evn} = \id_\evn + \bDelta_1, \qquad \tbpsi_{\leq \evn}^\sT \tbpsi_{\leq \evn} = \id_\evn + \bDelta_2, \label{eq:tbpsi_tbphi_concentration}
\end{equation}
with $\| \bDelta_i \|_\op = o_{d,\P}(1)$ for $i = 1,2$. Furthermore, by Proposition \ref{prop:concentration_bZ} (stated in Section \ref{sec:RFRR_ZZ_concentration}), we have
\begin{equation}\label{eq:tbZ_concentration}
\tbZ_{>\evn}  \tbZ_{>\evn}^\sT = \kappa_{>\evn} \cdot (\id_n + \bDelta_{\bZ}), \qquad \| \tbZ_{>\evn}  \tbphi_{\leq \evn} \|_\op = \kappa_{>\evn}^{1/2} \cdot o_{d,\P}(1 ),
\end{equation}
with $\| \bDelta_\bZ \|_\op = o_{d,\P}(1)$ and where we recall $\kappa_{>\evn} = \Tr ( \Hop_{d , > \evn} )$.
Furthermore, Assumption \ref{ass:spectral_gap}.$(a)$ implies that 
\begin{equation}\label{eq:lower_spiked_values}
\sigma_{\min} ( \tbD_{\leq \evn}) = \min_{k \leq \evn} \{ \sqrt{n} |\lambda_{d,k} | \} = \omega_{d} ( 1) \cdot  \kappa_{>\evn}^{1/2}.
\end{equation} 

Hence, we expect $\tbZ = \tbZ_{\leq \evn} + \tbZ_{>\evn} $ to have $\evn$ large singular values $ \omega_{d} ( 1) \cdot  \kappa_{>\evn}^{1/2}$ associated to $\tbZ_{\leq \evn}$ with left and right singular vectors spanned approximately by $\tbpsi_{\leq \evn}$ and $\tbphi_{\leq \evn}$, and $n-\evn$ small singular values approximately equal to $\kappa_{>\evn}^{1/2}$ associated to $\tbZ_{> \evn}$.

\begin{proof}[Proof of Proposition \ref{prop:singular_values_bZ}]

\noindent
{\bf Claim $(a)$. Bound on the singular values. } 

Using Eqs.~\eqref{eq:tbpsi_tbphi_concentration} and \eqref{eq:lower_spiked_values}, we have
\[
\begin{aligned}
\tbZ_{\leq \evn} \tbZ_{\leq \evn}^\sT =&~  \tbpsi_{\leq \evn} \tbD_{\leq \evn} \tbphi_{\leq \evn}^\sT \tbphi_{\leq \evn} \tbD_{\leq \evn} \tbpsi_{\leq \evn}^\sT \\
=&~ \tbpsi_{\leq \evn}  \tbD_{\leq \evn}(\id_\evn + \bDelta) \tbD_{\leq \evn}  \tbpsi_{\leq \evn}^\sT \\
\succeq &~ \Omega_{d,\P} (1) \cdot  \tbpsi_{\leq \evn}  \tbD_{\leq \evn}^2 \tbpsi_{\leq \evn}^\sT \\
\succeq&~ \kappa_{>\evn} \cdot \omega_{d} (1 ) \cdot \tbpsi_{\leq \evn} \tbpsi_{\leq \evn}^\sT.
\end{aligned}
\]
Furthermore, by $\tbpsi_{\leq \evn}^\sT \tbpsi_{\leq \evn} = \id_\evn + \bDelta_2$, we deduce that the singular values of  $\tbZ_{\leq \evn}$ are lower bounded as follows
\begin{equation}\label{eq:sing_val_z_leq}
\min_{i \in [\evn]} \sigma_i ( \tbZ_{\leq \evn} ) = \kappa_{>\evn} \cdot \omega_{d,\P} (1 ).
\end{equation}

By Lemma \ref{lem:ineq_sing_value} stated below in Section \ref{sec:aux_structure_Z}, we have for $i \in [n]$,
\begin{equation}\label{eq:lemma_8}
| \sigma_i ( \tbZ ) - \sigma_i ( \tbZ_{\leq \evn} ) | \leq \| \tbZ_{>\evn} \|_{\op}.
\end{equation}
Recalling Eq.~\eqref{eq:tbZ_concentration}, $\| \tbZ_{>\evn} \|_{\op} = O_{d,\P} ( 1)\cdot  \kappa_{>\evn}^{1/2}$. Hence the first $\evn$ singular values verify
\begin{equation}
\sigma_i ( \tbZ ) \geq \sigma_i ( \tbZ_{\leq \evn} ) -  \kappa_{>\evn}^{1/2} \cdot O_{d,\P}(1).
\end{equation}
Using Eq.~\eqref{eq:sing_val_z_leq} implies $\sigma_{\min} ( \bLambda_1 ) = \min_{i \in [\evn]} \sigma_i ( \tbZ) = \kappa_{>\evn}^{1/2} \cdot \omega_{d,\P}(1)$. This proves Eq.~\eqref{eq:min_lambda_1}.

Using again Eq.~\eqref{eq:lemma_8}, the $n-\evn$ smallest singular values verify
\begin{equation}\label{eq:lamb2_upper}
\max_{ i = \evn+1, \ldots , n} \sigma_i ( \tbZ)\leq \kappa_{>\evn}^{1/2} \cdot (1 + o_{d,\P}(1)).
\end{equation}
In order to lower bound the $n-\evn$ smallest singular values, we lower bound the eigenvalues of $\tbZ\tbZ^\sT $. We have
\[
\tbZ \tbZ^\sT = \tbZ_{\leq \evn} \tbZ_{\leq \evn}^\sT + \tbZ_{> \evn} \tbZ_{\leq \evn}^\sT + \tbZ_{\leq \evn} \tbZ_{> \evn}^\sT + \tbZ_{> \evn} \tbZ_{> \evn}^\sT.
\]
Recalling Eq.~\eqref{eq:tbpsi_tbphi_concentration} and Eq.~\eqref{eq:tbZ_concentration}, we have
\[
\begin{aligned}
 \tbZ_{\leq \evn} \tbZ_{\leq \evn}^\sT = &~ \tbpsi_{\leq \evn} \tbD_{\leq \evn} ( \tbphi_{\leq \evn}^\sT \tbphi_{\leq \evn} ) \tbD_{\leq \evn} \tbpsi_{\leq \evn}^\sT = \tbpsi_{\leq \evn} \tbD_{\leq \evn} ( \id_\evn + \bDelta_1 ) \tbD_{\leq \evn} \tbpsi_{\leq \evn} , \\
 \tbZ_{> \evn} \tbZ_{> \evn}^\sT =&~  \kappa_{>\evn} \cdot ( \id_n + \bDelta_\bZ ),
\end{aligned}
\]
where $\| \bDelta_\bZ \|_{\op} = o_{d,\P} ( 1)$.

Denote $\bL = \tbpsi_{\leq \evn} \tbD_{\leq \evn} ( \id_\evn + \bDelta_1 )^{1/2}$ and $\bT = \tbZ_{>\evn} \tbphi_{\leq \evn} ( \id_\evn + \bDelta_1 )^{-1/2}$. By Eq.~\eqref{eq:tbZ_concentration}, we have $\|  \bT \bT^\sT \|_\op = \kappa_{>\evn} \cdot o_{d,\P}(1)$. Combining these remarks, we get
\[
\begin{aligned}
 \tbZ \tbZ^\sT = &~ \bL \bL^\sT + \bT \bL^\sT + \bL \bT^\sT + \bT \bT^\sT - \bT \bT^\sT +  \tbZ_{> \evn} \tbZ_{> \evn}^\sT\\
 =&~ ( \bL + \bT ) ( \bL + \bT)^\sT + \kappa_{>\evn} \cdot ( \id_n + \bDelta') \\
 \succeq &~ \kappa_{>\evn} \cdot ( \id_n + \bDelta'),
\end{aligned}
\]
where $\| \bDelta ' \|_{\op} = o_{d,\P} (1)$. We deduce that 
\[
\sigma_{\min} (  \tbZ ) = \min_{ i \in [n]} \sigma_i ( \tbZ)  \geq \kappa_{>\evn}^{1/2} \cdot (1 + o_{d,\P}(1)),
\]
 which combined with Eq.~\eqref{eq:lamb2_upper} yields Eq.~\eqref{eq:bound_lambda_2}.

\noindent
{\bf Part $(b)$. Left and right singular vectors. } 

Let us prove $\| \tbphi_{\leq \evn}^\sT \bQ_2 \|_{\op} = o_{d,\P}(1)$. The proof for $\tbpsi_{\leq \evn}^\sT \bP_2$ follows from the same argument by replacing $\tbZ$ by $\tbZ^\sT$ and using the bound $\| \tbZ_{>\evn} \tbphi_{\leq \evn} \|_{\op}  =  \kappa_{>\evn}^{1/2} \cdot o_{d,\P} ( 1 )$, cf. Eq.~\eqref{eq:tbZ_concentration}.

Let us consider a sequence $\bu \in \R^{n-\evn}$ (where we keep the dependency on $d$ implicit) such that $\| \bu \|_2 = 1$ and $\| \tbphi_{\leq \evn}^\sT \bQ_2 \bu \|_2 = \| \tbphi_{\leq \evn}^\sT \bQ_2 \|_{\op}$. For convenience, denote $\tbu = \tbphi_{\leq \evn}^\sT \bQ_2 \bu $. We have
\begin{equation}\label{eq:sing_vec_P2}
\begin{aligned}
 \bu^\sT \bLambda_2^2 \bu =&~ \bu^\sT \bQ_2^\sT \tbZ^\sT \tbZ \bQ_2 \bu \\
 = &~ \bu^\sT \bQ_2^\sT ( \tbZ_{\leq \evn}^\sT \tbZ_{\leq \evn} + \tbZ_{> \evn}^\sT \tbZ_{\leq \evn} + \tbZ_{\leq \evn}^\sT \tbZ_{> \evn}^\sT + \tbZ_{> \evn}^\sT \tbZ_{> \evn} )\bQ_2 \bu \\
= &~ \tbu^\sT \tbD_{\leq \evn} ( \id_\evn + \bDelta_2) \tbD_{\leq \evn} \tbu + 2 \tbu^\sT \tbD_{\leq \evn} ( \tbpsi_{\leq \evn}^\sT \tbZ_{>\evn} \bu) + \| \tbZ_{>\evn} \bQ_2 \bu \|_2^2.
\end{aligned}
\end{equation}
From step 1, we have $\bu^\sT \bLambda_2^2 \bu  =\kappa_{>\evn}  \cdot  O_{d,\P} (1 )$. Furthermore, 
\begin{equation}\label{eq:sing_vec_lower}
\begin{aligned}
\tbu^\sT \tbD_{\leq \evn} ( \id_\evn + \bDelta_2) \tbD_{\leq \evn} \tbu = &~ \Omega_{d,\P} (1) \cdot \| \tbD_{\leq \evn} \tbu \|_2^2, \\
\tbu^\sT \tbD_{\leq \evn} ( \tbpsi_{\leq \evn}^\sT \tbZ_{>\evn} \bu)  \geq  & - \| \tbD_{\leq \evn} \tbu \|_2   \| \tbpsi_{\leq \evn}^\sT \tbZ_{>\evn}  \|_{\op},\\
 \| \tbpsi_{\leq \evn}^\sT \tbZ_{>\evn}  \|_{\op}  \leq &~  \| \tbpsi_{\leq \evn} \|_{\op} \| \tbZ_{>\evn}  \|_{\op}  = \kappa_{>\evn}^{1/2} \cdot O_{d,\P} ( 1).
\end{aligned}
\end{equation}
Therefore, using the bounds \eqref{eq:sing_vec_lower} in Eq.~\eqref{eq:sing_vec_P2}, we get
\[
\Omega_{d,\P} (1) \cdot \| \tbD_{\leq \evn} \tbu \|_2^2  - 2  \| \tbD_{\leq \evn} \tbu \|_2   \| \tbpsi_{\leq \evn}^\sT \tbZ_{>\evn}  \|_{\op}  \leq \kappa_{>\evn}  \cdot  O_{d,\P} (1 ). 
\]
Hence, 
\begin{equation}\label{eq:ineq_Du}
\| \tbD_{\leq \evn} \tbu \|_2 = O_{d,\P}  \Big( \max \big( \kappa_{>\evn}^{1/2} ,  \| \tbpsi_{\leq \evn}^\sT \tbZ_{>\evn}  \|_{\op}  \big) \Big)= \kappa_{>\evn}^{1/2} \cdot O_{d,\P} (1 ).
\end{equation}
Using the bound $\| \tbD_{\leq \evn} \tbu \|_2 =   \kappa_{>\evn}^{1/2} \cdot \omega_{d} ( 1 ) \cdot \| \tbu \|_2 = \kappa_{>\evn}^{1/2} \cdot \omega_{d} ( 1 ) \cdot \| \bQ_2^\sT \tbphi_{\leq \evn} \|_{\op}$ in Eq.~\eqref{eq:ineq_Du}, we deduce that $ \| \bQ_2^\sT \tbphi_{\leq \evn} \|_{\op} = o_{d,\P}(1)$. This concludes the proof of Proposition \ref{prop:singular_values_bZ}.$(b)$.

\noindent
{\bf Part $(c)$. Cross term bound. } 

This is a direct application of Lemma \ref{lem:linear_algebra} (stated below in Section \ref{sec:aux_structure_Z}) with matrix $\kappa_{>\evn}^{- 1/2} \tbZ  =\kappa_{>\evn}^{- 1/2} \tbZ_{\leq \evn} + \kappa_{>\evn}^{- 1/2} \tbZ_{>\evn}$. Indeed, Eq.~\eqref{eq:sing_val_z_leq} implies that $\sigma_{\min} ( \kappa_{>\evn}^{- 1/2} \tbZ_{\leq \evn} ) = \omega_{d,\P}(1)$ and Eq.~\eqref{eq:tbZ_concentration} gives $\| \kappa_{>\evn}^{- 1} \tbZ_{>\evn} \tbZ_{>\evn}^\sT - \id_n  \|_{\op} = o_{d,\P} (1)$. Furthermore, the right singular vectors $\bV_0$ of $\tbZ_{\leq \evn}$ are spanned by the left singular vectors of $\tbphi_{\leq \evn}$. From Eq.~\eqref{eq:tbZ_concentration}, we have $\| \tbZ_{> \evn} \tbphi_{\leq \evn} \|_{\op} = \kappa_{>\evn}^{1/2} \cdot o_{d,\P}(1 )$. Combined with $\| \tbphi_{\leq \evn}^\sT \tbphi_{\leq \evn} - \id_\evn \|_\op = o_{d,\P}(1)$, we get $\| \kappa_{>\evn}^{-1/2} \tbZ_{>\evn} \bV_0 \|_\op = o_{d,\P}(1 )$.
\end{proof}

\subsubsection{Auxiliary lemmas}
\label{sec:aux_structure_Z}

We recall the following classical perturbation theory result.

\begin{theorem}[Sin($\Theta$) theorem for rectangular matrices \cite{wedin1972perturbation}]\label{thm:sin_theta}
Let $\bA_0$ be a $n \times N$-matrix with singular value decomposition
\[
\bA_0 = \bU_0 \bSigma_0 \bV_0^\sT,
\]
where $\bU_0 \in \R^{n \times \evn}$, $\bV_0 \in \R^{N \times \evn}$ verify $\evn \leq \min ( n , N)$ and $\bU_0^\sT \bU_0 = \bV_0^\sT \bV_0 = \id_\evn$, and $\bSigma_0 = \diag ( (\sigma_i ( \bA_0) )_{i \in [\evn]} )$ are the singular values. Let $\bM$ be a perturbation $n \times N$-matrix and consider $\bB = \bA_0 + \bM$ with singular value decomposition
\[
\bB = \bP \bSigma \bQ = [\bP_1 , \bP_2] \diag ( \bLambda_1 , \bLambda_2 ) [\bQ_1 , \bQ_2 ]^\sT,
\]
where $\bP_1 \in \R^{n \times \evn}, \bQ_1 \in \R^{N \times \evn}, \bP_2 \in \R^{n \times (n-\evn)}, \bQ_2 \in \R^{N \times (n-\evn)}$. Assume that $\sigma_{\min} ( \bLambda_1 ) > 0 $. Then 
\begin{equation}\label{eq:sin_theta}
\max ( \| ( \id_n - \bU_0 \bU_0^\sT) \bP_1 \|_\op, \| ( \id_N - \bV_0 \bV_0^\sT) \bQ_1 \|_\op ) \leq \frac{\max ( \| \bM \bQ_1 \|_\op , \| \bM^\sT \bP_1 \|_\op )}{\sigma_{\min} ( \bLambda_1 ) }.
\end{equation}
\end{theorem}

\begin{lemma}[Weyl's inequality]\label{lem:ineq_sing_value}
Consider $\bA_0 , \bM \in \R^{n \times N}$ and define $\bB = \bA_0 + \bM$. Then for any $i \in [\min(n, N)]$, we have
\begin{equation}\label{eq:ineq_sing_value}
| \sigma_i ( \bB ) - \sigma_i ( \bA_0 ) | \leq \| \bM \|_\op.
\end{equation}
\end{lemma}

The next lemma implies  that the projection of the noise matrix $\bM$ on the top left singular vectors of the full matrix is approximately in the  space orthogonal to the right singular vectors.
\begin{lemma}[Null space of right singular vectors]\label{lem:linear_algebra}
Let $\{ N(d) \}_{d \geq 1}$, $\{ n(d) \}_{d \geq 1}$ and $\{ \evn(d) \}_{d \geq 1}$ be three sequences of integers. For convenience, we denote $N = N(d)$, $n = n(d)$ and $\evn = \evn(d)$. Assume that $N \geq n+\evn$ and $n \geq \evn$. Consider the following sequence of random spiked matrices:
\[
\bB := \bB(d) = \bA_0 + \bM =  \bU_0 \bSigma_0 \bV_0^\sT + \bM \in \R^{n \times N},
\]
where $\bU_0 \bSigma_0 \bV_0^\sT$ is the singular value decomposition of the rank $\evn$ matrix $\bA_0$ with $\bU_0 \in \R^{n \times \evn}$, $\bV_0 \in \R^{N \times \evn}$ and $\bU_0^\sT \bU_0 = \bV_0^\sT \bV_0 = \id_\evn$, and $\bSigma_0 = \diag ( (\sigma_{0, i} (\bA_0) )_{i \in [\evn] } ) \in \R^{\evn \times \evn}$ are the singular values. Further assume that
\begin{itemize}
\item[(a)] $\sigma_{\min} (\bA_0 ) = \min_{i \in [\evn]} \sigma_{0,i} ( \bA_0) = \omega_{d,\P} (1 )$, 
\item[(b)] $\| \bM \bV_0 \|_\op = o_{d,\P} (1 )$,
\item[(c)] $\| \bM \bM^\sT - \id_n  \|_{\op} = o_{d,\P} (1)$.
\end{itemize}

Denote $\bB = \bP \bLambda \bQ^\sT = [\bP_1 , \bP_2] \diag ( \bLambda_1 , \bLambda_2 ) [ \bQ_1 , \bQ_2]^\sT$ the singular value decomposition of $\bB$ where $\bP_1 \in \R^{n \times \evn}$ and $\bQ_1 \in \R^{N \times \evn}$ correspond to the left and right singular vectors associated to the first $\evn$ singular values $\bLambda_1$, while $\bP_2 \in \R^{n \times (n-\evn)}$ and $\bQ_2 \in \R^{N \times (n-\evn)}$ correspond to the left and right singular vectors associated to the last $(n-\evn)$ singular values $\bLambda_2$.

Then we have
\begin{equation}
 \| \bP_1^\sT \bM \bQ \|_\op = o_{d,\P}(1).
 \label{eq:cross_svd_algebra}
\end{equation}
\end{lemma}

\begin{proof}[Proof of Lemma \ref{lem:linear_algebra}]

\noindent
{\bf Step 1. Simplification of the problem. } 

Without loss of generality, we can choose an orthonormal basis in $\reals^N$ so that,
in that basis
    \begin{equation}\label{eq:M_rotated}
      \bV_0 = \begin{bmatrix}\id_\evn \\ \bzero_{N-\evn, \evn} \end{bmatrix}\, ,\qquad
\bM = \begin{bmatrix}
\bM_{1} & \bM_{2}  & \bzero_{n, N-(n+\evn)}\\
\end{bmatrix},
\end{equation}
where $\bM_{1} \in \R^{n \times \evn}$ and $\bM_{2} \in \R^{n \times n}$. Because the space corresponding to the last $N - (n+\evn)$ coordinates of the row is in the right null space of both $\bA_0$ and $\bM$ we can forget about them and consider --without loss of generality--
$\bM = [ \bM_1 , \bM_2] \in \R^{n \times (n + \evn)}$, $N=n+\evn$.

From the assumption $\| \bM \bV_0 \|_\op = o_{d,\P}(1 )$, we have 
\begin{equation}\label{eq:M1_bound}
\| \bM_{1} \|_\op = o_{d,\P}(1 ).
\end{equation}
Furthermore, from the assumption $\| \bM \bM^\sT - \id_n  \|_{\op} = o_{d,\P} (1)$, we have
\begin{equation}\label{eq:M2_bound}
\| \bM_{2} \bM_2^\sT - \id_n \|_\op = o_{d,\P}(1 ).
\end{equation}

\noindent
{\bf Step 2. There exists an orthogonal matrix $\bR \in \R^{\evn \times \evn}$ such that $\| \bP_1 - \bU_0 \bR \|_\op = o_{d,\P}(1 )$. } 

Recall that $\bLambda_1 = \diag ((\sigma_{1,i} (\bB) )_{i \in [\evn]}  )$. By Lemma \ref{lem:ineq_sing_value}, we have for any $i\in [\evn]$,
\[
| \sigma_{1,i} (\bB) - \sigma_{0,i} (\bA_0) | \leq \| \bM \|_\op.
\]
Using the assumption $(a)$ that $\sigma_{\min} ( \bA_0 ) = \omega_{d,\P} (1 )$ and assumption $(c)$ $\| \bM \|_{\op} = O_{d,\P}(1)$, we deduce that
\begin{equation}\label{eq:lambda1_big}
\sigma_{\min} ( \bLambda_1 ) = \omega_{d,\P}(1).
\end{equation}
Furthermore $\| \bM \bQ_1 \|_\op \leq \| \bM \|_{\op} = O_{d,\P}(1)$ and similarly $\| \bM^\sT \bP_1 \|_\op= O_{d,\P}(1)$. We can therefore apply Theorem \ref{thm:sin_theta} which gives
\[
\| ( \id_n - \bU_0 \bU_0^\sT ) \bP_1 \|_\op = o_{d,\P}(1).
\]
Denote by $\bU_{0,\perp} \in \R^{n \times (n - \evn)}$ a matrix such that $[\bU_0 , \bU_{0, \perp} ]$ is orthogonal, the last equation implies $\| \bU_{0,\perp}^\sT \bP_1 \|_\op = o_{d,\P}(1 )$. Further,
\[
\bP_1^\sT \bU_0 \bU_0^\sT \bP_1 = \id_{\evn} - \bP_1^\sT(\id_{n} -\bU_{0}  \bU_{0}^\sT ) \bP_1,
\]
which shows that $\| \bP_1^\sT \bU_0 \bU_0^\sT \bP_1 - \id_{\evn} \|_{\op} =  o_{d,\P}(1 )$.
This implies $\bU_0^\sT \bP_1$ is an approximately orthogonal matrix. Namely,
let its singular value decomposition be $\bU_0^\sT \bP_1= \bR_1\bS\bR_2^{\sT}$. Then,
by defining  the orthogonal matrix $\bR:= \bR_1\bR_2^{\sT} \in \R^{\evn \times \evn}$, we have $\| \bP_1 - \bU_0 \bR \|_\op = o_{d,\P}(1 )$.

\noindent
{\bf Step 3. The null space of the right eigenvectors $\bQ$. } 

Let us explicitly describe the null space of $\bQ \in \R^{ (n+\evn) \times n}$ (recall that we removed the $N-(n+\evn)$ last coordinates of the columns). Consider $\bN_1 \in \R^{\evn \times \evn}$ a rank $\evn$ matrix and write $\bN_2 \in \R^{n \times \evn}$ as a function of $\bN_1$ such that $\ker (\bQ)$ is spanned by the columns of the matrix $\bN = \begin{bmatrix} \bN_1 \\ \bN_2
  \end{bmatrix}\in \R^{(n + \evn) \times \evn}$, i.e., $\bB \bN = \bzero$, that is
\[
\begin{bmatrix} \bU_0 \bSigma_0 + \bM_1 & \bM_2 \end{bmatrix} \begin{bmatrix}
\bN_1 \\
\bN_2 
\end{bmatrix}  = ( \bU_0 \bSigma_0 + \bM_1) \bN_1 +  \bM_2 \bN_2 = \bfzero\, .
\]
Projecting on the two orthogonal subspaces $\bU_{0}$ and $\bU_{0,\perp}$, this is equivalent to
\begin{equation}\label{eq:equ_ker_Q}
\bN_1 = - ( \bSigma_0 + \bU_0^\sT \bM_{1} )^{-1} \bU_{0}^\sT \bM_2 \bN_2, \qquad \bU_{0,\perp}^\sT \bM_1 \bN_1 = - \bU_{0,\perp}^\sT \bM_{2} \bN_2.
\end{equation}
Let us do the following reparametrization $\bN_2 = \bM_2^{-1} \tbN_2$ and fix $\bN_1 = - ( \bSigma_0 + \bU_0^\sT \bM_{1} )^{-1}$. Then Eq.~\eqref{eq:equ_ker_Q} gives
  \[
  \bU_0^\sT \tbN_2 = \id_{\evn} , \qquad \bU_{0,\perp}^\sT \tbN_2 = \bU_{0,\perp}^\sT \bM_1 ( \bSigma_0 + \bU_0^\sT \bM_{1} )^{-1},
  \]
which gives $\tbN_2 =   \bU_0+  \bU_{0,\perp} \bU_{0,\perp}^\sT \bM_1 ( \bSigma_0 + \bU_0^\sT \bM_{1} )^{-1} $, and
\[
\begin{aligned}
\bN_1    = &~ - ( \bSigma_0 + \bU_0^\sT \bM_{1} )^{-1} , \\
\bN_2 = &~  \bM_2^{-1} \bU_0  +  \bM_2^{-1}  \bU_{0,\perp}  \bU_{0,\perp}^\sT \bM_1 ( \bSigma_0 + \bU_0^\sT \bM_{1} )^{-1}.
\end{aligned}
\]
By the assumption $\lambda_{\min} ( \bSigma_0 ) = \omega_{d,\P} (1 )$ and Eq.~\eqref{eq:M1_bound}, we have $\| ( \bSigma_0 + \bU_0^\sT \bM_{1} )^{-1} \|_{\op} = o_{d,\P} ( 1 )$. Furthermore, from Eq.~\eqref{eq:M2_bound}, we have $\| \bM_2^{-1} - \bM_2^\sT \|_{\op} = o_{d,\P}(1)$. We deduce that
\begin{equation}
\begin{aligned}\label{eq:N12_bounds}
\| \bN^\sT -  \begin{bmatrix}
\bzero_{n,\evn} & \bU_0^\sT \bM_2
\end{bmatrix}\|_{\op}  =  o_{d,\P}(1).
\end{aligned}
\end{equation}

\noindent
{\bf Step 4. Concluding the proof. } 

By construction, $\bN^\sT \bQ = \bzero$ and using Eq.~\eqref{eq:N12_bounds}, we get
\begin{equation}\label{eq:null_space_equ}
\begin{aligned}
\| \bN^\sT \bQ -  \begin{bmatrix}
\bzero_{n,\evn} & \bU_0^\sT \bM_{2}
\end{bmatrix} \bQ \|_{\op}  =  \|  \begin{bmatrix}
\bzero_{n,\evn} & \bU_0^\sT \bM_{2}
\end{bmatrix} \bQ \|_{\op}  = o_{d,\P}(1 ).
\end{aligned}
\end{equation}
Furthermore using step $2$ and recalling that $\| \bM_{1} \|_\op  = o_{d,\P}(1)$,
\begin{equation}\label{eq:null_space_equ2}
\| \bP_1^\sT \bM  - \begin{bmatrix}
\bzero_{n,\evn} & \bR^\sT \bU_0^\sT \bM_{2}
\end{bmatrix} \|_{\op} = o_{d,\P}(1).
\end{equation}
Combining Eqs.~\eqref{eq:null_space_equ} and \eqref{eq:null_space_equ2}, we get
\[
\| \bR \bP_1^\sT \bM \bQ -  \begin{bmatrix}
\bzero_{n,\evn} & \bU_0^\sT \bM_{2}
\end{bmatrix} \bQ \|_\op = o_{d,\P}(1 ),
\]
and $\|  \bP_1^\sT \bM \bQ  \|_{\op} = \| \bR \bP_1^\sT \bM \bQ  \|_{\op} = o_{d,\P} (1)$, which concludes the proof.
\end{proof}

\subsection{Proof of Proposition \ref{prop:technical_facts}: technical bounds in the overparametrized regime}
\label{sec:technical_claims_RFRR}

We prove the claims of this proposition in a different order than stated.

\subsubsection{Proof of claim $(c)$}
\label{sec:claim_c}

First, notice that $\E [ \| \boldf \|_2^2 ] = n \| f_d \|_{L^2}^2$. Hence, by Markov's inequality, $ \| \boldf \|_2^2 = n \| f_d \|_{L^2}^2 \cdot O_{d,\P} (1)$. 

Let us now consider $\bpsi_{\leq \evn}^\sT \boldf_{> \evn} /n $. For any $\eta >0$, we have
\[
\begin{aligned}
\E \left[ \Vert \bpsi_{\leq \evn}^\sT  \boldf_{> \evn} \Vert_2^2 \right] / n^2 = & \E_{\bx} \Big[ \Big( \sum_{u \ge \evn +1}    \hf_{u} \bpsi_{u}^\sT \Big)  \bpsi_{\leq \evn }  \bpsi_{\leq \evn }^\sT   \Big( \sum_{v \ge \evn +1} \hf_v \bpsi_{v}   \Big) \Big] / n^2 \\
= &\sum_{u , v \geq \evn +1} \sum_{s = 0}^\evn \sum_{i,j \in [n]}  \Big\{  \E \Big[ \psi_u (\bx_i ) \psi_s (\bx_i) \psi_s (\bx_j) \psi_v ( \bx_j) \Big] / n^2 \Big\}\hf_u \hf_v \\
= &\sum_{u , v \geq \evn +1} \sum_{s = 0}^\evn \sum_{i \in [n]}  \Big\{  \E \Big[ \psi_u (\bx_i ) \psi_s (\bx_i) \psi_s (\bx_i) \psi_v ( \bx_i) \Big] / n^2 \Big\}\hf_u \hf_v \\
= & \frac{1}{n}\sum_{s =0 }^\evn \E_{\bx} \Big[ \big( \proj_{>\evn} f_d ( \bx ) \big)^2 \psi_s (\bx)^2 \Big] \le \frac{1}{n}\sum_{s = 0}^\evn \| \proj_{> \evn} f_d \|_{L^{2 + \eta}}^2 \| \psi_s \|_{L^{ (4 + 2\eta)/\eta}}^2 \\
\le&~ \Tilde{C} (\eta) \frac{\evn}{n} \| \proj_{> \evn} f_d \|_{L^{2 + \eta}}^2,
\end{aligned}
\]
where the last inequality uses the hypercontractivity assumption of Assumption \ref{ass:FMPCP}.$(a)$:
\[
 \| \psi_s \|_{L^{ (4 + 2\eta)/\eta}}^2 = \E_{\bx} [ \psi_s ( \bx)^{2\cdot \frac{2 + \eta}{\eta}} ]^{\frac{2\eta}{4+2\eta} } \leq C ((2+ \eta )/\eta) \E_\bx [  \psi_s ( \bx)^2 ] = C ((2+ \eta )/\eta),
\]
and $ \Tilde{C} (\eta) = C ((2+ \eta )/\eta)$. By Markov's inequality (using $\evn \leq n^{1 - \delta_0}$ in Assumption \ref{ass:spectral_gap}.$(a)$ for some fixed $\delta_0 >0$), we get
\[
\Vert \bpsi_{\leq \evn}^\sT   \boldf_{>\evn} / n \Vert_2 = o_{d,\P}( 1 ) \cdot \| \proj_{> \evn} f_d \|_{L^{2+\eta}}.
\]

\subsubsection{Proof of Proposition \ref{prop:technical_facts}.$(a)$}

Throughout the proof, we will generically denote $\bDelta$ any matrix with $\| \bDelta \|_\op = o_{d,\P}(1)$.  In particular, $\bDelta$ can change from line to line. For convenience, we will use the notations introduced in Section \ref{sec:structure_Z}.

\noindent
{\bf Step 0. Bound $\| \tbZ \hbUi \|_\op = \kappa_{>\evn}^{-1/2} \cdot O_{d,\P} (   1 )$. } 

Recall the definition $\hbU = \tbZ^\sT \tbZ + \lambda \id_N$ and the singular value decomposition $\tbZ   = \bP \bLambda \bQ^\sT$. Hence, we can rewrite
\[
 \tbZ \hbUi  = \bP \frac{\bLambda}{\bLambda^2 + \lambda} \bQ^\sT,
\] 
where we denoted by a slight abuse of notation $\bLambda / (\bLambda^2 + \lambda) := \diag ( ( \Lambda_i / ( \Lambda_i^2 + \lambda ) )_{i \in [n]} )$. From Proposition \ref{prop:singular_values_bZ}.$(a)$, $\sigma_{\min} ( \bLambda) = \kappa_{>\evn}^{1/2} \cdot (1 + o_{d,\P}(1) )$. We deduce that
\[
\| \tbZ \hbUi  \|_\op = \kappa_{>\evn}^{-1/2} \cdot O_{d,\P} (   1).
\]

\noindent
{\bf Step 1. Bound $\| \tbpsi_{\leq \evn}^\sT \tbZ \hbUi \tbphi_{\leq \evn} \tbD_{\leq \evn} -  \id_{\evn}   \|_\op =  o_{d,\P}(1)$. }

First notice that $\tbphi_{\leq \evn} \tbD_{\leq \evn} = \tbZ_{\leq \evn}^\sT ( \tbpsi_{\leq \evn}^\sT )^\dagger = (\tbZ - \tbZ_{>\evn} )^\sT ( \tbpsi_{\leq \evn}^\sT )^\dagger $. Furthermore, by Eq.~\eqref{eq:tbpsi_tbphi_concentration}, we have $( \tbpsi_{\leq \evn}^\sT )^\dagger = \tbpsi_{\leq \evn} + \bDelta$. Hence,
\begin{equation}\label{eq:two_parts}
\tbpsi_{\leq \evn}^\sT \tbZ \hbUi \tbphi_{\leq \evn} \tbD_{\leq \evn} = \tbpsi_{\leq \evn}^\sT \tbZ \hbUi \tbZ^\sT  ( \tbpsi_{\leq \evn}^\sT )^\dagger - \tbpsi_{\leq \evn}^\sT \tbZ \hbUi \tbZ_{>\evn}^\sT ( \tbpsi_{\leq \evn}^\sT )^\dagger   .
\end{equation}
Let us decompose the first term along the large singular values $\bLambda_1$ and small singular values $\bLambda_2$: 
\[
\begin{aligned}
\tbpsi_{\leq \evn}^\sT \tbZ \hbUi\tbZ^\sT  ( \tbpsi_{\leq \evn}^\sT )^\dagger = &~ \tbpsi_{\leq \evn}^\sT\bP \frac{\bLambda^2}{\bLambda^2 +  \lambda}  \bP^\sT ( \tbpsi_{\leq \evn}^\sT )^\dagger \\
= &~ \tbpsi_{\leq \evn}^\sT\bP_1 \frac{\bLambda_1^2}{\bLambda_1^2 +  \lambda}  \bP_1^\sT ( \tbpsi_{\leq \evn}^\sT )^\dagger +\tbpsi_{\leq \evn}^\sT\bP_2 \frac{\bLambda_2^2}{\bLambda_2^2 +  \lambda}  \bP_2^\sT ( \tbpsi_{\leq \evn}^\sT )^\dagger. 
\end{aligned}
\]
From Eqs.~\eqref{eq:min_lambda_1} and \eqref{eq:singular_vectors_Z} in Proposition \ref{prop:singular_values_bZ} and the assumption in the theorem $\lambda = O_{d} (1) \cdot \kappa_{>\evn}$, we have
\[
\Big\| \frac{\bLambda_1^2}{\bLambda_1^2 +  \lambda} - \id_{\evn} \Big\|_{\op} = o_{d,\P} ( 1 ), \qquad \| \tbpsi_{\leq \evn}^\sT\bP_2  \|_{\op} = o_{d,\P} ( 1 ).
\]
Hence,
\[
\Big\| \tbpsi_{\leq \evn}^\sT\bP_2 \frac{\bLambda_2^2}{\bLambda_2^2 +  \lambda}  \bP_2^\sT ( \tbpsi_{\leq \evn}^\sT )^\dagger \Big\|_{\op} \leq \| \tbpsi_{\leq \evn}^\sT\bP_2  \|_{\op}  \| ( \tbpsi_{\leq \evn}^\sT )^\dagger \|_{\op} = o_{d,\P} ( 1),
\]
and 
\[
\tbpsi_{\leq \evn}^\sT\bP_1 \frac{\bLambda_1^2}{\bLambda_1^2 +  \lambda}  \bP_1^\sT ( \tbpsi_{\leq \evn}^\sT )^\dagger = \tbpsi_{\leq \evn}^\sT\bP_1 \bP_1^\sT ( \tbpsi_{\leq \evn}^\sT )^\dagger  + \bDelta =  \tbpsi_{\leq \evn}^\sT\bP \bP^\sT ( \tbpsi_{\leq \evn}^\sT )^\dagger  + \bDelta '   = \id_\evn + \bDelta ',
\]
where $\| \bDelta \|_{\op}, \| \bDelta ' \|_{\op} = o_{d,\P} ( 1 )$. We deduce
\begin{equation}\label{eq:first_l_part}
\begin{aligned}
\Big\| \tbpsi_{\leq \evn}^\sT \tbZ \hbUi \tbZ^\sT  ( \tbpsi_{\leq \evn}^\sT )^\dagger - \id_\evn \Big\|_{\op} = o_{d,\P} ( 1 ).
\end{aligned}
\end{equation}

Consider the second term in Eq.~\eqref{eq:two_parts}:
\[
\begin{aligned}
\tbpsi_{\leq \evn}^\sT \tbZ \hbUi \tbZ_{>\evn}^\sT ( \tbpsi_{\leq \evn}^\sT )^\dagger = &~   \tbpsi_{\leq \evn}^\sT \bP_1 \frac{\bLambda_1}{\bLambda_1^2 +  \lambda}  \bQ_1^\sT \tbZ_{>\evn}^\sT ( \tbpsi_{\leq \evn}^\sT )^\dagger +  \tbpsi_{\leq \evn}^\sT \bP_2 \frac{\bLambda_2}{\bLambda_2^2 +  \lambda}  \bQ_2^\sT \tbZ_{>\evn}^\sT ( \tbpsi_{\leq \evn}^\sT )^\dagger.
\end{aligned}
\]
Using Eq.~\eqref{eq:min_lambda_1} in Proposition \ref{prop:singular_values_bZ}, we have $\sigma_{\min}(\bLambda_1) = \kappa_{>\evn}^{1/2} \cdot \omega_{d,\P} ( 1 )$. Then, recalling that $ \|  \tbZ_{>\evn} \|_\op = \kappa_{>\evn}^{1/2} \cdot O_{d,\P} ( 1 )$, we have
\[
\begin{aligned}
\Big\|    \tbpsi_{\leq \evn}^\sT \bP_1 \frac{\bLambda_1}{\bLambda_1^2 +  \lambda}  \bQ_1^\sT \tbZ_{>\evn}^\sT ( \tbpsi_{\leq \evn}^\sT )^\dagger \Big\|_\op \leq & ~\|    \tbpsi_{\leq \evn}^\sT \|_\op \| \bLambda_1/ (\bLambda_1^2 +  \lambda) \|_\op \|  \tbZ_{>\evn} \|_\op \| \tbpsi_{\leq \evn} + \bDelta \|_{\op} \\
= &~ O_{d,\P}(1) \cdot o_{d,\P} (  \kappa_{>\evn}^{-1/2}  ) \cdot O_{d,\P}( \kappa_{>\evn}^{1/2} ) \cdot O_{d,\P}(1) = o_{d,\P}(1 ).
\end{aligned}
\]
By Eqs.~\eqref{eq:bound_lambda_2} and \eqref{eq:singular_vectors_Z} in Proposition \ref{prop:singular_values_bZ}, we get
\[
\begin{aligned}
\Big\|    \tbpsi_{\leq \evn}^\sT \bP_2 \frac{\bLambda_2}{\bLambda_2^2 +  \lambda}  \bQ_2^\sT \tbZ_{>\evn}^\sT ( \tbpsi_{\leq \evn}^\sT )^\dagger \Big\|_\op \leq &~ \|    \tbpsi_{\leq \evn}^\sT  \bP_2 \|_\op \| \bLambda_2/ (\bLambda_2^2 +  \lambda) \|_\op \|  \tbZ_{>\evn} \|_\op \| \tbpsi_{\leq \evn} + \bDelta \|_{\op} \\
= &~ o_{d,\P}(1 ) \cdot  O_{d,\P} ( \kappa_{>\evn} ^{-1/2}  ) \cdot O_{d,\P}( \kappa_{>\evn}^{1/2} ) \cdot O_{d,\P}(1) = o_{d,\P}(1 ).
\end{aligned}
\]
We deduce that
\begin{equation}\label{eq:second_l_part}
\| \tbpsi_{\leq \evn}^\sT \tbZ \hbUi \tbZ_{>\evn}^\sT ( \tbpsi_{\leq \evn}^\sT )^\dagger \|_\op = o_{d,\P}(1 ).
\end{equation}
Combining Eqs.~\eqref{eq:first_l_part} and \eqref{eq:second_l_part} into Eq.~\eqref{eq:two_parts} yields
\[
\tbpsi_{\leq \evn}^\sT \tbZ \hbUi \tbphi_{\leq \evn} \tbD_{\leq \evn}  = \id_\evn + \bDelta,
\]
where $\| \bDelta \|_{\op} = o_{d,\P} (1 )$.

\noindent
{\bf Step 2. Bound $\| \tbD_{\leq \evn}  \tbphi_{\leq \evn}^\sT \hbUi   \tbZ^\sT \boldf_{> \evn} / \sqrt{n} \|_2  = \| \proj_{>\evn} f_d \|_{L^{2+\eta} } \cdot o_{d,\P} (1)$. } 

Let us denote $\tbf_{>\evn} = \boldf_{> \evn} / \sqrt{n}$ for convenience. 
Let us use again that $\tbphi_{\leq \evn} \tbD_{\leq \evn} = (\tbZ - \tbZ_{>\evn} )^\sT ( \tbpsi_{\leq \evn}^\sT )^\dagger $:
\begin{equation}\label{eq:Detc}
\begin{aligned}
\tbD_{\leq \evn}  \tbphi_{\leq \evn}^\sT \hbUi   \tbZ^\sT \tbf_{> \evn}  = &~ ( \tbpsi_{\leq \evn} )^\dagger \tbZ  \hbUi   \tbZ^\sT \tbf_{> \evn}  - ( \tbpsi_{\leq \evn} )^\dagger \tbZ_{>\evn}  \hbUi   \tbZ^\sT \tbf_{> \evn} .
\end{aligned}
\end{equation}
First notice that because $\| \tbpsi_{\leq \evn}^\sT \tbpsi_{\leq \evn} - \id_\evn \|_{\op} = o_{d,\P}(1)$, we have $\| \tbpsi_{\leq \evn}^\sT \bP_2 \|_{\op} = o_{d,\P} (1 )$ in Proposition \ref{prop:singular_values_bZ}.$(b)$ that
implies $\| ( \tbpsi_{\leq \evn} )^\dagger \bP_2 \|_{\op} = o_{d,\P} ( 1 )$ (for example by looking at the singular value decomposition of $\tbpsi_{\leq \evn} $). Similarly $\| \tbpsi_{\leq \evn}^\sT \tbf_{>\evn} \|_2 =  \| \proj_{>\evn} f_d \|_{L^{2+\eta} }  \cdot o_{d,\P} ( 1 )$
(Proposition \ref{prop:technical_facts}.$(c)$) implies $\| ( \tbpsi_{\leq \evn} )^\dagger \tbf_{>\evn} \|_2 =  \| \proj_{>\evn} f_d \|_{L^{2+\eta} }  \cdot o_{d,\P} ( 1 )$.
Using the same argument as in the proof of Eq.~\eqref{eq:first_l_part}, we have
\begin{equation}\label{eq:2_step_bound_1}
\begin{aligned}
&~\| ( \tbpsi_{\leq \evn} )^\dagger \tbZ  \hbUi   \tbZ^\sT \tbf_{> \evn}  \|_2 \\
= &~ \Big\|  (\tbpsi_{\leq \evn} )^{\dagger}   \bP \frac{\bLambda^2}{\bLambda^2 +  \lambda} \bP^\sT  \tbf_{> \evn} \Big\|_2  \\
\le &~  \| ( \tbpsi_{\leq \evn} )^\dagger  \tbf_{>\evn} \|_2 +
o_{d,\P}(1 ) \cdot \| ( \tbpsi_{\leq \evn} )^\dagger\|_{\op}\|  \tbf_{>\evn} \|_2 + \| \proj_{>\evn} f_d \|_{L^2} \cdot O_{d,\P}(1) \cdot \| ( \tbpsi_{\leq \evn} )^\dagger \bP_2 \|_{\op} \\
=&~ \| \proj_{>\evn} f_d \|_{L^{2+\eta} } \cdot o_{d,\P} ( 1 ).
\end{aligned}
\end{equation}

The second term \eqref{eq:Detc}  can be decomposed as
\[
( \tbpsi_{\leq \evn} )^\dagger \tbZ_{>\evn}  \hbUi   \tbZ^\sT \tbf_{> \evn}  = ( \tbpsi_{\leq \evn} )^\dagger \tbZ_{>\evn} \bQ_1 \frac{\bLambda_1}{\bLambda_1^2 +  \lambda} \bP_1^\sT  \tbf_{> \evn} +( \tbpsi_{\leq \evn} )^\dagger \tbZ_{>\evn} \bQ_2 \frac{\bLambda_2}{\bLambda_2^2 +  \lambda} \bP_2^\sT  \tbf_{> \evn}.
\]
Using that $\sigma_{\min}(\bLambda_1) = \kappa_{>\evn}^{1/2}  \cdot \omega_{d,\P} (1 )$ and $\| \tbZ_{>\evn} \|_{\op} = \kappa_{>\evn}^{1/2}  \cdot O_{d,\P} (1)$ yields 
\begin{equation}\label{eq:2_step_bound_2}
\begin{aligned}
\Big\|   ( \tbpsi_{\leq \evn} )^\dagger \tbZ_{>\evn} \bQ_1 \frac{\bLambda_1}{\bLambda_1^2 +  \lambda} \bP_1^\sT  \tbf_{> \evn}  \Big\|_\op \leq &~ \|   ( \tbpsi_{\leq \evn} )^\dagger \tbZ_{>\evn} \bQ_1 \|_\op \| \bLambda_1/ (\bLambda_1^2 +  \lambda) \|_\op \|  \bP_1^\sT  \tbf_{> \evn}  \|_\op  \\
= &~ O_{d,\P}( \kappa_{>\evn}^{1/2}  ) \cdot  o_{d,\P} ( \kappa_{>\evn}^{-1/2}  ) \cdot  O_{d,\P}(\| \proj_{>\evn} f_d \|_{L^2} )\\
 = &~ \| \proj_{>\evn} f_d \|_{L^2} \cdot o_{d,\P} ( 1 ).
\end{aligned}
\end{equation}
For the second term, recall that $( \tbpsi_{\leq \evn} )^\dagger = \tbpsi_{\leq \evn}^\sT + \bDelta$ and introduce $\bP \bP^\sT = \bP_1 \bP_1^\sT + \bP_2 \bP_2^\sT = \id_n$:
\begin{equation}\label{eq:2_step_bound_3}
\begin{aligned}
& \Big\| ( \tbpsi_{\leq \evn} )^\dagger \tbZ_{>\evn} \bQ_2 \frac{\bLambda_2}{\bLambda_2^2 +  \lambda} \bP_2^\sT  \tbf_{> \evn} \Big\|_\op \\
= &~ \Big\| ( \tbpsi_{\leq \evn} )^\dagger [  \bP_1 \bP_1^\sT + \bP_2 \bP_2^\sT  ] \tbZ_{>\evn} \bQ_2 \frac{\bLambda_2}{\bLambda_2^2 +  \lambda} \bP_2^\sT  \tbf_{> \evn} \Big\|_\op \\
 \leq &~ \| ( \tbpsi_{\leq \evn} )^\dagger  \bP_1 \|_\op \| \bP_1^\sT  \tbZ_{>\evn} \bQ_2  \|_\op \| \bLambda_2/(\bLambda_2^2 +  \lambda) \|_\op \|  \bP_2^\sT  \tbf_{> \evn} \|_\op  \\
 &~ + \|( \tbpsi_{\leq \evn} )^\dagger  \bP_2 \|_\op \| \bP_2^\sT  \tbZ_{>\evn} \bQ_2  \|_\op \| \bLambda_2/(\bLambda_2^2 +  \lambda)\Big\|_\op \|  \bP_2^\sT  \tbf_{> \evn} \|_\op \\
 = &~  o_{d,\P} ( \kappa_{>\evn}^{1/2} ) \cdot O_{d,\P} ( \kappa_{>\evn}^{1/2}  ) \cdot \| \proj_{>\evn} f_d \|_{L^2} \\
 =&~\| \proj_{>\evn} f_d \|_{L^2} \cdot o_{d,\P} ( 1 ).
\end{aligned}
\end{equation}
where we used Eq.~\eqref{eq:cross_term_Z} in Proposition \ref{prop:singular_values_bZ},
and $\sigma_{\min}(\bLambda_2)=\kappa_{>\evn}^{-1/2} \cdot \Omega_{d,\P}(1)$
to obtain the second to last line. Combining Eqs.~\eqref{eq:2_step_bound_1}, \eqref{eq:2_step_bound_2} and \eqref{eq:2_step_bound_3} yields the result.

\noindent
{\bf Step 3. Bound $\sqrt{n} \| \bZ \hbUi \bphi_{\leq \evn} \bD_{\leq \evn} / N \|_\op = O_{d,\P}(1 )$. } 

First notice that $\| \tbZ_{>\evn} \tbphi_{\leq \evn} \|_{\op} = \kappa_{>\evn}^{1/2} \cdot o_{d,\P} (1 )$ implies $\| \tbZ_{>\evn} (\tbphi^\sT_{\leq \evn} )^\dagger\|_{\op} = \kappa_{>\evn}^{1/2} \cdot o_{d,\P} (1 )$, where we used that $\| \tbphi^\sT_{\leq \evn}  \tbphi_{\leq \evn}  - \id_{\evn} \|_\op = o_{d,\P}(1)$.

Using $\tbphi_{\leq \evn} \tbD_{\leq \evn} = ( \tbZ - \tbZ_{>\evn}) (\tbphi^\sT_{\leq \evn} )^\dagger$, we have
\[
\begin{aligned}
\| \tbZ \hbUi \tbphi_{\leq \evn} \tbD_{\leq \evn} \|_{\op} \leq &~\| \tbZ \hbUi\tbZ (\tbphi^\sT_{\leq \evn} )^\dagger\|_{\op} + \| \tbZ \hbUi\tbZ_{>\evn} (\tbphi^\sT_{\leq \evn} )^\dagger\|_{\op} \\
\leq &~ \| \bLambda^2 / ( \bLambda^2 +  \lambda ) \|_\op \| (\tbphi^\sT_{\leq \evn} )^\dagger\|_{\op} + \| \tbZ \hbUi \|_\op \| \tbZ_{>\evn} (\tbphi^\sT_{\leq \evn} )^\dagger\|_{\op} \\
= &~ O_{d,\P}(1) + O_{d,\P} ( \kappa_{>\evn}^{-1/2}  ) \cdot  o_{d,\P} (\kappa_{>\evn}^{1/2}  ) \\
=&~ O_{d,\P}(1),
\end{aligned}
\]
which concludes the proof of the claims in Proposition \ref{prop:technical_facts}.$(a)$.

\subsubsection{Proof of Proposition \ref{prop:technical_facts}.$(b)$}

Denote 
\[
\begin{aligned}
\bD_{\evn:\evN} =& \diag(\lambda_{d,\evn+1}, \lambda_{d,\evn+2} , \ldots , \lambda_{d,\evN} ) \in \R^{(\evN - \evn) \times (\evN-\evn)},  \\
\bphi_{\evn:\evN} =& (\bphi_{k } (\btheta_i ) )_{i \in [N],  k =\evn+1, \ldots , \evN} \in \R^{N \times (\evN - \evn)}.
\end{aligned}
\]
Applying Theorem \ref{prop:expression_U} to $\bU_{>\evn}$ (where the assumptions are satisfied by Assumptions \ref{ass:FMPCP}.$(a)$ and $(b)$ and Assumption \ref{ass:spectral_gap}.$(a)$), we get with Assumption \ref{ass:FMPCP}.$(d)$,
\[
\bU_{>\evn} = \bphi_{\evn:\evN} \bD_{\evn:\evN}^2 \bphi_{\evn:\evN}^\sT + \kappa_{>\evN} ( \id_N + \bDelta ),
\]
where $\| \bDelta \|_\op = o_{d,\P}(1)$ and $\kappa_{>\evN} = \Trace ( \Hop_{d,>\evN} )$.
By assumption, we have $N \geq n^{1 + \delta_0}$ for some fixed $\delta_0 >0$ and therefore
\begin{equation}\label{eq:b_technical_1}
\frac{n}{N }\| \kappa_{>\evN} ( \id_N + \bDelta ) \|_{\op}  =  \kappa_{>\evN} \cdot o_{d,\P} (1 ).
\end{equation}

By Proposition \ref{prop:YY_new} (assumptions satisfied by Assumptions \ref{ass:FMPCP}.$(a)$ and \ref{ass:spectral_gap}.$(a)$), we get
\[
\| \bphi_{\evn:\evN}^\sT \bphi_{\evn:\evN}/N - \id_{\evN-\evn} \|_\op = o_{d,\P}(1).
\]
Furthermore, by Assumption \ref{ass:spectral_gap}.$(a)$, we have $n^{1 + \delta_0} \cdot \| \Hop_{d,>\evn} \|_\op = O_{d}(1) \cdot \kappa_{>\evn}$ for a fixed $\delta_0 >0$. Therefore $n  \| \bD_{\evn:\evN}^2  \|_\op = \kappa_{>\evn} \cdot o_{d} (1)$. Hence,
\begin{equation}\label{eq:b_technical_2}
\frac{n}{N} \|  \bphi_{\evn:\evN} \bD_{\evn:\evN}^2 \bphi_{\evn:\evN}^\sT \|_\op \leq \| \bphi_{\evn:\evN} / \sqrt{N} \|_\op^2 \| n \bD_{\evn:\evN}^2 \|_\op  = \kappa_{>\evn} \cdot o_{d,\P}(1 ).
\end{equation}
Combining Eqs.~\eqref{eq:b_technical_1} and \eqref{eq:b_technical_2} yields
\[
\frac{n}{N} \| \bU_{>\evn} \|_\op = \kappa_{>\evn} \cdot  o_{d,\P}(1 ).
\]

\subsubsection{Proof of Proposition \ref{prop:technical_facts}.$(d)$}

Recall 
\[
\bV_{>\evn} = \sum_{k = \evn+1}^\infty \hf_{d,k} \lambda_{d,k} \bphi_k.
\]
Taking the expectation over $(\btheta_1 , \ldots , \btheta_N)$, we get
\[
\frac{n}{N}\E [\| \bV_{> \evn} \|_2^2 ]  = n \sum_{k \ge \evn+1} \lambda_{d,k}^2 \hf_{d,k}^2 \leq n \cdot \| \Hop_{d,>\evn} \|_\op \cdot \| \proj_{>\evn} f_d \|_{L^2}^2.
\]
From condition \eqref{eq:NumberOfSamples} in Assumption \ref{ass:spectral_gap}.$(a)$, we have $n^{1+\delta_0}  \| \Hop_{d,>\evn} \|_\op = O_d (1) \cdot \kappa_{>\evn}$, and we conclude with Markov's inequality that
\[
\sqrt{\frac{n}{N}} \| \bV_{> \evn} \|_2 = \sqrt{\kappa_{>\evn} } \|\proj_{> \evn} f_d\|_{L^2} \cdot o_{d,\P} (1 ).
\]

\subsubsection{Bounds in the underparametrized regime}
\label{sec:bounds_underparametrized}

In the underparametrized case, we further prove the following lemma.

\begin{lemma}\label{lem:under_bound}
Follow the assumptions of Theorem \ref{thm:RFK_generalization} in the underparametrized case as well as the notations in Section \ref{sec:proof_RFRR_over}. Then, we have
\begin{align}
\| \bZ_{>\evN}^\sT \boldf_{>\evN} / n \|_2 =&~ \kappa_{>\evN}^{1/2} \cdot \| \proj_{> \evN} f_d \|_{L^{2+\eta} } \cdot o_{d,\P}(1), \label{eq:under_bound_1}\\
 \| \hbUi \bZ^\sT \boldf / n \|_{\op} =&~ \kappa_{>\evN}^{- 1/2}  \cdot o_{d,\P}(1) \cdot( \| f_d \|_{L^2} +  \| \proj_{> \evN} f_d \|_{L^{2+\eta} }).\label{eq:under_bound_2}
\end{align}
\end{lemma}

\begin{proof}[Proof of Lemma \ref{lem:under_bound}]

\noindent
{\bf Step 1. Bound $\| \bZ_{>\evN}^\sT \boldf_{>\evN} / n \|_2 =\kappa_{>\evN}^{1/2} \cdot \| \proj_{> \evN} f_d \|_{L^{2 + \eta} } \cdot o_{d,\P}(1)$. } 

Recall the decomposition of $\bZ_{>\evN}$ in the eigenbasis of functions:
\[
\bZ_{>\evN} = \sum_{k = \evN +1}^\infty \lambda_{d,k} \bpsi_{k} \bphi_{k}^\sT.
\]
Consider the expected square norm (with respect to $\bTheta = (\btheta_j)_{j \in [N]}$)
\[
\begin{aligned}
\E_{\bTheta} \big[ \| \bZ_{>\evN}^\sT \boldf_{>\evN} \|_2^2 \big] = &~ \sum_{k, \ell = \evN+1}^\infty \lambda_{d,k} \lambda_{d,\ell} \E_{\bTheta} [ \boldf_{>\evN}^\sT \bpsi_{d,k} \bphi_{d,k}^\sT \bphi_{d,\ell} \bpsi_{d,\ell}\boldf_{>\evN} ] \\
=&~ N  \sum_{k = \evN+1}^\infty \lambda_{d,k}^2   (\boldf_{>\evN}^\sT \bpsi_{d,k})^2\\
\end{aligned}
\]
where we used that $\E_{\bTheta} [ \bphi_{d,k}^\sT \bphi_{d,\ell} ] = N \delta_{k,\ell}$ by orthonormality of $\{ \phi_{d,k} \}_{k \ge 1}$. Expanding with respect to the $\bx_i$'s, we get
\[
\begin{aligned}
\E_{\bTheta} \big[ \| \bZ_{>\evN}^\sT \boldf_{>\evN} \|_2^2 \big] =&~ N \sum_{i \in [n]} \big\{ H_{d,> \evN:\evn} (\bx_i , \bx_i) [\proj_{>\evN} f_d (\bx_i)]^2 +H_{d,>\evn}  (\bx_i , \bx_i) [\proj_{>\evN} f_d (\bx_i)]^2 \Big\} \\
&~+ N\sum_{i \neq j \in [n]}  \sum_{k = \evN+1}^\infty \lambda_{d,k}^2 \psi_{d,k} (\bx_i) \proj_{>\evN}f_d (\bx_i) \cdot \psi_{d,k} (\bx_j) \proj_{>\evN}f_d (\bx_j) ,
\end{aligned}
\]
where we recall
\[
\begin{aligned}
H_{d,\evN:u} (\bx_i , \bx_i) = &~ \sum_{k = \evN+1}^u \lambda_{d,k}^2\psi_{d,k} (\bx_i)^2,\\
H_{d,>u} (\bx_i , \bx_i) = &~ \sum_{k = u+1}^\infty \lambda_{d,k}^2\psi_{d,k} (\bx_i)^2.
\end{aligned}
\]

Consider the first term depending on $H_{d,\evN:\evn} $. Using the same computation as in the proof of Proposition \ref{prop:technical_facts}.$(c)$ and Lemma \ref{lem:hypercontractivity_basis_kernel} (with the hypercontractivity assumption up to $u \geq \evn$ of Assumption \ref{ass:FMPCP}.$(a)$), by H\"older's inequality we have for the $q$
\[
\begin{aligned}
\E \big[ H_{d,\evN:\evn} (\bx , \bx) [\proj_{>\evN} f_d (\bx)]^2 \big] \leq&~ \|  H_{d,\evN:\evn} \|_{L^{1+2/\eta}} \| \proj_{>\evN} f_d \|^2_{L^{2+ \eta}} \\
\leq&~ C ( 1 + 2/\eta )^2 \cdot \E_{\bx} [ H_{d,\evN:\evn} (\bx , \bx) ] \cdot \| \proj_{>\evN} f_d \|^2_{L^{2+ \eta}}.
\end{aligned}
\]
We deduce by Markov's inequality that the first term is bounded by 
\begin{equation}\label{eq:under_step1_1}
\sum_{i \in [n]}  H_{d,\evN:\evn} (\bx_i , \bx_i) [\proj_{>\evN} f_d (\bx_i)]^2 = O_{d,\P} (1) \cdot n \cdot \Trace ( \Hop_{d,\evN:\evn} ) \cdot \| \proj_{>\evN} f_d \|^2_{L^{2+ \eta}}.
\end{equation}
For the second term, recall that by Assumption \ref{ass:FMPCP}.$(d)$, we have
\[
\max_{\bx_i \in [n]} H_{d,>\evn} ( \bx_i , \bx_i ) = O_{d,\P}(1) \cdot \Trace ( \Hop_{d,>\evn} ).
\]
Hence
\[
\sum_{i \in [n]} H_{d,>\evn}  (\bx_i , \bx_i) [\proj_{>\evN} f_d (\bx_i)]^2 = O_{d,\P} (1) \cdot \Trace ( \Hop_{d,>\evn} ) \cdot \sum_{i \in [n]}  [\proj_{>\evN} f_d (\bx_i)]^2,
\]
and by Markov's inequality 
\begin{equation}\label{eq:under_step1_2}
\sum_{i \in [n]} H_{d,>\evn}  (\bx_i , \bx_i) [\proj_{>\evN} f_d (\bx_i)]^2 = O_{d,\P} (1) \cdot n \cdot \Trace ( \Hop_{d,>\evn} ) \cdot \| \proj_{>\evN} f_d  \|_{L^2}^2.
\end{equation}
Taking the expectation of the third term gives
\begin{equation}\label{eq:under_step1_3}
\begin{aligned}
n(n-1) \sum_{k = \evN+1}^\infty \lambda_{d,k}^2 \E \big[ \psi_{d,k} (\bx) [\proj_{>\evN} f_d (x)] \big]^2 
=&~ n(n-1)  \sum_{k = \evN+1}^\infty \lambda_{d,k}^2 \hf_{d,k}^2 \\
\leq&~ n(n-1) \| \Hop_{d,>u}  \|_{\op} \| \proj_{ > \evN} f_d \|_{L^2}^2 .
\end{aligned}
\end{equation}
Merging Eqs.~\eqref{eq:under_step1_1}, \eqref{eq:under_step1_2} and \eqref{eq:under_step1_3}, we get
\[
\begin{aligned}
\E_{\bTheta} \big[ \| \bZ_{>\evN}^\sT \boldf_{>\evN} / n \|_2^2 \big]  \leq&~ \frac{N}{n} \cdot O_{d,\P} (1)  \cdot \Trace (\Hop_{d,>\evN} ) \cdot \| \proj_{>\evN} f_d \|^2_{L^{2 + \eta}} + N\| \Hop_{d,>u}  \|_{\op} \cdot \| \proj_{ > \evN} f_d \|_{L^2}^2  \\
=&~ o_{d} (1 ) \cdot \Trace (\Hop_{d,>\evN} )\cdot \| \proj_{>\evN} f_d \|^2_{L^{2 + \eta}} ,
\end{aligned}
\]
where we used Assumption \ref{ass:spectral_gap}.$(b)$ ($N \cdot  \| \Hop_{d,>u}  \|_{\op} =
O_{d,\P} (N^{-\delta_0})  \Trace(\Hop_{d,>u})$ as well as  $n \geq N^{1 + \delta_0}$ for a fixed $\delta_0>0$). Using Markov's inequality proves Eq.~\eqref{eq:under_bound_1}.

\noindent
{\bf Step 2. Bound on $ \| \hbUi \bZ^\sT \boldf / n \|_{2}$. } 

By Proposition \ref{prop:singular_values_bZ}.$(a)$ in the underparametrized case, we have
\begin{equation}
\hbUi \bZ^\sT \boldf / n  = \bQ_1 \frac{\bLambda_1}{\bLambda_1^2 + \lambda} \bP_1^\sT \boldf/\sqrt{n} + \bQ_2 \frac{\bLambda_2}{\bLambda_2^2 + \lambda} \bP_2^\sT \boldf/\sqrt{n},\label{eq:UZf}
\end{equation}
where $\sigma_{\min} ( \bLambda_1 ) = \omega_{d,\P} (1) \cdot \kappa_{>\evN}^{1/2}$ and $\sigma_{\min} ( \bLambda_2 ) = \kappa_{>\evN}^{1/2} \cdot (1 + o_{d,\P}(1))$. In particular, this shows that
\begin{equation}\label{eq:under_pro_1}
\Big\| \bQ_1 \frac{\bLambda_1}{\bLambda_1^2 + \lambda} \bP_1^\sT \boldf/\sqrt{n} \Big\|_{2} \leq \sigma_{\min} ( \bLambda_1 )^{-1} \| \boldf/\sqrt{n} \|_2 \leq o_{d,\P}  (1 ) \cdot \kappa_{>\evN}^{-1/2} \cdot \| f_d \|_{L^2}.
\end{equation}
For the second term \eqref{eq:UZf}, decompose $\boldf = \boldf_{\leq \evN} + \boldf_{>\evN}$. Recall $\boldf_{\leq \evN} = \bpsi_{\leq \evN} \hat \boldf_{\leq \evN}$. By Proposition \ref{prop:singular_values_bZ}.$(b)$, we have $\| \bP_2^\sT \bpsi_{\leq \evn} / \sqrt{n} \|_{\op} = o_{d,\P} (1)$. Furthermore, using Eq.~\eqref{eq:under_bound_1},
namely $\| \bZ_{>\evN}^\sT \boldf_{>\evN} / n \|_2 =\kappa_{>\evN}^{1/2} \cdot \| \proj_{> \evN} f_d \|_{L^{2+\eta}} \cdot o_{d,\P}(1)$, we get $ \| \bP_2^\sT \boldf_{>\evN} / \sqrt{n} \|_{\op} = \| \proj_{> \evN} f_d \|_{L^{2+\eta}} \cdot o_{d,\P}(1)$. We deduce 
\begin{equation}\label{eq:under_pro_2}
\begin{aligned}
\Big\| \bQ_2 \frac{\bLambda_2}{\bLambda_2^2 + \lambda} \bP_2^\sT \boldf/\sqrt{n} \Big\|_{2} \leq&~ \sigma_{\min} ( \bLambda_2 )^{-1} ( \| \bP_2^\sT \boldf_{\leq \evN} / \sqrt{n} \|_2 + \| \bP_2^\sT \boldf_{> \evN} / \sqrt{n} \|_2) \\
=&~ O_{d,\P} ( \kappa_{>\evN}^{-1/2} )  \cdot o_{d,\P}(1) \cdot (\|  f_d \|_{L^2} + \| \proj_{> \evN} f_d \|_{L^{2+\eta}}) \\
=&~  \kappa_{>\evN}^{-1/2}(\|  f_d \|_{L^2} + \| \proj_{> \evN} f_d \|_{L^{2+\eta}}) \cdot o_{d,\P}(1).
\end{aligned}
\end{equation}
Combining Eqs.~\eqref{eq:under_pro_1} and \eqref{eq:under_pro_2} yields Eq.~\eqref{eq:under_bound_2}.
\end{proof}

\subsection{Concentration of the random features kernel matrix $Z^\sT Z$}
\label{sec:RFRR_ZZ_concentration}

We recall the following standard result on concentration of random matrices with independent rows:

\begin{lemma}[\cite{vershynin2010introduction} Theorem 5.45]\label{lem:vershynin_matrix_conc}
Let $\bA$ be a $p \times q$ matrix whose rows $\ba_i$ are independent random vectors in $\R^q$ with common second moment matrix $\bSigma = \E [ \ba_i \otimes \ba_i ]$. Let $\Gamma := \E [ \max_{i \in [p]} \| \ba_i \|_2^2 ]$. Then 
\[
\E \big[ \| \bA^\sT \bA / p - \bSigma \|_\op \big] \leq \max ( \| \bSigma \|_\op^{1/2}  \eta, \eta^2 ),
\]
where $\eta = C \sqrt{ \frac{\Gamma \log ( \min (p,q)) }{p} }$ and $C$ is an absolute constant.
\end{lemma}

We will also use the following corollary for asymmetric matrices:

\begin{corollary}\label{cor:asym_conc}
Let $\bA$ be a $n \times N$ matrix whose rows $\ba_i$ are independent random vectors in $\R^N$ with common second moment matrix $\bSigma_\ba = \E [ \ba_i \otimes \ba_i ]$. Let $\bB$ be a $n \times \evn$ matrix whose rows $\bb_i$ are independent random vectors in $\R^\evn$ with common second moment matrix $\bSigma_\bb = \E [ \bb_i \otimes \bb_i ]$. Let $\Gamma_\ba := \E [ \max_{i \in [n]} \| \ba_i \|_2^2 ]$ and $\Gamma_\bb := \E [ \max_{i \in [n]} \| \bb_i \|_2^2 ]$. Denote $\bSigma_{\ba\bb} = \E [ \ba_i \otimes \bb_i ]$. Then,
\begin{align}\label{eq:asym_conc}
\E \big[ \| \bA^\sT \bB / n - \bSigma_{\ba\bb} \|_\op \big] \leq \max \big( (\| \bSigma_\ba \|_\op + \| \bSigma_\bb \|_\op )^{1/2}  \eta, \eta^2 \big),
\end{align}
where $\eta = C \sqrt{ \frac{(\Gamma_\ba + \Gamma_\bb) \log ( \min (n,N,\evn)) }{n} }$ and $C$ is an absolute constant.
\end{corollary}

\begin{proof}[Proof of Corollary \ref{cor:asym_conc}]
Define $\bC = [\bA , \bB] \in \R^{n \times (N+\evn) }$ whose rows $\bc_i = [ \ba_i, \bb_i]$ are independent random vectors in $\R^{N+\evn}$ with common second matrix $\bSigma_\bc = \begin{bmatrix}\bSigma_\ba &\bSigma_{\ba\bb} \\ \bSigma_{\bb\ba} & \bSigma_\bb\end{bmatrix}$. By Lemma \ref{lem:vershynin_matrix_conc}, we have
\[
\E \big[ \| \bC^\sT \bC / n - \bSigma_\bc \|_\op \big] \leq \max ( \| \bSigma_\bc \|_\op^{1/2}  \eta, \eta^2 ),
\]
where $\eta = C \sqrt{ \frac{\Gamma \log ( \min (n,N+\evn)) }{n} }$ with
\[
\Gamma = \E [ \max_{i \in [n]} \| \bc_i \|_2^2 ] \leq  \E [ \max_{i \in [n]} \| \ba_i \|_2^2 ] +  \E [ \max_{i \in [n]} \| \bb_i \|_2^2 ] \leq \Gamma_\ba + \Gamma_\bb.
\]
Notice that $\| \bSigma_\bc \|_{\op} \leq C ( \| \bSigma_\ba \|_\op + \| \bSigma_\bb \|_\op )$, and
\[
\| \bA^\sT \bB / n - \bSigma_{\ba\bb} \|_\op \leq \| \bC^\sT \bC / n - \bSigma_\bc \|_\op.
\]
Combining these bounds yields Eq.~\eqref{eq:asym_conc}.
\end{proof}

Consider the feature matrix $\bZ = (\sigma_d ( \bx_i ; \btheta_j ) )_{i \in [n], j \in [N]}$. We recall the decomposition $\bZ = \bZ_{\leq \evn} + \bZ_{>\evn}$ into a low and high degree parts:
\[
\bZ_{\leq \evn} = \bpsi_{\leq \evn} \bD_{\leq \evn} \bphi_{\leq \evn}^\sT, \qquad \bZ_{>\evn} = \sum_{k \geq \evn+1}  \lambda_{d,k} \bpsi_k \bphi_k^\sT.
\]
We prove the following concentration result on $\bZ_{>\evn}$.

\begin{proposition}[Concentration $\bZ$ matrix]\label{prop:concentration_bZ}
Consider the overparametrized case $N(d) \geq n(d)^{1+ \delta_0}$ for some fixed $\delta_0>0$. Let $\{ \sigma_d \}_{d\ge 1}$ be a sequence of activation functions satisfying the feature map concentration (Assumption \ref{ass:FMPCP}) and the spectral gap (Assumption \ref{ass:spectral_gap}) at level $\{ (N(d), \evN (d) , n(d), \evn (d) ) \}_{d \ge 1}$. Then, we have
\begin{equation}
\frac{\bZ_{>\evn} \bZ_{>\evn}^\sT }{N} = \kappa_{>\evn } \cdot (\id_n + \bDelta_{\bZ} ),
\end{equation}
where $\kappa_{> \evn} = \Trace ( \Hop_{d , > \evn}) $ and $\| \bDelta_{\bZ} \|_{\op} = o_{d,\P}( 1 )$. Furthermore, 
\begin{align}
\Big\| \frac{\bZ_{>\evn} \bphi_{\leq \evn} }{N} \Big\|_\op = \kappa_{>\evn}^{1/2} \cdot o_{d,\P}(1 ).
\end{align}
\end{proposition}

\begin{proof}[Proof of Proposition \ref{prop:concentration_bZ}]
For convenience, we will drop the subscript $d$.

\noindent
{\bf Step 1. Bound on  $\| \bZ_{>\evn} \bZ_{>\evn}^\sT /N - \kappa_{>\evn} \id_n\|_\op$. } 

Denote $\bA^\sT = \bZ_{>\evn} = [ \ba_1 , \ldots , \ba_N ] \in \R^{n\times N}$ with $\ba_i = ( \sigma_{>\evn} ( \bx_1 ; \btheta_i ) , \ldots , \sigma_{>\evn} (\bx_n ; \btheta_i ) ) \in \R^n$. Conditioned on $(\bx_1 , \ldots , \bx_n)$, the rows $\ba_i$ are independent with common second moment matrix
\[
\E_{\bx} [ \ba_i \otimes \ba_i ] = \bH_{>\evn},
\]
where $\bH_{> \evn} = (H_{>\evn, ij})_{1\leq i,j \leq N}$ with $H_{>\evn ,ij} = \E_\btheta [ \sigma_{>\evn} ( \bx_i ; \btheta ) \sigma_{>\evn} (\bx_j ; \btheta ) ]$. By applying Theorem \ref{prop:expression_U} to the kernel matrix $\bH_{>\evn}$ (assumptions satisfied by Assumptions \ref{ass:FMPCP} and \ref{ass:spectral_gap}), we have $\bH_{>\evn } = \kappa_{>\evn} \cdot ( \id_n + \bDelta_H )$ where $\| \bDelta_H \|_{\op} = o_{d,\P}(1)$. Therefore it is sufficient to show that 
\[
\Big\| \frac{\bZ_{>\evn} \bZ_{>\evn}^\sT }{N} - \bH_{> \evn} \Big\|_{\op} = o_{d,\P}( 1 ). 
\]
Let us decompose $\sigma_{>\evn}$ into a low and high degree parts $\sigma_{>\evn} = \sigma_{\evn:u} + \sigma_{>u}$ (recall that $u(d) > \evn(d)$):
\[
\begin{aligned}
\sigma_{\evn:u} ( \bx ; \btheta) = & \sum_{k = \evn+1}^{u} \lambda_{d,k} \psi_k ( \bx) \phi_k (\btheta), \\
\sigma_{>u} ( \bx ; \btheta) = & \sum_{k = u+1}^{\infty} \lambda_{d,k} \psi_k ( \bx) \phi_k (\btheta).
\end{aligned}
\]
Let $\uba_i =  ( \sigma_{\evn:u} ( \bx_1 ; \btheta_i ) , \ldots , \sigma_{\evn:u} (\bx_n ; \btheta_i ) ) \in \R^n$ and $\oba_i =  ( \sigma_{>u} ( \bx_1 ; \btheta_i ) , \ldots , \sigma_{>u} (\bx_n ; \btheta_i ) ) \in \R^n$,
$\ba_i=\uba_i+\oba_i$. Then
\[
\Gamma = \E_{\btheta} [ \max_{i \in [N]} \| \ba_i \|_2^2 ] \leq 2 \E_\btheta [ \max_{i \in [N]} \| \oba_i \|_2^2 ] + 2 \E_\btheta [ \max_{i \in [N]} \| \uba_i \|_2^2 ].
\]
Let $q>0$ be an integer as in Assumption \ref{ass:FMPCP}.$(c)$. We have
\[
\begin{aligned}
 \E_{\btheta} \Big[ \max_{i \in [N]} \| \oba_i \|_2^2 \Big] \leq  \E_\btheta \Big[ \max_{i \in [N]} \| \oba_i \|_2^{q} \Big]^{1/q} \leq & N^{1/q} \E_\btheta [ \| \uba_i \|_2^{2q} ]^{1/q}.
 \end{aligned}
\]
By Jensen's inequality and Assumption \ref{ass:FMPCP}.$(c)$, there exists a fixed $\delta_0 >0$ such that
\[
\begin{aligned}
 \E_{\bx, \btheta} \Big[ \| \uba_i \|_2^{2q} \Big] = &~  \E_{\bx, \btheta} \Big[ \Big( \sum_{j \in [n]} \sigma_{>u} ( \bx_j ; \btheta )^2  \Big)^q \Big] \\
 \leq &~ n^{q-1}  \E_{\bx, \btheta} \Big[  \sum_{j \in [n]} \sigma_{>u} ( \bx_j ; \btheta )^{2q} \Big]\\
 \leq &~ n^q \E_{\bx, \btheta} [ \sigma_{>u} ( \bx_j ; \btheta )^{2q} ] = O_{d} ( 1 ) \cdot n^{q(1 +2\delta_0)} \cdot \kappa_{>u}^{q},
 \end{aligned}
\]
where $\kappa_{>u} = \Trace ( \Hop_{>u} ) = \sum_{k = u+1}^\infty \lambda_k^2$.
Hence, by Markov's inequality, we get
\begin{equation}\label{eq:bound_Gam_1}
 \E_{\btheta} [ \max_{i \in [N]} \| \oba_i \|_2^2 ]  = O_{d,\P}(1) \cdot N^{1/q} n^{1 + 2\delta_0} \cdot \kappa_{>u}.
\end{equation}
Similarly, by the hypercontractivity assumption (Assumption \ref{ass:FMPCP}.$(a)$), we have
\[
 \E_{\bx} \Big[ \E_{\btheta} \Big[ \max_{i \in [N]} \| \uba_i \|_2^2 \Big] \Big] \leq C_{q} N^{1/q} \E_{\bx, \btheta} \big[ \| \uba_i \|_2^2 \big] =  C_{q} N^{1/q} n  \cdot \kappa_{\evn:u},
\]
where $\kappa_{\evn:u} = \sum_{k = \evn+1}^u \lambda_k^2$.
Hence, by Markov's inequality, we get 
\begin{equation}\label{eq:bound_Gam_2}
 \E_{\btheta} \Big[ \max_{i \in [N]} \| \uba_i \|_2^2 \Big] = O_{d,\P}(1) \cdot N^{1/q} n  \cdot \kappa_{\evn:u} .
\end{equation}
Combining Eqs.~\eqref{eq:bound_Gam_1} and \eqref{eq:bound_Gam_2}, we get
\[
\Gamma_\ba =  O_{d,\P}(1) \cdot  N^{1/q} n^{1 + 2\delta_0}  \kappa_{>\evn}.
\]
We can therefore apply Lemma \ref{lem:vershynin_matrix_conc}. Recalling $\| \bH_{>\evn} \|_\op = O_{d,\P} (1)  \cdot \kappa_{>\evn}$, we have
\[
\E_{\btheta} \Big[ \big\| \bZ_{>\evn} \bZ_{>\evn}^\sT /N - \bH_{>\evn} \big\|_\op \Big] \leq O_{d,\P} (1) \cdot \max ( \kappa_{>\evn}^{1/2} \eta, \eta^2),
\]
with $\eta = \big( \kappa_{>\evn} N^{1/q - 1} n^{1 + 2 \delta_0} \log (N) \big)^{1/2} = \kappa_{>\evn}^{1/2} \cdot o_{d,\P}(  1 )$ by the choice of $q$ in Assumption \ref{ass:FMPCP}.$(c)$. We conclude 
\[
\big\| \bZ_{>\evn} \bZ_{>\evn}^\sT /N - \bH_{>\evn} \big\|_\op = \kappa_{>\evn} \cdot o_{d,\P}(  1 ).
\]

\noindent
{\bf Step 2. Bound on  $\| \bZ_{>\evn} \bphi_{ \leq \evn} /N \|_\op$. } 

Consider $\bB = \kappa_{>\evn}^{1/2} \bphi_{\leq \evn} = [\bb_1 , \ldots , \bb_N ]^\sT \R^{N \times \evn}$ where $\bb_i = \kappa_{>\evn}^{1/2} [ \bphi_1 ( \btheta_i) , \ldots , \bphi_\evn ( \btheta_i ) ] \in \R^\evn$ are independent rows with second moment matrix $\bSigma_\bb = \E [ \bb_i \otimes \bb_i ] = \kappa_{>\evn} \id_\evn$. Furthermore, by the hypercontractivity assumption (Assumption \ref{ass:FMPCP}.$(a)$), we have
\[
\Gamma_\bb = \E_{\btheta} \Big[ \max_{i \in [N]} \| \bb_i \|_2^2 \Big] \leq C_q N^{1/q} \E_{\btheta} \Big[ \| \bb_i \|_2^2 \Big] = C_q N^{1/q} \evn \cdot \kappa_{>\evn} .
\]
Notice that $\E [ \ba_i \otimes \bb_i ] = 0$. Furthermore, recalling the previous step, we have $\| \bSigma_\ba \|_\op = \| \bH_{>\evn} \|_{\op} = O_{d,\P} ( 1 )  \cdot  \kappa_{>\evn}$ and
\[
\begin{aligned}
\| \bSigma_\ba \|_\op + \| \bSigma_\bb \|_\op = & O_{d,\P} ( 1 )  \cdot  \kappa_{>\evn},\\
\Gamma_\ba + \Gamma_\bb  = & O_{d,\P}(1) \cdot  N^{1/q}  n^{1 + 2 \delta_0} \cdot \kappa_{>\evn}.
\end{aligned}
\]
Then by Corollary \ref{cor:asym_conc} applied to $\bA^\sT \bB / N$ and recalling the assumption on $q$ in Assumption \ref{ass:FMPCP}.$(c)$, we have
\[
\E_\btheta [ \| \bZ_\evn \bphi_{\leq \evn} / N  \|_\op ] = o_{d,\P} ( 1 ) \cdot \kappa_{>\evn}^{1/2},
\]
which concludes the proof by Markov's inequality.

\end{proof}

\section{Generalization error of kernel ridge regression: Proof of Theorem \ref{thm:upper_bound_KRR}}\label{sec:proof_KR}

In this section, we prove Theorem \ref{thm:upper_bound_KRR}. We will then prove a different version of the same theorem
in Section \ref{sec:KRR_weaker_gap}, under
somewhat different assumptions. Namely, we will relax Assumption \ref{ass:KRR}.$(c)$ and instead impose a gap condition on the eigenvalues
of the kernel.


\subsection{Proof of Theorem \ref{thm:upper_bound_KRR}}
\label{sec:proof_KRR_app}

In this section, we prove Theorem \ref{thm:upper_bound_KRR_app}. 
Throughout the proof, we will denote $\bDelta$ any matrix with $\| \bDelta \|_\op = o_{d,\P}(1)$. In particular, $\bDelta$ can change from one line to line. We defer the proofs of some more technical results to Section \ref{sec:KRR_upper_auxiliary}.

\noindent
{\bf Step 1. Expressing the risk in terms of empirical kernel matrix. }

Recall that the KRR estimator is given by
\[
\hf_{\lambda} ( \bx) = \by^\sT    ( \bH + \lambda \id_N )^{-1} \bh (\bx),
\]
where $\by = (y_1 , \ldots , y_n )$ and $\bH = (H ( \bx_i , \bx_j)
)_{i, j \in[n]}$, $\bh (\bx) = ( H_d ( \bx , \bx_1) , \ldots , H_d ( \bx , \bx_n ) ) \in \R^n$.
The resulting test error  is 
\[
\begin{aligned}
R_{\KR}(f_d, \bX, \lambda) \equiv&~ \E_\bx\Big[ \Big( f_d(\bx) - \by^\sT (\bH + \lambda \id_n)^{-1} \bh(\bx) \Big)^2 \Big]\\
=&~ \E_\bx [ f_d(\bx)^2] - 2 \by^\sT (\bH + \lambda \id_n)^{-1} \bE + \by^\sT (\bH + \lambda \id_n)^{-1} \bM (\bH + \lambda \id_n)^{-1} \by,
\end{aligned}
\]
where $\bE = (E_1, \ldots, E_n)^\sT$, $\bM = (M_{ij})_{ij \in [n]}$ and $\bH = (H_{ij})_{ij \in [n]}$ are defined by
\[
\begin{aligned}
E_i =&~ \E_\bx[f_d(\bx) H_d(\bx, \bx_i)], \\
M_{ij} =&~ \E_{\bx}[H_d(\bx_i, \bx) H_d(\bx_j, \bx)], \\
H_{ij} = &~ H_d(\bx_i, \bx_j).
\end{aligned}
\]

We recall that the eigendecomposition of $H_d$ is given by
\[
H_d(\bx, \by) = \sum_{k=1}^\infty \lambda_{d, k}^2 \psi_k(\bx) \psi_k(\by).
\]
We write the orthogonal decomposition of $f_{d}$ in the basis $\{ \psi_k \}_{k\ge 1}$ as
\[
f_{d}(\bx) = \sum_{k = 1}^\infty \hat f_{d, k} \psi_k(\bx). 
\]
Define
\[
\begin{aligned}
\bpsi_k =&~ (\psi_k(\bx_1), \ldots, \psi_k(\bx_n))^\sT \in \R^n, \\
\bD_{\le \evn} =&~ \diag(\lambda_{d, 1}, \lambda_{d, 2}, \ldots, \lambda_{d, \evn}) \in \R^{\evn \times \evn},\\
\bPsi_{\le \evn} =&~ (\psi_{k}(\bx_i))_{i \in [n],  k \in [\evn]} \in \R^{n \times \evn}, \\
\hat \boldf_{\le \evn} =&~ (\hat f_{d, 1}, \hat f_{d, 2}, \ldots, \hat f_{d, \evn})^\sT \in \R^{\evn}. 
\end{aligned}
\]

We decompose the vectors and matrices $\boldf$, $\bE$, $\bH$, and $\bM$ in terms of orthogonal basis
\begin{equation}\label{eq:DecompositionFEHM}
\begin{aligned}
  &\boldf =\boldf_{\le \evn}+\boldf_{>m},\qquad
  &\boldf_{\le \evn}= \bPsi_{\le \evn} \hat \boldf_{\le \evn},\qquad
  &\boldf_{>\evn} = \sum_{k = \evn + 1}^\infty \hat f_{d, k} \bpsi_k , \\
  &\bE =\bE_{\le \evn}+\bE_{>m},\qquad
  &\bE_{\le \evn}=~ \bPsi_{\le \evn} \bD_{\le \evn}^2 \hat \boldf_{\le \evn} ,\qquad
  & \bE_{>\evn}=\sum_{k = \evn + 1}^\infty \lambda_{d, k}^2 \hat f_{d, k} \bpsi_k ,\\
  &\bH =\bH_{\le \evn}+\bH_{>m},\qquad
  &\bH_{\le\evn} =~ \bPsi_{\le \evn} \bD_{\le \evn}^2 \bPsi_{\le\evn}^\sT,\qquad
  &\bH_{>\evn}\sum_{k=\evn + 1}^\infty  \lambda_{d, k}^2 \bpsi_k \bpsi_k^\sT, \\
   &\bM =\bM_{\le \evn}+\bM_{>m},\qquad
   &\bM_{\le \evn} =~ \bPsi_{\le \evn} \bD_{\le \evn}^4 \bPsi_{\le \evn}^\sT,\qquad
   &\bM_{>\evn} =\sum_{k=\evn + 1}^\infty  \lambda_{d, k}^4 \bpsi_k \bpsi_k^\sT. 
\end{aligned}
\end{equation}

Applying Theorem \ref{prop:expression_U} with respect to the operator $\Hop_d$ and $\Hop_d^2$  where the assumptions are satisfied by Assumptions \ref{ass:KRR}.$(a)$, \ref{ass:KRR}.$(b)$, cf. Eqs.~\eqref{eqn:hypercontractivity_in_KRR} and \eqref{eq:ass_kernel_b2}, and \ref{ass:eig_decay}.$(a)$, cf. Eq.~\eqref{ass:n_parameters_KRR_lower_2}, and using Assumption \ref{ass:KRR}.$(c)$, the kernel matrices $\bH$ and $\bM$ can be rewritten as
\begin{equation}\label{eq:HMdecompositionKRR}
\begin{aligned}
\bH =&~ \bPsi_{\le \evn} \bD_{\le \evn}^2 \bPsi_{\le \evn}^\sT + \kappa_H (\id+ \bDelta_H), \\
\bM =&~ \bPsi_{\le \evn} \bD_{\le \evn}^4 \bPsi_{\le \evn}^\sT + \kappa_M (\id + \bDelta_M),
\end{aligned}
\end{equation}
where 
\[
\begin{aligned}
\kappa_H =&~ \Trace(\Hop_{d, > \evn})= \sum_{k \ge \evn + 1} \lambda_{d, k}^2,\\
\kappa_M =&~ \Trace(\Hop_{d, > \evn}^2) =\sum_{k \ge \evn + 1} \lambda_{d, k}^4,
\end{aligned}
\]
and
\begin{equation}\label{eq:DeltaSmall}
\max\{ \| \bDelta_M \|_{\op}, \| \bDelta_H \|_{\op} \} = o_{d, \P}(1).
\end{equation}

Let us introduce the shrinkage matrix
\begin{align}\label{eq:shrinkage_matrix}
\bS_{\leq \evn} = \Big( \id_\evn + \frac{ \kappa_H + \lambda }{ n}  \bD_{\leq \evn}^{-2}  \Big)^{-1} = \diag ( ( s_j )_{j \in [\evn]} ) \qquad \text{where } s_j = \frac{\lambda_{d,j}^2 }{\lambda_{d,j}^2 +  \frac{ \kappa_H + \lambda }{ n}  } \, .
\end{align}

\noindent
{\bf Step 2. Decompose the risk}

Recalling $\by = \boldf + \beps$, we decompose the risk as follows
\[
\begin{aligned}
R_{\KR}(f_d, \bX, \lambda) =&~ \| f_d \|_{L^2}^2  - 2 T_1 + T_2 +  T_3 - 2 T_4 + 2 T_5. 
\end{aligned}
\]
where
\[
\begin{aligned}
T_1 =&~ \boldf^\sT (\bH + \lambda \id_n)^{-1} \bE, \\
T_2 =&~  \boldf^\sT (\bH + \lambda \id_n)^{-1} \bM (\bH + \lambda \id_n)^{-1} \boldf, \\
T_3 =&~ \beps^\sT (\bH + \lambda \id_n)^{-1} \bM (\bH + \lambda \id_n)^{-1} \beps,\\ 
T_4 =&~ \beps^\sT (\bH + \lambda \id_n)^{-1} \bE,\\
T_5 =&~ \beps^\sT (\bH + \lambda \id_n)^{-1} \bM (\bH + \lambda \id_n)^{-1} \boldf. 
\end{aligned}
\]

\noindent
{\bf Step 3. Term $T_2$}

Note we have 
\[
T_2 = T_{21}  + T_{22} + T_{23}, 
\]
where
\begin{equation}\label{eq:T2i_def}
\begin{aligned}
T_{21} =&~\boldf_{\le\evn}^\sT (\bH + \lambda \id_n)^{-1} \bM (\bH + \lambda \id_n)^{-1} \boldf_{\le \evn},\\
T_{22} =&~ 2 \boldf_{\le\evn}^\sT (\bH + \lambda \id_n)^{-1} \bM (\bH + \lambda \id_n)^{-1} \boldf_{> \evn},\\
T_{23} =&~\boldf_{> \evn}^\sT (\bH + \lambda \id_n)^{-1} \bM (\bH + \lambda \id_n)^{-1} \boldf_{> \evn}.\\
\end{aligned}
\end{equation}
By Lemma \ref{lem:key_H_U_H_bound} which is stated  in Section \ref{sec:KRR_upper_auxiliary} below, we have 
\begin{equation}\label{eqn:H_U_H_bound}
\| n (\bH + \lambda \id_n)^{-1} \bM (\bH + \lambda \id_n)^{-1} - \bPsi_{\le \evn} \bS_{\leq \evn}^2 \bPsi_{\le \evn}^\sT / n\|_{\op} = o_{d, \P}(1),
\end{equation}
hence
\[
\begin{aligned}
T_{21} =&~ \hat \boldf_{\le \evn}^\sT \bPsi_{\le \evn}^\sT (\bH + \lambda \id_n)^{-1} \bM (\bH + \lambda \id_n)^{-1} \bPsi_{\le \evn} \hat \boldf_{\le \evn} \\
=&~ \hat \boldf_{\le \evn}^\sT\bPsi_{\le \evn}^\sT \bPsi_{\le \evn}^\sT \bS_{\leq\evn} ^2 \bPsi_{\le \evn}  \bPsi_{\le \evn} \hat \boldf_{\le \evn} / n^2 + [\| \bPsi_{\le \evn} \hat \boldf_{\le \evn} \|_2^2  / n] \cdot o_{d, \P}(1). 
\end{aligned}
\]
By Assumption \ref{ass:KRR}.$(a)$, the conditions of Theorem \ref{prop:expression_U}.$(b)$ are satisfied, and we have (with $\| \bDelta \|_{\op} = o_{d, \P}(1)$)
\[
\begin{aligned}
\hat \boldf_{\le \evn}^\sT \bPsi_{\le \evn}^\sT \bPsi_{\le \evn} \bS_{\leq\evn} ^2 \bPsi_{\le \evn}^\sT\bPsi_{\le \evn} \hat \boldf_{\le \evn} / n^2 =&~ \hat \boldf_{\le \evn}^\sT(\id + \bDelta) \bS_{\leq\evn} ^2  (\id + \bDelta) \hat \boldf_{\le \evn}  = \| \bS_{\leq\evn} \hat \boldf_{\le \evn} \|_2^2 + o_{d, \P}(1) \cdot  \| \hat \boldf_{\le \evn} \|_2^2. 
\end{aligned}
\]
Moreover, we have 
\[
\| \bPsi_{\le \evn} \hat \boldf_{\le \evn} \|_2^2 / n = \hat \boldf_{\le \evn}^\sT(\id + \bDelta) \hat \boldf_{\le \evn}  = \| \hat \boldf_{\le \evn} \|_2^2 (1 + o_{d, \P}(1)). 
\]
As a result, we have
\begin{align}\label{eqn:term_R21}
T_{21} =&~ \| \bS_{\leq\evn} \hat \boldf_{\le \evn} \|_2^2 + o_{d, \P}(1) \cdot  \| \hat \boldf_{\le \evn} \|_2^2  =  \| \bS_{\leq\evn} \hat \boldf_{\le \evn} \|_2^2 +  o_{d, \P}(1) \cdot \| \proj_{\le \evn} f_d \|_{L^2}^2. 
\end{align}

By Eq. (\ref{eqn:H_U_H_bound}) again, we have 
\[
\begin{aligned}
T_{23} =&~ \Big(\sum_{k\ge \evn+1} \hat f_k \bpsi_k^\sT \Big) (\bH + \lambda \id_n)^{-1} \bM (\bH + \lambda \id_n)^{-1}\Big(\sum_{k\ge \evn+1} \bpsi_k \hat f_k \Big)\\
=&~ \Big(\sum_{k\ge \evn+1} \hat f_k \bpsi_k^\sT \Big) \bPsi_{\leq \evn}\bS_{\leq\evn}^2 \bPsi_{\leq \evn}^\sT \Big(\sum_{k\ge \evn+1} \bpsi_k \hat f_k \Big) / n^2 +  \Big[\Big\|\sum_{k\ge \evn+1} \bpsi_k \hat f_k \Big\|_2^2 / n\Big] \cdot o_{d, \P}(1).
\end{aligned}
\]
Note that $\bS_{\leq\evn} \preceq \id_{\evn}$ and we have 
\[
\begin{aligned}
&~ \E\Big[\Big(\sum_{k\ge \evn+1} \hat f_k \bpsi_k^\sT \Big) \bPsi_{\le \evn} \bS_{\leq\evn}^2 \bPsi_{\le \evn}^\sT \Big(\sum_{k\ge \evn+1} \bpsi_k \hat f_k \Big) \Big] /n^2  \\
\leq &~ \E\Big[\Big(\sum_{k\ge \evn+1} \hat f_k \bpsi_k^\sT \Big) \bPsi_{\le \evn} \bPsi_{\le \evn}^\sT \Big(\sum_{k\ge \evn+1} \bpsi_k \hat f_k \Big) \Big] /n^2  \\
=&~ \sum_{u, v \ge \evn+1} \sum_{s = 1}^\evn \sum_{i, j \in [n]}  \Big\{ \E \Big[ \psi_u(\bx_i)  \psi_{s}(\bx_i) \psi_{s}(\bx_j)  \psi_v(\bx_j) \Big] /n^2 \Big\} \hat f_v \hat f_u \\
=& \sum_{u, v \ge \evn+1} \sum_{s = 1}^\evn \sum_{i \in [n]}  \Big\{ \E \Big[\psi_u(\bx_i)  \psi_{s}(\bx_i) \psi_{s}(\bx_i)  \psi_v(\bx_i) \Big] /n^2 \Big\} \hat f_v \hat f_u \\
=&~  \frac{1}{n}\sum_{s = 1}^\evn \E_\bx \Big[ \big(\proj_{> \evn}f_d(\bx) \big)^2  \psi_{s}(\bx)^2 \Big] \le \frac{1}{n}\sum_{s = 1}^\evn \| \proj_{> \evn} f_d \|_{L^{2 + \eta}}^2 \| \psi_s \|_{L^{ (4 + 2\eta)/\eta}}^2 \\
\le&~ C(\eta) \frac{\evn}{n} \| \proj_{> \evn} f_d \|_{L^{2 + \eta}}^2,
\end{aligned}
\]
where the last inequality used the hypercontractivity assumption as in Assumption \ref{ass:KRR}.$(a)$.  Moreover
\[
  \E \Big[\frac{1}{n}\Big\|\sum_{k\ge \evn+1} \bpsi_k \hat f_k \Big\|_2^2
  \Big] = \sum_{k = \evn+ 1}^\infty \hat f_k^2 = \| \proj_{> \evn} f_d \|_{L^2}^2. 
\]
Using the last two displays, and the fact that $\evn(d)\le n(d)^{1-\delta}$ by Assumption \ref{ass:eig_decay}.$(b)$,
\begin{align}\label{eqn:term_R23}
T_{23} = o_{d, \P}(1) \cdot \| \proj_{> \evn} f_d \|_{L^{2 + \eta}}^2. 
\end{align}
Using Cauchy-Schwarz inequality for $T_{22}$, we get
\begin{align}\label{eqn:term_R22}
T_{22} \le 2 (T_{21} T_{23})^{1/2} = o_{d, \P}( 1) \cdot \| \proj_{\le \evn} f_d \|_{L^2}  \| \proj_{> \evn} f_d \|_{L^{2 + \eta}} .  
\end{align}
As a result, combining Eqs. (\ref{eqn:term_R21}), (\ref{eqn:term_R23}) and (\ref{eqn:term_R22}), we have 
\begin{equation}\label{eqn:KRR_term_T2}
T_2 = \| \bS_{\leq\evn} \hat \boldf_{\le \evn} \|_2^2 + o_{d, \P}(1) \cdot (\| f_d \|_{L^2}^2 + \| \proj_{>\evN} f_d \|_{L^{2 + \eta} }^2). 
\end{equation}

\noindent
{\bf Step 4. Term $T_{1}$. }

We have 
\[
T_1 = T_{11} + T_{12} + T_{13}, 
\]
where
\[
\begin{aligned}
T_{11} =&~ \boldf_{\le \evn}^\sT (\bH + \lambda \id_n)^{-1} \bE_{\le \evn}, \\
T_{12} =&~ \boldf_{> \evn}^\sT (\bH + \lambda \id_n)^{-1} \bE_{\le \evn}, \\
T_{13} =&~ \boldf^\sT (\bH + \lambda \id_n)^{-1} \bE_{> \evn}. \\
\end{aligned}
\]
By Lemma \ref{lem:lem_for_error_bound_R11} stated in Section \ref{sec:KRR_upper_auxiliary} below, we have 
\[
\| \bPsi_{\le \evn}^\sT( \bH + \lambda \id_n)^{-1}\bPsi_{\le \evn} \bD_{\le \evn}^2 - \bS_{\leq \evn} \|_{\op} = o_{d, \P}(1), 
\]
so that 
\begin{align}
  T_{11} =&~ \hat \boldf_{\le \evn}^\sT  \bPsi_{\le \evn}^\sT( \bH + \lambda \id_n)^{-1}\bPsi_{\le \evn} \bD_{\le \evn}^2 \hat \boldf_{\le \evn} \nonumber\\
  =&~ \| \bS_{\leq \evn}^{1/2} \hat \boldf_{\le \evn} \|_2^2  + o_{d, \P}(1) \cdot \|\hat \boldf_{\le \evn} \|_2^2 =  \| \bS_{\leq \evn}^{1/2} \hat \boldf_{\le \evn} \|_2^2  + o_{d, \P}(1) \cdot  \| \proj_{\le \evn} f_d \|_2^2 . \label{eqn:term_R11}
\end{align}

Using Cauchy-Schwarz inequality for $T_{12}$, and by the expression $\bM = \bPsi_{\le \evn} \bD_{\le \evn}^4 \bPsi_{\le \evn}^\sT +
\kappa_M(\id_M + \bDelta_M)$, cf. Eq.~\eqref{eq:HMdecompositionKRR}, we get with high probability
\begin{equation}\label{eqn:term_R12}
\begin{aligned}
\vert T_{12} \vert =&~ \Big\vert \sum_{k=\evn + 1}^\infty \hat f_k \bpsi_k^\sT (\bH + \lambda \id_n)^{-1} \bPsi_{\le \evn} \bD_{\le \evn}^2  \hat \boldf_{\le \evn} \Big\vert \\
\stackrel{(a)}{\le}&~ \Big\| \sum_{k=\evn + 1}^\infty \hat f_k \bpsi_k^\sT (\bH + \lambda \id_n)^{-1} \bPsi_{\le \evn} \bD_{\le \evn}^2 \Big\|_2 \| \hat \boldf_{\le \evn} \|_2 \\
\stackrel{(b)}{=}&~ \Big[ \Big( \sum_{k=\evn + 1}^\infty \hat f_k \bpsi_k^\sT \Big) (\bH + \lambda \id_n)^{-1} \bPsi_{\le \evn} \bD_{\le \evn}^4 \bPsi_{\le \evn}^\sT (\bH + \lambda \id_n)^{-1} \Big( \sum_{k=\evn + 1}^\infty \hat f_k \bpsi_k \Big)\Big]^{1/2} \| \hat \boldf_{\le \evn} \|_2\\
\stackrel{(c)}{\le} &~ \Big[ \Big( \sum_{k=\evn + 1}^\infty \hat f_k \bpsi_k^\sT \Big) (\bH + \lambda \id_n)^{-1} \bM (\bH + \lambda \id_n)^{-1} \Big( \sum_{k=\evn + 1}^\infty \hat f_k \bpsi_k \Big) \Big]^{1/2} \| \hat \boldf_{\le \evn} \|_2\\
\stackrel{(d)}{=}&~  T_{23}^{1/2} \| \hat \boldf_{\le \evn} \|_2 \stackrel{(e)}{=} o_{d, \P}(1) \cdot \| \proj_{\le \evn} f_d \|_{L^2} \| \proj_{> \evn} f_d \|_{L^{2 + \eta}}.
\end{aligned}
\end{equation}
Here $(a)$ follows by Cauchy-Schwarz; $(b)$ by the definition of norm; $(c)$ because
$\bM \succeq \bPsi_{\le \evn} \bD_{\le \evn}^4 \bPsi_{\le \evn}^\sT + \kappa_M(\id + \bDelta_M) \succeq \bPsi_{\le \evn} \bD_{\le \evn}^4 \bPsi_{\le \evn}^\sT$; $(d)$ by the definition of $T_{23}$ as in Eq. \eqref{eq:T2i_def}; $(e)$ by Eq.~\eqref{eqn:term_R23}.

For term $T_{13}$, we have 
\[
\begin{aligned}
\vert T_{13} \vert =&~  \vert \boldf^\sT (\bH + \lambda \id_n)^{-1} \bE_{> \evn}\vert \le  \| \boldf \|_2 \| (\bH + \lambda \id_n)^{-1} \|_{\op} \| \bE_{> \evn} \|_2. 
\end{aligned}
\]
Note that we have $\E[\| \boldf \|_2^2] = n \| f_d \|_{L^2}^2$. Further by Eq.~\eqref{eq:HMdecompositionKRR}, we have $\| (\bH + \lambda \id_n)^{-1} \|_{\op} \le 2/(\kappa_H + \lambda)$ with high probability. Finally, we have
\[
\E [ \| \bE_{> \evn} \|_2^2 ] = n \sum_{k = \evn + 1}^\infty \lambda_{d, k}^4 \hat f_k^2 \le n \Big[ \max_{k \ge \evn + 1} \lambda_{d, k}^4 \Big] \| \proj_{> \evn} f_d \|_{L^2}^2. 
\]
As a result, we have
\begin{equation}\label{eqn:term_R13}
\begin{aligned}
\vert T_{13} \vert \le&~ O_{d, \P}(1) \cdot  \| \proj_{> \evn} f_d \|_{L^2} \| f_d \|_{L^2} \Big[n^2 \max_{k \ge \evn + 1} \lambda_{d, k}^4 \Big]^{1/2} / (\kappa_H + \lambda) \\
=&~ O_{d, \P}(1) \cdot \| \proj_{> \evn} f_d \|_{L^2} \| f_d \|_{L^2} \Big[ n \max_{k \ge \evn + 1} \lambda_{d, k}^2 \Big]/ \Big( \sum_{k \ge \evn + 1} \lambda_{d, k}^2 + \lambda \Big)\\
=&~ o_{d, \P}(1) \cdot \| \proj_{> \evn} f_d \|_{L^2} \| f_d \|_{L^2},
\end{aligned}
\end{equation}
where the last equality used Eq.~\eqref{ass:n_parameters_KRR_lower_2} in Assumption \ref{ass:eig_decay}.$(a)$ and the fact that $\lambda \in [0, \Trace(\Hop_{d, > \evn})]$. Combining Eqs. (\ref{eqn:term_R11}), (\ref{eqn:term_R12}) and (\ref{eqn:term_R13}), we get 
\begin{equation}\label{eqn:KRR_term_T1}
T_1 = \| \bS_{\leq \evn}^{1/2} \hat \boldf_{\le \evn} \|_2^2  + o_{d, \P}(1) \cdot (\| f_d \|_{L^2}^2 + \| \proj_{>\evN} f_d \|_{L^{2 + \eta} }^2). 
\end{equation}

\noindent
{\bf Step 5. Terms $T_3$. } 

By Lemma \ref{lem:key_H_U_H_bound} again, we have 
\[
\begin{aligned}
\frac{1}{\noise^2}\E_\beps[T_3]  =&~ \Trace((\bH + \lambda \id_n)^{-1} \bM (\bH + \lambda \id_n)^{-1}) = \Trace(\bPsi_{\le \evn} \bS_{\leq \evn}^2 \bPsi_{\le \evn}^\sT / n^2) + o_{d, \P}(1), 
\end{aligned}
\]
By Proposition \ref{prop:YY_new} and noting that $\bS_{\leq \evn} \preceq \id_{\evn}$, we have 
\[
   \frac{1}{n^2}\Trace(\bPsi_{\le \evn} \bS_{\leq \evn}^2 \bPsi_{\le \evn}^\sT ) \leq \frac{1}{n^2}\Trace(\bPsi_{\le \evn} \bPsi_{\le \evn}^\sT ) = \frac{1}{n^2}\Trace(\bPsi_{\le \evn}^\sT \bPsi_{\le \evn})  = \frac{1}{n^2}n \evn \big(1 + o_{d, \P}(1)
  \big)= o_{d, \P}(1). 
\]
This gives
\begin{align}\label{eqn:term_varT3}
T_3 = o_{d, \P}(1) \cdot \noise^2. 
\end{align}

\noindent
{\bf Step 6. Terms $T_4$. }  

Note that
\[
\begin{aligned}
  \frac{1}{\noise^2}\E_{\beps} [T_4^2 ]= &~ \frac{1}{\noise^2}
  \E_{\beps} [ \beps^\sT (\bH + \lambda \id_n)^{-1} \bE \bE^\sT  (\bH + \lambda \id_n)^{-1} \beps ]\\
= &~ \bE^\sT  (\bH + \lambda \id_n)^{-2 }  \bE.
\end{aligned}
\]
Notice that $\bM\succeq \bPsi_{\leq L } \bD_{\leq L}^4  \bPsi_{\leq L }^{\sT}$ for any $L\in\naturals$,
by the decomposition of Eq.~\eqref{eq:DecompositionFEHM}. Therefore:
\begin{equation}
\begin{aligned}
\| \bD_{\leq L}^2 \bPsi_{\leq L }^{\sT} (\bH + \lambda \id_n)^{-2 } \bPsi_{\leq L } \bD_{\leq L}^2 \|_{\op} = &~   \| (\bH + \lambda \id_n)^{-1  }\bPsi_{\leq L } \bD_{\leq L}^4  \bPsi_{\leq L }^{\sT} (\bH + \lambda \id_n)^{-1  } \|_{\op} \\
\leq&~ \| (\bH + \lambda \id_n )^{-1} \bM ( \bH + \lambda \id_n )^{-1}  \|_{\op}.\label{eq:UBDPsi}
\end{aligned}
\end{equation}
Further notice that, using  Lemma \ref{lem:key_H_U_H_bound} (stated below) followed by
Proposition \ref{prop:YY_new}, we get
\begin{equation}\label{eq:HMH}
\begin{aligned}
  \| (\bH + \lambda \id_n)^{-1} \bM (\bH + \lambda \id_n)^{-1} \|_{\op} = & 
  \| \bPsi_{\leq \evn} \bS_{\leq \evn}^2 \bPsi_{\leq \evn}^{\sT} / n \|_{\op}/n + o_{d, \P}(1/n ) \\
  \leq &  \| \bPsi_{\leq \evn} \bPsi_{\leq \evn}^{\sT} / n \|_{\op}/n + o_{d,\P}(1) = o_{d,\P} (1)\, .
  \end{aligned}
  \end{equation}
Hence, 
\[ 
\begin{aligned}
\bE^\sT  (\bH + \lambda \id_n)^{-2 }  \bE \stackrel{(a)}{=} &~ \lim_{L \to \infty} \bE^\sT_{\leq L}  (\bH + \lambda \id_n)^{-2 }  \bE_{\leq L} \\
\stackrel{(b)}{=} &~ \lim_{L \to \infty} \hat \boldf^\sT_{\leq L}  [\bD_{\leq L}^2 \bPsi_{\leq L }^{\sT} (\bH + \lambda \id_n)^{-2 } \bPsi_{\leq L } \bD_{\leq L}^2] \hat \boldf_{\leq L} \\
\stackrel{(c)}{\leq} &~ \limsup_{L \to \infty}  \| \bD_{\leq L}^2 \bPsi_{\leq L }^{\sT} (\bH + \lambda \id_n)^{-2 } \bPsi_{\leq L } \bD_{\leq L}^2  \|_{\op} \cdot \lim_{L \to \infty} \| \hat \boldf_{\leq L} \|_2^2 \\
\stackrel{(d)}{\leq} &~ \| (\bH + \lambda \id_n )^{-1} \bM ( \bH + \lambda \id_n )^{-1}  \|_{\op}\cdot \| f_d \|_{L^2}^2\\
\stackrel{(e)}{\leq} &~  o_{d,\P}(1)\, \cdot \| f_d \|_{L^2}^2,
\end{aligned}
\]
where the limits for $L\to\infty$ exist with high probability. In particular, $(a)$ holds with high probability
since $\| \bE^\sT_{\leq L} -\bE\|_2\to 0$ as $L\to\infty$, and $\| (\bH + \lambda \id_n)^{-2 } \|_{\op}\le 1/\lambda_{\min}(\bH)^2\le
(2/\kappa_H)^2$, by the decomposition \eqref{eq:HMdecompositionKRR}, together with the fact that
$\| \bDelta_H \|_{\op}=o_{d,\P}(1)$, cf. Eq.~\eqref{eq:DeltaSmall}. Further, $(b)$ is by definition of $\bE_{\leq L}$; $(c)$
by definition of operator norm; $(d)$ by Eq.~\eqref{eq:UBDPsi}; $(e)$ by Eq.~\eqref{eq:HMH}.

We thus obtain
\begin{align}\label{eqn:term_varT4}
T_4 = o_{d, \P}(1) \cdot \noise \cdot \| f_d \|_{L^2} = o_{d,\P} (1) \cdot ( \noise^2 +  \| f_d \|_{L^2}^2 ). 
\end{align}

\noindent
{\bf Step 7. Terms $T_5$. } 

We decompose $T_5$ using $\boldf = \boldf_{\leq \evn} + \boldf_{> \evn}$,
\[
T_5 = T_{51} + T_{52},
\]
where
\[
\begin{aligned}
T_{51} = &~ \beps^\sT (\bH + \lambda \id_n )^{-1} \bM ( \bH + \lambda \id_n )^{-1}  \boldf_{\leq \evn}, \\
T_{52} = &~ \beps^\sT (\bH + \lambda \id_n )^{-1} \bM ( \bH + \lambda \id_n )^{-1}  \boldf_{> \evn} .
\end{aligned}
\]
First notice that, by Eq.~\eqref{eq:HMH},
\[
\begin{aligned}
\| \bM^{1/2} ( \bH + \lambda \id_n)^{-2} \bM^{1/2} \|_{\op} =&~ \| (\bH + \lambda \id_n)^{-1} \bM (\bH + \lambda \id_n)^{-1} \|_{\op} =  o_{d,\P} (1) .
\end{aligned}
\]
Then by Lemma \ref{lem:key_H_U_H_bound}, we get
\[
\begin{aligned}
\frac{1}{\noise^2} \E_{\beps} [T_{51}^2 ] = &~ \frac{1}{\noise^2} \E_{\beps} [ \beps^\sT (\bH + \lambda \id_n)^{-1} \bM (\bH + \lambda \id_n)^{-1} \boldf_{\leq \evn} \boldf_{\leq \evn}^\sT (\bH + \lambda \id_n)^{-1} \bM (\bH + \lambda \id_n)^{-1} \beps ] \\
= &~ \boldf^\sT_{\leq \evn} [ (\bH + \lambda \id_n)^{-1} \bM (\bH + \lambda \id_n)^{-1} ]^2 \boldf_{\leq \evn} \\
\leq &  \| \bM^{1/2} ( \bH + \lambda \id_n)^{-2} \bM^{1/2} \|_{\op} \| \bM^{1/2} ( \bH + \lambda \id_n)^{-1 } \boldf_{\leq \evn} \|_{2}^2 \\
= &~  o_{d,\P}(1)  \cdot T_{21} \\
= &~  o_{d,\P} (1)  \cdot \| \proj_{\leq \evn } f_d \|_{L^2}^2.
\end{aligned}
\]
where the last equality follows by Eq.~\eqref{eqn:term_R21}. Similarly, we get 
\[
\begin{aligned}
\E_{\beps} [T_{52}^2 ]/\noise^2   = &~   o_{d,\P} (1) \cdot T_{23}  = o_{d,\P}  (1) \cdot \| \proj_{> \evn} f_d \|_{L^2}^2.
\end{aligned}
\]
By Markov's inequality, we deduce that
\begin{align}\label{eqn:term_varT5}
T_5 = o_{d, \P}(1) \cdot \noise ( \| \proj_{\leq \evn} f_d \|_{L^2} + \| \proj_{> \evn} f_d \|_{L^2} )= o_{d,\P} (1) \cdot ( \noise^2 +  \| f_d \|_{L^2}^2 ). 
\end{align}

\noindent
{\bf Step 8. Finish the proof. }

Combining Eqs.~(\ref{eqn:KRR_term_T2}), (\ref{eqn:KRR_term_T1}), (\ref{eqn:term_varT3}), (\ref{eqn:term_varT4}) and (\ref{eqn:term_varT5}), we have
\[
\begin{aligned}
R_{\KR}( f_d, \bX, \lambda) = &~ \| f_d \|_{L^2}^2 - 2 T_{1} + T_{2} + T_3 - 2 T_4 + 2T_5 \\
= &~ \| \hat \boldf_{\le \evn} \|_2^2  - 2 \| \bS_{\leq\evn}^{1/2} \hat \boldf_{\le \evn} \|_2^2 + \| \bS_{\leq\evn} \hat \boldf_{\le \evn} \|_2^2 + \| \proj_{> \evn} f_d \|_{L^2}^2 \\
&~ + o_{d, \P}(1) \cdot (\| f_d \|_{L^2}^2 + \| \proj_{>\evN} f_d \|_{L^{2 + \eta} }^2 + \noise^2)  \\
= &~ \| ( \id - \bS_{\leq \evn} ) \hat \boldf_{\le \evn} \|_2^2  + \| \proj_{> \evn} f_d \|_{L^2}^2 + o_{d, \P}(1) \cdot (\| f_d \|_{L^2}^2 + \| \proj_{>\evN} f_d \|_{L^{2 + \eta} }^2 + \noise^2).
\end{aligned}
\]
Recall the expression \eqref{eq:sol_population_KRR} of $\hf_{\gamma^{\seff}}^{\seff}$:
\[
    \hf^{\seff}_{\gamma^{\seff} } = \sum_{j=1}^{\infty}\frac{\lambda_{d,j}^2}{\lambda_{d,j}^2+\frac{\gamma^{\seff} }{n}}\, \hat f_{d,k} \psi_{d,j}\, ,
    \]
 with $\gamma^{\seff} = \lambda + \kappa_H$. From Assumption \ref{ass:eig_decay}.$(a)$, we have $\max_{j > \evn} \lambda_{d,j}^2 = o_{d} (1) \cdot \kappa_H/ n$ and we deduce
 \[
\| ( \id - \bS_{\leq \evn} ) \hat \boldf_{\le \evn} \|_2^2  + \| \proj_{> \evn} f_d \|_{L^2}^2 = \| f_d -  \hf^{\seff}_{\gamma^{\seff} } \|_{L^2}^2 + \| f_d \|_{L^2}^2 \cdot o_{d,\P}(1).
 \]
 We conclude
 \[
 R_{\KR}( f_d, \bX, \lambda)  = \| f_d -\hf^{\seff}_{\gamma^{\seff} } \|_{L^2}^2+ o_{d, \P}(1) \cdot (\| f_d \|_{L^2}^2 + \| \proj_{>\evN} f_d \|_{L^{2 + \eta} }^2 + \noise^2).
 \]
Proceeding analogously (with $\hf^{\seff}_{\gamma^{\seff} }$ replacing $f_d$) we obtain
\[
\| \hat f_{\lambda} - \hf^{\seff}_{\gamma^{\seff} } \|_{L^2}^2 = o_{d, \P}(1) \cdot (\| f_d \|_{L^2}^2 + \| \proj_{>\evN} f_d \|_{L^{2 + \eta} }^2 + \noise^2).
\]

\subsubsection{Auxiliary lemmas}\label{sec:KRR_upper_auxiliary}

\begin{lemma}\label{lem:key_H_U_H_bound}
Follow the assumptions and notations in the proof of Theorem \ref{thm:upper_bound_KRR}. We have 
\[
\| n (\bH + \lambda \id_n)^{-1} \bM (\bH + \lambda \id_n)^{-1} - \bPsi_{\le \evn} \bS_{\leq \evn}^2 \bPsi_{\le \evn}^\sT / n \|_{\op} = o_{d, \P}(1). 
\]
where $\bS_{\leq \evn}$ is the shrinkage matrix defined in Equation \eqref{eq:shrinkage_matrix}.
\end{lemma}

\begin{proof}[Proof of Lemma \ref{lem:key_H_U_H_bound}]
We simplify the notations by defining $\bpsi_k = (\psi_k(\bx_i))_{i \in [n]} \in \R^{n}$
$\bPsi = \bpsi_{\le \evn} \in \R^{n \times \evn}$,
$\bD =\bD_{\le \evn }= \diag(\lambda_{d, 1}, \ldots, \lambda_{d, \evn}) \in \R^{\evn\times \evn}$.

Then recalling Eq.~\eqref{eq:HMdecompositionKRR}, we have
\begin{equation}\label{eq:HM_decompo}
\begin{aligned}
\bH =&~ \bPsi \bD^2 \bPsi^\sT + \kappa_H (\id + \bDelta_H), \\
\bM =&~ \bPsi \bD^4 \bPsi^\sT + \kappa_M (\id + \bDelta_M), 
\end{aligned}
\end{equation}
where $\kappa_H = \Trace(\Hop_{d, > \evn})$ and $\kappa_M = \Trace(\Hop_{d, > \evn}^2)$, and
 \begin{equation}\label{eq:DeltaSmall2}
\max\{ \| \bDelta_M \|_{\op}, \| \bDelta_H \|_{\op} \} = o_{d, \P}(1). 
\end{equation}

As a result, we have
\[
\begin{aligned}
n (\bH + \lambda \id_n)^{-1} \bM (\bH + \lambda \id_n)^{-1} =&~ T_1 + T_2, 
\end{aligned}
\]
where
\[
\begin{aligned}
T_1 =&~  n   \kappa_M (\bH + \lambda \id_n)^{-1}  (\id_M + \bDelta_M) (\bH + \lambda \id_n)^{-1}, \\
T_2 =&~ n (\bH + \lambda \id_n)^{-1} \bPsi \bD^4 \bPsi^\sT  (\bH + \lambda \id_n)^{-1}.
\end{aligned}
\] 

\noindent
{\bf Step 1. Bound term $T_1$. } 

For $T_1$, by Eqs.~\eqref{eq:HM_decompo} and \eqref{eq:DeltaSmall2}, we have,
\begin{equation}\label{eqn:key_H_U_H_bound_T1_1}
\begin{aligned}
\| T_1 \|_{\op} \le&~ n \kappa_M \| (\bH + \lambda \id_n)^{-1} \|_{\op}^2 \| \id + \bDelta_M \|_{\op} = O_{d, \P}(1) \cdot  n [ \kappa_M / \kappa_H^2].
\end{aligned}
\end{equation}
By Eq.~\eqref{ass:n_parameters_KRR_lower_2} in Assumption \ref{ass:eig_decay}.$(a)$, we have
\[
 \frac{\kappa_M}{\kappa_H^2} = \frac{\Trace(\Hop_{d, > \evn}^2)}{\Trace(\Hop_{d, > \evn})^2} \leq \frac{\|\Hop_{d, > \evn} \|_{\op} }{ \Trace ( \Hop_{d, > \evn} )} = O_d ( n^{-1 - \delta_0} ).
\]
We conclude that
\[
\| T_1 \|_{\op} = o_{d, \P}(1). 
\]

\noindent
{\bf Step 2. Bound term $T_2$. } 

For $T_2$, by the Sherman-Morrison-Woodbury formula, we have, setting   $\bDelta_H ' =  \kappa_H\bDelta_H/(\lambda+\kappa_H)$, 
\[
\begin{aligned}
T_2 =&~  n (\id + \bDelta_H ' )^{-1} \bPsi ( (\kappa_H + \lambda) \bD^{-2} + \bPsi^\sT (\id + \bDelta_H ' )^{-1} \bPsi)^{-2} \bPsi^\sT (\id + \bDelta_H ' )^{-1} \\
=& \bE \bPsi \bR^{2} \bPsi^\sT \bE / n, 
\end{aligned}
\]
where
\[
\begin{aligned}
\bE =&~ (\id + \bDelta_H ' )^{-1}, \\
\bR =&~ [ (\kappa_H + \lambda) (n \bD^2 )^{-1} + \bPsi^\sT \bE \bPsi / n ]^{-1}. 
\end{aligned}
\]
Denote $\bS : = \bS_{\leq \evn} = [ \id_\evn + ( \kappa_H + \lambda) ( n \bD^2)^{-1} ]^{-1}$. We have 
\[
\begin{aligned}
\| T_2 - \bPsi^\sT \bS^2 \bPsi / n \|_{\op} \le&~ (1 + \| \bE \|_{\op}) \| \bE - \id \|_{\op} \| \bPsi \bR^{2} \bPsi^\sT / n\| +  \| \bPsi \bR^{-2} \bPsi^\sT / n - \bPsi \bS^2\bPsi^\sT / n\|_{\op} \\
\le&~ (1 + \| \bE \|_{\op}) \| \bE - \id \|_{\op} \| \bPsi \bR^{2} \bPsi^\sT / n\|_{\op} + \| \bPsi \bPsi^\sT / n\|_{\op}  \| \bR^2 - \bS^2 \|_{\op}. \\
\end{aligned}
\]
Recalling Eq.~\eqref{eq:DeltaSmall2}, we have $\| \bE - \id \|_{\op} = o_{d, \P}(1)$ and by Theorem \ref{prop:expression_U}.$(b)$, we have $\| \bPsi^\sT \bPsi/ n - \id \|_{\op} = o_{d, \P}(1)$. Furthermore
\[
\| \bR^2 - \bS^2 \|_{\op} = \| \bR - \bS \|_\op (\| \bR \|_{\op}+ \|\bS \|_{\op} 0\leq \| \bR \|_\op \| \bS \|_\op (\| \bR \|_\op + \| \bS \|_\op ) \| \bR^{-1}  - \bS^{-1} \|_\op.
\]
We have
\[
\begin{aligned}
\| \bR^{-1} - \bS^{-1} \|_{\op} \le&~  \| \bPsi^\sT \bE \bPsi / n - \bPsi^\sT \bPsi / n \|_{\op} + \| \bPsi^\sT \bPsi / n - \id \|_{\op} \\
\le&~ \| \bPsi^\sT \bPsi / n \|_{\op} \| \bE - \id \|_{\op} + \| \bPsi^\sT \bPsi / n - \id \|_{\op} = o_{d, \P}(1). 
\end{aligned}
\]
Furthermore $\| \bS \|_\op \leq 1$ and by the above computation, $\| \bR \|_\op \leq 1 + o_{d,\P} (1)$.
Combining the above inequalities, we have 
\[
\| T_2 - \bPsi^\sT \bS^2 \bPsi / n \|_{\op} = o_{d, \P}(1). 
\]
This gives 
\[
\| n (\bH + \lambda \id_n)^{-1} \bM (\bH + \lambda \id_n)^{-1} - \bPsi \bS^2 \bPsi^\sT / n \|_{\op} \le \| T_1 \|_{\op} + \| T_2 - \bPsi \bS^2 \bPsi^\sT / n  \|_{\op} = o_{d, \P}(1). 
\]
This completes the proof. 
\end{proof}

\begin{lemma}\label{lem:lem_for_error_bound_R11}
Follow the assumptions and notations in the proof of Theorem \ref{thm:upper_bound_KRR}. We have
\[
\|  \bS_{\leq\evn} - \bPsi_{\le \evn}^\sT( \bH + \lambda \id_n)^{-1}\bPsi_{\le \evn} \bD_{\le \evn}^2 \|_{\op} = o_{d, \P}(1),
\]
where $\bS_{\leq \evn}$ is the shrinkage matrix defined in Equation \eqref{eq:shrinkage_matrix}.
\end{lemma}

\begin{proof}[Proof of Lemma \ref{lem:lem_for_error_bound_R11}]~ We will follow the notations in the proof of Proposition \ref{lem:key_H_U_H_bound}. Applying Theorem \ref{prop:expression_U}  with respect to operator $\Hop_d$ and by Eq.~\eqref{ass:concentration_orthogonal_KRR_lower_e2} in Assumption \ref{ass:KRR}.$(c1)$, we have 
\[
\bH + \lambda \id_n = \bPsi \bD^2 \bPsi^\sT + \kappa_H (\id_n + \bDelta_H) + \lambda \id_n = \bPsi \bD^2 \bPsi^\sT + ( \kappa_H + \lambda) (\id_n + \bDelta_H '),
\]
 where $\| \bDelta_H \|_{\op}, \| \bDelta_H ' \|_\op = o_{d, \P}(1)$. By the Sherman-Morrison-Woodbury formula, we have
\[
\bPsi^\sT[ \bPsi \bD^2 \bPsi^\sT + (\kappa_H + \lambda) ( \id_n + \bDelta_H') ]^{-1}\bPsi \bD^2 = \bPsi^\sT \bE^{-1} \bPsi  \bR /n,
\]
where $\bE = \id_n + \bDelta_H$ and $\bR = [ ( \kappa_H+\lambda) (n \bD^2)^{-1} + \bPsi^\sT \bE^{-1} \bPsi / n ]^{-1}$. We have
\[
\begin{aligned}
\| \bS - \bPsi^\sT \bE^{-1} \bPsi  \bR /n  \|_{\op} \leq \| \bR \|_\op \| \bPsi^\sT \bE^{-1} \bPsi / n - \id \|_\op + \| \bR - \bS \|_\op.
\end{aligned}
\]
In the proof of Lemma \ref{lem:key_H_U_H_bound}, we already showed that $\| \bPsi^\sT \bE^{-1} \bPsi / n - \id \|_\op = o_{d,\P} (1)$, $\| \bR - \bS \|_\op = o_{d,\P}(1)$ and $\| \bR \|_\op = O_{d,\P}(1)$, which concludes the proof.
\end{proof}

\subsection{Kernel ridge regression under relaxed assumptions on the diagonal}
\label{sec:KRR_weaker_gap}

In this section, we state and prove a version of Theorem \ref{thm:upper_bound_KRR} that holds under weaker  assumptions.
Namely, instead of the concentration bound in Assumption \ref{ass:KRR}.$(c)$ we only require that the diagonal terms
are upper bounded by a sub-polynomial factors times their expectation. Instead, we assume a spectral gap condition that
was not required in the previous section.

We will first describe the modified assumption, then state the new version of the theorem.
The proof is very similar to the one in the previous section. We will therefore use the same notations and only sketch the differences.

\begin{assumption}[Relaxed kernel concentration at level $\{ (n(d) , \evn (d)) \}_{d \ge 1}$]\label{ass:KRR_app} 
  We assume the kernel concentration property at level $\{ (n(d) , \evn (d)) \}_{d \ge 1}$, as stated in
  Assumption \ref{ass:KRR}, with condition $(c)$ replaced by the following 
\begin{itemize}
\item[(c')] (Upper bound on the diagonal elements of the kernel) For $(\bx_i)_{i \in [n(d)]} \sim_{iid} \nu_d$ and any $\delta>0$, we have
\begin{align}
  \max_{i \in [n(d)]} \E_{\bx \sim u(d) }\big[H_{d, > u(d) }(\bx_i, \bx)^2\big]  = &~O_{d, \P}(n(d)^\delta) \cdot \E_{\bx, \bx' \sim \nu_d}\big[H_{d, > u(d) }(\bx, \bx')^2\big], \label{ass:concentration_orthogonal_KRR_lower_e1}\\
  \max_{i \in [n(d)]}  H_{d, > u(d) }(\bx_i, \bx_i)  =&~ O_{d, \P}(n(d)^\delta)\cdot \E_{\bx \sim \nu_d}[H_{d, > u(d) }(\bx, \bx)].  \label{ass:concentration_orthogonal_KRR_lower_e2}
\end{align}
\end{itemize}
\end{assumption}

\begin{assumption}[Eigenvalue condition at level $\{ (n(d) , \evn (d)) \}_{d \ge 1}$]\label{ass:eig_cond_app}
  We assume the eigenvalue condition Assumption \ref{ass:eig_decay} and, in addition, the following to hold
  \begin{itemize}
\item[(c)] There exists a fixed $\delta_0 > 0$, such that 
  \[
    n^{1-\delta_0}\ge \frac{1}{\lambda_{d,\evn(d)}^2} \sum_{k=\evn(d)+1} \lambda_{d,k}^2\, .
\]
\end{itemize}
\end{assumption}

\begin{assumption}[Regularization and lower bound on diagonal elements]\label{ass:KRR_lower_diagonal} Consider the regularization parameter $\lambda \in \R_{\ge 0}$. 
We assume that one of the following holds:
\begin{itemize}
\item[(i)] For $(\bx_i)_{i \in [n(d)]} \sim_{iid} \nu_d$ and any $\delta>0$, we have
\begin{align}
  \min_{i \in [n(d)]} \E_{\bx \sim \nu_d}[H_{d, > \evn(d)}(\bx_i, \bx)^2]  = &\Omega_{d, \P}(n(d)^{-\delta} ) \cdot \E_{\bx, \bx' \sim \nu_d}[H_{d, > \evn (d)}(\bx, \bx')^2], \label{ass:lower_bound_KRR_diagonal_1}\\
  \min_{i \in [n(d)]}  H_{d, > \evn (d)}(\bx_i, \bx_i)  =&~ \Omega_{d, \P}(n(d)^{-\delta} ) \cdot \E_{\bx}[H_{d, > \evn(d)}(\bx, \bx)], \label{ass:lower_bound_KRR_diagonal_2}
\end{align}
and $\lambda = O_d (1) \cdot \Trace(\Hop_{d, > \evn(d)})$ (in particular, taking $\lambda = 0$ is fine).
\item[(ii)] We have $\lambda = \Theta_d (1) \cdot \Trace(\Hop_{d, > \evn(d)})$.
\end{itemize} 
\end{assumption}

\begin{theorem}\label{thm:upper_bound_KRR_app}
    Let $\{ f_d \in \cD_d \}_{d \ge 1}$ be a sequence of functions, $(\bx_i)_{i \in [n(d)]} \sim \nu_d$ independently, $\{ \Hop_d \}_{d\ge 1}$ be a sequence of kernel operators such that $\{ ( \Hop_d , n(d),\evn(d) )\}_{d \ge 1}$ and the regularization parameter $\lambda$ satisfy
    Assumptions \ref{ass:KRR_app},  \ref{ass:eig_cond_app}, and \ref{ass:KRR_lower_diagonal}. Then for any fixed $\eta >0$, we have
\begin{align}
\vert R_{\KR}(f_d, \bX, \bTheta, \lambda) - \| \proj_{> \evn } f_d \|_{L^2}^2 \vert = o_{d,\P}(1)\cdot (\| f_d \|_{L^2}^2+\| \proj_{>\evn}f_d \|_{L^{2 + \eta} }^2 + \noise^2). 
\end{align}
\end{theorem}

\subsubsection{Proof outline for Theorem \ref{thm:upper_bound_KRR_app}}

Throughout this section, we will denote $\delta_0 >0$ a fixed constant and $\delta >0$ a constant that can be arbitrarily small. The value of $\delta_0$ is allowed to change from line to line. 

By the spectral gap condition (Assumption \ref{ass:eig_cond_app}), the population estimator $\hf_{\gamma^{\seff}}^{\seff}$ defined in Eq.~\eqref{eq:sol_population_KRR} is approximately given by
\[
\| \hf_{\gamma^{\seff}}^{\seff} -  \proj_{\leq \evn} f_d \|_{L^2}^2 = o_{d,\P} (1) \cdot \| f_d \|_{L^2}^2.
\]
Similarly, the shrinkage matrix defined in Eq.~\eqref{eq:shrinkage_matrix} verifies
\[
\| \bS_{\leq \evn } - \id_{\evn} \|_{\op} = O_{d,\P} ( n^{-\delta_0} ).
\]
From Theorem \ref{prop:expression_U} applied to the operator $\Hop_d$ and $\Hop_d^2$, the kernel matrices $\bH$ and $\bM$ can be rewritten as
\begin{equation}\label{eq:HMdecompositionKRR_c2}
\begin{aligned}
\bH =&~ \bPsi_{\le \evn} \bD_{\le \evn}^2 \bPsi_{\le \evn}^\sT + \kappa_H (\bLambda_H + \bDelta_H), \\
\bM =&~ \bPsi_{\le \evn} \bD_{\le \evn}^4 \bPsi_{\le \evn}^\sT + \kappa_M (\bLambda_M + \bDelta_M),
\end{aligned}
\end{equation}
where 
\[
\begin{aligned}
\kappa_H =&~ \Trace(\Hop_{d, > \evn})= \sum_{k \ge \evn + 1} \lambda_{d, k}^2,\\
\kappa_M =&~ \Trace(\Hop_{d, > \evn}^2) =\sum_{k \ge \evn + 1} \lambda_{d, k}^4,
\end{aligned}
\]
and
\[
\begin{aligned}
\bLambda_H = & \diag \big( \big\{  H_{d, > \evn} ( \bx_i , \bx_i)/\kappa_H \big\}_{i \in [n]} \big), \\
  \bLambda_M = & \diag \Big( \Big\{ \E_{\bx} [ H_{d, > \evn} ( \bx_i , \bx)^2 ]/\kappa_M \Big\}_{i \in [n]} \Big),
\end{aligned}
\]
and there exists a fixed $\delta_0 >0$ such that 
\begin{equation}\label{eq:DeltaSmall22}
\max\{ \| \bDelta_M \|_{\op}, \| \bDelta_H \|_{\op} \} = O_{d, \P} (n^{-\delta_0 } ).
\end{equation}
From Lemma \ref{lem:bound_max_hyper} applied to $\bLambda_H$ and $\bLambda_M$ with Assumptions \ref{ass:KRR}.$(a)$ and \ref{ass:KRR_app}.$(c')$, we have
\begin{equation}\label{eq:upper_bound_diag_HM}
\begin{aligned}
\bLambda_H \preceq & \, O_{d,\P} ( n^\delta) \cdot  \id_n, \\ \bLambda_M \preceq & \, O_{d,\P} ( n^\delta)  \cdot  \id_n.
\end{aligned}
\end{equation}
Furthermore from Assumption \ref{ass:KRR_lower_diagonal} and Eq.~\eqref{eq:DeltaSmall22}, we have for any $\delta < \delta  '$,
\begin{equation}\label{eq:lower_bound_diag_H}
\bH + \lambda \id_n \succeq \kappa_H (\bLambda_H + \bDelta_H) + \lambda \id_n \succeq \Omega_{d,\P} (n^{-\delta}) \cdot \kappa_H \cdot \id_n.
\end{equation}

The handling of the bounds on $T_1$, $T_2$, $T_3$, $T_4$ and $T_5$ follows from the same computation as Section \ref{sec:proof_KRR_app} where every $o_{d,\P}(1)$ is replaced by $O_{d,\P}( n^{-\delta_0})$ for some fixed $\delta_0>0$ while every $O_{d,\P}(1)$ is replaced by $O_{d,\P}(n^\delta)$ with $\delta >0$ arbitrary small. In particular, bounds of the form  $O_{d,\P}(1) \cdot o_{d,\P}(1)$ should be replaced by $ O_{d,\P}(n^\delta) \cdot O_{d,\P}( n^{-\delta_0})$, and taking $\delta >0$ sufficiently small yields a bound $o_{d,\P}(1)$ (see the proofs bellow for some examples).

Below we detail the proof of the updated auxiliary lemmas from Section \ref{sec:KRR_upper_auxiliary}. Eq.~\eqref{eq:c2_bound_1} is used to bound the term $T_{21}$, Eq.~\eqref{eq:c2_bound_2} is used to bound the term $T_{23}$, while Eq.~\eqref{eq:c2_bound_3} is used to bound the term $T_{3}$, $T_4$ and $T_5$.

\begin{lemma}\label{lem:key_H_U_H_bound_c2}
Follow the assumptions of Theorem \ref{thm:upper_bound_KRR_app} and the same notations as in Section \ref{sec:proof_KRR_app}. Define
\[
\bG = n (\bH + \lambda \id_n)^{-1} \bM (\bH + \lambda \id_n)^{-1}.
\]
Then, there exists a fixed $\delta_0>0$ such that for any $\delta >0$,
\begin{align}
\|  \bpsi_{\leq \evn}^\sT \bG \bpsi_{\leq \evn} / n  - \id_{\evn} \|_{\op} =&~ O_{d, \P}(n^{-\delta_0} ), \label{eq:c2_bound_1} \\
\boldf_{>\evn}^\sT \bG \boldf_{>\evn}/n = &~ O_{d,\P} (n^{-\delta_0} ) \cdot \| \proj_{>\evn} f_d \|_{L^{2 + \eta}}^2,\label{eq:c2_bound_2} \\
\| \bG \|_{\op} \leq &~ O_{d,\P} (n^\delta).\label{eq:c2_bound_3} 
\end{align}
\end{lemma}

\begin{proof}[Proof of Lemma \ref{lem:key_H_U_H_bound_c2}]
Recall that we denote $\delta_0 >0$ a fixed constant and $\delta >0$ a constant that can be arbitrarily small. The value of $\delta_0$ is allowed to change from line to line.

Following the notations as in the proof of Lemma \ref{lem:key_H_U_H_bound}, we have
\[
\begin{aligned}
\bH =&~ \bPsi \bD^2 \bPsi^\sT + \kappa_H (\bLambda_H + \bDelta_H), \\
\bM =&~ \bPsi \bD^4 \bPsi^\sT + \kappa_M (\bLambda_M + \bDelta_M).
\end{aligned}
\]
Consider the same decomposition $\bG = \bT_1 + \bT_2$ as in the proof of Lemma \ref{lem:key_H_U_H_bound}, where
\[
\begin{aligned}
\bT_1 =&~  n   \kappa_M (\bH + \lambda \id_n)^{-1}  (\bLambda_M + \bDelta_M) (\bH + \lambda \id_n)^{-1}, \\
\bT_2 =&~ n (\bH + \lambda \id_n)^{-1} \bPsi \bD^4 \bPsi^\sT  (\bH + \lambda \id_n)^{-1}.
\end{aligned}
\] 

\noindent
{\bf Step 1. Bound term $\bT_1$. } 

For $\bT_1$, by Eqs.~\eqref{eq:upper_bound_diag_HM} and \eqref{eq:lower_bound_diag_H}, we have for any $\delta >0$,
\begin{equation}\label{eqn:key_H_U_H_bound_T1_1_c2}
\begin{aligned}
\| \bT_1 \|_{\op} \le&~ n \kappa_M \| (\bH + \lambda \id_n)^{-1} \|_{\op}^2 \| (\bLambda_M + \bDelta_M) \|_{\op} \\
\le&~ O_{d, \P}(1) \cdot  n \kappa_M \cdot n^{2\delta} \kappa_H^{-2} \cdot n^{2\delta} \\
\le &~ O_{d, \P}(1) \cdot n^{1+ 4\delta} \frac{\kappa_M}{\kappa_H^2}.   
\end{aligned}
\end{equation}
By Eq.~\eqref{ass:n_parameters_KRR_lower_2} in Assumption \ref{ass:eig_decay}.$(a)$, we have
\[
 \frac{\kappa_M}{\kappa_H^2} = \frac{\Trace(\Hop_{d, > \evn}^2)}{\Trace(\Hop_{d, > \evn})^2} \leq \frac{\|\Hop_{d, > \evn} \|_{\op} }{ \Trace ( \Hop_{d, > \evn} )} = O_d ( n^{-1 - \delta_0} ).
\]
Hence, taking $\delta$ sufficiently small in Eq.~\eqref{eqn:key_H_U_H_bound_T1_1} yields 
\[
\| \bT_1 \|_{\op} = O_{d, \P}(n^{-\delta_0} ). 
\]

\noindent
{\bf Step 2. Simplifying the term $\bT_2$. } 

Introduce $\bA = \bLambda_H + \bDelta  + (\lambda / \kappa_H) \cdot \id_n$ so that
\[
\bH + \lambda \id_n = \bPsi \bD^2 \bPsi^\sT + \kappa_H \bA.
\]
By the Sherman-Morrison-Woodbury formula, we have 
\[
\begin{aligned}
\bT_2 =&~ \bA^{-1} \bpsi \big( \kappa_H (n \bD^2)^{-1} + \bpsi^\sT \bA^{-1} \bpsi /n \big)^{-2} \bpsi^\sT \bA^{-1} /n.
\end{aligned}
\]
From Assumption \ref{ass:eig_cond_app}.$(b)$, we have $\kappa_H (n \bD^2)^{-1} \preceq \Omega_{d,\P} ( n^{-\delta_0} ) \cdot \id_n$. Furthermore, recalling that $\| \bpsi^\sT \bpsi - \id_{\evn} \|_{\op} = o_{d,\P}(1)$ and $\bA^{-1} \succeq O_{d,\P}(n^{-\delta}) \id_n$ for any $\delta >0$, we deduce (for example from Lemma \ref{lem:ineq_sing_value})
\[
\big\| \bT_2 -  \bA^{-1} \bpsi \big(  \bpsi^\sT \bA^{-1} \bpsi /n \big)^{-2} \bpsi^\sT \bA^{-1} /n \big\|_{\op} = O_{d,\P} ( n^{-\delta_0} ).
\]
Denote $\bS = \bLambda_H + (\lambda / \kappa_H) \cdot \id_n$ the diagonal matrix such that we have $\| \bA - \bS \|_{\op} = O_{d,\P} (n^{-\delta_0} )$. We have $ \Omega_{d} (n^{-\delta}) \cdot \id_n \preceq \bS \preceq O_{d} (n^\delta) \cdot \id_n$. Similarly to the previous line, we get
\[
\big\| \bA^{-1} \bpsi \big(  \bpsi^\sT \bA^{-1} \bpsi /n \big)^{-2} \bpsi^\sT \bA^{-1} /n - \bS^{-1}  \bpsi \big(  \bpsi^\sT \bS^{-1} \bpsi /n \big)^{-2} \bpsi^\sT \bS^{-1} /n \big\|_{\op} = O_{d,\P} (n^{-\delta_0}).
\]
Denote $\bR =  \bS^{-1}  \bpsi \big(  \bpsi^\sT \bS^{-1} \bpsi /n \big)^{-2} \bpsi^\sT \bS^{-1} /n$.

\noindent
{\bf Step 3. Proving the bounds. } 

First notice that because $ \Omega_{d} (n^{-\delta}) \cdot \id_n \preceq \bS \preceq O_{d} (n^\delta) \cdot \id_n$ and $\| \bpsi^\sT \bpsi - \id_{\evn} \|_{\op} = o_{d,\P}(1)$, we have
\[
\| \bG \|_{\op} \leq \|\bT_1\|_{\op} + \| \bT_2 - \bR\|_{\op} + \| \bR\|_{\op} = O_{d,\P} (n^\delta),
\]
for any $\delta >0$, which proves Eq.~\eqref{eq:c2_bound_3}. Similarly,
\[
\| \bpsi^\sT \bG \bpsi/n - \id_{\evn} \|_{\op} \leq  (  \|\bT_1\|_{\op} + \| \bT_2 - \bR\|_{\op} ) \| \bpsi/\sqrt{n} \|_{\op}^2 + \| \bpsi^\sT \bR \bpsi/n - \id_{\evn} \|_{\op} = O_{d,\P} ( n^{-\delta_0}),
\]
which proves Eq.~\eqref{eq:c2_bound_1}.

Notice that 
\[
\bR = \bS^{-1}  \bpsi \big(  \bpsi^\sT \bS^{-1} \bpsi /n \big)^{-2} \bpsi^\sT \bS^{-1} /n \preceq O_{d,\P} (n^\delta) \cdot. \bS^{-1}  \bpsi \bpsi^\sT \bS^{-1} /n 
\]
Denote $\bS = \diag ( (s_i)_{i \in [n]} )$ and recall the decomposition
\[
\boldf_{>\evn} = \sum_{k = \evn +1}^\infty \hat f_k  \bpsi_k.
\]
We have
\[
\begin{aligned}
\E \Big[ \boldf_{>\evn}^\sT \bS^{-1}  \bpsi \bpsi^\sT \bS^{-1} \boldf_{>\evn} \Big]/n^2 = &~\E\Big[\Big(\sum_{k\ge \evn+1} \hat f_k \bpsi_k^\sT \Big)  \bS^{-1}  \bPsi_{\le \evn}  \bPsi_{\le \evn}^\sT  \bS^{-1}  \Big(\sum_{k\ge \evn+1} \bpsi_k \hat f_k \Big) \Big] /n^2 \\
=&~ \sum_{u, v \ge \evn+1} \sum_{t = 1}^\evn \sum_{i, j \in [n]}  \Big\{  s_i^{-1} s_j^{-1} \E \Big[ \psi_u(\bx_i)  \psi_{t}(\bx_i) \psi_{t}(\bx_j)  \psi_v(\bx_j) \Big] /n^2 \Big\} \hat f_v \hat f_u \\
=& \sum_{u, v \ge \evn+1} \sum_{t = 1}^\evn \sum_{i \in [n]} s_i^{-2}  \Big\{ \E \Big[\psi_u(\bx_i)  \psi_{t}(\bx_i) \psi_{t}(\bx_i)  \psi_v(\bx_i) \Big] /n^2 \Big\} \hat f_v \hat f_u \\
=&~  O_{d} (n^\delta) \cdot \frac{1}{n}\sum_{s = 1}^\evn \E_\bx \Big[ \big(\proj_{> \evn}f_d(\bx) \big)^2  \psi_{s}(\bx)^2 \Big] \\
\le&~ O_{d} (n^\delta) \cdot  \frac{1}{n}\sum_{t = 1}^\evn \| \proj_{> \evn} f_d \|_{L^{2 + \eta}}^2 \| \psi_t \|_{L^{ (4 + 2\eta)/\eta}}^2 \\
\le&~ O_{d} (n^\delta) \cdot \frac{\evn}{n} \| \proj_{> \evn} f_d \|_{L^{2 + eta}}^2 = O_{d} (n^{-\delta_0} ) \cdot \| \proj_{> \evn} f_d \|_{L^{2 + \eta}}^2.
\end{aligned}
\]
We deduce by Markov's inequality that
\[
| \boldf_{>\evn}^\sT \bR\boldf_{>\evn}/ n | \leq |\boldf_{>\evn}^\sT \bS^{-1}  \bpsi \bpsi^\sT \bS^{-1} \boldf_{>\evn} / n | = O_{d} (n^{-\delta_0} ) \cdot \| \proj_{> \evn} f_d \|_{L^{2 + \eta}}^2.
\]
We deduce that
\[
|\boldf_{>\evn}^\sT \bG \boldf_{>\evn} | \leq ( \| \bT_1\|_{\op} + \| \bT_2 - \bR \|_{\op} ) \| \boldf_{>\evn} /\sqrt{n} \|_2^2 + | \boldf_{>\evn}^\sT \bR\boldf_{>\evn}/ n |  = O_{d} (n^{-\delta_0} ) \cdot \| \proj_{> \evn} f_d \|_{L^{2 + \eta}}^2,
\]
which concludes the proof.
\end{proof}

\begin{lemma}\label{lem:lem_for_error_bound_R11_c2}
Follow the assumptions of Theorem \ref{thm:upper_bound_KRR_app} and the same notations as in Section \ref{sec:proof_KRR_app}. There exists a fixed $\delta_0>0$ such that
\[
\|  \id_{\leq\evn} - \bPsi_{\le \evn}^\sT( \bH + \lambda \id_n)^{-1}\bPsi_{\le \evn} \bD_{\le \evn}^2 \|_{\op} = O_{d, \P}(n^{-\delta_0}).
\]
\end{lemma}

\begin{proof}[Proof of Lemma \ref{lem:lem_for_error_bound_R11_c2}]~ We follow the same argument as in Lemma \ref{lem:lem_for_error_bound_R11}. We have 
\[
\bH + \lambda \id_n = \bPsi \bD^2 \bPsi^\sT + \kappa_H \cdot \bA,
\]
where we denoted $\bA = \bLambda_H + \bDelta  + (\lambda / \kappa_H) \cdot \id$. By the Sherman-Morrison-Woodbury formula, we have
\[
\bPsi^\sT[  \bPsi \bD^2 \bPsi^\sT + \bA ]^{-1} \bPsi \bD^2 = \bPsi^\sT \bA^{-1} \bPsi  [\kappa_H (n \bD^2)^{-1} +   \bPsi^\sT \bA^{-1} \bPsi   /n]^{-1}  /n.
\]
Hence
\[
\begin{aligned}
\| \id_{\evn} -\bPsi^\sT \bA^{-1} \bPsi  [\kappa_H (n \bD^2)^{-1} +   \bPsi^\sT \bA^{-1} \bPsi   /n]^{-1}  /n \|_{\op} =  \| \kappa_H (n \bD^2)^{-1}   ( \kappa_H (n \bD^2)^{-1} +   \bPsi^\sT \bA^{-1} \bPsi   /n)^{-1}    \|_\op.
\end{aligned}
\]
We have by Assumption \ref{ass:eig_cond_app}, $n \bD^2 \succeq \Omega_{d} ( n^{\delta_0} ) \cdot \kappa_H \cdot \id_\evn$. Furthermore, by Eq.~\eqref{eq:lower_bound_diag_H}, we have $\bA^{-1} \succeq \Omega_{d} ( n^{- \delta} ) \cdot \kappa_H \cdot \id_n$. Using that $\|\bPsi^\sT \bPsi - \id_\evn \|_{\op} = o_{d,\P}(1)$, we deduce that for any $\delta >0$,
\[
\| ( \kappa_H (n \bD^2)^{-1} +   \bPsi^\sT \bA^{-1} \bPsi   /n)^{-1}    \|_\op = O_{d,\P} (n^{\delta}).
\]
We deduce that 
\[
\| \id_{\evn} -\bPsi^\sT \bA^{-1} \bPsi  [\kappa_H (n \bD^2)^{-1} +   \bPsi^\sT \bA^{-1} \bPsi   /n]^{-1}  /n \|_{\op} = O_{d,\P} ( n^{-\delta_0} ) \cdot O_{d,\P} (n^\delta).
\]
Taking $\delta$ sufficiently small concludes the proof.
\end{proof}

\section{Proof of Theorem \ref{prop:RFK_sphere}: generalization error of RFRR on the sphere and hypercube}\label{sec:proof_RFK_sphere}

We check that Assumption \ref{ass:activation_RFK_sphere} implies the assumptions of Theorem \ref{thm:RFK_generalization} on the sphere (Section \ref{sec:sphere_case}) and on the hypercube (Section \ref{sec:hypercube_case}).

\subsection{On the sphere}\label{sec:sphere_case}

\begin{proof}[Proof of Theorem \ref{prop:RFK_sphere} on the sphere]

Consider the spherical case $\btheta, \bx \sim \Unif ( \S^{d-1} ( \sqrt{d} ))$ and $ d^{\lvn+\delta_0}  \leq n \leq d^{\lvn +1 - \delta_0}$ and $d^{\lvN+\delta_0} \leq N \leq d^{\lvN +1 - \delta_0}$. Take $\sigma_d ( \bx; \btheta) = \barsigma_d ( \< \bx , \btheta \> / \sqrt{d})$ for some activation function $\barsigma_d : \R \to \R$ satisfying Assumption \ref{ass:activation_RFK_sphere} at level $(\lvn , \lvN)$ (see Section \ref{sec:RF_Examples} in the main text).

\noindent
{\bf Step 1. Diagonalization of the activation function and choosing $\evn = \evn (d)$, $\evN = \evN(d)$. }

By rotational invariance, we can decompose $\barsigma$ in the basis of spherical harmonics (see Section \ref{sec:functions_sphere})
\[
\sigma_d ( \bx ; \btheta ) = \barsigma_d ( \<\bx , \btheta \> / \sqrt{d}) = \sum_{k=0}^\infty  \xi_{d,k}  B(\S^{d-1} ; k) Q_k^{(d)} ( \< \bx , \btheta \> ) =  \sum_{k = 0}^\infty  \xi_{d,k} \sum_{s \in [B(d,k)] } Y_{ks} ( \bx) Y_{ks} (\btheta)\, ,
\]
where the distinct eigenvalues are $\xi_{d,k}$ with degeneracy 
\[
B(\S^{d-1} ; k) = \frac{d-2+2k}{d-2}\binom{d-3+k}{k} \, .
\]
We have for fixed $k$, $ B(\S^{d-1} ; k) = \Theta_d (  d^{k} )$. Furthermore, we have uniformly $\sup_{k \ge \ell} B(\S^{d-1} ; k)^{-1} = O_d ( d^{-\ell} )$ (see Lemma 1 in \cite{ghorbani2019linearized}). Notice that by Assumption \ref{ass:activation_RFK_sphere} (see for example Lemma 5 in \cite{ghorbani2019linearized}), there exists a constant $C>0$ such that
\begin{equation}
\| \barsigma_d \|_{L^2}^2 = \sum_{k=0}^\infty \xi_{d,k}^2  B(\S^{d-1} ; k)  \leq C,\label{eq:BoundL2Sigma}
\end{equation}
which implies that $\xi_{d,k}^2 = O_d ( B(\S^{d-1} ; k)^{-1}  )$. In particular,
\begin{align}
\sup_{k > \lvn} \xi_{d,k}^2 = &~ O_{d} ( d^{-\lvn -1}), \label{eq:Uppers} \\
\sup_{k > \lvN} \xi_{d,k}^2 = &~ O_{d} ( d^{-\lvN -1}), \label{eq:UpperS}
\end{align}
Furthermore, by noting that $\xi^2_{d,k} =  B(\S^{d-1} ; k)^{-1} \|\oproj_{k}\barsigma_{d } (\< \be,\, \cdot\, \>) \|^2_{L^2}$, conditions \eqref{eq:ass_sphere_b1}, \eqref{eq:ass_sphere_b1p} and \eqref{eq:ass_sphere_b2} can be rewritten
as follows in terms of the coefficients $(\xi_{d,k})_{k\ge 0}$:
\begin{align}
 \min_{k \le \lvn}\xi_{d,k}^2  = &~ \Omega_d(d^{-\lvn}),  \label{eq:AssXi1}\\
  \min_{k \le \lvN}\xi_{d,k}^2  = &~ \Omega_d(d^{-\lvN}),  \label{eq:AssXi1p}\\
  \sum_{k=2\max(\lvn,\lvN)+2}^\infty \xi_{d,k}^2  B(\S^{d-1} ; k) = &~ \Omega_d (1). \label{eq:AssXi2}
\end{align}

Denote $\{ \lambda_{d,j} \}_{j \ge 1}$ the eigenvalues $\{  \xi_{d,k}  \}_{k \ge 0}$ with their degeneracy in non increasing order of their absolute value. Set $\evN$ and $\evn$ to be the number of eigenvalues associated to spherical harmonics of degree less or equal to $\lvN$ and $\lvn$ respectively, i.e.,
\begin{equation}\label{eq:bound_m_M}
\evN = \sum_{k = 0}^{\lvN} B(\S^{d-1} ; k) = \Theta_d ( d^{\lvN} )\, , \qquad \evn = \sum_{k= 0}^{\lvn} B(\S^{d-1} ; k) = \Theta_d ( d^{\lvn})\, .
\end{equation}
Notice that Eqs.~\eqref{eq:Uppers} and \eqref{eq:AssXi1} imply that $(\lambda_{d,j})_{j \leq \evn}$ corresponds exactly to all the eigenvalues associated to spherical harmonics of degree less or equal to $\lvn$. Similarly Eqs.~\eqref{eq:UpperS} and \eqref{eq:AssXi1p} imply that $(\lambda_{d,j})_{j \leq \evN}$ corresponds exactly to all the eigenvalues associated to spherical harmonics of degree less or equal to $\lvN$.

Notice that the diagonal elements of the truncated kernels are given by (for any $\bx,\btheta \in \S^{d-1} (\sqrt{d})$)
\begin{equation}\label{eq:diag_kernels_sphere}
\begin{aligned}
H_{d,>\evn} ( \bx , \bx) =&~ \sum_{k=\lvn +1 }^\infty \xi_{d,k}^2  B(\S^{d-1} ; k) Q_k^{(d)} ( \< \bx , \bx \> ) = \sum_{k=\lvn +1}^\infty \xi_{d,k}^2  B(\S^{d-1} ; k)  = \| \oproj_{>\lvn} \barsigma_d \|_{L^2}^2 = \Tr ( \Hop_{d,>\evn} ), \\
U_{d,>\evN} ( \btheta , \btheta) =&~ \sum_{k=\lvN +1 }^\infty \xi_{d,k}^2  B(\S^{d-1} ; k) Q_k^{(d)} ( \< \btheta , \btheta \> ) = \sum_{k=\lvN +1}^\infty \xi_{d,k}^2  B(\S^{d-1} ; k)  = \| \oproj_{>\lvN} \barsigma_d \|_{L^2}^2 = \Tr ( \Uop_{d,>\evN} ), \\
\end{aligned}
\end{equation}
where we used that $Q^{(d)}_k (d) = 1$.

\noindent
{\bf Step 2. Checking the assumptions at level $\{ (N(d), \evN(d), n(d) , \evn(d) ) \}_{d \ge 1}$. }

We are now in position to verify the assumptions of Theorem \ref{thm:RFK_generalization}. Choose $u := u(d)$ to be the number of eigenvalues with absolute value $\Omega_d ( d^{-2 \max(\lvn,\lvN) - 2 + \delta} )$ for some $\delta>0$ that will be chosen small enough, see Eq.~\eqref{eq:proper_eig_sphere}. In particular, $(\lambda_{d,j})_{j \in [u]}$ contains all the eigenvalues associated to the spherical harmonics of degree less or equal to $\max (\lvN , \lvn)$, and none of the eigenvalues associated to spherical harmonics of degree $2\max (\lvN , \lvn)+2$ and bigger. We therefore must have $u \geq \max (\evN(d) , \evn(d))$.
%
%

Let us verify the conditions of $(N,\evN,n,\evn)$-FMCP in Assumption \ref{ass:FMPCP} with the sequence of integers $u(d)$:
\begin{itemize}
\item[$(a)$] The hypercontractivity of the space of polynomials of degree less or equal $2\max (\lvN, \lvn) +1$ is a consequence of a classical result due to Beckner \cite{beckner1992sobolev} (see Section \ref{app:hypercontractivity}).

\item[$(b)$] Let us  lower bound the right-hand side of Eq.~\eqref{eq:ConcentratioB}. We have
\[
\begin{aligned}
\sum_{j=u(d)+1}^{\infty} \lambda_{d,j}^2\geq &~ \sum_{k=2 \max(\lvn , \lvN)+2}^{\infty} \xi_{d,k}^2B(\S^{d-1};k) = \Omega_d (1), \\
\sum_{j=u(d)+1}^{\infty} \lambda_{d,j}^4 \leq &~ \Big\{ \sup_{j > u} \lambda_{d,j}^2  \Big\}\cdot \sum_{j=u(d)+1}^{\infty} \lambda_{d,j}^2 = O_d ( d^{-2 \max(\lvn,\lvN) - 2 + \delta} ) \cdot  \sum_{j=u(d)+1}^{\infty} \lambda_{d,j}^2.
\end{aligned}
\]
Hence,
    \begin{equation}\label{eq:proper_eig_sphere}
  \begin{aligned}
    \frac{\Big(\sum_{j=u(d)+1}^{\infty} \lambda_{d,j}^2\Big)^2}{\sum_{j=u(d)+1}^{\infty} \lambda_{d,j}^4} &=   \Omega_d (d^{2 \max(\lvn,\lvN) + 2 - \delta} ) \geq \max( n, N)^{2 + \delta },
  \end{aligned}
  \end{equation}
  for $\delta >0$ small enough, where we recall that $n \leq d^{\lvn +1 - \delta_0}$ and $N \leq d^{\lvN +1 - \delta_0}$ for some fixed $\delta_0 >0$.

\item[$(c)$] From Eq.~\eqref{eq:ass_sphere_b2} in Assumption \ref{ass:activation_RFK_sphere}, we only need to check that for $q$ such that 
\[
\min(n,N) \max(N,n)^{1/q - 1} \log(\max(N,n) ) = o_d (1),
\]
we have
  \begin{align*}
\E_{\bx,\btheta} \big[ [\proj_{> u } \barsigma_{d}] ( \<\bx, \btheta \> / \sqrt{d} )^{2q} \big]^{1/(2q)} =   O_d(1).
    \end{align*}
    Denote $S$ the set of eigenvalues $\lambda_{d,j}$, with $j >u$, associated to spherical harmonics of degree less of equal to $2 \max(\lvn , \lvN ) +1$. By triangular inequality, we have 
    \[
    \begin{aligned}
   &~ \E_{\bx,\btheta} \big[ [\proj_{> u } \barsigma_{d}] ( \<\bx, \btheta \> / \sqrt{d} )^{2q} \big]^{1/(2q)} \\
    \leq &~ \E_{\bx,\btheta} \big[ [\proj_{S} \barsigma_{d}] ( \<\bx, \btheta \> / \sqrt{d} )^{2q} \big]^{1/(2q)} + \E_{\bx,\btheta} \big[ [\oproj_{> 2 \max(\lvn , \lvN ) +1 } \barsigma_{d}] ( \<\bx, \btheta \> / \sqrt{d} )^{2q} \big]^{1/(2q)} = O_d(1),
   \end{aligned}
    \]
    where we used that $\proj_{S} \barsigma_{d}$ is a polynomial of degree less or equal to  $2\max (\lvN, \lvn) +1$ in each variable $\bx$ and $\btheta$ and satisfies the hypercontractivity property (see Lemma \ref{lem:hypercontractivity_basis_kernel}), i.e.,
    \[
    \begin{aligned}
    \E_{\bx,\btheta} \big[ [\proj_{S} \barsigma_{d}] ( \<\bx, \btheta \> / \sqrt{d} )^{2q} \big] = &~  O_d(1) \cdot \E_{\bx} \big[ H_{d,S} (\bx , \bx)^q \big] = O_d(1) \cdot \Tr ( \Hop_{d,S} )^q = O_d(1),
    \end{aligned}
    \]
    while the bound on $\oproj_{> 2 \max(\lvn , \lvN ) +1 } \barsigma_{d}$ follows from Assumption \ref{ass:activation_RFK_sphere}.$(a)$ and Lemma \ref{lem:assumption2a_spherical} stated below.

\item[$(d)$] This is automatically verified because the diagonal elements are constant in this case (Eq.~\eqref{eq:diag_kernels_sphere}).
\end{itemize}
Next, we check Assumption \ref{ass:spectral_gap} at level $(N,\evN,n,\evn)$. Consider the overparametrized case $N(d) \ge n(d)$, and therefore $\evN \ge \evn$. The underparametrized case $N(d)\le n(d)$ is treated analogously.
\begin{itemize} 
\item[$(i)$] The eigenvalue sums in Eq.~\eqref{eq:NumberOfSamples} can be estimated as follows
\begin{align}
  \frac{1}{\lambda_{d,\evn(d)}^2}\sum_{k=\evn(d)+1}^{\infty}\lambda_{d,k}^2&=\frac{1}{\xi_{d,\lvn}^2}\sum_{k=\lvn+1}^{\infty}\xi^2_{d,k}B(\S^{d-1};k)  =
  O_d(d^{\lvn})\, ,\label{eq:CheckingSphere1}\\
  \frac{1}{\lambda_{d,\evn(d)+1}^2}\sum_{k=\evn(d)+1}^{\infty}\lambda_{d,k}^2&=
                                                                               \frac{1}{\xi_{d,\lvn+1}^2}\sum_{k=\lvn+1}^{\infty}\xi^2_{d,k}B(\S^{d-1};k)  \ge B(\S^{d-1};\lvn+1)
                                                                               =\Omega_d(d^{\lvn+1})\, .\label{eq:CheckingSphere2}
\end{align}
The last equality in \eqref{eq:CheckingSphere1}  follows from Eq.~\eqref{eq:BoundL2Sigma} and the assumption \eqref{eq:AssXi1},
Hence condition \eqref{eq:NumberOfSamples} in Assumption  \ref{ass:spectral_gap} is satisfied since, by the statement
of Theorem \ref{prop:RFK_sphere}, we assume $d^{\lvn+\delta} \leq n \le d^{\lvn+1-\delta}$. Furthermore, by Eq.~\eqref{eq:bound_m_M}, we have $\evn \leq n^{1 - \delta}$ for some $\delta>0$ chosen small enough.
\item[$(ii)$] The eigenvalue sum in Eq.~\eqref{eq:NumberOfFeatures} is 
  \begin{align}
    \frac{1}{\lambda_{d,\evN(d)+1}^2}\sum_{k=\evN(d)+1}^{\infty}\lambda_{d,k}^2=
    \frac{1}{\xi_{d,\lvN+1}^2}\sum_{k=\lvN+1}^{\infty}\xi^2_{d,k}B(\S^{d-1};k)  \ge B(\S^{d-1};\lvN+1) =\Omega_d(d^{\lvN+1})\, .
  \end{align}
  Hence  condition \eqref{eq:NumberOfFeatures} in Assumption  \ref{ass:spectral_gap} is satisfied since,  by the statement
  of Theorem \ref{prop:RFK_sphere}, $N\le d^{\lvN+1-\delta_0}$. By Eq.~\eqref{eq:bound_m_M}, we have $\evN \leq N^{1 - \delta}$ for some $\delta>0$ chosen small enough.
\end{itemize}

\end{proof}

\begin{lemma}\label{lem:assumption2a_spherical}
Consider $m,\ell$ two fixed integers. Assume $| \barsigma_d ( x) | \leq c_0 \exp ( c_1 x^2 /(4m) )$ with $c_0 >0$ and $c_1 < 1$. Then
\[
\E_{x_1 \sim \tau_d^1 } \big[ \barsigma_{d,>\ell} ( x_1 )^{2m} \big] = O_{d} (1),
\]
where $\tau_d^1$ is the marginal distribution of $\< \be ,\bx \>$ with $\| \be \|_2 = 1$ and $\bx \sim \Unif ( \S^{d-1} (\sqrt{d}))$, and we denoted $ \barsigma_{d,>\ell} = \oproj_{>\ell} \barsigma_d$.
\end{lemma} 

\begin{proof}[Proof of Lemma \ref{lem:assumption2a_spherical}]
Recall that
\begin{equation}\label{eq:decompo_sigma_u}
 \barsigma_{d,>\ell}   (x) = \barsigma_d (x) - \sum_{k = 0}^\ell \xi_{d,k} (\sigma) B(\S^{d-1} ; k) Q_k^{(d)} (\sqrt{d} x),
\end{equation}
where $\xi_{d,k} (\sigma)^2 B(d,k) \leq \|  \barsigma_d  \|_{L^2}^2 \leq C$ for some constant $C>0$ (using $| \barsigma_d ( x) | \leq c_0 \exp ( c_1 x^2 /(4m) )$) and $\sqrt{B(\S^{d-1} ; k)} Q_k^{(d)}$ is a degree-$k$ polynomial that converges to the Hermite polynomial $\He_k / \sqrt{k!}$ (see Section \ref{sec:Hermite}). Therefore, $ \barsigma_{d,>\ell}$ is equal to $ \barsigma_d$ plus a polynomial of degree $\ell$ with bounded coefficients. In particular, from the assumption $| \barsigma_d ( x) | \leq c_0 \exp ( c_1 x^2 /(4m) )$, we deduce there exists $c_0 ' >0$ such that $|  \barsigma_{d,>\ell}  (x )| \leq c_0' \exp ( c_1 x^2 /(4m) )$, whence:
\begin{equation}\label{eq:domination_sigma_2m}
\begin{aligned}
 \Big\vert \E_{x_1 } [  \barsigma_{d,>\ell}  ( x_1 )^{2m} ] \Big\vert 
\leq  (c_0 ')^{2m} \E_{x_1 } [ \exp ( c_1 x_1^2 /2) ].
\end{aligned}
\end{equation}
Furthermore, recall that $ \tau_d^1 (\de x ) = C_d (1 - x^2/d )^{(d-3)/2} \ones_{x \in [-\sqrt{d},\sqrt{d} ]} \de x \leq C \exp (-x^2/2) \de x$. We can therefore upper bound the right hand side of Eq.~\eqref{eq:domination_sigma_2m} and use dominated convergence, which concludes the proof.
\end{proof}

\subsection{On the hypercube}\label{sec:hypercube_case}

The proof for the hypercube $\Cube$ follows from the same proof as for the sphere.
We refer tp \cite{o2014analysis} for background on Fourier analysis on $\Cube$, and Section \ref{sec:functions_hypercube} for
notations that make the analogy with the sphere transparent. In particular, an analogous of Lemma \ref{lem:assumption2a_spherical}
follows by noticing that the law
$\< \ones, \bx \> /\sqrt{d}$ is a  standardized binomial, which can be in terms of the standard normal distribution, times polynomial factors.
The only difference comes from the degeneracy 
\[
B(\Cube ; k) = {{d}\choose{\ell}}.
\]
Hence Assumption \ref{ass:activation_RFK_sphere}.$(a)$ only implies $\xi^2_{d,d-\ell} = O_d ( d^{-\ell})$ for the last coefficients, which is the reason for the further requirement Assumption \ref{ass:activation_RFK_sphere}.$(c)$.

Let us check that Assumption \ref{ass:activation_RFK_sphere}.$(c)$ holds for a class of smooth activation functions.
We believe that indeed this assumption holds much more generally, but we leave such generalizations to future work.
\begin{lemma}\label{lem:bound_eigenvalues_last_Fourier}
 Consider $\ell$ a fixed integer. Assume there exist constants $c_0>0$ and $c_1 <1$ such that $| \barsigma^{(\ell)} (x) | \le c_0 \exp (c_1 x^2 / 4)$ for all $x\in\reals$. 
Then, we have
\[
\max_{k\le \ell} \xi_{d,d-k} (\barsigma)^2 = O_d ( d^{-\ell} ),
\]
where $\xi_{d,d-k}(\barsigma) = \< \barsigma(\< \be,\, \cdot\,\>), Q_{d-k}(\sqrt{d} \< \be, \cdot\>) \>_{L^2(\Cube)}$, and $Q_k$
is the $k$-th hypercubic Gegenbauer polynomial (see Appendix \ref{sec:functions_hypercube}).
\end{lemma}

\begin{proof}[Proof of Lemma \ref{lem:bound_eigenvalues_last_Fourier}]
By the mean value theorem, we have for any $k \leq \ell$,
\[
\begin{aligned}
 \xi_{d-k,d} (\barsigma) = & \E_{\bx} [ \barsigma ( \< \ones, \bx\> /\sqrt{d} ) Q_{d -k }^{(d)} ( \< \ones , \bx \> ) ] \\
 = & \E_{\bx} \Big[ x_1 \cdots x_{d-k} \barsigma \Big( \frac{x_1 + \ldots + x_d}{\sqrt{d}} \Big) \Big] \\
 = & \frac{1}{2} \E_{x_2, \ldots , x_3} \Big[  x_2 \ldots x_{d-k} \Big( \barsigma \Big( \frac{1+ x_2 + \ldots + x_d}{\sqrt{d}} \Big)  - \barsigma \Big( \frac{-1 + x_2 + \ldots + x_d}{\sqrt{d}} \Big)  \Big) \Big] \\
 = & \frac{1}{\sqrt{d} } \E_{x_2 , \ldots ,  x_d} \Big[  x_2 \ldots x_{d-k} \barsigma^{(1)} ( \zeta^1 ( x_2 , \ldots , x_d) ) \Big],
 \end{aligned}
\]
where on the third line we integrated over the first coordinate $x_1$ and on the last line $|  \zeta^1 ( x_2 , \ldots , x_d) ) -  (x_2 + \ldots  + x_d )/\sqrt{d} | \leq 1/\sqrt{d}$. By iterating this computation $\ell$ times, we get
\[
\begin{aligned}
  \xi_{d-k,d} (\barsigma) 
 = & \frac{1}{d^{\ell/2}}  \E_{x_{\ell+1} , \ldots ,  x_d} \Big[  x_{\ell+1} \ldots x_{d-k} \barsigma^{(\ell)} ( \zeta^\ell ( x_{\ell+1} , \ldots , x_d) ) \Big] ,
 \end{aligned}
\]
where $| \zeta^{\ell} ( x_{\ell+1} , \ldots , x_d)  - (x_{\ell+1} + \ldots  + x_d )/\sqrt{d} | \leq \ell /\sqrt{d}$. Hence,
\[
\begin{aligned}
 | \xi_{d-k,d} (\barsigma) |
 \leq & \frac{1}{d^{\ell/2}}  \E_{x_{\ell+1} , \ldots ,  x_d} \Big[  | \barsigma^{(\ell)} ( \zeta^\ell ( x_{\ell+1} , \ldots , x_d) ) |  \Big] \\
 \leq &  \frac{1}{d^{\ell/2}}  \E_{X = (x_{\ell+1} + \ldots + x_d) /\sqrt{d} } [ c_0 \exp( c_1 X^2/ 2 + c_1 \ell^2/d ) ] \\
 =& O_{d}( d^{-\ell/2} ),
 \end{aligned}
\]
where we used that $X$ converges weakly to the standard normal distribution and dominated convergence.
\end{proof}

\section{Technical background}\label{sec:technical_background}

\subsection{Functions on the sphere}\label{sec:functions_sphere}

\subsubsection{Functional spaces over the sphere}

For $d \ge 3$, we let $\S^{d-1}(r) = \{\bx \in \R^{d}: \| \bx \|_2 = r\}$ denote the sphere with radius $r$ in $\reals^d$.
We will mostly work with the sphere of radius $\sqrt d$, $\S^{d-1}(\sqrt{d})$ and will denote by $\tau_{d}$  the uniform probability measure on $\S^{d-1}(\sqrt d)$. 
All functions in this section are assumed to be elements of $ L^2(\S^{d-1}(\sqrt d) ,\tau_d)$, with scalar product and norm denoted as $\<\,\cdot\,,\,\cdot\,\>_{L^2}$
and $\|\,\cdot\,\|_{L^2}$:
\begin{align}
\<f,g\>_{L^2} \equiv \int_{\S^{d-1}(\sqrt d)} f(\bx) \, g(\bx)\, \tau_d(\de \bx)\,.
\end{align}

For $\ell\in\integers_{\ge 0}$, let $\tilde{V}_{d,\ell}$ be the space of homogeneous harmonic polynomials of degree $\ell$ on $\reals^d$ (i.e. homogeneous
polynomials $q(\bx)$ satisfying $\Delta q(\bx) = 0$), and denote by $V_{d,\ell}$ the linear space of functions obtained by restricting the polynomials in $\tilde{V}_{d,\ell}$
to $\S^{d-1}(\sqrt d)$. With these definitions, we have the following orthogonal decomposition
\begin{align}
L^2(\S^{d-1}(\sqrt d) ,\tau_d) = \bigoplus_{\ell=0}^{\infty} V_{d,\ell}\, . \label{eq:SpinDecomposition}
\end{align}
The dimension of each subspace is given by
\begin{align}
\dim(V_{d,\ell}) = B(\S^{d-1}; \ell) = \frac{2 \ell + d - 2}{d - 2} { \ell + d - 3 \choose \ell} \, .
\end{align}
For each $\ell\in \integers_{\ge 0}$, the spherical harmonics $\{ Y_{\ell, j}^{(d)}\}_{1\le j \le B(\S^{d-1}; \ell)}$ form an orthonormal basis of $V_{d,\ell}$:
\[
\<Y^{(d)}_{ki}, Y^{(d)}_{sj}\>_{L^2} = \delta_{ij} \delta_{ks}.
\]
Note that our convention is different from the more standard one, that defines the spherical harmonics as functions on $\S^{d-1}(1)$.
It is immediate to pass from one convention to the other by a simple scaling. We will drop the superscript $d$ and write $Y_{\ell, j} = Y_{\ell, j}^{(d)}$ whenever clear from the context.

We denote by $\oproj_k$  the orthogonal projections to $V_{d,k}$ in $L^2(\S^{d-1}(\sqrt d),\tau_d)$. This can be written in terms of spherical harmonics as
\begin{align}
\oproj_k f(\bx) \equiv& \sum_{l=1}^{B(\S^{d-1}; k)} \< f, Y_{kl}\>_{L^2} Y_{kl}(\bx). 
\end{align}
We also define
$\oproj_{\le \ell}\equiv \sum_{k =0}^\ell \oproj_k$, $\oproj_{>\ell} \equiv \id -\oproj_{\le \ell} = \sum_{k =\ell+1}^\infty \oproj_k$,
and $\oproj_{<\ell}\equiv \oproj_{\le \ell-1}$, $\oproj_{\ge \ell}\equiv \oproj_{>\ell-1}$.

\subsubsection{Gegenbauer polynomials}
\label{sec:Gegenbauer}

The $\ell$-th Gegenbauer polynomial $Q_\ell^{(d)}$ is a polynomial of degree $\ell$. Consistently
with our convention for spherical harmonics, we view $Q_\ell^{(d)}$ as a function $Q_{\ell}^{(d)}: [-d,d]\to \reals$. The set $\{ Q_\ell^{(d)}\}_{\ell\ge 0}$
forms an orthogonal basis on $L^2([-d,d],\tilde\tau^1_{d})$, where $\tilde\tau^1_{d}$ is the distribution of $\sqrt{d}\<\bx,\be_1\>$ when $\bx\sim \tau_d$,
satisfying the normalization condition:
\begin{align}
\< Q^{(d)}_k(\sqrt{d}\< \be_1, \cdot\>), Q^{(d)}_j(\sqrt{d}\< \be_1, \cdot\>) \>_{L^2(\S^{d-1}(\sqrt d))} = \frac{1}{B(\S^{d-1};k)}\, \delta_{jk} \, .  \label{eq:GegenbauerNormalization}
\end{align}
In particular, these polynomials are normalized so that  $Q_\ell^{(d)}(d) = 1$. 
As above, we will omit the superscript $(d)$ in $Q_\ell^{(d)}$ when clear from the context.

Gegenbauer polynomials are directly related to spherical harmonics as follows. Fix $\bv\in\S^{d-1}(\sqrt{d})$ and 
consider the subspace of  $V_{\ell}$ formed by all functions that are invariant under rotations in $\reals^d$ that keep $\bv$ unchanged.
It is not hard to see that this subspace has dimension one, and coincides with the span of the function $Q_{\ell}^{(d)}(\<\bv,\,\cdot\,\>)$.

We will use the following properties of Gegenbauer polynomials
\begin{enumerate}
\item For $\bx, \by \in \S^{d-1}(\sqrt d)$
\begin{align}
\< Q_j^{(d)}(\< \bx, \cdot\>), Q_k^{(d)}(\< \by, \cdot\>) \>_{L^2} = \frac{1}{B(\S^{d-1}; k)}\delta_{jk}  Q_k^{(d)}(\< \bx, \by\>).  \label{eq:ProductGegenbauer}
\end{align}
\item For $\bx, \by \in \S^{d-1}(\sqrt d)$
\begin{align}
Q_k^{(d)}(\< \bx, \by\> ) = \frac{1}{B(\S^{d-1}; k)} \sum_{i =1}^{ B(\S^{d-1}; k)} Y_{ki}^{(d)}(\bx) Y_{ki}^{(d)}(\by). \label{eq:GegenbauerHarmonics}
\end{align}
\end{enumerate}
These properties imply that ---up to a constant--- $Q_k^{(d)}(\< \bx, \by\> )$ is a representation of the projector onto 
the subspace of degree -$k$ spherical harmonics
\begin{align}
(\oproj_k f)(\bx) = B(\S^{d-1}; k) \int_{\S^{d-1}(\sqrt{d})} \, Q_k^{(d)}(\< \bx, \by\> )\,  f(\by)\, \tau_d(\de\by)\, .\label{eq:ProjectorGegenbauer}
\end{align}
For a function $\barsigma \in L^2([-\sqrt d, \sqrt d], \tau^1_{d})$ (where $\tau^1_{d}$ is the distribution of $\< \be_!, \bx \> $ when $\bx \sim_{iid} \Unif(\S^{d-1}(\sqrt d))$), denoting its spherical harmonics coefficients $\xi_{d, k}(\barsigma)$ to be 
\begin{align}\label{eqn:technical_lambda_sigma}
\xi_{d, k}(\barsigma) = \int_{[-\sqrt d , \sqrt d]} \barsigma(x) Q_k^{(d)}(\sqrt d x) \tau^1_{d}(\de x),
\end{align}
then we have the following equation holds in $L^2([-\sqrt d, \sqrt d],\tau^1_{d-1})$ sense
\[
\barsigma(x) = \sum_{k = 0}^\infty \xi_{d, k}(\barsigma) B(\S^{d-1}; k) Q_k^{(d)}(\sqrt d x). 
\]

To  any rotationally invariant kernel $H_d(\bx_1, \bx_2) = h_d(\< \bx_1, \bx_2\> / d)$,
with $h_d(\sqrt{d}\, \cdot \, ) \in L^2([-\sqrt{d},\sqrt{d}],\tau^1_{d})$,
we can associate a self adjoint operator $\cuH_d:L^2(\S^{d-1}(\sqrt{d}))\to L^2(\S^{d-1}(\sqrt{d}))$
via
\begin{align}
\cuH_df(\bx) \equiv \int_{\S^{d-1}(\sqrt{d})} h_d(\<\bx,\bx_1\>/d)\, f(\bx_1) \, \tau_d(\de \bx_1)\, .
\end{align}
By rotational invariance,   the space $V_{k}$ of homogeneous polynomials of degree $k$ is an eigenspace of
$\cuH_d$, and we will denote the corresponding eigenvalue by $\xi_{d,k}(h_d)$. In other words
$\cuH_df(\bx) \equiv \sum_{k=0}^{\infty} \xi_{d,k}(h_d) \oproj_{k}f$.   The eigenvalues can be computed via
\begin{align}
  \xi_{d, k}(h_d) = \int_{[-\sqrt d , \sqrt d]} h_d\big(x/\sqrt{d}\big) Q_k^{(d)}(\sqrt d x) \tau^1_{d-1}(\de x)\, .
\end{align}

\subsubsection{Hermite polynomials}
\label{sec:Hermite}

The Hermite polynomials $\{\bbHe_k\}_{k\ge 0}$ form an orthogonal basis of $L^2(\reals,\gamma)$, where $\gamma(\de x) = e^{-x^2/2}\de x/\sqrt{2\pi}$ 
is the standard Gaussian measure, and $\bbHe_k$ has degree $k$. We will follow the classical normalization (here and below, expectation is with respect to
$G\sim\normal(0,1)$):
\begin{align}
\E\big\{\bbHe_j(G) \,\bbHe_k(G)\big\} = k!\, \delta_{jk}\, .
\end{align}
As a consequence, for any function $g\in L^2(\reals,\gamma)$, we have the decomposition
\begin{align}\label{eqn:sigma_He_decomposition}
g(x) = \sum_{k=0}^{\infty}\frac{\mu_k(g)}{k!}\, \bbHe_k(x)\, ,\;\;\;\;\;\; \mu_k(g) \equiv \E\big\{g(G)\, \bbHe_k(G)\}\, .
\end{align}

The Hermite polynomials can be obtained as high-dimensional limits of the Gegenbauer polynomials introduced in the previous section. Indeed, the Gegenbauer polynomials (up to a $\sqrt d$ scaling in domain) are constructed by Gram-Schmidt orthogonalization of the monomials $\{x^k\}_{k\ge 0}$ with respect to the measure 
$\tilde\tau^1_{d}$, while Hermite polynomial are obtained by Gram-Schmidt orthogonalization with respect to $\gamma$. Since $\tilde\tau^1_{d}\Rightarrow \gamma$
(here $\Rightarrow$ denotes weak convergence),
it is immediate to show that, for any fixed integer $k$, 
\begin{align}
\lim_{d \to \infty} \Coeff\{ Q_k^{(d)}( \sqrt d x) \, B(\S^{d-1}; k)^{1/2} \} = \Coeff\left\{ \frac{1}{(k!)^{1/2}}\,\bbHe_k(x) \right\}\, .\label{eq:Gegen-to-Hermite}
\end{align}
Here and below, for $P$ a polynomial, $\Coeff\{ P(x) \}$ is  the vector of the coefficients of $P$. As a consequence,
for any fixed integer $k$, we have
\begin{align}\label{eqn:mu_lambda_relationship}
\mu_k(\barsigma) = \lim_{d \to \infty} \xi_{d,k}(\barsigma) (B(\S^{d-1}; k)k!)^{1/2}, 
\end{align}
where $\mu_k(\barsigma)$ and $\xi_{d,k}(\barsigma)$ are given in Eq. (\ref{eqn:sigma_He_decomposition}) and (\ref{eqn:technical_lambda_sigma}). 

\subsection{Functions on the hypercube}\label{sec:functions_hypercube}

Fourier analysis on the hypercube is a well studied subject \cite{o2014analysis}. The purpose of this section is to introduce
some notations that  make the correspondence with proofs on the sphere straightforward.
For convenience, we will adopt the same notations as for their spherical case. 

\subsubsection{Fourier basis}

Denote $\Cube = \{ -1 , +1 \}^d $ the hypercube in $d$ dimension. Let us denote $\tau_{d}$ to be the uniform probability measure on $\Cube$. All the functions will be assumed to be elements of $L^2 (\Cube, \tau_d)$ (which contains all the bounded functions $f : \Cube \to \R$), with scalar product and norm denoted as $\< \cdot , \cdot \>_{L^2}$ and $\| \cdot \|_{L^2}$:
\[
\< f , g \>_{L^2} \equiv \int_{\Cube} f(\bx) g(\bx) \tau_d ( \de \bx) = \frac{1}{2^n} \sum_{\bx \in \Cube} f(\bx) g(\bx).
\]
Notice that $L^2 ( \Cube, \tau_d)$ is a $2^n$ dimensional linear space. By analogy with the spherical case we decompose $L^2 ( \Cube, \tau_d)$ as a direct sum of $d+1$ linear spaces obtained from polynomials of degree $\ell = 0, \ldots , d$
\[
L^2 ( \Cube, \tau_d) = \bigoplus_{\ell = 0}^d V_{d,\ell}.
\]

For each $\ell \in \{ 0 , \ldots , d \}$, consider the Fourier basis $\{ Y_{\ell, S}^{(d)} \}_{S \subseteq [d], |S | =\ell}$ of degree $\ell$, where for a set $S \subseteq [d]$, the basis is given by
\[
Y_{\ell, S}^{(d)} (\bx) \equiv x^S \equiv \prod_{i \in S} x_i.
\]
It is easy to verify that (notice that $x_i^k = x_i$ if $k$ is odd and  $x_i^k = 1$ if $k$ is even)
\[
\< Y_{\ell, S}^{(d)} , Y_{k, S'}^{(d)} \>_{L^2} = \E[ x^{S} \times x^{S'} ] = \delta_{\ell,k} \delta_{S,S'}. 
\]
Hence $\{ Y_{\ell, S}^{(d)} \}_{S \subseteq [d], |S | =\ell}$ form an orthonormal  basis of $V_{d,\ell}$ and 
\[
\dim (V_{d,\ell} ) = B(\Cube;\ell) = {{d}\choose{\ell}}.
\]
As above, we will omit the superscript $(d)$ in $Y_{\ell, S}^{(d)}$ when clear from the context.

\subsubsection{Hypercubic Gegenbauer}

We consider the following family of polynomials $\{ Q^{(d)}_\ell \}_{\ell = 0 , \ldots, d}$ that we will call hypercubic Gegenbauer, defined as
\[
Q^{(d)}_\ell (  \< \bx , \by \> ) = \frac{1}{B(\Cube; \ell)} \sum_{S \subseteq [d], |S| = \ell} Y_{\ell,S}^{(d)} ( \bx ) Y_{\ell,S}^{(d)} ( \by ). 
\]
Notice that the right hand side only depends on $ \< \bx , \by \>$ and therefore these polynomials are uniquely defined. In particular,
\[
\< Q_\ell^{(d)} ( \< \ones , \cdot \> ) , Q_k^{(d)} ( \< \ones , \cdot \> ) \>_{L^2} = \frac{1}{B(\Cube;k)} \delta_{\ell k}.
\]
Hence $\{ Q^{(d)}_\ell \}_{\ell = 0 , \ldots, d}$ form an orthogonal basis of $L^2 ( \{ -d , -d+2 , \ldots , d-2 ,d\}, \Tilde \tau_d^1 )$ where $\Tilde \tau_d^1$ is the distribution of $\< \ones , \bx \>$ when $\bx \sim \tau_d$, i.e., $\Tilde \tau_d^1 \sim 2 \text{Bin}(d, 1/2) - d/2$.

We have
\[
\< Q_\ell^{(d)} ( \< \bx , \cdot \> ) , Q_k^{(d)} ( \< \by , \cdot \> ) \>_{L^2} =   \frac{1}{B(\Cube;k)} Q_{k} ( \< \bx , \by \> )\delta_{\ell k} .
\]
For a function $\barsigma ( \cdot / \sqrt{d} ) \in L^2 ( \{ -d , -d+2 , \ldots , d-2 ,d\}, \Tilde \tau_d^1 )$, denote its hypercubic Gegenbauer coefficients $\xi_{d,k} ( \barsigma)$ to be
\[
\xi_{d,k} (\barsigma ) = \int_{\{-d, -d+2 , \ldots , d-2 , d\} } \barsigma(x / \sqrt{d} ) Q_k^{(d)} ( x) \Tilde \tau_d^1 (\de x).
\]

Notice that by weak convergence of $\< \ones, \bx \> / \sqrt{d}$ to the normal distribution, we have also convergence of the (rescaled) hypercubic Gegenbauer polynomials to the Hermite polynomials, i.e., for any fixed $k$, we have
\begin{align}
\lim_{d \to \infty} \Coeff\{ Q_k^{(d)}( \sqrt d x) \, B(\Cube; k)^{1/2} \} = \Coeff\left\{ \frac{1}{(k!)^{1/2}}\,\bbHe_k(x) \right\}\, .\label{eq:Hyper-Gegen-to-Hermite}
\end{align}

\subsection{Hypercontractivity of Gaussian measure and uniform distributions on the sphere and the hypercube}
\label{app:hypercontractivity}

By Holder's inequality, we have $\| f \|_{L^p} \le \| f \|_{L^q}$ for any $f$ and any $p \le q$. The reverse inequality does not hold in general, even up to a constant. However, for some measures, the reverse inequality will hold for some sufficiently nice functions. These measures satisfy the celebrated hypercontractivity properties \cite{gross1975logarithmic, bonami1970etude, beckner1975inequalities, beckner1992sobolev}. 

\begin{lemma}[Hypercube hypercontractivity \cite{beckner1975inequalities}] For any $\ell = \{ 0 , \ldots , d \}$ and $f_d \in L^2 ( \Cube)$ to be a degree $\ell$ polynomial, then for any integer $q\ge 2$, we have
\[
\| f_d \|_{L^q ( \Cube)}^2 \leq (q-1)^\ell \cdot \| f_d \|^2_{L^2 (\Cube)}.
\]
\end{lemma}

\begin{lemma}[Spherical hypercontractivity \cite{beckner1992sobolev}]\label{lem:hypercontractivity_sphere}
For any $\ell \in \N$ and $f_d \in L^2(\S^{d-1})$ to be a degree $\ell$ polynomial, for any $q \ge 2$, we have 
\[
\| f_d \|_{L^q(\S^{d-1})}^2 \le (q - 1)^\ell \cdot \| f_d \|_{L^2(\S^{d-1})}^2. 
\]
\end{lemma}

\begin{lemma}[Gaussian hypercontractivity]\label{lem:hypercontractivity_Gaussian}
For any $\ell \in \N$ and $f \in L^2(\R, \gamma)$ to be a degree $\ell$ polynomial on $\R$, where $\gamma$ is the standard Gaussian distribution. Then for any $q \ge 2$, we have 
\[
\| f \|_{L^q(\R, \gamma)}^2 \le (q - 1)^{\ell} \cdot \| f \|_{L^2(\R, \gamma)}^2. 
\]
\end{lemma}

The Gaussian hypercontractivity is a direct consequence of hypercube hypercontractivity.

\end{document}